\documentclass[reqno]{amsart}  
\usepackage[OT2,T1]{fontenc}
\theoremstyle{plain}
\newtheorem{theorem}{Theorem}[section]
\newtheorem{lemma}[theorem]{Lemma}
\newtheorem{proposition}[theorem]{Proposition}

\newtheorem{definition}[theorem]{Definition}
\newtheorem{corollary}[theorem]{Corollary}

\theoremstyle{remark}

\newtheorem{remark}[theorem]{Remark}
\newtheorem{example}[theorem]{Example}

\numberwithin{equation}{section}

\theoremstyle{definition}

\newcommand{\row}{\operatorname{row}}
\newcommand{\col}{\operatorname{col}}

\newcommand{\cyr}{\fontencoding{OT2}\fontfamily{wncyr}\selectfont}
\newcommand{\ba}{{\mathbf a}}

\newcommand{\cB}{{\mathcal B}}

\newcommand{\cD}{{\mathcal D}}
\newcommand{\cE}{{\mathcal E}}
\newcommand{\cF}{{\mathcal F}}
\newcommand{\cH}{{\mathcal H}}
\newcommand{\cK}{{\mathcal K}}
\newcommand{\cL}{{\mathcal L}}
\newcommand{\cM}{{\mathcal M}}
\newcommand{\cR}{{\mathcal R}}
\newcommand{\cS}{{\mathcal S}}

\newcommand{\cT}{{\mathcal T}}
\newcommand{\cU}{{\mathcal U}}
\newcommand{\cV}{{\mathcal V}}
\newcommand{\cW}{{\mathcal W}}
\newcommand{\cX}{{\mathcal X}}
\newcommand{\cY}{{\mathcal Y}}

\newcommand{\bbC}{{\mathbb C}}
\newcommand{\bbT}{{\mathbb T}}
\newcommand{\bbD}{{\mathbb D}}

\newcommand{\bbZ}{{\mathbb Z}}

\newcommand{\be}{{\mathbf e}}

\newcommand{\sbm}[1]{\left[\begin{smallmatrix} #1
        \end{smallmatrix}\right]}

\newcommand{\pos}{\ge}
\newcommand{\Sz}{\text{\rm Sz}}
\newcommand{\diag}{\text{\rm diag}}
\newcommand{\slim}{\operatorname{strong\ limit}}

\begin{document}
\title[Scattering and formal reproducing kernel Hilbert spaces]{Scattering systems with several 
evolutions and formal reproducing kernel Hilbert spaces}
\author[J.A.~Ball]{Joseph A. Ball}
\address{Department of Mathematics \\ Virginia Tech \\
Blacksburg, Virginia 24061} \email{ball@math.vt.edu}
\author[D.S.~Kaliuzhnyi-Verbovetskyi]{Dmitry S.
Kaliuzhnyi-Verbovetskyi}
\address{Department of Mathematics \\
Drexel University\\
3141 Chestnut St.\\
  Philadelphia, PA, 19104}
\email{dmitryk@math.drexel.edu}
\author[C.~Sadosky]{\fbox{Cora Sadosky}}
\author[V.~Vinnikov]{Victor Vinnikov}
\address{Department of  Mathematics  \\ Ben Gurion
University of the Negev \\
 84105 Beer-Sheva, Israel}
\email{vinnikov@math.bgu.ac.il}

\thanks{JB, CS, and VV were partially supported by US--Israel  BSF
grant 2002414. JB, DK-V, and VV were partially supported by
US--Israel  BSF grant 2010432. DK-V was partially supported by NSF
grant DMS-0901628. }


\begin{abstract}
A Schur-class function in $d$ variables is defined to be an
analytic contractive-operator valued function on the unit
polydisk.  Such a function is said to be in the Schur--Agler class
if it is contractive when evaluated on any commutative $d$-tuple
of strict contractions on a Hilbert space. It is known that the
Schur--Agler class is a strictly proper subclass of the Schur
class if the number of variables $d$ is more than two. The
Schur--Agler class is also characterized as those functions
arising as the transfer function of a certain type
(Givone--Roesser) of conservative multidimensional linear system.
Previous work of the authors identified the Schur--Agler class as
those Schur-class functions which arise as the scattering matrix
for a certain type of (not necessarily minimal) Lax--Phillips
multievolution scattering system having some additional geometric
structure. The present paper links this additional geometric
scattering structure directly with a known reproducing-kernel
characterization of the Schur--Agler class.
We use extensively the technique of formal reproducing kernel Hilbert spaces
that was previously introduced by the authors and that allows us to manipulate 
formal power series
in several commuting variables and their inverses (e.g., Fourier series of 
elements of $L^2$ on a torus)
in the same way as one manipulates analytic functions in the usual setting 
of reproducing kernel Hilbert spaces.
\end{abstract}

\maketitle

\noindent Cora Sadosky passed away in December 2010. We lost a
collaborator of many years, and a wonderful friend. Like many
others, we miss her.

\hfill{Joseph A. Ball, Dmitry S. Kaliuzhnyi-Verbovetskyi, and
Victor Vinnikov}

\section{Introduction} \label{S:intro}

Let $\cE$ and $\cE_{*}$ be two Hilbert spaces and let $\cS(\cE,
\cE_{*})$ denote the \emph{Schur class} of holomorphic functions
on the unit disk $\bbD$ in the complex plane with values equal to
contraction operators between $\cE$ and $\cE_{*}$.  The Schur
class was originally introduced by Schur for the scalar case $\cE
= \cE_{*} = \bbC$ and is the natural class of functions for which
one formulates interpolation problems of Nevanlinna--Pick and
Carath\'eodory--Fej\'er type.  Recently there has been a
resurgence of interest in matrix- and operator-valued versions of
this class due to connections with manifold applications in
operator theory, mathematical physics and linear system theory.
Specifically, Schur-class functions arise (1) as the {\em
characteristic function} of a Hilbert-space contraction operator
(see \cite{NF} as well as \cite{dBR1, dBR2, LivMR8:588d,
LivMR20:7221}), (2) as the  scattering matrix of a (discrete-time)
Lax--Phillips scattering system (see \cite{LP} as well as
\cite{LivMR19:221d}), and (3) as the transfer function of
discrete-time conservative linear system (see \cite{AV, Staffans}
as well as \cite{LivMR22:11818, LivMR38:1922}).  Over the last few
decades the connections between these various theories have come
be better understood (see e.g.~\cite{AA, Helton-scat, BC}).  Let
us only mention the key structure theorem for Schur-class
functions.

\begin{theorem}  \label{T:1DSchur}
    Let $S \colon \bbD \to \cL(\cE, \cE_{*})$ be an operator-valued
    function defined on the unit disk $\bbD$.  Then the following
    conditions are equivalent:
    \begin{enumerate}
    \item $S$ is in the Schur class $\cS(\cE, \cE_{*})$.

    \item The kernel $K_{S}(z,w) = (I - S(z) S(w)^{*})/(1 - z
    \overline{w})$ is a positive kernel, i.e. has a factorization
   $$
   \frac{ I - S(z) S(w)^{*}}{1 - z \overline{w}} = H(z) H(w)^{*}
   $$
   for some function $H \colon \bbD \to \cL(\cH', \cE_{*})$ for some
   Hilbert space $\cH'$.

   \item There exist a Hilbert space $\cH$ and a unitary colligation
   $$
     U = \begin{bmatrix} A & B \\ C & D \end{bmatrix} \colon
     \begin{bmatrix} \cH \\ \cE \end{bmatrix} \to \begin{bmatrix}
     \cH \\ \cE_{*} \end{bmatrix}
   $$
   so that $S(z)$ can be realized in the form
   $$
     S(z) = D + z C(I - zA)^{-1} B.
   $$
   \end{enumerate}
   \end{theorem}

There have recently appeared several types of generalizations of
the Schur-class functions to several-variable contexts (see
\cite{Ball, IJCsummary} for surveys).  We focus here on the
generalizations where the unit disk $\bbD$ is replaced by the unit
polydisk
$$
  \bbD^{d} = \{ z = (z_{1}, \dots, z_{d}) \in \bbC^{d} \colon
  |z_{k}| < 1 \text{ for } k = 1, \dots, d \}
 $$
 in
$d$-dimensional complex space $\bbC^{d}$.  One can define the
$d$-variable Schur class $\cS_{d}(\cE, \cE_{*})$ to consist of
 functions $S \colon \bbD^{d} \to \cL(\cE, \cE_{*})$ of $d$ complex
 variables which are analytic on $\bbD^{d}$
 and with values equal to contraction operators from $\cE$ to
 $\cE_{*}$.  However, if $d>2$ it is known that the analogue of
 Theorem \ref{T:1DSchur} fails.  Instead, we define the {\em
 Schur--Agler class} $\mathcal{SA}_{d}(\cE, \cE_{*})$ (usually
 shortened to $\mathcal{SA}(\cE, \cE_{*})$ as the number of variables
 $d$ will be fixed throughout) as the space of functions $S(z) =
 \sum_{n \in \bbZ^{d}_{+}} S_{n} z^{n}$ (with the standard
 multivariable notation $z^{n} = z_{1}^{n_{1}} \cdots z_{d}^{n_{d}})$
 holomorphic on $\bbD^{d}$ with values equal to operators from
 $\cE$ to  $\cE_{*}$ such that
 $$
  S(T): = \sum_{n \in \bbZ^{d}_{+}} S_{n} \otimes T^{n} \in \cL(\cE
  \otimes \cK, \cE_{*} \otimes \cK)
 $$
 is a contraction ($\|S(T)\| \le 1$) whenever $T = (T_{1}, \dots,
 T_{d})$ is a commuting $d$-tuple of strict contractions on some
 Hilbert space $\cK$.  (Here $T^{n} = T_{1}^{n_{1}} \cdots
 T_{d}^{n_{d}}$ if $n = (n_{1}, \dots, n_{d}) \in \bbZ^{d}_{+}$.)
 Note that any Schur-Agler-class function is also Schur-class since
 one can take in particular $\cK = \bbC$ and $T = z = (z_{1}, \dots,
 z_{d})$ a $d$-tuple of
 scalar operators.   The following analogue of Theorem \ref{T:1DSchur}
 appears in \cite{Ag90, AgMcC} (see \cite{BT} for the additional
 condition (2$^{\prime}$)).

 \begin{theorem}  \label{T:Agler}
     Let $S \colon \bbD^{d} \to \cL(\cE, \cE_{*})$ be an
operator-valued
     function defined on the unit polydisk $\bbD^{d}$.  Then the
     following are equivalent:
     \begin{enumerate}
     \item[(1)] The function $S$ is in the Schur--Agler class
     $\mathcal{SA}(\cE, \cE_{*})$.

     \item[(2)] There exist auxiliary Hilbert spaces $\cH'_{1}$, \dots,
     $\cH'_{d}$ and operator-valued functions $H_{k} \colon \bbD^{d} \to
     \cL(\cH'_{k}, \cE_{*})$ for $k = 1, \dots, d$ so that $S$ has the
    so-called Agler decomposition
   \begin{equation}  \label{Aglerdecom}
     I - S(z) S(w)^{*} = \sum_{k=1}^{d} (1 - z_{k} \overline{w_{k}})
     H_{k}(z) H_{k}(w)^{*}.
   \end{equation}

   \item[(2$^{\prime}$)] There exist auxiliary Hilbert spaces $
   \cH'_{1}, \dots, \cH'_{d}$ and operator-valued functions $H_{k} =
   \sbm{H^{1}_{k} \\ H^{2}_{k}} \colon {\mathbb D}^{d} \to
   \cL(\cH'_{k}, \sbm{\cE_{*} \\ \cE})$  for $k = 1, \dots, d$ so that
   $S$ has the {\em augmented Agler decomposition}
   \begin{align}
     &  \begin{bmatrix}  I - S(z) S(w)^{*} & S(z) - S(\overline{w}) \\
     S(\overline{z})^{*} -
       S(w)^{*} & I - S(\overline{z})^{*}S(\overline{w})  \end{bmatrix}
\notag \\
       & \qquad =
     \sum_{k=1}^{d}  \begin{bmatrix}  H^{1}_{k}(z) \\ H^{2}_{k}(z)
     \end{bmatrix} \begin{bmatrix} H^{1}_{k}(w)^{*} & H^{2}_{k}(w)^{*}
 \end{bmatrix}  \circ \begin{bmatrix} 1 - z_{k} \overline{w}_{k} &
 z_{k} - \overline{w}_{k} \\ z_{k} - \overline{w}_{k} &
 1 - z_{k} \overline{w}_{k}  \end{bmatrix}
 \label{augAglerdecom-Intro}
 \end{align}
 where $\circ$ is the Schur or entrywise matrix product.

  \item[(3)] There exist Hilbert spaces $\cH_{1}, \dots, \cH_{d}$ and
a
   unitary colligation $U$ of the structured form
   \begin{equation}  \label{colligation}
   U = \begin{bmatrix} A & B \\ C & D \end{bmatrix} =
   \begin{bmatrix}  A_{11} & \cdots & A_{1d} & B_{1} \\
                    \vdots &       & \vdots & \vdots \\
            A_{d1} & \cdots & A_{dd} & B_{d} \\
            C_{1} & \cdots & C_{d} & D \end{bmatrix}  \colon
            \begin{bmatrix} \cH_{1} \\ \vdots \\ \cH_{d} \\ \cE
\end{bmatrix}
  \to \begin{bmatrix} \cH_{1} \\ \vdots \\ \cH_{d} \\ \cE_{*}
  \end{bmatrix}
 \end{equation}
 so that $S(z)$ is realized in the form
 \begin{equation}  \label{realization}
   S(z) = D + C (I - Z_{\diag}(z) A)^{-1} Z_{\diag}(z) B
 \end{equation}
 where we have set
 $$
   Z_{\diag}(z) = \begin{bmatrix} z_{1}I_{\cH_{1}} & & \\ & \ddots &
\\
          & & z_{d} I_{\cH_{d}} \end{bmatrix}.
 $$
 \end{enumerate}
 \end{theorem}
 We remark that the proof of (3) $\Rightarrow$ (2) or (3)
 $\Rightarrow$ (2$^{\prime}$) is particularly transparent:  given
 that \eqref{realization} holds with $U$ as in \eqref{colligation}
 unitary, one uses the relations
  \begin{align}
    &  AA^{*} + BB^{*} = I,  \qquad AC^{*} + B D^{*} = 0, \qquad
    CC^{*} + DD^{*} = I, \notag \\
    & A^{*}A + C^{*}C = I, \qquad A^{*}B + C^{*} D = 0, \qquad
    B^{*}B + D^{*} D = I
    \label{unitary}
   \end{align}
   arising from the unitary property of $U$ to check that
   \eqref{Aglerdecom} holds with
   \begin{equation*}
       H_{k}(z) = C (I - Z_{\diag}(z) A)^{-1} P_{k}
    \end{equation*}
    and that \eqref{augAglerdecom-Intro} holds with
    \begin{equation}  \label{H-intro2}
  H^{1}_{k}(z) = C (I - Z_{\diag}(z) A)^{-1}P_{k}, \qquad
  H^{2}_{k}(z) = B^{*}(I - Z_{\diag}(\overline{z})^{*}A^{*})^{-1}
P_{k}
  \end{equation}
  with $P_{k}$ the orthogonal projection of $\cH =
\bigoplus_{j=1}^{d} \cH_{j}$
  onto its $k$-th coordinate space $\cH_{k}$ (identified as a
  subspace of $\cH$).
  The original proof of the converse direction (2) $\Rightarrow$
(3)
   involves what we call the ``lurking isometry argument'' (see
\cite{Ball}):
   from the given Agler decomposition \eqref{Aglerdecom} there is
   associated a natural partially defined isometry $V$; one can then
   identify a one-to-one correspondence between unitary extensions $U$
   (possibly outside the original Hilbert space) of $V$ and
realizations
   of $S(z)$ in the form \eqref{realization} associated with the
   structured unitary colligation \eqref{colligation}.
 In the sequel, we shall refer to a structured unitary colligation $U$
 of the form \eqref{colligation} as a {\em Givone--Roesser unitary}
 (GR-unitary) colligation\footnote{We note that the term {\em Roesser
 unitary colligation} was used in \cite{MammaBear, BabyBear}.}.
(Multidimensional linear systems whose coefficient matrix has a
structure of \eqref{colligation}, though without any norm
constraints on its blocks, and whose transfer function has the
form \eqref{realization}, were introduced in \cite{GR1,GR2}; see
\eqref{sys-free} for the explicit form of the system equations.)

We also note that the equivalence (1)$\Leftrightarrow$(3) in
Theorem \ref{T:Agler} for the case of rational inner matrix-valued
functions on the bidisk $\mathbb{D}^2$ was independently proved by
Kummert \cite{Kummert} in a stronger form: the spaces
$\mathcal{H}_1$, $\mathcal{H}_2$ in (3) are finite-dimensional.
This strengthening of Agler's theorem, which also includes item
(2) of Theorem \ref{T:Agler} with rational matrix-valued functions
$H_1$, \ldots, $H_d$, appears also in
\cite{GW,BabyBear,Knese2008}. Moreover, it turns out that this
version of Theorem \ref{T:Agler} for the case of rational inner
matrix-valued functions can be extended to any number of variables
\cite{Knese2011}.

 A new approach to the issue of unitary realizations of
 Schur-Agler-class
 functions was made in \cite{BabyBear}.  A
 multievolution Lax--Phillips scattering system ${\mathfrak S}$
(associated with
 function theory on the polydisk) was defined in \cite{CS1, CS2} (see
also
 \cite{S-MSRI}).  A minimal such scattering system
 ${\mathfrak S}$ is completely determined (up to unitary equivalence)
 by its scattering matrix $S$ which can be any Schur-class function.
 If the scattering system has some additional geometric structure
 (``colligation geometry''), one
 can identify a unitary colligation $U$ of the form
 \eqref{colligation} which is embedded in ${\mathfrak S}$.  When this
 is the case, the scattering matrix is also the transfer function
 of the unitary colligation $U$ embedded in ${\mathfrak S}$
 and thus $S$ is in fact in the Schur--Agler class
 $\mathcal{SA}(\cE, \cE_{*})$ rather than just the Schur class
 $\cS(\cE, \cE_{*})$. (Notice that scattering systems associated with
a different type of colligation
 geometry were considered in \cite{K-JOT2000}.)  In \cite{BabyBear} we
arrived at a new criterion for a
 Schur-class function $S$ to be a Schur-Agler-class function: {\em
 the Schur-class function $S$ is in the Schur--Agler class if and only
 if there is a multievolution scattering system (not necessarily
 minimal) with scattering
 function $S$ possessing the special colligation geometry. }
 By combining the results of \cite{Ag90, AgMcC, BT} with
\cite{BabyBear},
 it is clear that, given a Schur-class function $S$ on the polydisk
$\bbD^{d}$,
 there is a multievolution scattering system ${\mathfrak S}$ having
 scattering matrix $S$ along with the special colligation geometry if
and
 only if $S$ has an Agler decomposition \eqref{Aglerdecom} (or,
 equivalently, an augmented Agler decomposition
 \eqref{augAglerdecom-Intro}).  One of
 the goals of this paper is to construct such a multievolution
 scattering system ${\mathfrak S}$ (in functional-model form)
directly from a given Agler
 decomposition \eqref{Aglerdecom} for $S$.

 Given a Schur-class function $S$, in \cite{BabyBear} we presented
 three functional models (see \cite{NV1, NV2} for the $d=1$ case)
 for a minimal multievolution scattering system having $S$ (up to
 trivial identifications) as its scattering matrix, called the {\em
 Pavlov model}, the {\em de Branges--Rovnyak model} and the {\em
 Sz.-Nagy--Foia\c{s} model}.  We find that the most convenient one
for our
 purposes here is the de Branges--Rovnyak model.  The modification of
 the de Branges--Rovnyak model introduced in \cite{NV1, NV2} works
 with two components, one of which is analytic on the disk $\bbD$
while the
 other is conjugate-analytic on $\bbD$; it turns out that
 this modified version of the original de Branges--Rovnyak model is
 more convenient to work with for many manipulations and formulas.
 However there are difficulties with extending this formalism to the
 multivariable context; a function may be the product of an analytic
 function in one variable and a conjugate analytic function in
 another variable.  We here modify the de Branges--Rovnyak functional
 model once again by making use of the formal power series
 formalism from \cite{FRKHS}; with this formalism the extension to
 several variables can be done relatively smoothly.
 In particular, the formal version of the augmented Agler
 decomposition \eqref{augAglerdecom-Intro} plays a key role for the
 analysis to follow (see \eqref{enlargedAglerdecom} below).

 The formal reproducing-kernel formalism is the tool which helps
 clarify the more complicated connections between the various notions
of
 scattering and system minimality (scattering-minimal,
 closely-connected, strictly closely-connected) in the multivariable
 situation which were introduced but left mysterious in \cite{BT}.
 These complications are caused by the fact that the colligation
geometry
 may not mesh well with the scattering geometry; specifically, the
 projections coming from the colligation geometry may not commute
 with the projection onto the minimal part of the scattering system.

 The paper is organized as follows.  Section \ref{S:FRKHS} develops
the preliminaries on formal reproducing kernel Hilbert spaces.
 Here we also clarify and enhance some results from \cite{FRKHS} on
 multipliers in the context of formal reproducing kernel Hilbert
 spaces.  Section \ref{S:scatmodels} places the de Branges--Rovnyak model
 scattering system \cite{BabyBear} into a  reproducing kernel framework.
 Section \ref{S:scat-col}, in
 addition to recalling the admissible-trajectory construction
 from \cite{MammaBear} for embedding a given GR-unitary colligation
into a multievolution scattering system having scattering matrix
equal to
 the transfer function of the colligation, also presents a
 multivariable version of the Sch\"affer-matrix construction to
 achieve the same embedding.  Here we also recall the converse
 colligation-geometry characterization of those scattering systems
containing an
 embedded GR-unitary colligation.  In Section \ref{S:scat-col-model},
we obtain a reproducing-kernel analogue of this colligation
geometry.  Under an
 additional hypothesis which eliminates the presence of nuisance
 overlapping spaces, this leads to the true presence of
  the colligation geometry inside the de
 Branges--Rovnyak functional-model minimal multievolution scattering
 system; this is exactly the special case where a minimal
 multievolution scattering system carries a colligation geometry.
 Section \ref{S:minimal} assumes that we are given a GR-unitary
colligation
 $U = \sbm{ A & B \\ C & D}$ and sorts out various notions of
 minimality in terms of the operators $A,B,C,D$.  Section \ref{S:cc}
 shows how the original lurking-isometry approach can be carried out
 in functional-model form; the defect spaces involved in a known
 parametrization  of the set of all the  GR-unitary realizations
 of a given augmented Agler decomposition (see e.g.~\cite{BBQ2})
 can be identified with the overlapping
 spaces arising from the augmented Agler decomposition. In
particular, if the
 augmented Agler decomposition is strictly closely connected (a
 weaker form of minimal), then the associated closely-connected
 GR-unitary colligation realization
 is unique. Here we also discuss how these canonical functional models
 for a Schur-Agler-class function with given Agler decompostion
 relate to those obtained in \cite{BB-Gohberg}.
 The final Section \ref{S:Aglerdecom-scatsys} analyzes how to construct a
 functional-model multievolution scattering system carrying the
colligation
 geometry directly from a given augmented Agler decomposition. Here
 the Sch\"affer-matrix construction sheds some insight on the
 structure of the problem; the results of Section \ref{S:cc} in some
sense
 solve this problem.

 Since the paper is quite long we indicate several ``pointers'' that
the reader can use to jump directly to several key
 locations:
 \begin{itemize}
 \item
 Remark \ref{R:min-summary} summarizes the various notions of
minimality:
 for (augmented) Agler decompositions,
 for Givone--Roesser unitary colligations,
 and for multievolution scattering systems, and the interrelations
among these notions.
 \item
 Theorem \ref{T:minfuncmodel} identifies explicitly the colligation
geometry inside
 the de Branges--Rovnyak functional-model minimal multievolution
scattering
 system in terms of a minimal (augmented) Agler decomposition, and
provides a direct construction
 of scattering minimal GR-unitary realizations.
 \item
 Theorem \ref{T:parametrize} provides a parametrization of all
GR-unitary realizations
 of a given (augmented) Agler decomposition in terms of functional
models and overlapping spaces.
 \end{itemize}

 We mention that some analogous results concerning de Branges--Rovnyak
 func\-tional-model realizations for a given Schur-class function,
but in the
 context of Drury--Arveson spaces and coisometric colligations of
Fornasini--Marchesini type,
 have been obtained recently in \cite{ADR, BBF2a, BBF3}.

 We finally note that an analogue of the Agler-decomposition
characterization was recently obtained in \cite{GKVVW} for the
operator-valued $d$-variable Schur class using the methods
developed in the present paper (with a reference to its earlier
preprint version), namely the formal de Branges--Rovnyak model for
a scattering system associated with the given Schur-class
function. (This result was later reproved in \cite{Knese_CAOT}
using standard reproducing-kernel techniques, however only in the
case of scalar-valued functions.)

 The notation is mostly standard.  However there is the potential for
 confusion between internal and external orthogonal direct sums.  We
 use the notation $\oplus$ and $\bigoplus$ for internal orthogonal
 direct sums and $\widehat \oplus$ and $\widehat \bigoplus$ for
 external orthogonal direct sums;
 for a (finite or countably infinite) collection $\left\{\cH_j \colon j \in J\right\}$ of Hilbert spaces
 and $h_j \in \cH_j$, $j \in J$, with $\sum_{j \in J} \|h_j\|_{\cH_j}^2 < \infty$,
 we denote by $\widehat \bigoplus h_j$ the corresponding element of $\widehat \bigoplus \cH_j$.

\section{Function theory on $L^{2}(\bbT^{d},  \cF)$: the formal
reproducing kernel point of view} \label{S:FRKHS}

\subsection{Preliminaries on formal reproducing kernel Hilbert
spaces} \label{S:prelimFRKHS}

We recall the notion of a formal reproducing kernel Hilbert space
(FRKHS for short) in the commuting indeterminates $ z = (z_{1},
\dots, z_{d})$ from \cite{FRKHS}. For $n = (n_{1}, \dots, n_{d})
\in \bbZ^{d}$, we set $z^{n} = z_{1}^{n_{1}} \cdots z_{d}^{n_{d}}$
(note that we allow each integer $n_{k}$ to be positive as well as
negative or zero). We say that a Hilbert space whose elements are
formal power series $f(z) = \sum_{n \in \bbZ^{d}} f_{n} z^{n}$
with coefficients in the Hilbert space $\cF$ (i.e., $f \in
\cF[[z^{\pm 1}]]$ where here and throughout for brevity
$\cF[[z^{\pm 1}]]$ denotes the space of formal power series in
the indeterminates $z_{1}, \dots, z_{n}$ {\em and} their inverses
$z_{1}^{-1}, \dots, z_{d}^{-1}$ that we shall refer to as formal Laurent series\footnote{
We emphasize that unlike the usual algebraic definitions we impose no restrictions
on the support of our formal Laurent series; as a result the product operation
is only partially defined; see Section \ref{S:multipliers} below
and especially Remark \ref{R:trap}.})
is a {\em formal reproducing
kernel Hilbert space} if the linear operator $\pi_n\colon f
\mapsto f_{n}$ is continuous for each $n \in \bbZ^{d}$.  In this
case, for each $n \in \bbZ^{d}$ there is a formal power series
$K_{n}(z) \in \cL(\cF)[[z^{\pm 1}]]$ such that, for each vector $u
\in \cF$ we have
\begin{equation} \label{reproduce}
   \langle f, K_{n}u \rangle_{\cH} = \langle f_{n}, u \rangle_{\cF}.
\end{equation}
If we then define $K(z,w) = \sum_{n \in \bbZ^{d}} K_{n}(z) w^{-n}
\in
  \cL(\cF)[[z^{\pm 1},w^{\pm 1}]]$,
then \eqref{reproduce} can be written more compactly in a way
suggestive of the classical (non-formal) case
\begin{equation}  \label{reproduce'}
    \langle f(z), K(z,w) u \rangle_{\cH \times \cH[[w^{\pm 1}]]} =
    \langle f(w),  u \rangle_{\cF[[w^{\pm 1}]] \times \cF}.
\end{equation}
Here we use, for an arbitrary coefficient Hilbert space ${\mathcal
C}$,  the general notion of a pairing
$$
\langle \cdot, \cdot \rangle_{{\mathcal C} \times {\mathcal
C}[[z^{\prime \pm 1}]]} \colon {\mathcal C} \times {\mathcal
C}[[z^{\prime \pm 1}]] \to {\mathbb C}[[z^{\prime \pm 1}]] $$
defined by
$$
\Big\langle c, \sum_{n' \in {\mathbb Z}^{d}} f_{n'} z^{\prime -n'}
\Big\rangle_{{\mathcal C} \times {\mathcal C}[[z^{\prime \pm 1}]]}
= \sum_{n' \in {\mathbb Z}^{d}} \langle c, f_{n'}
\rangle_{{\mathcal C}} z^{\prime n'}. $$ When convenient we take
on occasion the pairing in the reverse order:
$$
\Big\langle \sum_{m' \in {\mathbb Z}^{d}} f_{m'} z^{\prime m'}, c
\Big\rangle_{{\mathcal C}[[z^{\prime \pm 1}]] \times {\mathcal C}}
= \sum_{m' \in {\mathbb Z}^{d}} \langle f_{m'}, c
\rangle_{{\mathcal C}} z^{\prime m'}.
$$
These are actually special cases of a more general partially
defined pairing on ${\mathcal C}[[z^{\prime \pm 1}]]$ with values
in ${\mathbb C}[[z^{\prime \pm 1}]]$ given by $$ \left\langle
\sum_{m' \in {\mathbb Z}^{d}} f_{m'} z^{\prime m'}, \sum_{n' \in
{\mathbb Z}^{d}} g_{n'} z^{\prime -n'} \right\rangle_{{\mathcal
C}[[z^{\prime \pm 1}]] \times {\mathcal C}[[ z^{\prime \pm 1}]]} =
\sum_{m' \in {\mathbb Z}^{d}} \left[ \sum_{n' \in {\mathbb Z}^{d}}
\langle f_{m'-n'}, g_{n'} \rangle_{{\mathcal C}} \right] z^{\prime
m'}
$$
where the sum required to compute the coefficient of $z^{\prime
m'}$ is finite under the assumption that at least one of $f$ or
$g$ is a polynomial in $z_{1}, \dots, z_{d}$ and $z_{1}^{-1},
\dots, z_{d}^{-1}$. When this is the case, we write $\cH = \cH(K)$
and we say that {\em $K \in \cL(\cF)[[z^{\pm 1}, w^{\pm 1}]]$ is
the reproducing kernel for the formal reproducing kernel Hilbert
space $\cH(K)$}.

The following proposition extends the reproducing property
\eqref{reproduce'} to the case where a vector $u\in\mathcal{F}$ is
replaced with a formal power series $u(w)\in\mathcal{F}[[w^{\pm
1}]]$.
\begin{proposition}\label{P:repr-ext}
Let $\mathcal{F}$ be a Hilbert space, let
$K(z,w)=\sum_{n,m\in\mathbb{Z}^d}K_{n,m}z^nw^{-m}\in\mathcal{L(F)}[[z^{\pm
1},w^{\pm 1}]]$ be the reproducing kernel for a formal reproducing
kernel Hilbert space $\mathcal{H}(K)$, where
 we will identify $K(z,w)$ with
$\sum_{m\in\mathbb{Z}^d}K_m(z)w^{-m}\in\mathcal{H}(K)[w^{\pm 1}]$.
Let $u(w)=\sum_{m\in\mathbb{Z}^d}u_mw^{-m}\in\mathcal{F}[[w^{\pm
1}]]$ be such that the series
$\sum_{\ell\in\mathbb{Z}^d}K_{m-\ell}(z)u_\ell$ converges weakly
in $\mathcal{H}(K)$ for every $m\in\mathbb{Z}^d$. Then
$K(z,w)u(w)\in \mathcal{H}(K)[w^{\pm 1}]$, for every
$f(z)=\sum_{n\in\mathbb{Z}^d}f_nz^n\in\mathcal{H}(K)$ and every
$m\in\mathbb{Z}^d$ the series $\sum_{\ell\in\mathbb{Z}^d}\langle
f_{m-\ell},u_\ell\rangle_\mathcal{F}$ converges, and
\begin{equation}\label{E:repr-ext}
\langle
f(z),K(z,w)u(w)\rangle_{\mathcal{H}(K)\times\mathcal{H}(K)[[w^{\pm
1}]]}= \langle f(w),u(w)\rangle_{\mathcal{F}[[w^{\pm
1}]]\times\mathcal{F}[[w^{\pm 1}]]}.
\end{equation}
\end{proposition}
\begin{proof}
The first statement follows from the identity
$$K(z,w)u(w)=\sum_{m\in\mathbb{Z}^d}\Big(\text{weak}\,\sum_{\ell\in\mathbb{Z}^d}K_{m-\ell}(z)u_\ell\Big)w^{-m}.$$
Next, for every $f\in\mathcal{H}(K)$ and $m\in\mathbb{Z}^d$,
$$\sum_{\ell\in\mathbb{Z}^d}\langle
f(z),K_{m-\ell}(z)u_\ell\rangle_{\mathcal{H}(K)}=\sum_{\ell\in\mathbb{Z}^d}\langle
f_{m-\ell},u_\ell\rangle_{\mathcal{F}},$$ and since the sum on the
left-hand side converges, so does the sum on the right-hand side.
Notice that
$$\langle f(w),u(w)\rangle_{\mathcal{F}[[w^{\pm
1}]]\times\mathcal{F}[[w^{\pm 1}]]}=\sum_{m\in\mathbb{Z}^d}
\Big(\sum_{\ell\in\mathbb{Z}^d}\langle
f_{m-\ell},u_\ell\rangle_{\mathcal{F}}\Big)w^m.$$ On the other
hand,
\begin{align*}
&\langle
f(z),K(z,w)u(w)\rangle_{\mathcal{H}(K)\times\mathcal{H}(K)[[w^{\pm
1}]]}\\
& \quad =\Big\langle
f(z),\sum_{m\in\mathbb{Z}^d}\Big(\text{weak}\,\sum_{\ell\in\mathbb{Z}^d}K_{m-\ell}(z)u_\ell\Big)w^{-m}
\Big\rangle_{\mathcal{H}(K)\times\mathcal{H}(K)[[w^{\pm 1}]]}\\
& \quad =\sum_{m\in\mathbb{Z}^d}\Big\langle f(z),\text{weak}\,
\sum_{\ell\in\mathbb{Z}^d}K_{m-\ell}(z)u_\ell\Big\rangle_{\mathcal{H}(K)}
w^{m}\\
& \quad =\sum_{m\in\mathbb{Z}^d}\Big(\sum_{\ell\in\mathbb{Z}^d}
\langle f(z),K_{m-\ell}(z)u_\ell\rangle_{\mathcal{H}(K)}\Big) w^{m}\\
& \quad
=\sum_{m\in\mathbb{Z}^d}\Big(\sum_{\ell\in\mathbb{Z}^d}\langle
f_{m-\ell},u_\ell\rangle_{\mathcal{F}}\Big)w^m,
\end{align*}
and \eqref{E:repr-ext} follows.
\end{proof}

 The following theorem from \cite{FRKHS} characterizes
which formal power series
$$K \in \cL(\cF)[[z^{\pm 1},w^{\pm 1}]]
$$
arise as the reproducing kernel for some formal reproducing kernel
Hilbert space $\cH(K)$.

\begin{theorem} \label{T:FRKHS} (see \cite[Theorem 2.1]{FRKHS})
    Suppose that $\cF$ is a Hilbert space and that we are given a
    formal power series $K \in \cL(\cF)[[z^{\pm 1},w^{\pm 1}]]$.  Then
    the following are equivalent.
    \begin{enumerate}
    \item
    $K$ is the reproducing kernel for a uniquely determined FRKHS
    ${\mathcal H}(K)$ of formal power
    series in the commuting variables
    $$
    z^{\pm 1} = (z_{1}, \dots, z_{d},
    z_{1}^{-1}, \dots, z_{d}^{-1})
    $$
    with
    coefficients in $\cF$.
    \item There is an auxiliary Hilbert space ${\mathcal H}'$ and a
formal
    power series $H(z) \in {\mathcal L}({\mathcal H}', {\mathcal
    F})[[z^{\pm 1}]]$ so that
    $$
    K(z,w) = H(z) H(w)^{*}
    $$
    where we use the convention
\begin{equation}  \label{convention}
    (w^{n})^{*} = w^{-n}, \qquad
    H(w)^{*} = \sum_{n \in \bbZ^{d}} H_{n}^{*} w^{-n} \text{ if }
    H(z) = \sum_{n \in \bbZ^{d}} H_{n} z^{n}.
  \end{equation}

    \item $K(z,w) = \sum_{n,n' \in \bbZ^{d}} K_{n,n'} z^{n} w^{-n'}$
    is a {\em positive kernel} in the sense that
    $$ \sum_{n,n'
    \in {\mathbb Z}^{d}} \langle K_{n,n'} u_{n'}, u_{n}
    \rangle_{{\mathcal F}} \ge 0
    $$
    for all finitely supported ${\mathcal F}$-valued functions $n \mapsto
    u_{n}$ on ${\mathbb Z}^{d}$.
    \end{enumerate}
    Moreover, in this case the FRKHS ${\mathcal H}(K)$ can be defined
    directly in terms of the formal power series $H(z)$ appearing in
    condition (2) by
    $$
    {\mathcal H}(K) = \{ H(z) h \colon h \in {\mathcal H}' \} $$
    with norm taken to be the pullback norm $$ \| H(z)h\|_{{\mathcal
    H}(K)} = \| Q h\|_{{\mathcal H}'} $$
    where $Q$ is the orthogonal projection of ${\mathcal H}'$ onto the
    orthogonal complement of the kernel of the map
    $M_H \colon {\mathcal H}' \mapsto {\mathcal F}[[z^{\pm 1}]]$ given by
    $$ M_H \colon h \mapsto H(z) h.
    $$
    \end{theorem}
\begin{remark} \label{R:special_Kolmogorov}
We can always choose a factorization $K(z,w)=H(z)H(w)^*$ with
$\ker M_H=\{0\}$ so that $M_H$ is an isometry of ${\mathcal H}'$
onto ${\mathcal H}(K)$. (A ``canonical'' choice is ${\mathcal
H}'={\mathcal H}(K)$ and $H_n = \pi_n \colon \sum_{n \in {\mathbb
Z}^d} f_n z^n \mapsto f_n$; it is then easily seen that $H_{n'}^*
u = \sum_{n \in {\mathbb Z}^d} (K_{n,n'} u) z^n$.)
\end{remark}
When the kernel $K \in \cL(\cF)[[z^{\pm 1},w^{\pm 1}]]$ satisfies
any of the equivalent conditions in Theorem \ref{T:FRKHS}, we
shall write
$$
   K \pos 0.
$$

One can also express the FRKHS $\cH(K)$ directly from the Laurent
series coefficients $K(z,w) = \sum_{n,n' \in \bbZ^{d}} K_{n,n'}
z^{n} w^{-n'}$ in a more operator-theoretic way as follows.
Suppose that $[K_{n,n'}]_{n,n'}$ is a matrix of operators on the
Hilbert space $\cF$ with rows and columns indexed by $\bbZ^{d}$.
Note that any such operator-matrix defines a linear operator,
denoted by $K$ with the expectation that this will cause no
confusion, from the space of polynomials $\cF[z^{\pm 1}]$ into the
space of formal power series $\cF[[z^{\pm 1}]]$ according to the
formula
$$
  (Kp)(z) = \sum_{n \in \bbZ^{d}} \left( \sum_{n' \in \bbZ^{d}}
  K_{n,n'} p_{n'} \right) z^{n} \text{ if } p(z) = \sum_{n' \in
  \bbZ^{d}} p_{n'} z^{n'}.
$$
Note that this formula involves only finite sums under the
assumption that $p(z)$ is a polynomial (so $p_{n'} = 0$ for all
but finitely many $n' \in \bbZ^{d}$).  We also note that there is
a natural pairing between $\cF[[z^{\pm 1}]]$ and $\cF[z^{\pm 1}]$:
\begin{equation*}
    \langle f(z), p(z) \rangle_{L^2} =
    \sum_{n \in \bbZ^{d}} \langle f_{n}, p_{n}\rangle_{\cF}
\end{equation*}
if $f(z) = \sum_{n \in \bbZ^{d}}f_{n} z^{n} \in \cF[[z^{\pm 1}]]$
    and  $p(z) = \sum_{n \in \bbZ^{d}} p_{n} z^{n} \in \cF[z^{\pm
1}]$. We say that the operator $K = [K_{n,n'}]_{n,n' \in
\cL(\cF)}$ is {\em positive-semidefinite} if
$$ \langle K p, p \rangle_{L^2} \ge 0
\text{ for all } p \in \cF[z^{\pm 1}].
$$

\begin{theorem} \label{T:FRKHS-Laurent}
    (See \cite[Proposition 2.2]{FRKHS}.)
    Let $[K_{n,n'}]_{n,n'}$ be a matrix of operators on the Hilbert
space
    $\cF$ with rows and columns indexed by $\bbZ^{d}$, let $K =
    [K_{n,n'}]_{n,n' \in \bbZ^{d}}$ be the associated operator from
    $\cF[z^{\pm 1}]$ into $\cF[[z^{\pm 1}]]$, and define a kernel
   $K(z,w) \in \mathcal{L}(\cF)[[z^{\pm 1}, w^{\pm 1}]]$ by
  $$ K(z,w) = \sum_{n,n' \in \bbZ^{d}} K_{n,n'} z^{n} w^{-n'}.
 $$
 Then the operator $K$ is positive-semidefinite if and only if the
    associated kernel $K(z,w)$ is a positive kernel (i.e., satisfies
    any one of the three equivalent conditions in Theorem
\ref{T:FRKHS}).
    In this case, the associated FRKHS $\cH(K)$ can be described as
    the closure of the linear manifold $K  \cF[z^{\pm 1}] \subset
    \cF[[z^{\pm 1}]]$ in the $\cH(K)$-inner product given by
 \begin{equation*}
  \langle Kp, Kp' \rangle_{\cH(K)} =
  \langle K p, p' \rangle_{\cF[[z^{\pm 1}]] \times \cF[z^{\pm 1}]}.
  \end{equation*}
  \end{theorem}

   Given any positive formal kernel
   $$
    K(z,w) = \sum_{n,m \in {\mathbb Z}^{d}} K_{n,m} z^{n} w^{-m} \in
    \cL(\cF)[[z^{\pm 1}, w^{\pm 1}]],
    $$
    we have already introduced the notation
    \begin{equation}  \label{notation}
    K_m(z) = \sum_{n \in {\mathbb Z}^{d}}
    K_{n,m} z^{n} \in \cL(\cF)[[z^{\pm}]].
    \end{equation}
    Then, in terms of the associated formal reproducing kernel Hilbert
    space $\cH(K)$ as in Theorem \ref{T:FRKHS}, we then have
    $$
    K_m h \in \cH(K) \text{ for each } m \in {\mathbb Z}^{d} \text{
and } h\in\mathcal{F}.
    $$

\subsection{Multipliers between formal reproducing kernel Hilbert
spaces} \label{S:multipliers}

Let now $K \in \cL(\cF)[[z^{\pm 1},w^{\pm 1}]]$ and $K' \in
\cL(\cF')[[z^{\pm 1},w^{\pm 1}]]$ be two arbitrary positive
kernels with associated FRKHSs $\cH(K)$ and $\cH(K')$. We say that
the formal power series
$$
S(z) = \sum_{n \in \bbZ^{d}} S_{n} z^{n}  \in \cL(\cF,
\cF')[[z^{\pm 1}]]
$$
is a {\em bounded multiplier}  from $\cH(K)$ into $\cH(K')$
(written $S \in \cM(K,K')$) if, for each $f \in \cH(K)$ the
product of formal Laurent series $S(z) \cdot f(z)$ is
well-defined, i.e.,
\begin{equation}\label{E:mult-coefs}
(Sf)_{n} = \sum_{\ell \in \bbZ^{d}} S_{n - \ell} f_{\ell}
\end{equation}
converges in  the weak topology on $\cF'$ for each $n \in
\bbZ^{d}$, the resulting series $(Sf)(z) = \sum_{n} (S f)_{n}
z^{n}$ is in $\cH(K')$, and the associated operator $M_{S} \colon
\cH(K) \to \cH(K')$ is bounded as an operator from $\cH(K)$ into
$\cH(K')$.

As explained in \cite{FRKHS}, many of the results concerning
formal reproducing kernel Hilbert spaces can be reduced to the
corresponding results for the classical case by observing that any
formal reproducing kernel Hilbert space can be viewed as a
classical reproducing kernel Hilbert space over ${\mathbb Z}^{d}$
with reproducing kernel $K(n,m) = K_{n,m}$ corresponding to
evaluation of the $n$-th Laurent coefficient $f_{n}$ (i.e., the
value of $f$ at the ``point'' $n$) of a generic element $f(z) =
\sum_{n \in {\mathbb Z}^{d}} f_{n} z^{n}$ of the space.   While
this technique works well for some results, its applicability for
the analysis of bounded multipliers is quite limited since the
multipliers of interest in the formal setting act via convolution
(rather then pointwise or Schur) multiplication when expressed in
terms of Laurent coefficients.  Since the convolution
multiplication involves a possibly infinite sum, one is
necessarily confronted with convergence issues (see Remark
\ref{R:trap} below)  and much more elaborate arguments to verify
results whose analogues in the classical setting are
straightforward.  A first instance of this phenomenon is the
following proposition.

\begin{proposition} \label{P:convergences}
Suppose that $K$ is positive kernel with coefficients in
$\cL(\cF)$, with associated FRKHS $\cH(K)$, and suppose that $S$
is a formal power series in $\cL(\cF, \cF')[[z^{\pm 1}]]$. If for
each $f \in \cH(K)$ the product of formal Laurent series $S(z)
\cdot f(z)$ is well-defined as in \eqref{E:mult-coefs}, then $S(z)
K(z,w) S(w)^{*}$ is a well-defined formal power series
\begin{equation} \label{eq:sum_carefully-1}
 S(z) K(z,w) S(w)^{*} = \sum_{n,m} \Gamma_{n,m} z^{n} w^{-m}
\end{equation}
with coefficients $\Gamma_{n,m}$ given by either of the following
two iterated sums
\begin{multline} \label{eq:sum_carefully-2}
 \Gamma_{n,m} = {\rm weak}\,\sum_{\ell}\Big({\rm weak}\,\sum_{\ell'}
S_{n-\ell} K_{\ell, \ell'} S_{m -
 \ell'}^{*}\Big)\\
 ={\rm weak}\,\sum_{\ell'}\Big( {\rm weak}\,\sum_{\ell}S_{n-\ell}
K_{\ell, \ell'} S_{m -
 \ell'}^{*}\Big),\quad n,m\in\mathbb{Z}^d.
\end{multline}
\end{proposition}
\begin{proof}
Indeed, write $K(z,w) = \sum_{n,m \in \bbZ^{d}} K_{n,m} z^{n}
w^{-m}$ where $K_{n,m}$ has the factored form
$$
K_{n,m} = H_{n} H_{m}^{*}
$$
for operators $H_{n} \in \cL(\cH', \cF)$, with $\cH'$  some
auxiliary Hilbert space.  Moreover, as explained in Theorem
\ref{T:FRKHS}, $f(z) = H(z)h=\Big(\sum_{n \in \bbZ^{d}}
H_{n}z^n\Big) h \in \cH(K)$ for each $h \in \cH'$.  The assumption
that $S$ is a multiplier for $\cH(K)$ tells us that the series
$\sum_{n \in \bbZ^{d}} S_{n - \ell} H_{\ell}$ converges in the
weak operator topology of $\cL(\cH', \cF)$. Then $\sum_{\ell' \in
\bbZ^{d}}  H_{\ell'}^{*}S_{m - \ell'}^{*}$ converges in the weak
operator topology. We note the following general principle:  {\em
if $\{A_{L}\}_{L \in {\mathbb Z}_{+}}$ and $\{B_{L'}\}_{L' \ge 0}$
are sequences of operators for which the weak limits
$\operatorname{weak} \lim_{L \to \infty} A_{L} = A$ and
$\operatorname{weak} \lim_{L' \to \infty} B_{L'} = B$ exist, then
$$
\operatorname{weak} \lim_{L \to \infty} \left( \operatorname{weak}
\lim_{L' \to \infty} A_{L} B_{L'} \right) = \operatorname{weak}
\lim_{L \to \infty} A_{L}B = AB
$$
and similarly,
$$
\operatorname{weak} \lim_{L' \to \infty} \left(\operatorname{weak}
\lim_{L \to \infty} A_{L} B_{L'} \right) = \operatorname{weak}
\lim_{L' \to \infty} A B_{L'} = AB,
$$
i.e.,
$$
\operatorname{weak} \lim_{L\to \infty} \left( \operatorname{weak}
\lim_{L' \to \infty} A_{L} B_{L'}\right) = AB =
\operatorname{weak} \lim_{L' \to \infty} \left(
\operatorname{weak} \lim_{L \to \infty} A_{L} B_{L'} \right).
$$}
For each fixed $n$ and $m$, we apply this general principle to the
operator sequences
$$
A_{L} = \sum_{\ell \colon |\ell| \le L} S_{n-\ell} H_{\ell}, \quad
B_{L'} = \sum_{\ell' \colon |\ell'| \le L'} H^{*}_{\ell'} S^{*}_{m
- \ell'}
$$
where in general we set
\begin{equation}  \label{def|ell|}
|\ell| = \sum_{i=1}^{d} |\ell_{i}| \text{ for } \ell = (\ell_{1},
\dots, \ell_{d}) \in {\mathbb Z}^{d}.
\end{equation}
This enables us to conclude that
 \begin{multline*}
 {\rm weak}\,\sum_{\ell}\Big({\rm weak}\,\sum_{\ell'} S_{n-\ell}
H_\ell H_{\ell'}^* S_{m -
 \ell'}^{*}\Big)\\
 ={\rm weak}\,\sum_{\ell'}\Big( {\rm weak}\,\sum_{\ell}S_{n-\ell}
H_\ell H_{\ell'}^* S_{m -
 \ell'}^{*}\Big)=\Gamma_{n,m},
\end{multline*}
as asserted\footnote{\label{F:old_sins} The statement of
\cite[Theorem 2.6]{FRKHS} is flawed in not specifying the topology
on ${\mathcal F}$ making the product $S(z)f(z)$ well-defined and
in using the strong operator topology and not specifying the order
of summation making the product $S(z) K(z,w) S(w)^*$ well-defined.
It is also flawed in asserting that the converse of Proposition
\ref{P:convergences}  holds whereas in reality this is false.  We
notice that the proof of \cite[Theorem 2.5]{FRKHS} is incorrect as
well, since it confuses coefficientwise (Schur) multiplication
with the infinite-sum convolution multiplication required in the
FRKHS setting here;  for a correct proof, see Proposition
\ref{P:pullback} below.}
\end{proof}

\begin{remark} \label{R:trap}
    While this formalism of formal Laurent series is often convenient
    for computations, it does have its traps. First of all, the
product of two formal power series is not always defined,
    and thus the associative and distributive properties may be
meaningless for some series. On the other hand, even
    when the relevant products of the series are well defined, they
may violate the associative law.
    As an example, let
    $k_{\Sz}(z,w) = \sum_{n \in \bbZ^{d}} z^{n} w^{-n}$ be the
    bilateral Szeg\H{o}
    kernel (see Section \ref{S:Szego} below).  Then is it easily
    verified that
    $$ \left( \sum_{\ell = 0}^{\infty} z_{1}^{\ell}
    w_{1}^{-\ell}\right) (1 - z_{1}w_{1}^{-1}) =1+
    \sum_{\ell=1}^{\infty}(1-1) z_{1}^{\ell} w_{1}^{-\ell} = 1
    $$
 and hence
 $$ \left[ \left(\sum_{\ell=0}^{\infty} z_{1}^{\ell}
 w_{1}^{-\ell}\right) (1 - z_{1} w_{1}^{-1}) \right] \cdot
k_{\Sz}(z,w)
 = 1 \cdot k_{\Sz}(z,w) = k_{\Sz}(z,w)
 $$
 while, on the other hand,
 \begin{multline*}
  \left( \sum_{\ell=0}^{\infty} z_{1}^{\ell} w_{1}^{-\ell}\right)
 \cdot \left[ (1 - z_{1}w_{1}^{-1}) k_{\Sz}(z,w) \right]=
 \left( \sum_{\ell=0}^{\infty} z_{1}^{\ell} w_{1}^{-\ell}\right)
  \sum_{n \in \bbZ^{d}}(1-1) z^{n} w^{-n}\\
  = \left( \sum_{\ell=0}^{\infty} z_{1}^{\ell} w_{1}^{-\ell}\right)
 \cdot 0 = 0,
 \end{multline*}
 and we have a violation of the associative law.

 If $F$, $G$, and $H$ are formal power series such that $(F+G)H$ is
defined while $FH$ or $GH$ are not, then the distributive property
does not make sense. E.g., for $F=-G=k_{Sz}$ and
$H=\sum_{\ell=0}^\infty z_{1}^{\ell} w_{1}^{-\ell}$, we have
$(F+G)H=0$ while the product
$$FH= \left(\sum_{n \in \bbZ^{d}} z^{n} w^{-n}\right) \left(
\sum_{\ell=0}^{\infty} z_{1}^{\ell} w_{1}^{-\ell} \right)=
 \sum_{m\in\mathbb{Z}^d}\left(\sum_{\ell=0}^\infty 1\right) z^{m}
w^{-m}
 $$
 is not defined. On the other hand, if both $FH$ and $GH$ are
defined, then so is $(F+G)H$, and $(F+G)H=FH+GH$. For simplicity,
we will show this for the series in $z^{\pm 1}$ only (the proof
for the case of series in $z^{\pm 1},w^{\pm 1}$ is analogous),
$F=F(z)$, $G=G(z)$, $H=H(z)$:
 \begin{multline*}
F(z)H(z)+G(z)H(z)=\sum_{n\in\mathbb{Z}^d}\left(\sum_{\ell\in\mathbb{Z}^d}
F_{n-\ell}H_\ell\right)z^n+
\sum_{n\in\mathbb{Z}^d}\left(\sum_{\ell\in\mathbb{Z}^d}G_{n-\ell}H_\ell\right)
z^n\\
=\sum_{n\in\mathbb{Z}^d}\left(\sum_{\ell\in\mathbb{Z}^d}F_{n-\ell}H_\ell+
\sum_{\ell\in\mathbb{Z}^d}G_{n-\ell}H_\ell\right)z^n
=\sum_{n\in\mathbb{Z}^d}\left(\sum_{\ell\in\mathbb{Z}^d}(F_{n-\ell}+G_{n-\ell})
H_\ell\right)z^n\\
 =[F(z)+G(z)]H(z).
 \end{multline*}
\end{remark}

    A well known result in the classical
    reproducing-kernel-Hilbert-space
    setting is:  {\em if the $\cL(\cF, \cF')$-valued function
    $S$ is such that $S(z) f(z) \in \cH(K')$ for each $f \in \cH(K)$
    (so $M_{S} \colon f(z) \mapsto S(z) \cdot f(z)$ is well defined
    as an operator from $\cH(K)$ to $\cH(K')$), then necessarily
    $M_{S}$ is bounded (i.e., $\| M_{S}\|_{\cL(\cH(K), \cH(K'))} < \infty$).}
    Indeed, it is easily verified that $M_{S}$ is closed and then the
    result follows from the Closed Graph Theorem (see
    e.g.~\cite[page 51]{Rudin}).  For the FRKHS setting, the parallel result
    holds, but under the hypothesis that the series $\sum_{\ell}
    S_{n-\ell} f_{\ell}$ defining the coefficients of $S(z) \cdot
    f(z)$ converges in the norm (not just the weak) topology of
    $\cF'$, and as a consequence of two applications of the
    Banach-Steinhaus Theorem rather than of the Closed Graph
    Theorem.
    Also it is convenient to assume that the coefficient
    space $\cF'$ is finite-dimensional, although this hypothesis can
    be weakened (see Remark \ref{R:auto-bded} below).
    In detail, we have the following result.

    \begin{theorem}   \label{T:auto-bded}
    Let $K(z,w) \in \cL(\cF)[[ z^{\pm 1}, w^{\pm 1}]]$, $K'(z,w)
    \in \cL(\cF')[[z^{\pm 1} w^{ \pm 1} ]]$ be formal positive
    kernels, and let $S(z) \in \cL(\cF, \cF')[[z,w]]$.  Assume
    that,
    for all $n$ and for all $f \in \cH(K)$, the infinite series
    $(Sf)_{\ell}: = \sum_{m \in {\mathbb Z}^{d}} S_{\ell - m} f_{m}$
    converges in the norm of $\cF'$, and that the resulting
    formal Laurent series $Sf: = \sum_{\ell \in {\mathbb Z}^{d}} (Sf)_{\ell}
    z^{\ell}$ is in $\cH(K')$.  Assume also that $\dim \cF' <
    \infty$. Then $M_{S} \colon f \mapsto S f$ is a bounded
    linear operator form $\cH(K)$ into $\cH(K')$.
    \end{theorem}

    For the proof we require the following lemma.

    \begin{lemma}  \label{L:approx}
Suppose that $K(z,w) \in \cL(\cF))[[z^{\pm 1}, w^{\pm 1}]]$ is a
formal positive kernel with associated FRKHS $\cH(K)$ and that
$f(z) = \sum_{n \in {\mathbb Z}^{d}} f_{n} z^{n}$ is an element of
$\cH(K)$.  Then we can recover $f$ from the collection of
coefficients $\{f_{\ell}\}_{\ell \in {\mathbb Z}^{d}}$ according
to the following approximation scheme:
\begin{equation}  \label{approx}
    f = \text{$\cH(K)$--}\lim_{L \to \infty} \operatorname{row} [
    \pi_{\ell}^{*}]_{\ell \colon |\ell| \le L} {\mathbb K}_{L}^{[-1]}
    \operatorname{col} [f_{\ell}]_{\ell \colon |\ell| \le L},
\end{equation}
where $\pi_{\ell} \colon f \mapsto f_{\ell}$ is the coefficient
evaluation map with adjoint
$$
  \pi_{\ell}^{*} \colon u \mapsto K_{\ell}(z) u = \sum_{n \in
  {\mathbb Z}^{d}} (K_{n,\ell}u) z^{n},
$$
where ${\mathbb K}_{L}$ is the operator on $\cF^{N_{L}}$ (where we
set $N_{L} = \#\{ \ell \in {\mathbb Z}^{d} \colon |\ell| \le L\}$
and the notation $|\ell|$ is as in \eqref{def|ell|}) with block
matrix decomposition given by
$$
{\mathbb K}_{L}:= \left[ K_{\ell', \ell} \right]_{|\ell'|, |\ell|
\le L},
$$
and where ${\mathbb K}_{L}^{[-1]}$ is the (possibly unbounded)
inverse of the injective selfadjoint operator ${\mathbb
K}_{L}|_{({\rm ker} {\mathbb K}_{L})^{\perp}}$.
\end{lemma}

\begin{proof}
    Let us set $\cH_{L} = \{ h \in \cH(K) \colon h_{\ell} = 0 \text{ for
} |\ell| \le L\}$ (where the notation $|\ell|$ is as in
\eqref{def|ell|}). Then the orthogonal complement
$\cH_{L}^{\perp}$ of $\cH_{L}$ is given by the span of the kernel
functions
\begin{equation}   \label{HLperp}
 \cH_{L}^{\perp} = \operatorname{span} \{ K_{\ell}(z) u \colon |\ell|
 \le L,\, u \in \cF\}.
\end{equation}
For $f \in \cH(K)$, set $g_{L} = P_{L} f$ where $P_{L} \colon
\cH(K) \to \cH_{L^{\perp}}$ is the orthogonal projection.  Then $f
- g_{L} \in \cH_{L}$ and consequently $(f)_{\ell} =
(g_{L})_{\ell}$ for $|\ell| \le L$.  Notice that the sequence of
subspaces $\{ \cH_{L}\}_{L = 1,2, \dots}$ is decreasing
($\cH_{L+1} \subset \cH_{L}$) with trivial intersection
($\bigcap_{L \ge 1} \cH_{L} = \{0\}$).  It follows that the
associated sequence of orthogonal projections $P_{L} \colon \cH(K)
\to \cH_{L}^{\perp}$ is increasing with strong limit equal to the
identity operator:
$$
   \operatorname{strong} \lim_{L \to \infty} P_{L} = I_{\cH(K)}.
$$
Let us set $N_{L} = \#\{ \ell \in {\mathbb Z}^{d} \colon |\ell|
\le L\}$ and define a block $N_{L} \times N_{L}$ matrix ${\mathbb
K}_{L}$ by
$$
    {\mathbb K}_{L}: = \left[ {\mathbb K}_{\ell', \ell}
    \right]_{|\ell'|, |\ell| \le L} \colon \cF^{\oplus N_{L}} \to
    \cF^{\oplus N_{L}}.
$$
We know as a consequence of the description of $\cH_{L}^{\perp}$
in \eqref{HLperp} that the vector $g_{L}$ has a presentation of
the form
$$
   g_{L} = \sum_{\ell \colon |\ell| \le L} K_{\ell}(z) u_{\ell}
$$
for some choice of vectors $u_{\ell} \in \cF$.  To solve for the
vectors $u_{\ell}$ (these vectors also depend on the choice of $L$
but we suppress this fact in the notation for simplicity),  we
note that, for any vector $v \in \cF$
\begin{align*}
    \langle f_{\ell'}, v \rangle_{\cF} & = \langle f, K_{\ell'}(z) v
    \rangle_{\cH(K)} \\
    & = \langle g_{L}, K_{\ell'}(z v
    \rangle_{\cH(K)} \\
    & = \left\langle \sum_{|\ell| \le L} K_{\ell}(z) u_{\ell},
    K_{\ell'}(z) v \right\rangle_{\cH(K)} \\
    & = \sum_{|\ell|\le L} \langle K_{\ell', \ell} u_{\ell}, v
    \rangle_{\cF}
\end{align*}
from which we deduce that
$$
  f_{\ell'} = \sum_{\ell \colon |\ell| \le L} K_{\ell', \ell}
  u_{\ell}.
$$
It thus follows that $g_{L} \in \operatorname{im} {\mathbb K}_{L}
\subset \cF^{\oplus N_{L}}$.  Let ${\mathbb K}_{L}^{[-1]} \colon
\cF_{L} \to \cF_{L} \subset \cF^{\oplus N_{L}}$ be the inverse of
${\mathbb K}_{L}|_{(\operatorname{ker} {\mathbb K}_{L})^{\perp}}$.
We see that $[u_{\ell}]_{|\ell| \le L} = {\mathbb K}_{L}^{[-1]}
[f_{\ell}]_{|\ell| \le L}$.  Therefore
\begin{align*}
    f & = \lim_{L \to \infty} P_{L} f = \lim_{L \to \infty}
    \sum_{\ell \colon |\ell| \le L} K_{\ell}(z) u_{\ell} \\
    & = \operatorname{row} [ \pi_{\ell}^{*}]_{|\ell| \le L} \cdot
    {\mathbb K}_{L}^{[-1]} \cdot [ f_{\ell} ]_{|\ell| \le L}
\end{align*}
and the formula \eqref{approx} follows.
\end{proof}

    \begin{proof}[Proof of Theorem \ref{T:auto-bded}]
    Notice that
$$
  (S f)_{\ell} = \lim_{M \to \infty} \sum_{m \colon |m| \le M} S_{\ell
  -m} f_{m} = \sum_{m \colon |m| \le M} \left( S_{\ell -m} \circ
  \pi_{m}\right)(f)
$$
where $\pi_{m} \colon \cH(K) \to \cF$ is the $m$-th Laurent
coefficient evaluation functional, assumed to be norm-continuous.
A first application of the Banach-Steinhaus theorem (see
\cite[Chapter III Section 14]{Conway}) gives us that the map
$$
\pi_{S, \ell} \colon f \mapsto (Sf)_{\ell}
$$
is bounded as a linear operator from $\cH(K)$ to $\cF'$ for each
$\ell \in {\mathbb Z}$.

We next apply Lemma \ref{L:approx} to $Sf \in \cH(K')$; thus we
get
\begin{align}
 Sf  & = \lim_{L \to \infty} \operatorname{row} [ \pi^{\prime *}_{\ell}
 ]_{|\ell| \le L} \cdot  {\mathbb K}^{\prime [-1]}_{L} \cdot
 \operatorname{col} [(Sf)_{\ell}]_{|\ell| \le L}   \notag \\
 & = \lim_{L \to \infty} \operatorname{row} [ \pi^{\prime *}_{\ell}
 ]_{|\ell| \le L} \cdot  {\mathbb K}^{\prime [-1]}_{L} \cdot
 \operatorname{col} \left[ \pi_{S,\ell}(f) \right]_{|\ell| \le L}
 \label{compmap}
\end{align}
with the obvious adaptation of the notation. The next key
observation is that the map $f \mapsto [f_{\ell}]_{|\ell| \le L}$
is bounded for each $L = 1, 2, \dots$; as we are assuming that
$\dim \cF' < \infty$, we are guaranteed that ${\mathbb K}^{ \prime
[-1]}_{L}$ is bounded for each $L$. Hence also the composition
$\operatorname{row} [ \pi_{\ell}^{*}]_{\ell \le L} \cdot {\mathbb
K}^{\prime [-1]}_{L}$ is bounded from $\cF^{\prime N_{L}}$ to
$\cH(K')$ for each $L$.  Finally, by the first application of the
Banach-Steinhaus theorem in the first paragraph of the proof, we
know that the map $f \mapsto \left[ \pi_{S, \ell}(f)
\right]_{|\ell| \le L}$ is bounded from $\cH(K)$ to $\cF^{\prime
N_{L}}$, and hence the composite map in \eqref{compmap} is bounded
for each $L$. A second application of the Banach-Steinhaus Theorem
now guarantees us that the map $M_{S} \colon f(z) \mapsto S(z)
f(z)$ is bounded from $\cH(K)$ into $\cH(K')$ as wanted.
\end{proof}

\begin{remark}  \label{R:auto-bded}
    A careful analysis of the proof of Theorem \ref{T:auto-bded} shows
that the conclusion still holds if one replaces the hypothesis
that $\cF'$ be finite-dimensional with the alternate hypothesis
\begin{itemize}
\item {\em the subspace $\cF'_{L} \subset \cF^{\prime \oplus
N_{L}}$ consisting of all vectors $[u']_{\ell}$ in $\cF^{\prime
\oplus N_{L}}$ of the form $[(Sf)_{\ell}]_{|\ell| \le L}$ for some
$f \in \cH_{L}^{\perp}$ is a closed subspace of $\cF'_{L}$.}
\end{itemize}
Note that this condition is automatic in case $\cF'$ itself is
finite-dimensional.

\end{remark}

We shall be particularly interested in the case where $\|M_{S}\|
\le 1$; in this case we write $S \in \cB \cM(K,K')$. We have the
following characterization of when a given $S$ is in $\cB
\cM(K,K')$ or is a coisometry.  We use the convention
\eqref{convention} with respect to both sets of variables $z =
(z_{1}^{\pm 1}, \dots, z_{d}^{\pm 1})$ and $w = (w_{1}^{\pm 1},
\dots, w_{d} ^{\pm 1})$:  if $S(z) = \sum_{n \in \bbZ^{d}} S_{n}
z^{n} \in \cL(\cF, \cF')[[ z^{\pm 1}]]$, then we define $S(z)^{*}
\in \cL(\cF', \cF)[[ z^{\pm 1}]]$ by
$$
  S(z)^{*} = \sum_{n \in \bbZ^{d}} S_{n}^{*} z^{-n}.
$$

\begin{proposition}  \label{P:multipliers}
    Suppose that $K,K'$ are two positive kernels with coefficients
    in $\cL(\cF)$ and $\cL(\cF')$ respectively and with respective
    associated FRKHSs $\cH(K)$ and $\cH(K')$, and suppose that $S$ is
    a formal power series in $\cL(\cF, \cF')[[z^{\pm 1}]]$.
    Then the following holds:
    \begin{enumerate}
    \item $S \in \cB \cM(K,K')$ if and only if $S(z) f(z)$ is a
    well-defined power series (i.e., weak $\sum_{\ell \in {\mathbb Z}^d}
    S_{m-\ell} f_\ell$ exists in $\cF'$ for all $m \in {\mathbb Z}^d$)
for each $f \in \cH(K)$ and hence also
    $S(z)K(z,w) S(w)^{*}$ is a well-defined formal power series as in
    \eqref{eq:sum_carefully-1}--\eqref{eq:sum_carefully-2}, and
moreover $K_{S} \pos 0$, where
   \begin{equation}  \label{KSdef}
   K_{S}(z,w):= K'(z,w) - S(z)K(z,w) S(w)^{*} \in
   \cL(\cF')[[z^{\pm 1},w^{\pm 1}]].
   \end{equation}
   In this case, $M_{S}$ is a coisometry if and only if $K_{S}(z,w) =
   0$.

   \item
   $S^{*} \in \cB \cM(K',K)$ if and only if $S(z)^* f(z)$ is a
well-defined power series for each $f \in \cH(K')$  and hence also
   $S(z)^{*}K'(z,w) S(w)$ is a well-defined formal power series
similarly to
    \eqref{eq:sum_carefully-1}--\eqref{eq:sum_carefully-2}, and
moreover $K_{S^{*}} \pos 0$, where
     \begin{equation*}
      K_{S^{*}}(z,w):= K(z,w) - S(z)^{*}K'(z,w) S(w) \in
      \cL(\cF)[[z^{\pm 1},w^{\pm 1}]].
     \end{equation*}
      In this case, $M_{S^{*}}$ is a coisometry if and only if
$K_{S^{*}}(z,w) =
      0$.
   \end{enumerate}
  \end{proposition}

  \begin{proof}
      Suppose that $S \in \cM(K,K')$.
      By Proposition \ref{P:convergences}, $S(z)K(z,w) S(w)^{*}$ is a well-defined
formal power series as in
      \eqref{eq:sum_carefully-1}--\eqref{eq:sum_carefully-2}. For all
$f \in {\mathcal H}(K)$
\begin{align*}
  & \langle f(z), (M_{S}^{*} \otimes I_{\bbC[[w^{\pm 1}]]}) K'(z,w)
    u' \rangle_{\cH(K) \times \cH(K)[[w^{\pm 1}]]} \\
  & \qquad \qquad =  \langle S(z) f(z), K'(z,w) u' \rangle_{\cH(K') \times
    \cH(K')[[w^{\pm 1}]]} \\
    & \qquad \qquad = \langle S(w) f(w), u' \rangle_{\cF'[[w^{\pm 1}]] \times \cF'}
\end{align*}
 where we used the reproducing kernel property of $K'$.
 Furthermore,
\begin{align} \label{eq:kernelcalc1}
   \langle S(w) f(w), u' \rangle_{\cF'[[w^{\pm 1}]] \times \cF'}
&  = \langle f(w), S(w)^{*} u' \rangle_{\cF[[w^{\pm 1}]] \times
    \cF[[w^{\pm 1}]]} \notag\\
    &  =
    \langle f(z), K(z,w) S(w)^{*}u' \rangle_{\cH(K) \times
    \cH(K)[[w^{\pm 1}]]}.
\end{align}
Here we used Proposition \ref{P:repr-ext} for the last equality.
Indeed,  as we saw in the proof of Proposition
\ref{P:convergences}, if $K(z,w)=H(z)H(w)^*$ as in Theorem
\ref{T:FRKHS}(2), then $H(w)^*S(w)^*u'\in\mathcal{H}[[w^{\pm 1}]]$
for every $u'\in\mathcal{F}'$, with the series
$\sum_{\ell'\in\mathbb{Z}^d}H_{\ell'}^*S^*_{m-\ell'}u'=\sum_{\ell'\in\mathbb{Z}^d}H_{m-\ell'}^*S^*_{\ell'}u'$
convergent weakly in $\mathcal{H}$ for every $m\in\mathbb{Z}^d$.
Then
$$
\sum_{\ell'\in\mathbb{Z}^d}K_{m-\ell'}(z)S^*_{\ell'}u'=
\sum_{\ell'\in\mathbb{Z}^d}H(z)H_{m-\ell'}^*S^*_{\ell'}u'
$$
converges weakly in $\mathcal{H}(K)$ and we may apply Proposition
\ref{P:repr-ext} with $u(w)=S(w)^*u'\in\mathcal{F}[[w^{\pm 1}]]$.
It follows that
  \begin{equation}  \label{kernelid}
      (M_{S}^{*} \otimes I_{\bbC[[w^{\pm 1}]]}) K'(z,w) u' =
     K(z,w) S(w)^{*} u'
   \end{equation}
 for all $u' \in \cF'$. Note that the action of $M_{S}^{*} \otimes
 I_{{\mathbb C}[[w^{\pm 1}]]}$ can be expanded out in terms of
 coefficients of $w^{m}$ as
 $$
 M_{S}^{*} \otimes I_{{\mathbb C}[[w^{\pm 1}]]} \colon \sum_{n}
 K'_{n}(z) u' w^{-n} \mapsto \sum_{n} \left( \sum_{\ell} K_{n-\ell}(z)
 S_{\ell}^{*} u' \right) w^{-n}.
 $$
 and hence we see that
 \begin{equation}   \label{defT0}
     M_{S}^{*} \colon K'_{n}(z) u' \mapsto \sum_{\ell} K_{n-\ell}(z)
     S_{\ell}^{*} u'.
 \end{equation}
 From the formula \eqref{defT0} we see that the fact that $M_{S}^{*}
 \colon \cH(K') \to \cH(K)$ has $\| M_{S} \| \le 1$ is equivalent to
 the expression
 \begin{align}
& \langle K'_n u'_n, K'_m u'_m \rangle_{\cF'} - \langle
\sum_{\ell} K_{n - \ell} S_{\ell}^{*} u_N', \sum_{\ell'} K_{m-
\ell'}
S^{*}_{\ell'} u'_m \rangle_{\cF} \notag \\
& \quad = \langle (K'_{mn} -  \sum_{\ell, \ell'} S_{\ell'}
K_{m-\ell', n-\ell} S_\ell^*) u'_n, u_m' \rangle_{\cF}.
\label{beastker}
\end{align}
being positive as a kernel on $({\mathbb Z}^d \times \cF')$ (here
we view $(n,u'_n)$ and $(m, u'_m)$ as generic elements of
${\mathbb Z}^d \times \cF'$).
 From the definitions we see that this expression is positive when
viewed as a kernel on ${\mathbb Z}^d \times \cF$ if and only if $
K_{S}(z,w) = K'(z,w) - S(z) K(z,e) S(w)^* \ge 0$ is a positive
formal kernel.  Note that $M_{S}$ being coisometric means that
$M_{S}^{*}$ is isometric, or \eqref{beastker} is the zero kernel;
this in turn corresponds to $K_{S}$ being the zero kernel. This
completes the proof of necessity in part (1) of Proposition
\ref{P:multipliers}.

To prove sufficiency in part (1) of Proposition
\ref{P:multipliers}, we now assume that $S(z) f(z)$ is well
defined for each $f \in \cH(K)$, so weak $\sum_{\ell} S_{m-\ell}
H_{\ell}$ exists for each $m \in {\mathbb Z}^{d}$ (where $K_{m,n}
= H_{m} H_{n}^{*}$ is the Kolmogorov decomposition for $K(z,w)$)
and hence also weak $\sum_{\ell} H^{*}_{\ell} S^{*}_{m - \ell}$
exists.  The formula \eqref{defT0} for the action of $M_{S}^{*}$
on kernel functions suggests that we define an operator $T_{0}$
mapping kernel functions  of $\cH(K')$ into $\cH(K)$ according the
formula in \eqref{defT0}:
\begin{equation}   \label{defT0'}
    T_{0} \colon K'_{n}(z) u' \mapsto \sum_{\ell} K_{n-\ell}(z)
    S_{\ell}^{*} u'.
\end{equation}
The computation in the first part of the proof (read backwards)
tells us that the assumption $K_{S} \ge 0$ implies that $T_{0}$
extends uniquely to a contraction operator from $\cH(K')$ into
$\cH(K)$. Furthermore the action of $T: = T_{0} \otimes
I_{{\mathbb C}[[w^{\pm 1}]]}$ on kernel functions $K'(z,w)u' \in
\cH(K')[[w^{\pm 1}]]$ is given as in formula \eqref{kernelid}:
\begin{equation}  \label{defT}
T \colon K'(z,w) u' \mapsto K(z,w) S(w)^{*}u'.
\end{equation}
>From this formula \eqref{defT} for $T$ it is easy to compute a
formula for the action of $T_{0}^{*}$: for $f \in \cH(K)$ and $u'
\in \cF'$ we have, using the reproducing kernel property of $K'$:
\begin{align*}
   & \langle (T_{0}^{*} f)(w), u' \rangle_{\cF'[w^{\pm 1}]] \times
   \cF'} \\
   & \quad = \langle T_{0}^{*} f,
    K'(z,w) u' \rangle_{\cH(K') \times \cH(K') [[w^{\pm 1}]]} \\
    & = \langle f, T K'(z,w) u' \rangle_{\cH(K) \times \cH(K)
    [[w^{\pm 1}]]}  \\
    & = \langle f, K(z,w) S(w)^{*} u' \rangle_{\cH(K) \times \cH(K)
     [[w^{\pm 1}]]}.
\end{align*}
>From \eqref{eq:kernelcalc1} read backwards, we finally conclude
that
$$
\langle (T_{0}^{*}f)(w), u' \rangle_{\cF'[[w^{\pm}]] \times \cF'}
= \langle S(w) f(w), u' \rangle_{\cF'[[w^{\pm 1}]] \times \cF'}
$$
from which we conclude that $T_{0}^{*} = M_{S} \colon f(z) \mapsto
S(z) f(z)$.  As we have already identified $T_{0}$ as mapping
$\cH(K')$ contractively into $\cH(K)$, it follows that $T_{0}^{*}
= M_{S} \colon \cH(K) \to \cH(K')$ is well defined with $\|M_{S}\|
\le 1$.

Part (2) amounts to part (1) applied to $S(z)^{*}$ in place of
  $S(z)$.
 \end{proof}

 \begin{remark}  \label{R:expiate_old_sins}
One might think that the kernel $K_{S}(z,w)$ given by
\eqref{KSdef} being well-defined and positive would be sufficient
for the results of Proposition \ref{P:multipliers} to be valid.
That this is not the case can be seen from the following example.

We take as our Hilbert space $\cX = \ell^{2}$ with standard
orthonormal basis $\{e_{n} \colon n = 0,1,2, \dots \}$ and we set
$P_{n}$ equal to the orthogonal projection onto the span of the
basis vector $e_{n}$. We define a formal power series $S(z) =
\sum_{n \in {\mathbb Z}} S_{n} z^{n}$ and a formal positive kernel
$K(z,w) = \sum_{n,m \in {\mathbb Z}} H_{n} H_{m}^{*} z^{n} w^{-m}$
by
$$
  S_{n}  = I_{\cX}, \quad H_{n}  = \begin{cases}  n^{2} P_{n} -
  (n-1)^{2} P_{n-1} &\text{for } n > 0, \\
  0 &\text{for } n \le 0. \end{cases}
$$
Then
$$
  \sum_{\ell \colon |\ell| \le L} H_{\ell} = \sum_{\ell = 1}^{L}
  [ \ell^{2}P_{\ell} - (\ell-1)^{2} P_{\ell-1}] = L^{2} P_{L}
$$
and hence
$$\lim_{L \to \infty} \sum_{\ell \colon |\ell| \le L} S_{n-\ell} H_{\ell} = \lim_{L \to
\infty} L^{2} P_{L} \text{ is not weakly convergent.}
$$
However
$$ \text{weak} \sum_{\ell} \left( \text{weak} \sum_{\ell'} S_{n-\ell}
H_{\ell} H^{*}_{\ell'} S^{*}_{m-\ell'} \right) = \text{weak}
\sum_{\ell} \left( \text{weak} \lim_{L' \to \infty} H_{\ell}
(L^{\prime 2} P_{L'} \right) = 0
$$
and similarly
$$
\text{weak} \sum_{\ell'} \left( \sum_{\ell} S_{n-\ell} H_{\ell}
H^{*}_{\ell'} S^{*}_{m - \ell'} \right) = 0.
$$
Moreover the kernel
$$ K_{S}(z,w) = K(z,w) - S(z) K(z,w) S(z)^{*} = K(z,w) \ge 0.
$$
Thus the kernel $K_{S}(z,w)$ is well defined and positive, but the
multiplier $M_{S} \colon f(z) \mapsto S(z) f(z)$ is not well
defined (much less bounded with norm at most 1) on $\cH(K)$.
\end{remark}

    \subsection{Formal reproducing kernel Hilbert spaces constructed
     from others} \label{S:FRKHSs-FRKSH}

    There are a couple of ways to construct more complicated FRKHSs
    from a stock of given (simpler) FRKHSs.  We mention what we shall call
    {\em lifted-norm FRKHSs} and {\em pullback FRKHSs} described in
    the next two subsections.

    \subsubsection{Lifted-norm FRKHSs}   \label{S:liftednorm}

    Suppose that $W = W^{*}$ is a self-adjoint operator on a FRKHS
       $\cH(K)$ which is positive semidefinite.  We
       define a space $\cH_{W}^{\ell}$ (the {\em lifted-norm FRKHS
       associated with $W$}) via the following recipe.  We take a
dense
       subset of $\cH^{\ell}_{W}$ to be
       $$
       \operatorname{im}W = \{ W  f \colon f \in \cH(K)\}
       $$
       with inner product given by
      $$
    \langle (W f)(z), (W g)(z) \rangle_{\cH^{\ell}_{W}} =
    \langle (Wf)(z), g(z) \rangle_{\cH(K)}
       $$
       The completion of $\cH^{\ell}_{W}$ can be identified as
       $\operatorname{im} ({W}^{1/2})$ with pullback inner product
       \begin{equation}  \label{Winnerprod}
     \langle  {W}^{1/2}f, {W}^{1/2} g \rangle_{\cH^{\ell}_{W}} =
     \langle Qf, g \rangle_{\cH(K)}
       \end{equation}
       where $Q \colon \cH(K) \to (\ker {W})^{\perp}$  is the
       orthogonal projection.  The following proposition summarizes
the
       properties of $\cH^\ell_{W}$ which we shall need in the sequel.

       \begin{proposition}  \label{P:liftednorm}
       Suppose that $W \in \cL(\cH(K))$ is
       positive semidefinite and the space
       $\cH^{\ell}_{W} = \operatorname{im} ({W}^{1/2})$  is defined
       with inner product given by  \eqref{Winnerprod}.   Then
$\cH^\ell_{W}$
       is a FRKHS with formal reproducing kernel $K^\ell_{W}(z,w) \in
       \cL(\cF)[[z^{\pm 1}, w^{\pm 1}]]$ given by
       $$
    K^{\ell}_{W}(z,w) = ((W \otimes I_{\bbC[[w^{\pm 1}]]}) K)(z,w).
       $$
       \end{proposition}

       \begin{proof}
       If $f_{n} = W^{\frac{1}{2}} g_{n}$ is a Cauchy sequence in
       $\cH^{\ell}_{W}$, then $Qg_{n}$ is Cauchy in $\cH(K)$ and
       hence converges to a $g \in \cH(K)$.  It then follows
       that $\{f_{n}\}$ converges to $W^{\frac{1}{2}}Qg =
       W^{\frac{1}{2}}g$ in $\cH^{\ell}_{W}$.  In this way we see
       that $\cH^{\ell}_{W}$ is complete.

 By construction $((W \otimes I_{\bbC[[w^{\pm 1}]]})K)(z,w) u \in
\cH^{\ell}_{W}[[w^{\pm 1}]]$ for
 each $u \in \cF$ and each $f \in \cH^{\ell}_{W}$ has the form
$W^{\frac{1}{2}} g$ for a $g \in \cH(K)$. We now compute
\begin{align*}
   & \langle f(z),  ((W \otimes I_{{\mathbb C}[[w^{\pm 1}]]})
    K)(z,w) u  \rangle_{\cH^{\ell}_{W} \times \cH^{\ell}_{W}[[w^{\pm
1}]]}  \\
    & \quad = \langle(W^{\frac{1}{2}}g)(z), ((W \otimes I_{{\mathbb
C}[[w^{\pm 1}]]})
 K)(z,w) u  \rangle_{\cH^{\ell}_{W} \times \cH^{\ell}_{W}[[w^{\pm
1}]]} \\
    &\quad  = \langle \langle (W^{\frac{1}{2}} g)(z), K(z,w) u
    \rangle_{\cH(K) \times \cH(K)[[w^{\pm}]]} \\
    & \quad = \langle f(z), K(z, w) u \rangle_{\cH(K) \times \cH(K)[[
    w^{ \pm 1}]]}
\end{align*}
and we see that indeed  $K^{\ell}_{W}(z,w):= (W \otimes
     I_{\bbC[[w^{\pm 1}]]}K)(z,w)$ plays
     the role of the formal reproducing kernel for the space
     $H^{\ell}_{W}$.

     \end{proof}

     \subsubsection{Pullback FRKHSs}  \label{S:pullback}
     Suppose now that $B \in \cM(K,K')$ is a bounded multiplier
     between two FRKHSs $\cH(K)$ and $\cH(K')$.
     We define the {\em pullback RKHS associated with $B$}
     by
     $$
      \cH^{p}_{B} = \operatorname{im} M_B = \{ B(z) f(z) \colon f \in
      \cH(K) \} \subset \cH(K')
      $$
      with inner product
      \begin{equation}  \label{Binnerprod}
      \langle B(z) f(z), B(z) g(z) \rangle_{\cH^{p}_{B}} = \langle
      (Qf, g \rangle_{\cH(K)}
      \end{equation}
      where $Q \in \cL(\cH(K))$ is the orthogonal projection
      onto $(\operatorname{ker} M_{B})^{\perp}$.
      The following Proposition summarizes the main features about
      pullback FRKHSs.

      \begin{proposition} \label{P:pullback}
      Suppose that $B \in \cM(K,K')$ and that
      $\cH^{p}_{B}$ is the associated pullback space with inner product
      given by \eqref{Binnerprod}.  Then
      \begin{enumerate}
\item $\cH^{p}_{B}$ is a FRKHS with
      reproducing kernel
      $$
      K^{p}_{B}(z,w) = B(z) K(z,w) B(w)^{*}.
      $$
 \item If $B \in \cB \cM(K,K')$ then $\cH^{p}_{B}$ is contractively
      included in $\cH(K')$:
      $$
       \| Bf \|_{\cH(K')} \le \|Bf\|_{\cH^{p}_{B}} \text{ for all } Bf
       \in \cH^{p}_{B}
      $$
      with equality for all $Bf \in \cH^{p}_{B}$ if and only if $M_{B}
      \colon \cH(K) \to \cH(K')$ is a coisometry.
  \item  If $B \in \mathcal{BM}(K, K')$, then the Brangesian
  complementary space $\left(\cH^{p}_{B}\right)^{\perp \rm{dB}}$
  defined as the space of all $f \in \cH(K)$ with finite
  $(\cH^{p}_{B})^{\perp \rm{dB}}$-norm:
 $$ \|f\|^{2}_{(\cH^{p}_{B})^{\perp \rm{dB}}} : = \sup \{ \| f +
 B g \|^{2}_{\cH(K')} - \| g\|^{2}_{\cH(K)} \colon g \in \cH(K) \} <
 \infty
 $$
 is the lifted norm space $\cH^{\ell}_{I - M_{B} M_{B}^{*}}$ and is
 equipped with the reproducing kernel
 $$
    K^{\ell}_{I - M_{B} M_{B}^{*}}(z,w) = K'(z,w) - B(z) K(z,w)
    B(w)^{*}.
 $$
\end{enumerate}
\end{proposition}

\begin{proof} To prove the first statement (1),  notice that by Proposition
    \ref{P:repr-ext}, the product $B(z) K(z,w) B(w)^{*} u'$ is well defined and
    is in $\cH^{p}_{B}[[w^{\pm 1}]]$.  Calculate
 \begin{align*}
  & \langle B(z) f(z), B(z) K(z,w) B(w)^{*} u' \rangle_{\cH^{p}_{B} \times
  \cH^{p}_{B}[[ w^{\pm 1}]]}  \\
  & \qquad \qquad = \langle Q f(z), K(z,w) B(w)^{*} u'
  \rangle_{\cH(K) \times \cH(K)[[w^{\pm 1}]]} \\
  & \qquad \qquad = \langle B(w) f(w), u' \rangle_{\cF'[[w^{\pm 1}]]
  \times \cF'}
 \end{align*}
 and  the identity $K^{p})_{B}(z,w) = B(z) K(z,w) B(w)^{*}$ follows.

Notice next that statement (2) is immediate from the definitions.

For the proof of the first part of statement (3), we refer the
reader to Sarason's book \cite[Chapter 1]{Sarason}  for a general
operator-theoretic treatment which handles the FRKHS setup here as
a particular case (see also \cite[Theorem 2.5]{FRKHS}). The
statement about the reproducing kernels follows from Proposition
\ref{P:liftednorm} and the identity \eqref{kernelid}.
\end{proof}

\begin{remark}   \label{R:equiv-prop}
    Proposition \ref{P:pullback} (1) is essentially equivalent to
    Proposition \ref{P:multipliers}.  Indeed, assume Proposition
    \ref{P:multipliers}.  Let $B \in \cM(K,K')$.  Define a kernel
    $\widetilde K$ by $\widetilde
    K(z,w) = B(z) K(z,w) B(w)^{*}$.  It follows from Proposition
    \ref{P:multipliers} that $M_{B} \colon f(z) \mapsto B(z) f(z)$ is
    a coisometry from $\cH(K)$ onto $\cH(\widetilde K)$, or
    equivalently, $\cH(\widetilde K) = \cH^{p}_{B}$.  Conversely,
    assume Proposition \ref{P:pullback} (1).  Let $S$ be such that
    $S(z) f(z)$ is well defined for all $f \in \cH(K)$.  Set
    $\widetilde K(z,w) = S(z) K(z,w)S(w)^{*}$.
    By Proposition \ref{P:pullback} (1), $\cH(\widetilde K) =
    \cH^{p}_{B}$.  If $K_{S} = K' - \widetilde K$ (see
    \eqref{KSdef}) is a positive kernel, then $\cH(\widetilde K)$ is
    contractively included in $\cH(K')$ (see \cite[Theorem
  2.5]{FRKHS}), meaning that $S \in \mathcal{BM}(K,K')$.
    \end{remark}

      \begin{remark}  \label{R:SK=KS}
Suppose that $B\in\mathcal{M}(K)$ ($=\mathcal{M}(K,K)$) and
$M_B=M_B^*\ge 0$. By \eqref{kernelid}, we have
        \begin{multline}
            B(z) K(z,w) = \Big((M_B\otimes I_{\mathbb{C}[[w^{\pm
1}]]})K\Big)(z,w)\\
 =\Big((M_B^*\otimes I_{\mathbb{C}[[w^{\pm
 1}]]})K\Big)(z,w)=K(z,w)B(w)^*,
 \label{SK=KS} \end{multline}
Then $W:= (M_B)^2=M_{B^2} \ge 0$ and from  the
          discussion above we see that the lifted norm space
          $\cH^{\ell}_{W}$ with kernel $K^{\ell}_{W}(z,w) = B(z)^2
K(z,w)$ is
          the same as the pullback space $\cH^{p}_{B}$ with kernel
          $K^{p}_{B}(z,w) = B(z) K(z,w) B(w)^{*}$.  The
          identity \eqref{SK=KS} explains why these two forms of the
kernel
          agree.
          \end{remark}

      \subsubsection{Overlapping spaces}  \label{S:overlapping}
      Suppose that we are given a finite or infinite countable
collection $\{\cH(K_{j})\colon j\in J\}$ of FRKHSs with the same
coefficient space $\cF$;
we adapt the convention \eqref{notation} by writing
\begin{equation} \label{notation'}
K_{j}(z,w) = \sum_{n,m \in {\mathbb Z}^{d}} [K_{j}]_{n,m} z^{n}
w^{-m}, \quad [ K_{j}]_{m}(z) = \sum_{n \in {\mathbb Z}^{d}}
[K_{j}]_{n,m} z^{n}.
\end{equation}
We let $K(z,w) =\sum_{j\in
J}K_{j}(z,w)$, where we assume that the series converges in
$\cL(\cF)[[z^{\pm 1},w^{\pm 1}]]$ coefficientwise in the weak operator topology, i.e., for every
$n,m\in\mathbb{Z}^d$ the series of coefficients $\sum_{j\in
J}[K_j]_{n,m}$ converges in the weak operator topology of $\cL(\cF)$\footnote{
Equivalently, the series converges in the strong operator topology.
Indeed, 
$\left[\begin{smallmatrix} [K_j]_{n,n} & [K_j]_{n,m} \\ [K_j]_{m,n} & [K_j]_{m,m} \end{smallmatrix}\right] \geq 0$,
and it is well known that for a series of nonnegative operators (or for an increasing sequence of selfadjoint operators)
on a Hilbert space, weak operator topology and strong operator topology convergence coincide.}. 
Then by Theorem \ref{T:FRKHS} $K(z,w)$ is also a positive kernel with associated
FRKHS $\cH(K)$;
to avoid confusion, for the kernel $K$ we use the more elaborate
notation
\begin{equation}  \label{notation''}
    K(z,w) = \sum_{n,m \in {\mathbb Z}^{d}} [K]_{n,m} z^{n} w^{-m},
    \quad  [K]_{m}(z) = \sum_{n \in {\mathbb Z}^{d}} [K]_{n,m} z^{n}.
\end{equation}
It is often
      of interest to understand the relation between $\cH(K)$ and the
      subsidiary spaces $\cH(K_{j})$ ($j\in J$).  

We introduce the factorizations
      \begin{equation}  \label{fact1}
       K_{j}(z,w) = H_{j}(z) H_{j}(w)^{*} \text{ for some } H_j(z) = \sum_{n\in{\mathbb Z}^d} [H_j]_n z^n \in
       \cL(\cH_{j}', \cF)[[z^{\pm 1}]]
      \end{equation}
      so that as in Theorem \ref{T:FRKHS} we know that $\cH(K_{j}) =
      \{H_{j}(z) h_{j} \colon h_{j} \in \cH_{j}' \}$. 
Since $[K_j]_{n,m} = [H_j]_n [H_j]_m^*$ we have that
$\sum_{j\in J} [H_j]_n [H_j]_m^*$ converges in the weak operator topology for all
$n,m\in\mathbb{Z}^d$, and in particular 
$\sum_{j\in J} \langle [H_j]_n [H_j]_n^* v, v \rangle$ converges 
for all $n\in\mathbb{Z}^d$ and all $v\in\cF$. It follows that $\col_{j\in J}\left[[H_j]_n^*\right]$
is a well defined bounded linear operator from $\cF$ to $\widehat\bigoplus_{j\in J}\cH_{j}'$
with adjoint
$$
[H]_n = \row_{j\in J}\left[[H_j]_n\right] \in \cL\Big(\widehat\bigoplus_{j\in J}\cH_{j}',\cF)
$$
given more precisely by $\widehat\bigoplus_{j \in J} h_j' \mapsto \text{weak} \sum_{j\in J} [H_j]_n h_j'$.
From
\eqref{fact1} we see now that $K(z,w)$ has the factorization
      \begin{equation*}
          K(z,w) =  H(z) H(w)^{*} \text{ where }
          H(z) = \row_{j\in J}[H_j(z)]\in\cL\Big(\widehat\bigoplus_{j\in J}
\cH_{j}', \cF\Big)[[ z^{\pm 1} ]].
       \end{equation*}
      Hence, again from Theorem \ref{T:FRKHS} we read off that $\cH(K)$
      can be identified as
      \begin{align*}
      \cH(K) & = \Big\{ H(z)h=\sum_{j\in J}H_{j}(z) h_{j} \colon
h=\widehat\bigoplus_{j\in J}h_{j} \in
      \widehat\bigoplus_{j\in J}\cH_{j}'\Big\} \\
      & =\Big\{ \sum_{j\in J}f_j\colon \widehat{\bigoplus_{j\in
J}}f_j\in\widehat{\bigoplus_{j\in J}}\cH(K_{j})\Big\},
      \end{align*}
where the series converge in $\cF[[z^{\pm 1}]]$ coefficientwise in weak topology.
      However,
      only in special circumstances is this sum identifiable
      with a direct sum since the sum map ${\mathbf s} \colon
      \widehat \bigoplus_{j\in J} \cH(K_{j}) \to \cH(K)$ given by
      \begin{equation}  \label{sum}
       {\mathbf s} \colon \widehat{\bigoplus_{j\in J}}f_j(z) \mapsto
\sum_{j\in J}f_j(z)
      \end{equation}
      may have a kernel.  In fact, again as a consequence of Theorem
      \ref{T:FRKHS}, we see that the norm on elements of $\cH(K)$ is
      given by
      \begin{align*}
\Big\|\sum_{j\in J}H_j(z)h_j\Big\|_{\mathcal{H}(K)}
&=\Big\|P_{(\ker\mathbf{s})^\perp}\widehat{\bigoplus_{j\in
J}}H_j(z)h_j\Big\|_{\widehat\bigoplus_{j\in J}\mathcal{H}(K_j)}\\
&=\Big\|P_{(\ker H(z))^\perp}\widehat\bigoplus_{j\in
J}h_j\Big\|_{\widehat\bigoplus_{j\in J}\mathcal{H}_j'}.
      \end{align*}
To quantify the overlapping of the spaces $\cH(K_{j})$ ($j\in J$),
in the spirit of the work of de Branges and Rovnyak \cite{dBR1,
dBR2} (see also \cite{Ball78}), we introduce the {\em overlapping
space} $\cL(\{K_{j}\colon j\in J\})$ defined by
\begin{equation}  \label{overlapping}
\cL(\{K_{j}\colon j\in J\})  = \ker {\mathbf s}
\end{equation}
with norm inherited from $\widehat\bigoplus_{j\in J} \cH(K_{j})$.
Then $\cL(\{K_{j}\colon j\in J\})$ is again a FRKHS whose
reproducing kernel $K_{\cL(\{K_{j}\colon j\in J\})}$ can be
computed explicitly in some special cases (see
\cite{dBR1,dBR2,Ball78, Cuntz-rep}). In any case we have the
following result:

      \begin{proposition} \label{P:overlapping}
          Suppose that $\{K_{j}\colon j\in J\}$ is a finite or infinite
countable
          collection of positive kernels with the common coefficient
space $\mathcal{F}$, and we set $K(z,w) =\sum_{j\in J}K_j(z,w)$
where we assume that the series converges coefficientwise in the weak operator topology in
$\mathcal{L(F)}[[z^{\pm 1},w^{\pm 1}]]$.
           Then:
      \begin{enumerate}
          \item The sum map ${\mathbf s}$ given by
          \eqref{sum} is a coisometry from $\widehat
          \bigoplus_{j\in J} \cH(K_{j})$ onto $\cH(K)$ with initial
          space ${\mathcal D}_{{\mathbf s}}$ given by
       \begin{equation*}
           {\mathcal D}_{{\mathbf s}} = \overline{\operatorname{span}}
           \left\{
\widehat{\bigoplus_{j\in J}}[K_{j}]_{m}(z) u \colon u \in
          \cF, \, m \in {\mathbb Z}^{d} \right\}.
    \end{equation*}

    \item
      There is a unitary identification map
      $$ \tau \colon \widehat{\bigoplus_{j\in J}} \cH(K_{j}) \to \cH(K)
      \widehat\oplus\cL(\{K_{j}\colon j\in J\})
      $$
       given by
      $$
       \tau \colon \widehat{ \bigoplus_{j\in J}} f_{j} \mapsto
        \begin{bmatrix} {\mathbf s} \\ P_{\ker
     {\mathbf s}} \end{bmatrix}
     \left( \widehat{ \bigoplus_{j\in J}} f_{j} \right) =
        \begin{bmatrix} \sum_{j\in J}f_{j}\\
        P_{\operatorname{ker} {\mathbf
        s}}\left( \widehat\bigoplus_{j\in J} f_{j} \right) \end{bmatrix}
      $$
      where the overlapping space $\cL(\{K_{j}\colon j\in J\})$ is
      given by \eqref{overlapping}.
      In particular, if $\cL(\{K_{j}\colon j\in J\}) = \{0\}$, then
      $$
       \cH(K) =
       \bigoplus_{j\in J} \cH(K_{j}).
      $$
      \end{enumerate}
      \end{proposition}

      \begin{proof}  In view of the discussion preceding the
        statement, we only check the assertion concerning the
        initial space.  Note that the computation
     \begin{align*}
      &   \Big\langle {\mathbf s} \Big(\widehat \bigoplus_{j\in J}
f_{j}\Big),
         K(\cdot, w) u \Big\rangle_{\cH(K) \times \cH(K)[[w^{\pm
         1}]]}
          = \sum_{j\in J} \langle f_{j}(w), u
         \rangle_{\cF[[w^{\pm1}]] \times \cF} \\
         & \qquad  = \sum_{j\in J} \langle f_{j}, K_{j}(\cdot, w) u
         \rangle_{\cH(K_{j}) \times \cH(K_{j}) [[ w^{\pm 1}]]}  \\
        & \qquad  = \Big\langle \widehat \bigoplus_{j\in J} f_{j},
\widehat\bigoplus_{j\in J} K_{j}(\cdot, w) u
         \Big\rangle_{\widehat \bigoplus_{j\in J} \cH(K_{j})
         \times \widehat \bigoplus_{j\in J} \cH(K_{j})[[w^{\pm 1}]]}
         \end{align*}
     shows that ${\mathbf s}^{*}(K(\cdot, w) u)
=\widehat\bigoplus_{j\in J} K_{j}(\cdot, w) u$. This in turn
amounts to the
     simultaneous set of equalities
     $$
     [{\mathbf s}^{*}([K]_{m}(\cdot) u)]_{m}(z) =
\widehat{\bigoplus_{j\in J}}[K_{j}]_{m}(z)u
      \text{ for all } u \in \cF \text{ and } m \in {\mathbb Z}^{d}.
     $$
     Hence $(\operatorname{ker}
     {\mathbf s})^{\perp} =  \overline{\operatorname{im}}\, {\mathbf
     s}^{*} = {\mathcal D}_{{\mathbf s}}$ as asserted.
    \end{proof}

   \subsection{The Szeg\H{o} kernel and associated FRKHSs} \label{S:Szego}
 We apply these ideas in particular to the case $k_{\Sz}(z,w) =
\sum_{n \in \bbZ^{d}} z^{n} w^{-n}$ (the {\em bilateral Szeg\H{o}
kernel})
    in the space $ \bbC[[z^{\pm 1}, w^{\pm 1}]]$ (here we
    identify the complex numbers $\bbC$ with the space $\cL(\bbC)$ of
    linear operators on $\bbC$).  Note that $k_{\Sz}(z,w) = h_{\Sz}(z)
    h_{\Sz}(w)^{*}$
    if we set
    $$
      h_{\Sz}(z) = \operatorname{row}_{n \in \bbZ^{d}} [ z^{n} ]
       \in \cL(\ell^{2}(\bbZ^{d}), \bbC)[[z^{ \pm 1}]].
    $$
    Hence condition (2) in Theorem \ref{T:FRKHS} is verified for
    $k_{\Sz}$ and
    it follows that $k_{\Sz}$ is the reproducing kernel for a formal
reproducing kernel Hilbert space
    $\cH(k_{\Sz})$.  To identify $\cH(k_{\Sz})$ explicitly, note that
    $k_{\Sz}(z,w) =
    \sum_{n \in \bbZ^{d}} k_{\Sz,n}(z) w^{-n}$ where $k_{\Sz,n}(z) =
z^{n}$.
    Hence, for $f(z) = \sum_{n \in \bbZ^{d}} f_{n} z^{n} \in
    \cH(k_{\Sz})$ we
    have the identity
    $$
      \langle f, z^{n} \rangle_{\cH(k_{\Sz})} = f_{n} \text{ for all
} n \in
      \bbZ^{d}
    $$
    from which we see that $\cH(k_{\Sz})$ can be identified with the
Lebesgue
    space $L^{2}(\bbT^{d})$ of measurable functions on the torus
$\bbT^{d}$
    with modulus-squared integrable with respect to Lebesgue measure
via
    identification of a function $f \in L^{2}(\bbT^{d})$ with its
    $d$-variable  Fourier series $f(z) \sim  \sum_{n \in \bbZ^{d}}
f_{n}
    z^{n}$ (with $f_{n} = \int_{\bbT^{d}} f(\zeta) \zeta^{-n}
    \ |d\zeta|/(2 \pi)^{d}$).  The fact that the operator $M_{z_{k}}$
    of multiplication by the
    variable $z_{k}$ is unitary on $L^{2}(\bbT^{d})$ implies then that
    $M_{z_{k}}$ is unitary on $\cH(k_{\Sz})$; this can be seen
directly
    in the FRKHS context as an application of
    Proposition \ref{P:multipliers} combined with the easily verified
    identities
    $$ k_{\Sz}(z,w) = z_{k} k_{\Sz}(z,w) w_{k}^{-1}, \qquad
      k_{\Sz}(z,w) = z_{k}^{-1} k_{\Sz}(z,w) w_{k}
     $$
    and  the observation that $M_{z_k}^*=M_{z_k^{-1}}$   for  $k = 1,
\dots, d$; this is also a special case of Proposition
\ref{P:Szmultipliers} below.

Let $\cF$ be a separable coefficient Hilbert space. Following
Sz.-Nagy-Foia\c{s} \cite[Chapter V]{NF} where the case $d=1$ is
handled or specializing the very general setting of Hille-Phillips
\cite[Section 3.1]{HP} or the somewhat different general setting
of Dixmier \cite[Part II]{Dix}, we say that a $\cF$-valued
function $\zeta \mapsto f(\zeta)$ defined on ${\mathbb T}^{d}$ is
{\em measurable in the weak sense} if
\begin{itemize}
    \item {\em the scalar-valued function $\zeta \mapsto \langle
    f(\zeta)u, \widetilde u \rangle_{\cF}$ is measurable for each
    choice of $u, \widetilde u \in \cF$}
\end{itemize}
and is {\em measurable in the strong sense} if
\begin{itemize}
    \item {\em the function $\zeta \mapsto f(\zeta)$ is the
    almost-everywhere pointwise limits of simple functions
    $$
    f(\zeta) = \lim_{n \to \infty} f_{n}(\zeta)
    $$
    where each $f_{n}$ simple means that $f_{n}$ is a finite linear
    combination  (with vector coefficients from $\cF$) of characteristic
    functions of measurable sets.}
  \end{itemize}
The fact that $\cF$ is separable implies that these two notions of
measurable are equivalent (see \cite[Corollary 2 page 73]{HP}) and
we simply say that $f$ is measurable for either of these notions.
>From the Parseval relation
$$
  \| f(\zeta) \|^{2} = \sum_{n=0}^{\infty} |\langle f(\zeta) e_{n},
  e_{n} \rangle|^{2}
$$
where $\{e_{n}\}_{n \in {\mathbb Z}_{+}}$ is any orthonormal basis
for $\cF$, we see immediately that $\zeta \mapsto \|
f(\zeta)\|^{2}$ is measurable in the standard sense for a
scalar-valued function.\footnote{This remains true for
Banach-space valued functions $\zeta \mapsto f(\zeta)$ but by a
different argument (see \cite[Theorem 3.52 page 72]{HP}).} Hence
it makes sense to define a Hilbert space $L^{2}({\mathbb T}^{d},
\cF)$ as the space of all measurable $\cF$-valued functions $\zeta
\mapsto f(\zeta)$ which are norm-squared integrable:
$$
  \| f \|^{2}_{L^{2}({\mathbb T}^{d}, \cF)} : = \int_{{\mathbb
  T}^{d}} \| f(\zeta)\|^{2}\, {\tt d}\zeta < \infty.
$$
Furthermore, it can be shown (see \cite[page 190]{NF} for the case
$d=1$) that $L^{2}({\mathbb T}^{d}, \cF)$ has an orthogonal
decomposition
$$
L^{2}({\mathbb T}^{d}, \cF) = \bigoplus_{k \in {\mathbb Z}^{d}}
{\mathfrak E}_{k}
$$
where we set
$$
 {\mathfrak E}_{k} = \zeta^{k} \cF \subset L^{2}({\mathbb T}^{d}, \cF)
$$
for $k  \in {\mathbb Z}^{d}$.  From this decomposition it is
easily verified that we can identify the function space
$L^{2}({\mathbb T}^{d}, \cF)$ with the formal reproducing kernel
Hilbert space $\cH(k_{\Sz} I_{\cF})$ by identifying the function
$\zeta^{k}u \in L^{2}({\mathbb T}^{d}, \cF)$ with the formal
monomial $z^{k} u \in \cH(k_{\Sz}I_{\cF})$ for each $u \in \cF$
and then by extending by linearity and taking of limits.  What is
the same, we identify the function $f \in L^{2}({\mathbb T}^{d},
\cF)$ with the formal series $\widehat f(z) = \sum_{n \in {\mathbb
Z}^{d}} f_{n} z^{n}$ by making use of the Fourier series
representation for $f$
$$
  f(z) \sim \sum_{n \in {\mathbb Z}^{d}} f_{n} \zeta^{n}
$$
where the Fourier coefficients $f_{n}$ are given by
$$
  f_{n} = \int_{{\mathbb T}} f(\zeta) \zeta^{-n}\, {\tt d} |\zeta|/(2
  \pi)^{d}
  \text{ (weak integral)},
$$
just as explained above for the scalar-valued case.

Now let us suppose that we have two separable Hilbert coefficient
spaces $\cF$ and $\cF'$ and suppose that $\zeta \mapsto S(\zeta)$
is an $\cL(\cF, \cF')$-operator-valued function on ${\mathbb
T}^{d}$.  We say that $S$ is {\em measurable} if $S$ is weakly
measurable in the sense that $\zeta \mapsto \langle S(\zeta) u, u'
\rangle_{\cF'}$ is a measurable scalar-valued function for each $u
\in \cF$ and $u' \in \cF'$. A consequence of Theorem 1 in
\cite[Part II, Chapter 2, Section 5]{Dix} is that bounded linear
operators $T$ from $L^{2}({\mathbb T}^{d}, \cF)$ to
$L^{2}({\mathbb T}^{d}, \cF')$ intertwining the scalar
multiplication operators $M_{z_{k}}$ on $L^{2}({\mathbb T}^{d},
\cF)$ and $L^{2}({\mathbb T}^{d}, \cF')$
$$
  T (M_{z_{k}} \otimes I_{\cF}) = (M_{z_{k}} \otimes I_{\cF'}) T
$$
($k=1, \dots, d$) are characterized as operators of the form $T =
M_{S}$
$$
  M_{S} \colon f(\zeta) \mapsto S(\zeta) f(\zeta)
$$
for an essentially bounded measurable function $\zeta \mapsto
S(\zeta) \in \cL(\cF, \cF')$, with operator norm of $T$ equal to
the essential infinity norm of $S$
$$
  \| T\|_{\rm op} = \| S \|_{\infty}
$$
(see Proposition 2 in \cite[Part II Chapter 2 Section 3]{Dix}).
For convenience we denote the space of all such essentially
bounded measurable $\cL(\cF, \cF')$-valued functions $S$ by
$L^{\infty}({\mathbb T}^{d}, \cL(\cF, \cF'))$. Under the
isomorphism $f(\zeta) \sim f(z)$ of $L^{2}({\mathbb T}^{d}, \cF)$
with $\cH(k_{\Sz} I_{\cF})$ and of $L^{2}({\mathbb T}^{d}, \cF')$
with $\cH(k_{\Sz} I_{\cF'})$, it is clear that such multiplication
operators $M_{S}$ coincide with the space of multipliers
$\cM(k_{\Sz} I_{\cF}, k_{\Sz} I_{\cF'})$ and that this
correspondence  is again given by the Fourier series
representation:
\begin{align*}
  S(\zeta) \in L^{\infty}({\mathbb T}^{d}, \cL(\cF, \cF')) & \sim S(z)
  =  \sum_{n \in {\mathbb Z}^{d}} S_{n} z^{n} \in \cM(k_{\Sz}I_{\cF},
  k_{\Sz} I_{\cF'}) \\
  & \text{where } S_{n} =
   \int_{{\mathbb T}^{d}} F(\zeta) \zeta^{-n}\,
  {\tt d} |\zeta|/(2 \pi)^{d}.
\end{align*}

Proposition \ref{P:multipliers}
    specialized to this situation gives the following.

    \begin{proposition}  \label{P:Szmultipliers}
    Suppose that $S$ is an element of
     $L^{\infty}(\bbT^{d}, \cL(\cF, \cF'))$ with
    associated Fourier series $S(z) \sim \sum_{n \in \bbZ^{d}} S_{n}
    z^{n}$ viewed as a formal Laurent series in $\cL(\cF, \cF')[[z^{\pm
    1}]]$.  Then:
    \begin{enumerate}

        \item The function $S$ has $\| S \|_{\infty} \le 1$ if and only
        if its Fourier series $S(z) \sim \sum_{n \in \bbZ^{d}} S_{n}
        z^{n}$  is such that any, and hence both, of the kernels
        $$ K_{S}(z,w) = k_{\Sz}(z,w)( I_{\cF'} - S(z) S(w)^{*}),
        \qquad K_{S^{*}}(z,w) = k_{\Sz}(z,w)(I_{\cF} -S(z)^{*} S(w))
        $$
        are positive kernels (in $\cL(\cF')[[z^{\pm 1}, w^{\pm 1}]]$
        and $\cL(\cF)[[z^{\pm 1}, w^{\pm 1}]]$ respectively).

 \item $S$ has isometric values almost everywhere on $\bbT^{d}$
 ($S(\zeta)^{*}S(\zeta) = I_{\cF}$ for a.e.~$\zeta \in  \bbT^{d}$) if
 and only if its Fourier series $S(z) \sim \sum_{n \in \bbZ^{d}} S_{n} z^{n}$
 satisfies the formal Laurent series identity
$$
 k_{\Sz}(z,w)(I_{\cF} - S(z)^{*}S(w)) = 0.
 $$

 \item $S$ has coisometric values almost every on $\bbT^{d}$
 ($S(\zeta)S(\zeta)^{*} = I_{\cF'}$ for a.e.~$\zeta \in \bbT^{d}$) if
and only if its Fourier series $S(z) \sim \sum_{n \in \bbZ^{d}}
S_{n} z^{n}$ satisfies the formal Laurent series identity
$$
 k_{\Sz}(z,w)(I_{\cF} - S(z)S(w)^{*}) = 0.
$$
    \end{enumerate}
    \end{proposition}

    \begin{remark}  \label{SzSK=KS}
      In Proposition \ref{P:Szmultipliers}, we use the fact that
$M_S^*=M_{S^*}$ in this case, and it follows from \eqref{kernelid}
that
    $$
    k_{\Sz}(z,w)S(z)   = k_{\Sz}(z,w) S(w)
    $$
(notice that formal power series with operator coefficients
commute with formal power series with scalar coefficients)
    i.e.,
    $$
    \sum_{n, n' \in \bbZ^{d}} S_{n} z^{n+n'} w^{-n'} =
     \sum_{n,n' \in \bbZ^{d}} S_{n'}z^{n}  w^{-n+n'}
     $$
      whenever $S(z) \sim \sum_{n \in \bbZ^{d}} S_{n} z^{n} \in
      L^{\infty}(\bbT^{d}, \cL(\cF, \cF'))$.  Of course this
      identity can be verified directly by performing the appropriate
      change of variable in the summation index.
      \end{remark}

    Suppose that $S = S^{*} \in L^{\infty}(\bbT^{d}, \cL(\cF))$
    with positive-semidefinite values $S(\zeta) \ge 0$ for
    a.e.~$\zeta \in \bbT^{d}$.  Then in particular the multiplication
    operator $W=M_{S}$ on $L^{2}(\bbT^{d}, \cF) \sim
\cH(k_{\Sz} I_{\cF})$ is positive semidefinite and we may define
the lifted norm
    space $\cH^{\ell}_{W}$ as in Section \ref{S:liftednorm}.
    The following proposition lists a few properties of this space.

    \begin{proposition} \label{P:Szliftednorm}
    Suppose that $S \in \cL(\cF)[[z^{\pm 1}]]$ is
    such that $W=M_{S}$ is positive semidefinite on $\cH(k_{\rm Sz}
    I_{\cF})$ and the space $\cH_{W}^\ell = \operatorname{im} ({W}^{1/2})$
    with inner product given by  \eqref{Winnerprod} with $K=k_{\rm
Sz}$.   Then $\cH_{W}^\ell$
    is a FRKHS with formal reproducing kernel $K_{W}^\ell(z,w) \in
    \cL(\cF)[[z^{\pm 1}, w^{\pm 1}]]$ given by
    $$
     K^{\ell}_{W}(z,w) =k_{\Sz}(z,w) S(z)  =  k_{\Sz}(z,w)S(w).
    $$
    \end{proposition}

    \begin{proof}
    Note the equality of the two expressions for $K^{\ell}_{W}(z,w)$
follows
    from Remark \ref{SzSK=KS}. The remaining points of the proposition
    are just specializations of the various results in Proposition
    \ref{P:liftednorm}.
  \end{proof}

  Suppose now that $B \in L^{\infty}(\bbT^{d}, \cL(\cF, \cF'))
\sim
  \cM(k_{\Sz} I_{ \cF}, k_{\Sz} I_{\cF'})$.  Then we are in
  the situation of Proposition \ref{P:pullback}: we may define the
  pullback FRKHS $\cH^{p}_{B} = M_{B} \cdot (\cH(k_{\Sz}I_{\cF})$
  with inner product given by
   \begin{equation*}
   \langle B(z) f(z), B(z) f(z) \rangle_{\cH^{p}_{B}} = \langle
   (Qf)(z), (Qf)(z) \rangle_{\cH(k_{\Sz} I_{\cF})}
   \end{equation*}
   where $Q \in \cL(\cH(k_{\Sz} I_{\cF})$ is the orthogonal
projection
   onto $(\operatorname{ker} M_{B})^{\perp}$.
   The following proposition specializes the summary in Proposition
   \ref{P:pullback} to this special setting.

   \begin{proposition}  \label{P:Szpullback}
       For $B \in L^{\infty}(\cF, \cF')$, the pullback FRKHS
       $\cH^{p}_{B}$ has reproducing kernel
       $$
       K^{p}_{B}(z,w) = k_{\Sz}(z,w)B(z) B(w)^{*}
       $$
       and the Brangesian complementary space
$(\cH^{p}_{B})^{\perp dB}$
       has reproducing kernel
       $$ K^{p \perp}_{B}(z,w) = k_{\Sz}(z,w)(I - B(z) B(w)^{*}) .
       $$
    \end{proposition}

    More generally, following \cite{BabyBear} and abusing notation
    somewhat, for $\Omega$ equal to any subset of $\bbZ^{d}$ we let
    $H^{2}(\Omega)$ be the subspace of $L^{2}(\bbT)$ consisting of
    functions $f \in L^{2}(\bbT)$ with Fourier coefficients supported
    on $\Omega$:  $f(z) \sim \sum_{n \in \Omega} f_{n} z^{n}$.  Then
    it is easily seen that $H^{2}(\Omega)$ may be viewed (by viewing
    the Fourier series of $f \in H^{2}(\Omega)$ as a formal series
    in $\bbC[[z^{\pm 1}]]$) as a FRKHS with reproducing kernel equal
    to the truncated Szeg\H{o} kernel
    $$
      k_{\Sz,\Omega}(z,w) = \sum_{n \in \Omega} z^{n} w^{-n}.
    $$
    Note that in this notation, the Hardy space over the
    polydisk, usually written as $H^{2}(\bbD^{d})$ is now written as
    $H^{2}(\bbZ^{d}_{+})$ and $L^{2}(\bbT^{d})$ is also written as
    $H^{2}(\bbZ^{d})$.

  \section{Formal de Branges--Rovnyak model for a scattering system}
  \label{S:scatmodels}

  We recall the definition of a multievolution scattering system (with
  polydisk model) from \cite{BabyBear}.  We define a {\em
  $d$-evolution scattering system} (of Lax--Phillips type with
polydisk
  model) to be a collection
  \begin{equation}  \label{scatsys}
  {\mathfrak S} = (\cK; \cU = (\cU_{1}, \dots, \cU_{d}); \cF, \cF_{*})
  \end{equation}
  where $\cK$ is a Hilbert space (the {\em ambient space}) for the
scattering system, ${\mathcal U}_{1}, \dots, {\mathcal U}_{d}$ is
a $d$-tuple of commuting unitary operators on ${\mathcal K}$ (the
{\em evolutions} of the system), and such that ${\mathcal F}$ and
${\mathcal F}_{*}$ are wandering subspaces for ${\mathcal U} =
({\mathcal U}_{1}, \dots, {\mathcal U}_{d})$, i.e.,
\begin{equation*}
       {\mathcal U}^{n} {\mathcal F} \perp {\mathcal F}, \qquad
       {\mathcal U}^{n} {\mathcal F}_{*} \perp {\mathcal F}_{*}
       \text{ for all } n \in {\mathbb Z}^{d} \text{ with } n \ne 0.
\end{equation*}
We say that a subset $\Omega$ of the integer lattice $\bbZ^{d}$ is
a {\em shift-invariant sublattice} if $\Omega$ is invariant for
each of the shift operators
\begin{equation*}
 \sigma_{k} \colon n \mapsto n + \be_{k}
\end{equation*}
for $k = 1, \dots, d$, where we let
\begin{equation}  \label{bek}
     \be_{k} = (0, \dots, 0, 1, 0, \dots, 0)
\end{equation}
be the $k$-th standard basis vector for ${\mathbb R}^{d}$ ($k = 1,
\dots, d$). The main examples of shift-invariant sublattices are
\begin{align*}
  & \Omega = \bbZ^{d}, \qquad \Omega = \emptyset, \\
  & \Omega =  \bbZ^{d}_{+}: = \{ n = (n_{1}, \dots, n_{d}) \in
\bbZ^{d} \colon
    n_{k} \ge 0 \text{ for } k = 1, \dots, d \}, \\
  &  \Omega = \bbZ^{d} \setminus \bbZ^{d}_{+} =
  \{ n = (n_{1}, \dots, n_{d}) \in \bbZ^{d} \colon n_{k} < 0 \text{
  for at least one } k = 1, \dots, d \}, \\
  & \Omega = \Omega^{B}: = \{ n = (n_{1}, \dots, n_{d}) \in \bbZ^{d}
  \colon n_{1} + \cdots + n_{d} \ge 0 \}
\end{align*}
where the $B$ in the notation $\Omega^{B}$ suggests {\em
balanced}. Given such a shift-invariant sublattice $\Omega$, we
define associated {\em outgoing subspace} $\cW^{\Omega}$ and {\em
incoming subspace} $\cW_{*}^{\Omega}$ by
$$ \cW^{\Omega}: = \bigoplus_{n \in \Omega} \cU^{n} \cF, \qquad
\cW_{*}^{\Omega}:= \bigoplus_{n \in \bbZ^{d}\setminus \Omega}
\cU^{n} \cF_{*}.
$$
We say that the multievolution scattering system ${\mathfrak S}$
is {\em $\Omega$-orthogonal} if
\begin{equation*}
    \cW^{\Omega} \perp \cW^{\Omega}_{*}.
\end{equation*}
It turns out (see \cite[Proposition 3.1]{BabyBear}) that
${\mathfrak S}$ is $\Omega$-orthogonal for all shift-invariant
sublattices $\Omega$ as soon as it is $\bbZ^{d}_{+}$-orthogonal.
We therefore say simply that the $d$-evolution scattering system
${\mathfrak S}$ is {\em orthogonal} if ${\mathfrak S}$ is
$\bbZ^{d}_{+}$-orthogonal (and therefore $\Omega$-orthogonal for
all shift-invariant sublattices $\Omega$).  Thus, in any
orthogonal multievolution scattering system ${\mathfrak S}$, the
ambient space $\cK$ has an orthogonal decomposition
\begin{equation*}
 \cK = \cW^{\Omega}_{*} \oplus \cV^{\Omega} \oplus \cW^{\Omega}
\end{equation*}
associated with any shift-invariant lattice $\Omega \subset
\bbZ^{d}$, where
$$
\cV^{\Omega}: = \cK \ominus [ \cW^{\Omega} \oplus
\cW^{\Omega}_{*}]
$$
is the {\em scattering space} associated with $\Omega$.  In the
sequel we shall have as a standing assumption that the
multievolution scattering system ${\mathfrak S}$ is orthogonal.

Another property which often comes up is minimality.
\begin{definition}  \label{D:minimal-scat}
We shall say that the multievolution scattering system ${\mathfrak
S}$ is {\em minimal} if the smallest subspace of $\cK$ containing
$\cF$ and $\cF_{*}$ and invariant for each of $\cU_{1}, \dots,
\cU_{d}$ and for each of $\cU_{1}^{*}, \dots, \cU_{d}^{*}$ is the
whole space $\cK$. Equivalently, $\widetilde \cW + \widetilde
\cW_{*}$ is dense in $\cK$, where we have set
$$ \widetilde \cW:= \cW^{\bbZ^{d}} = \bigoplus_{n \in \bbZ^{d}}
\cU^{n} \cF, \qquad \widetilde \cW_{*}:= \cW^{\emptyset}_{*} =
\bigoplus_{n \in \bbZ^{d}} \cU^{n } \cF_{*}.
$$
\end{definition}

Given any such orthogonal scattering system ${\mathfrak S}$, we
associate Fourier representation operators
$$ \Phi \colon \cK \to L^{2}(\bbT^{d}, \cF) \sim \cH(k_{\Sz} I_{\cF}), \qquad
\Phi_{*} \colon \cK \to L^{2}(\bbT^{d}, \cF_{*}) \sim \cH(k_{\Sz}
I_{\cF_{*}})
$$
by
\begin{equation}  \label{Fourierrep}
    \Phi \colon k \mapsto \sum_{n \in \bbZ^{d}} (P_{\cF} \cU^{-n}k)
    z^{n}, \qquad \Phi_{*} \colon k \mapsto \sum_{n \in \bbZ^{d}}
    (P_{\cF_{*}} \cU^{-n}k) z^{n}.
 \end{equation}
Then $\Phi$ is a coisometry from $\cK$ onto $L^{2}(\bbT^{d},\cF)$
with initial space equal to $\widetilde \cW$ while $\Phi_{*}$ is a
coisometry form $\cK$ onto $L^{2}(\bbT^{d}, \cF_{*})$ with initial
space equal to $\widetilde \cW_{*}$.  What is more, if $\Omega$ is
any subset of $\bbZ^{d}$, then $\Phi|_{\cW^{\Omega}}$ is unitary
from $\cW^{\Omega}$ onto $H^{2}(\Omega, \cF)$ and
$\Phi_{*}|_{\cW^{\Omega}_{*}}$ is unitary from $\cW^{\Omega}_{*}$
onto $H^{2}(\bbZ^{d} \setminus \Omega, \cF^{*})$. A consequence of
the orthogonality assumption on ${\mathfrak S}$ is that
$\operatorname{im} \Phi^{*}|_{H^{2}(\bbZ^{d}_{+}, \cF)}
 = \cW^{\Omega}$ is orthogonal to $\cW^{\Omega}_{*}$.  As a
 consequence
 $$
 \Phi_{*} \Phi^{*} \colon H^{2}(\bbZ^{d}_{+}, \cF) \to
 H^{2}(\bbZ^{d}_{+}, \cF_{*})
 $$
 and commutes with $M_{z_{k}}$ for $k = 1, \dots, d$.  We conclude
 that $\Phi_{*} \Phi^{*}$ has the form
 $ \Phi_{*} \Phi^{*} = M_{S}$ for a function $S \in
 H^{\infty}(\bbZ^{d}_{+}, \cL(\cF, \cF_{*}))$ (i.e., $S \in
 H^{\infty}(\bbD^{d}, \cL(\cF, \cF_{*}))$ in the more standard
 notation is the boundary-value function of a bounded analytic
 function on the unit polydisk $\bbD^{d}$---see \cite[Section V.3]{NF}).
 Moreover, since $M_{S} =\Phi_{*} \Phi^{*}$ is the product of partial isometries,
 necessarily $M_{S}$ has norm at most 1 and $S$ has $\| S \|_{\infty}
 \le 1$, i.e., $S \in \cB H^{\infty}(\bbZ^{d}_{+}, \cL(\cF,
 \cF_{*}))$. This polydisk Schur-class
 function $S$ is called the {\em scattering matrix} of the
 $d$-evolution scattering system ${\mathfrak S}$.

 We recall the de Branges--Rovnyak model for the multievolution
 scattering system from \cite{MammaBear} associated with any polydisk
 Schur-class function $S \in \cS(\cF, \cF_{*})$.  The de Branges--Rovnyak multievolution
scattering
 system ${\mathfrak S}_{dBR}^{S}$ associated with the polydisk
 Schur-class function $S \in \cS(\cF, \cF_{*})$ is defined to be
 \begin{equation}  \label{frakSdBR}
 {\mathfrak S}^{S}_{dBR} = (\cK_{dBR}^{S}; \cU^{S}_{dBR} =
 ( \cU^{S}_{dBR,1}, \dots, \cU^{S}_{dBR,d}); \cF^{S}_{dBR},
 \cF^{S}_{dBR *})
 \end{equation}
 where
 $$
 \cK^{S}_{dBR} = \operatorname{im} \begin{bmatrix}  I & M_{S} \\
 M_{S}^{*} & I \end{bmatrix}^{1/2} \subset \begin{bmatrix}
 L^{2}(\bbT^{d}, \cF_{*}) \\ L^{2}(\bbT^{d}, \cF)
 \end{bmatrix}
 $$
 with lifted inner product
 $$
 \left \langle \begin{bmatrix} I & M_{S} \\ M_{S}^{*} & I
 \end{bmatrix} ^{1/2} \begin{bmatrix} f \\ g \end{bmatrix},
 \begin{bmatrix} I & M_{S} \\ M_{S}^{*} & I
  \end{bmatrix} ^{1/2} \begin{bmatrix} f' \\ g' \end{bmatrix}
  \right\rangle_{\cK^{S}_{dBR}} =
  \left\langle Q \begin{bmatrix} f \\ g \end{bmatrix}, \begin{bmatrix}
  f' \\ g' \end{bmatrix} \right\rangle_{L^{2}(\bbT^{d},\cF_{*}
   \oplus \cF)}
  $$
  where $Q$ is the orthogonal projection onto the orthogonal
  complement of \newline
  $\ker \sbm{ I & M_{S} \\ M_{S}^{*} & I}^{1/2}$, $\cU^{S}_{dBR,k} =
  M_{z_{k}} \otimes I_{\cK^{S}_{dBR}}$ on $\cK^{S}_{dBR}$, and where
we
  set
  $$
  \cF^{S}_{dBR} = \begin{bmatrix} S \\ I \end{bmatrix} \cF, \qquad
  \cF^{S}_{dBR *} = \begin{bmatrix} I \\ S^{*} \end{bmatrix} \cF_{*}.
  $$
  It is easily checked that ${\mathfrak S}^{S}_{dBR}$ is a minimal
  scattering system for any polydisk Schur-class function $S$.
  Conversely, if ${\mathfrak S}$ as in
  \eqref{scatsys} is any $d$-evolution scattering system with
  scattering matrix $S$, then it is shown in \cite{BabyBear} that the
map
  \begin{equation}  \label{PidBR}
   \Pi_{dBR} \colon \Phi_{*}^{*} f + \Phi^{*} g \mapsto
   \begin{bmatrix} I & M_{S} \\ M_{S^{*}} & I \end{bmatrix}
   \begin{bmatrix} f \\ g \end{bmatrix}
  \end{equation}
   extends by continuity to define a unitary map from the minimal
   part $\cK_{\text{min}}: = \operatorname{clos}(
   \operatorname{im} \Phi^{*}_{*} + \operatorname{im} \Phi^{*})$ of
   the ambient space $\cK$ of ${\mathfrak S}$ onto $\cK^{S}_{dBR}$
   such that
   \begin{align}
      & \Pi_{dBR} \cU_{k} = \cU^{S}_{dBR,k} \Pi_{dBR} \text{ for }
   k = 1, \dots, d, \notag\\
  & \Pi_{dBR} \colon \cF \to \cF^{S}_{dBR}, \qquad
   \Pi_{dBR} \colon \cF_{*} \to \cF^{S}_{dBR *} \text{ unitarily}
   \label{Piaction}
   \end{align}
   and the scattering matrix $S_{{\mathfrak S}^{S}_{dBR}}$
   associated with the de Branges--Rovnyak model scattering system
   ${\mathfrak S}^{S}_{dBR}$ {\em coincides} with
   the original scattering
   function $S$ in the sense that
   $$ S_{{\mathfrak S}^{S}_{dBR}}(z) = i^{S}_{dBR*} S(z)
   (i^{S}_{dBR})^{*}
   $$
   where $i^{S}_{dBR}$ and $i^{S}_{dBR*}$ are the unitary
   identification maps appearing in \eqref{Piaction}
   $$
   i^{S}_{dBR} = \Pi_{dBR}|_{\cF}: u \mapsto \begin{bmatrix} S \\ I
   \end{bmatrix} u, \qquad
   i^{S}_{dBR *} = \Pi_{dBR}|_{\cF_{*}} \colon u_{*} \mapsto
   \begin{bmatrix} I \\ S^{*} \end{bmatrix} u_{*}.
   $$
   Let us also mention that $\Pi_{dBR}$ can be given more explicitly
   than the formula \eqref{PidBR} above given in \cite{BabyBear}, as
given
   in the following Lemma.

   \begin{lemma}  \label{L:PidBR}  Suppose that ${\mathfrak S}$ is a
   multievolution scattering system with map $\Pi_{dBR}$ determined
   on $\cK_{\text{min}}$ by \eqref{PidBR} and extended by linearity to
   the whole space $\cK$ by the requirement that $\Pi_{dBR}|_{\cK
   \ominus \cK_{\text{min}}} = 0$.  Then $\Pi_{dBR}$ can be given by
   the alternative formula
   \begin{equation}  \label{PidBR'}
     \Pi_{dBR} = \begin{bmatrix} \Phi_{*} \\ \Phi \end{bmatrix}
   \colon  \cK \to \begin{bmatrix}  L^{2}(\bbT^{d}, \cF_{*})
   \\ L^{2}(\bbT^{d}, \cF) \end{bmatrix}.
   \end{equation}
   \end{lemma}

   \begin{proof} For $f \in L^{2}(\bbT^{d}, \cF_{*})$ and $g
   \in L^{2}(\bbT^{d}, \cF)$ we have
   \begin{align*}
   \begin{bmatrix} \Phi_{*} \\ \Phi \end{bmatrix} (\Phi^{*}_{*}f +
\Phi^{*} g)
       &  = \begin{bmatrix} \Phi_{*} \\ \Phi \end{bmatrix}
       \begin{bmatrix} \Phi_{*}^{*} & \Phi^{*} \end{bmatrix}
       \begin{bmatrix} f \\ g \end{bmatrix}
   = \begin{bmatrix}  \Phi_{*} \Phi^{*}_{*} & \Phi_{*} \Phi^{*} \\
     \Phi \Phi^{*}_{*} & \Phi \Phi^{*} \end{bmatrix} \begin{bmatrix}
     f \\ g \end{bmatrix} \\
     & = \begin{bmatrix} I & M_{S} \\ M_{S}^{*} & I \end{bmatrix}
     \begin{bmatrix} f \\ g \end{bmatrix}
      = \Pi_{dBR}( \Phi_{*}^{*} f + \Phi^{*} g)
  \end{align*}
  and it follows that \eqref{PidBR} and \eqref{PidBR'} agree on
  $\cK_{\text{min}}$.  For $k \perp \cK_{\text{min}}$, by construction
  both $\Phi k = 0$ and $\Phi_{*} k = 0$.  The lemma now follows.
  \end{proof}

 For our purposes here, it will be convenient to use the
 identifications discussed in Section \ref{S:Szego} and view Fourier series
 as formal Laurent series and identify the various spaces
 $L^{2}(\bbT^{d}, \cF)$, $L^{\infty}(\bbT^{d}, \cL(\cF, \cF_{*}))$,
 $H^{2}(\Omega, \cF)$ and
 $H^{\infty}(\Omega, \cL(\cF, \cF_{*}))$ as the corresponding
 spaces of formal Laurent series $\cH(k_{\Sz} I_{\cF})$,
 $\cM(k_{\Sz} I_{\cF}, k_{\Sz}  I_{\cF'})$,
 $\cH(k_{\Sz,\Omega}  I_{\cF})$ and $\cM(k_{\Sz, \Omega} I_{\cF},
 k_{\Sz, \Omega} I_{\cF_{*}})$ respectively. This is done in
 the next proposition.

 \begin{proposition}  \label{P:dBRFRKHSs}
     After identification $F(z) \sim \sum_{n \in \bbZ^{d}} F_{n}
     z^{n}$ of an $L^{2}$-function with its Fourier series viewed as
     a formal Laurent series, the spaces $\cK^{S}_{dBR}$,
     $\cF^{S}_{dBR}$, $\cF^{S}_{dBR*}$, $\cW^{S,\Omega}_{dBR} =
     \Pi_{dBR}(\cW^{\Omega})$, $\cW^{S, \Omega}_{dBR*} =
     \Pi_{dBR}(\cW^{\Omega}_{*})$ and $\cV^{S, \Omega}_{dBR} =
     \Pi_{dBR}(\cV^{\Omega})$ can be viewed as FRKHSs
     $\cH(K_{\cK^{S}_{dBR}})$, $\cH(K_{\cF^{S}_{dBR}})$,
     $\cH(K_{\cF^{S}_{dBR}*})$, $\cH(K_{\cW^{S, \Omega}_{dBR}})$,
     $\cH(K_{\cW^{S, \Omega}_{dBR}*})$ and
     $\cH(K_{\cV^{S,\Omega}_{dBR}})$ where the respective reproducing
     kernels are given by
     \begin{align}
     &  K_{\cK^{S}_{dBR}}(z,w) = k_{\Sz}(z,w)\begin{bmatrix} I & S(z)
\\
             S(z)^{*} & I \end{bmatrix}, \qquad
       K_{\cF^{S}_{dBR}}(z,w) = \begin{bmatrix}
             S(z) S(w)^{*} & S(z) \\ S(w)^{*} & I \end{bmatrix},
         \notag \\
   &  K_{\cF^{S}_{dBR}*}(z,w) = \begin{bmatrix} I & S(w) \\
               S(z)^{*} & S(z)^{*}S(w) \end{bmatrix},  \notag  \\
      & K_{\cW^{S, \Omega}_{dBR}}(z,w) =  k_{\Sz,
\Omega}(z,w)\begin{bmatrix} S(z) S(w)^{*} & S(z)
                \\ S(w)^{*} & I \end{bmatrix}
      = \sum_{n \in \Omega} z^{n} K_{\cF^{S}_{dBR}}(z,w) w^{-n}, \notag
\\
      & K_{\cW^{S, \Omega}_{dBR}*}(z,w) = k_{\Sz, \bbZ^{d}
             \setminus \Omega}(z,w)\begin{bmatrix} I & S(w) \\
             S(z)^{*} & S(z)^{*} S(w) \end{bmatrix}  =
         \sum_{n \in \bbZ^{d} \setminus \Omega} z^{n}
         K_{\cF^{S}_{dBR}*}(z,w) w^{-n}, \notag \\
      & K_{\cV^{S, \Omega}_{dBR}}(z,w) = K_{\cK^{S}_{dBR}}(z,w) -
                     K_{\cW^{S, \Omega}_{dBR}}(z,w) - K_{\cW^{S,
                     \Omega}_{dBR}*}(z,w) \notag \\
      & \qquad = \begin{bmatrix} k_{\Sz,
                  \Omega}(z,w)(I - S(z) S(w)^{*})  &  k_{\Sz,
\bbZ^{d} \setminus
                    \Omega}(z,w)(S(z) - S(w)) \\
                k_{\Sz, \Omega}(z,w) (S(z)^{*} - S(w)^{*}) &
                 k_{\Sz, \bbZ^{d} \setminus \Omega}(z,w) (I -
S(z)^{*} S(w))
                 \end{bmatrix} \notag \\
      & \qquad =
             \begin{bmatrix}  k_{\Sz, \Omega}(z,w)(I - S(z) S(w)^{*})
               & - k_{\Sz,  \Omega}(z,w)(S(z) - S(w)) \\
               -k_{\Sz,\bbZ^{d} \setminus \Omega}(z,w)(S(z)^{*} -
S(w)^{*}) &
            k_{\Sz, \bbZ^{d} \setminus \Omega}(z,w) (I - S(z)^{*} S(w))
             \end{bmatrix} \notag \\
   & \qquad =
  k_{\Sz, \Omega}(z,w)  \begin{bmatrix} I - S(z) S(w)^{*} & S(w) -
S(z) \\ S(z)^{*} -
   S(w)^{*} & S(z)^{*}S(w) - I \end{bmatrix}.
  \label{dBRkernels}
  \end{align}
  \end{proposition}

  \begin{proof}
      By construction $\cK^{S}_{dBR}$ is the lifted norm space
      associated with the multiplier $\sbm{ I & S(z) \\ S(z)^{*} & I
}$
      acting on $\cH(k_{\Sz}  I_{\cF_{*} \oplus \cF})$.
      Hence Proposition \ref{P:liftednorm} implies
      the formula for $K_{\cK^{S}_{dBR}}(z,w)$ in \eqref{dBRkernels}.
      Note next that, for $u, u' \in \cF$,
      \begin{align*}
     \left\langle  \begin{bmatrix} S(z) \\ I \end{bmatrix} u,
      \begin{bmatrix} S(z) \\ I \end{bmatrix} u'
      \right \rangle_{\cF^{S}_{dBR}}
      &  =
        \left \langle \begin{bmatrix} I & S(z) \\ S(z)^{*} & I
        \end{bmatrix} \begin{bmatrix} 0 \\ u \end{bmatrix},
    \begin{bmatrix} I & S(z) \\ S(z)^{*} & I
        \end{bmatrix} \begin{bmatrix} 0 \\ u' \end{bmatrix}
        \right\rangle_{\cF^{S}_{dBR}} \\
  &  =
  \left\langle \begin{bmatrix} I & S(z) \\ S(z)^{*} & I
     \end{bmatrix} \begin{bmatrix} 0 \\ u \end{bmatrix},
     \begin{bmatrix} 0 \\ u' \end{bmatrix}
         \right \rangle_{\cH(k_{\Sz} I_{\cF_{*} \oplus \cF})} \\
  & =
  \langle u, u' \rangle_{\cF}
  \end{align*}
  and hence $\cF^{S}_{dBR}$ is just the pullback space coming from
  multiplier $B(z) = \sbm{S(z) \\ I }$ mapping $\cF$ into
  $\cH(k_{\Sz} I_{\cF_{*} \oplus \cF})$.  It follows from
  Proposition \ref{P:pullback} that
  $$
  K_{\cF^{S}_{dBR}}(z,w) =k_{\Sz, \emptyset}(z,w) B(z)  B(w)^{*}
  = \begin{bmatrix} S(z) \\ I \end{bmatrix} \begin{bmatrix} S(w)^{*} &
  I \end{bmatrix} =
  \begin{bmatrix} S(z) S(w)^{*} & S(z) \\ S(w)^{*} & I \end{bmatrix}
   $$
  in agreement with the formula for $K_{\cF^{S}_{dBR}}(z,w)$ given
  in \eqref{dBRkernels}.  The remaining formulas in
  \eqref{dBRkernels} can be derived similarly.  The equivalence of the
  two expressions for $K_{\cV^{S, \Omega}_{dBR}}(z,w)$ follows from
  the identity
  $$
     k_{\Sz}(z,w)(S(z) - S(w)) = 0
  $$
  (see
  Remark \ref{SzSK=KS}).  This completes the proof of Proposition
  \ref{P:dBRFRKHSs}.  The equivalence of the first two expressions for
  $K_{\cV^{S, \Omega}_{dBR}}(z,w)$ also enables one to verify
directly the
  selfadjointness property:  $K_{\cV^{S, \Omega}_{dBR}}(z,w) =
  (K_{\cV^{S, \Omega}_{dBR}}(w,z))^{*}$.
   \end{proof}

   \begin{remark}  \label{R:dBRscatmodel}
       It is straightforward to check directly from the definitions
       that
    $$
     z^{n} \cdot \cF^{S}_{dBR} \perp z^{n'} \cdot \cF_{dBR}^{S}
\text{ and }
     z^{n} \cdot \cF^{S}_{dBR*} \perp z^{n'} \cdot \cF^{S}_{dBR*}
\text{
     in } \cK^{S}_{dBR} \text{ for } n \ne n' \text{ in } \bbZ^{d}.
    $$
    By minimal-decomposition theory the identities
    \begin{align*}
    & K_{\cW^{S, \Omega}_{dBR}}(z,w) = \sum_{n \in \Omega } z^{n}
    K_{\cF^{S}_{dBR}}(z,w) w^{-k}, \\
   &  K_{\cW^{S, \Omega}_{dBR}*}(z,w) = \sum_{n \in \bbZ^{d} \setminus
    \Omega} z^{n}  K_{\cF^{S}_{dBR}}(z,w) w^{-k}
    \end{align*}
    appearing in \eqref{dBRkernels} holding for a shift-invariant
    sublattice $\Omega$
    imply the internal orthogonal-sum decompositions
    \begin{equation}  \label{shiftdecom}
  \cH(K_{\cW^{S, \Omega}_{dBR}}) =  \bigoplus_{n \in \Omega}
    z^{n} \cH(K_{\cF^{S}_{dBR}}), \qquad
    \cH(K_{\cW^{S, \Omega}_{dBR}*}) = \bigoplus_{n \in \bbZ^{d}
\setminus \Omega}
   z^{n} \cH(K_{\cF^{S}_{dBR*}})
    \end{equation}
     as is to be expected.
    We also know from
    $$
    \cK = \cW_{*}^{\Omega} \oplus \cV^{\Omega}
    \oplus \cW^{\Omega}.
    $$
  that
  \begin{equation}  \label{dBRkerneldecom'}
  \cH(K_{\cK^{S}_{dBR}}) = \cH(K_{\cW^{S,\Omega}_{dBR}*}) \oplus
  \cH(K_{\cV^{S, \Omega}_{dBR}}) \oplus \cH(K_{\cW^{S,
\Omega}_{dBR}}).
  \end{equation}
  We emphasize that the decompositions in \eqref{shiftdecom} and
  \eqref{dBRkerneldecom'} are {\em internal} orthogonal direct sums.
  The decomposition \eqref{dBRkerneldecom'} is consistent with the
last of the identities in
  \eqref{dBRkernels} with the added information that the spaces
  $$
  \cH(K_{\cW^{S,\Omega}_{dBR}*}),\qquad
  \cH(K_{\cV^{S, \Omega}_{dBR}}), \qquad
  \cH(K_{\cW^{S, \Omega}_{dBR}})
  $$
  have no overlap.

     \end{remark}

 \section{Scattering system containing an embedded unitary
colligation}
 \label{S:scat-col}

 Suppose now that $U = \sbm{ A & B \\ C & D }$ as in
\eqref{colligation}
 is a Givone--Roesser unitary colligation.  As in \cite{BabyBear} we
 associate the system of equations
 \begin{equation} \label{polysys}
 \Sigma = \Sigma(U) \left\{  \begin{array}{rcl}
 x_{1}(n + \be_{1}) & = & A_{11} x_{1}(n) + \dots + A_{1d} x_{d}(n) +
 B_{1} u(n) \\
 \vdots & & \vdots \\
 x_{d}(n + \be_{d}) & =& A_{d1} x_{1}(n) + \dots + A_{dd} x_{d}(n) +
 B_{d} u(n) \\
 y(n) & = & C_{1} x_{1}(n) + \dots + C_{d} x_{d}(n) + D u(n).
 \end{array}  \right.
 \end{equation}
 where the basis vectors  $\be_{k}$ are as
 in \eqref{bek} and for any $n\in\mathbb{Z}^d$ one has
$u(n)\in\mathcal{E}$, $y(n)\in\mathcal{E}_*$,
$x_k(n)\in\mathcal{H}_k$, $k=1,\ldots,d$.
 Since we are assuming that $U$ is unitary, we may also write the
 system equations as a backward recursion
 \begin{equation} \label{polysysback}
 \Sigma = \Sigma(U) \left\{  \begin{array}{rcl}
 x_{1}(n) & = & A_{11}^{*} x_{1}(n + \be_{1}) + \dots + A_{d1}^{*}
 x_{d}(n + \be_{d}) + C_{1}^{*} y(n) \\
   & \vdots &  \\
 x_{d}(n) & =& A_{1d}^{*} x_{1}(n + \be_{1}) + \dots + A_{dd}^{*}
 x_{d}(n + \be_{d}) + C_{d}^{*} y(n) \\
 u(n) & = & B_{1}^{*} x_{1}(n + \be_{1}) + \dots + B_{d}^{*}
 x_{d}((n + \be_{d}) + D^{*} y(n).
 \end{array}  \right.
 \end{equation}
 The form of the system equations here
 are of {\em Givone--Roesser} type (see \cite{GR1, GR2}) with the
 additional property that the connecting matrix (or colligation) $U$
 is unitary---hence we refer to such a system as a {\em conservative
 Givone--Roesser system} (see \cite{MammaBear} for more details on the
 system-theoretic aspects of such systems---where the term {\em
 Roesser system} is used instead).

 We now recall from \cite{BabyBear} how one can associate a
 multievolution Lax--Phillips scattering system
 $$
  {\mathfrak S}(\Sigma(U)) =
  ( \cT; \cU = (\cU_{1}, \dots, \cU_{d}); \cF, \cF_{*})
  $$
 with any conservative Givone--Roesser system $\Sigma(U)$.
 The ambient space for the scattering system ${\mathfrak
S}(\Sigma(U))$ is the space $\cT$ of
 {\em admissible trajectories} of $\Sigma(U)$ defined as follows.
 By a {\em trajectory} of the system $\Sigma(U)$ we mean a $\cE
 \times \cH_{1} \times \cdots\times \cH_{d} \times \cE_{*}$-valued
function
 $$
  n \mapsto  (u(n), x_{1}(n), \dots, x_{d}(n), y(n))
 $$
 on $\bbZ^{d}$ which satisfies the system of equations
 \eqref{polysys} (or equivalently \eqref{polysysback}) for all $n
 \in \bbZ^{d}$.  A given trajectory $(u,x,y)$ is said to be {\em
 admissible} if it has finite energy:
 \begin{equation}  \label{calTnorm}
   \| (u,x,y)\|^{2}_{\cT} : = \|
 u|_{\Omega}\|^{2}_{\ell^{2}(\Omega, \cE)} + \| x|_{\partial
 \Omega}\|^{2}_{\partial \Omega} + \| y|_{\bbZ^{d}\setminus \Omega}
 \|^{2}_{\ell^{2}(\bbZ^{d}\setminus \Omega,  \cE_{*})} < \infty.
 \end{equation}
 In general we use the notation $\ell^{2}(\Omega, \cF)$ to indicate
  the space of $\cF$-valued functions on the set $\Omega$ which are
  norm-square integrable with respect to the discrete measure on
  $\Omega$:
  $$
  \ell^{2}(\Omega, \cF) = \Big\{ \{x(\omega)\}_{\omega \in \Omega}
\colon
  x(\omega) \in \cF \text{ with } \sum_{\omega \in \Omega} \|
  x(\omega)\|^{2}_{\cF} < \infty \Big\}.
  $$
 Of the three terms occurring in the definition \eqref{calTnorm} of
$\| (u,x,y)
 \|^{2}_{\cT}$ the first and last are fairly evident
 $$
   \| u|_{\Omega}\|^{2}_{\ell^{2}(\Omega, \cE)} = \sum_{n \in
   \Omega} \| u(n) \|^{2}_{\cE}, \qquad
   \| y|_{\bbZ^{d} \setminus
   \Omega}\|^{2}_{\ell^{2}(\bbZ^{d}\setminus \Omega,
   \cE_{*})} = \sum_{n \in \bbZ^{d}\setminus \Omega} \| y(n)
   \|^{2}_{\cE_{*}}
  $$
  while the precise meaning of the middle term is explained below.
  The space of all admissible trajectories is denoted by $\cT$ and is
a
  Hilbert space in the $\cT$-norm given by \eqref{calTnorm}.

  To explain the meaning of $\| x|_{\partial \Omega} \|^{2}_{\partial
  \Omega}$, we need to recall the decomposition of boundary of
  a shift-invariant sublattice $\Omega$ in {\it finite} and
  {\it infinite} parts, as introduced in \cite{MammaBear}.
  The finite boundary of $\Omega$ is given by
  $$   \partial \Omega_{\text{fin}} = \bigcup_{k=1}^{d} \partial_{k}
     \Omega_{\text{fin}}
  $$
       where
  $$    \partial_{k} \Omega_{\text{fin}}  = \{(n_{1}, \dots, n_{d})
     \in {\mathbb Z}^{d} \colon
     n_{k} = \mathop{\text{min}} \{ m
     \colon (n_{1}, \dots, n_{k-1}, m, n_{k+1}, \dots, n_{d}) \in
     \Omega\} \},
  $$
  while the infinite boundary $\partial \Omega_{\infty}$ is the union
of
  the boundaries at plus and minus infinity:
  $$\partial  \Omega_{\infty} = \partial \Omega_{+\infty}
  \cup \partial \Omega_{-\infty}.
  $$
  The boundary of $\Omega$ at plus infinity is in turn given by the
union of
  plus-infinity $k$-boundaries
  $$ \partial \Omega_{+\infty} = \bigcup_{k=1}^{d} \partial_{k}
\Omega_{+\infty}
  $$
  where an element of the plus-infinity $k$-boundary is defined to be
  the collection of
  lines  of the form\footnote{Note that a line $\ell_{n,k}$ as in
\eqref{ellnk}
  is parametrized  by the pair $n \in {\mathbb Z}^{d}$ and $k \in
\{1, \dots, d\}$ but
  the resulting line $\ell_{n,k}$ is independent of the particular
  value $n_{k}$ of the $k$-th coordinate of the $d$-tuple $n \in
  {\mathbb Z}^{d}$; such lines are really determined by the
  ($k-1$)-tuple $\widehat n^{k} = (n_{1}, \dots, n_{k-1}, n_{k+1},
\dots, n_{d})$
  together with the choice of $k \in \{1, \dots, d\}$.  The
  oversimplified notation $\ell_{n,k}$ nevertheless is convenient for
  the manipulations to follow. In \cite{MammaBear, BabyBear} the
  $(d-1)$-tuple $\widehat n^{k}$ was used as the index rather the the
  $d$-tuple $n$.}
  \begin{equation}  \label{ellnk}
  \ell_{n,k} = n + {\mathbb Z} \cdot \be_{k} \text{ for some } n \in
\bbZ^{d}
  \text{ and } k \in \{1, \dots, d\}
  \end{equation}
  which have trivial intersection with $\Omega$:
  $$ \partial_{k} \Omega_{+\infty} = \{ \ell_{n,k} \colon n \in
  {\mathbb Z}^{d} \text{ such that } \ell_{n,k} \cap \Omega =
  \emptyset \}.
  $$
  Similarly, the boundary of $\Omega$ at minus infinity is defined to
be
  the union of minus-infinity $k$-boundaries
  $$ \partial \Omega_{-\infty} = \bigcup_{k=1}^{d} \partial_{k}
  \Omega_{-\infty}
  $$
  with the minus-infinity $k$-boundary in turn defined to be the set
of
  all lines $\ell_{n,k}$ (for some $n \in {\mathbb Z}^{d}$) which are
  completely contained in $\Omega$:
  $$ \partial_{k} \Omega_{-\infty} = \{ \ell_{n,k} \colon \ell_{n,k}
  \subset \Omega \}.
  $$

  For $x(n) = \sbm{ x_{1}(n) \\ \vdots \\ x_{d}(n)}$ an
  ${\mathcal H}$-valued function on ${\mathbb Z}^{d}$,
  where ${\mathcal H} = \sbm{ {\mathcal H}_{1} \\ \vdots \\
  {\mathcal H}_{d} }$ as in the state space for a Givone--Roesser
  system, and $\Omega$ is a shift-invariant sublattice,
  we now are able to define the middle term on
  the right-hand side of \eqref{calTnorm} by
  \begin{align*}
      \| x|_{\partial \Omega} \|^{2}_{\partial \Omega} & =
  \sum_{k=1}^{d} \left\{ \sum_{ n \in \partial_{k}
\Omega_{\text{fin}}}
  \| x_{k}(n) \|^{2} +
  \sum_{ n \colon \ell_{n,k} \in \partial_{k}\Omega_{-\infty} }
  \lim_{m \to -\infty} \|
  (x_{k}|_{\ell_{n,k}})(m) \|^{2} \right. \\
  & \qquad + \left. \sum_{n \colon \ell_{n,k} \in
\partial_{k}\Omega_{+\infty}}
  \lim_{m \to \infty} \|(x_{k}|_{\ell_{n,k}})(m) \|^{2} \right\},
  \end{align*}
  where we have set in general
  $$
  (x_{k}|_{\ell_{n,k}})(m) = x_{k}(n + (m - n_{k}) \be_{k})
  = x_{k}(n_{1}, \dots, n_{k-1}, m, n_{k+1}, \dots n_{d}).
  $$
   Note that this quantity $\|x|_{\partial \Omega}
  \|^{2}_{\partial \Omega}$ is always defined (finite or
  infinite), and that by the Monotone Convergence Theorem, the order
of
  the summation and limit signs in the last two terms is immaterial.
  When this quantity is finite, we say that $x$ is {\em norm
  square-summable over} $\partial \Omega$.
 Note that the quantity $\|x|_{\partial
  \Omega}\|^{2}_{\partial \Omega}$ is the sum of more elementary
  pieces:
  \begin{align*}
  \|x|_{\partial \Omega}\|^{2}_{\partial \Omega} & =
       \| x|_{\partial
  \Omega_{\text{fin}}}\|^{2}_{\partial \Omega_{\text{fin}}}
  + \|x|_{\partial \Omega_{-\infty}}\|^{2}_{\partial \Omega_{-\infty}}
  +\|x|_{\partial \Omega_{+\infty}}\|^{2}_{\partial \Omega_{+\infty}}
  \end{align*}
where
  \begin{align*}
     \|x|_{\partial \Omega_{\text{fin}}}\|^{2}_{\partial
     \Omega} & = \sum_{k=1}^{d}
     \|x_{k}|_{\partial_{k}\Omega_{\text{fin}}}\|^{2}_{\ell^{2}
     (\partial\Omega, \cH_{k})}  \\
  \| x|_{\partial \Omega_{-\infty}}\|^{2}_{\partial \Omega} & =
     \sum_{k=1}^{d} \ \sum_{n \colon \ell_{n,k} \in \partial_{k}
     \Omega_{-\infty}} \| x_{k}|_{\ell_{n,k}}
     \|^{2}_{- \infty}   \\
      \|x|_{\partial \Omega_{+\infty}}\|^{2}_{\partial
      \Omega} & =
      \sum_{k=1}^{d} \ \sum_{n \colon \ell_{n,k} \in \partial_{k}
      \Omega_{+\infty}} \| x_{k}|_{\ell_{n,k}}\|^{2}_{+\infty}
      \end{align*}
 and finally,
  \begin{align*}
  \|x_{k}|_{\partial_{k}\Omega_{\text{fin}}}\|^{2}_{\partial
        \Omega}
     & = \|x_{k}|_{\partial_{k}\Omega_{\text{fin}}}\|^{2}_{\ell^{2}
     (\partial_{k}\Omega, {\mathcal H}_{k})}
     := \sum_{n \in \partial_{k}\Omega_{\text{fin}}} \| x_{k}(n)
     \|^{2}_{{\mathcal H}_{k}} \\
     \|x_{k}|_{\ell_{n,k}}\|^{2}_{-\infty} & = \lim_{m \to
     -\infty} \| (x_{k}|_{\ell_{n,k}})(m)\|^{2}_{{\mathcal H}_{k}} \\
     \|x_{k}|_{\ell_{n,k}}\|^{2}_{+\infty} & =
     \lim_{m \to +\infty} \|(x_{k}|_{\ell_{n,k}})(m)\|^{2}_{{\mathcal
     H}_{k}}.
  \end{align*}
  It is shown in \cite[Proposition 3.4]{MammaBear} that
  admissibility and in fact the
  $\cT$-norm given by \eqref{calTnorm} is
   independent of the choice of shift-invariant sublattice.
  Then the space ${\mathcal T} = {\mathcal
  T}^{\Sigma}$ consisting of all admissible system trajectories
  $(u,x,y)$ is a Hilbert space in the ${\mathcal T}$-norm.

  To make sense of initial conditions at minus infinity or of final
  conditions at plus infinity for system trajectories, it is
  convenient to introduce so-called {\em residual spaces}
  ${\mathcal R}_{k}$ and ${\mathcal R}_{*k}$ by
  \begin{align*}
  {\mathcal R}_{k} & = \{ \vec h = \{h(t)\}_{t=-\infty}^{\infty}
\colon
  h(t) \in {\mathcal H}_{k} \text{ with } h(t) =
  A_{kk}^{*}h(t+1) \text{ for all } t \in {\mathbb Z} \text{ and } \\
  & \qquad \| \vec h\|_{{\mathcal R}_{k}}^{2} := \lim_{t \to \infty}
\|
  h(t) \|^{2}_{{\mathcal H}_{k}} < \infty \}, \\
  {\mathcal R}_{*k} & =
  \{ \vec h = \{ h(t) \}_{t=-\infty}^{\infty}  \colon
  h(t) \in {\mathcal H}_{k} \text{ with }
  h(t+1) = A_{kk} h(t) \text{ for all } t \in {\mathbb Z}
  \text{ and } \\
  & \qquad \| \vec h \|^{2}_{{\mathcal R}_{*k}} := \lim_{t \to
-\infty}
  \|h(t)\|^{2}_{{\mathcal H}_{k}} < \infty \}. \end{align*}
  We also need a formal definition of the notion of two vector-valued
  sequences having the same asymptotics at plus or minus infinity:
   given two ${\mathcal H}$-valued sequences $x$ and $x'$ we say
  that {\em $x$ is asymptotic to $x'$ at plus-infinity}, denoted as
  $x \underset{+\infty} \sim x'$ {\em  if} $\lim_{t \to
  +\infty} \| x(t) - x'(t) \|^{2}_{{\mathcal H}} = 0$. Similarly, we
  say that {\em $x$ is asymptotic to $x'$ at minus-infinity}, denoted
  as $x \underset{-\infty} \sim x'$, {\em if } $\lim_{t \to
  -\infty} \| x(t) - x'(t) \|^{2}_{{\mathcal H}} = 0$.
  Then, as is also shown in \cite{MammaBear},  {\em given a
GR-unitary colligation
  $U$ and a shift-invariant sublattice $\Omega$ as above, then,
  corresponding to any
  $\cH_{k}$-valued sequence of the form $x_{k}|_{\ell_{n,k}}$ for an
  admissible trajectory $(u,x,y)$ and a line $\ell_{n,k} \in
  \partial_{k} \Omega_{+\infty}$, there is a sequence $\vec h_{k} \in
  \cR_{k}$ so that $x_{k}|_{\ell_{n,k}} \underset{+\infty} \sim \vec
  h_k$.}  Similarly, {\em given any line $\ell_{n,k} \in \partial_{k}
  \Omega_{-\infty}$, there is a sequence $\vec h_{*k} \in \cR_{*k}$
  so that $x_{k}|_{\ell_{n,k}} \underset{-\infty} \sim \vec h_{*k}$.}
  Conversely, one can solve the initial value problem for any
  square-summable initial data: {\em given initial/final state data,
  future input string and past output string relative to a given
  invariant sublattice $\Omega$ of the form}
  \begin{align*}
    & \begin{bmatrix} \widehat \bigoplus_{k=1}^{d}
      x_{k}^{\partial_{k}\Omega_{\text{fin}},0} \\ \widehat
\bigoplus_{k=1}^{d}
      x_{k}^{(0,+\infty),0)} \\
      \widehat \bigoplus_{k=1}^{d} x_{k}^{(-\infty,0),0} \end{bmatrix}
      \in \begin{bmatrix} \widehat \bigoplus_{k=1}^{d}
\ell^{2}(\partial_{k}
      \Omega_{\text{fin}}, \cH_{k}) \\
      \widehat \bigoplus_{k=1}^{d} \ell^{2}(\partial_{k}
\Omega_{+\infty},
      \cR_{k}) \\
     \widehat  \bigoplus_{k=1}^{d} \ell^{2}(\partial_{k}
\Omega_{-\infty},
      \cR_{*k})
  \end{bmatrix} \\
  & u^{0} \in \ell^{2}(\Omega, \cU), \qquad y^{0} \in
  \ell^{2}(\bbZ^{d} \setminus \Omega, \cY),
  \end{align*}
  {\em there is a unique admissible system trajectory $(u,x,y)$ so
that}
  \begin{align*}
 & u|_{\Omega} = u^{0}, \qquad  y|_{\bbZ^{d} \setminus \Omega} =
 y^{0}, \qquad  x|_{\partial_{k}\Omega_{\text{fin}}} =
 x_{k}^{\partial_{k} \Omega_{\text{fin}},0},  \\
 & x|_{\ell_{n,k}} \underset{+ \infty}\sim [ x_{k}^{(0,
 +\infty),0}]_{\ell_{n,k}} \text{ for each } \ell_{n,k} \in
 \partial_{k} \Omega_{+\infty}, \\
&  x|_{\ell_{n,k}} \underset{- \infty}\sim [ x_{k}^{(
-\infty,0),0}]_{\ell_{n,k}}
 \text{ for each } \ell_{n,k} \in \partial_{k} \Omega_{-\infty}.
 \end{align*}
 Hence we see that the space $\cT$ is isometrically isomorphic to the
 space its $\Omega$-coordinate version
 \begin{align*}
  \cT_{\Omega}  : = & \ell^{2}(\Omega, \cE) \widehat \oplus \widehat
\bigoplus_{k=1}^{d}
  \ell^{2}(\partial_{k}\Omega_{\text{fin}}, \cH_{k}) \widehat \oplus
  \widehat \bigoplus_{k=1}^{d} \ell^{2}(\partial_{k}
  \Omega_{+\infty}, \cR_{k})  \\
  & \qquad \widehat \oplus
\widehat \bigoplus_{k=1}^{d} \ell^{2}(\partial_{k}
\Omega_{-\infty}, \cR_{*k}) \widehat \oplus \ell^{2}(\bbZ^{d}
\setminus \Omega, \cE_*)
\end{align*}
under the map $\Gamma \colon \cT \to \cT_{\Omega}$ given by
\begin{equation}  \label{Gamma}
\Gamma \colon (u,x,y) \mapsto
\begin{bmatrix} u|_{\Omega} \\ \widehat \bigoplus_{k=1}^{d}
x|_{\partial_{k}
    \Omega_{\text{fin}}} \\ \widehat \bigoplus_{k=1}^{d} \{
    x|_{\ell_{n,k}}\}_{\ell_{n,k} \in \partial_{k}\Omega_{+\infty}}
    \\ \widehat \bigoplus_{k=1}^{d} \{
    x|_{\ell_{n,k}}\}_{\ell_{n,k} \in \partial_{k}\Omega_{-\infty}}
    \\
    y|_{\bbZ^{d} \setminus \Omega}
    \end{bmatrix}
\end{equation}
with the inverse $\Gamma^{-1}$ of $\Gamma$ computed by solving the
initial value problem described above.

  Moreover, by using the independence of the $\cT$-norm on the choice
of
  shift-invariant sublattice $\Omega$, it is not difficult to show
  that the shift operators
  \begin{equation} \label{U-calU}
     {\mathcal U}_{k} \colon (u(\cdot), x(\cdot), y(\cdot)) \mapsto
  (u'(\cdot), x'(\cdot), y'(\cdot))
  \end{equation}
  where
  \begin{align*}
     u'(n) & = u(n - \be_{k}) = u(\sigma_{k}^{-1}(n))   \\
     x'(n) & = x(n - \be_{k}) = x(\sigma_{k}^{-1}(n))  \\
     y'(n) & = y(n - \be_{k}) = y(\sigma_{k}^{-1}(n)).
     \end{align*}
  for $k=1, \dots, d$ are well-defined and unitary on $\cT$ with
  inverses given by
  $$
  {\mathcal U}_{k}^{-1} \colon (u,x,y) \mapsto  (u'',x'', y'')
  $$
  with
  \begin{align*} \notag
     u''(n) & = u(n + \be_{k}) = u(\sigma_{k}(n))  \notag  \\
     x''(n) & = x(n + \be_{k}) = x(\sigma_{k}(n)) \notag \\
     y''(n) & = y(n + \be_{k}) = y( \sigma_{k}(n)).
  \end{align*}
  In addition we define subspaces $\cF$ and $\cF_{*}$ of $\cT$ by
  \begin{align*}
       &  {\mathcal F}  = \{ (u,x,y) \colon u|_{\Omega
     \setminus \{0\}} = 0, \ \|x|_{\partial \Omega}\|^{2}_{\partial
     \Omega} = 0,\ y|_{{\mathbb Z}^{d} \setminus \Omega } =
     0\}  \text{ where }  \notag
      \\
      & \qquad \Omega \text{ is any shift-invariant sublattice with }
    0 \in  \Omega    \\
       &  {\mathcal F}_{*}  = \{ (u,x,y) \colon u|_{\Omega'} = 0,\
     \|x|_{\partial \Omega'}\|^{2}_{\partial \Omega'} = 0, \
     y|_{({\mathbb Z}^{d} \setminus \Omega') \setminus \{0\}} = 0 \}
     \text{ where }  \notag   \\
       & \qquad \Omega' \text{ is any shift-invariant sublattice with
}
       0 \in {\mathbb Z}^{d} \setminus \Omega'.
  \end{align*}
  Then, by \cite[Theorem 4.4]{BabyBear},
  \begin{equation*}
  {\mathfrak S}(\Sigma(U)) = (\cT; \cU = (\cU_{1}, \dots,
  \cU_{d}); \cF, \cF_{*})
  \end{equation*}
  is a multievolution Lax--Phillips scattering system in the sense
  defined in Section \ref{S:scatmodels} above.  Moreover, the
  scattering matrix $S(z)$ of the scattering system ${\mathfrak
  S}(\Sigma(U))$ can be identified with the transfer function
  $T_{\Sigma(U)}(z)$ of the conservative Givone--Roesser system
$\Sigma(U)$
  after some trivial modifications:  if we let
  \begin{align}
  & \Omega = \text{ any shift-invariant sublattice with } 0 \in
  \Omega, \notag \\
  & \Omega' = \text{ any shift-invariant sublattice with }0 \notin
  \Omega',
  \label{Omega/Omega'}
  \end{align}
  then
  \begin{equation}  \label{trans-scat-func}
   S_{{\mathfrak S}(\Sigma(U))}(z) i = i_{*} T_{\Sigma(U)}(z)
  \text{ for almost all } z \in \bbT^{d}
  \end{equation}
  where $i \colon \cE \to \cF$ and $i_{*} \colon \cE_{*} \to \cF_{*}$
  are the unitary identification maps\footnote{It turns out that
  the operators $i$ and $i_{*}$ defined via \eqref{iEF} and
\eqref{iE*F*} are
  independent of the particular choice of shift-invariant sublattices
  $\Omega$ and $\Omega'$
  satisfying \eqref{Omega/Omega'}.}
  \begin{align}
    &  i \colon e \mapsto (u(\cdot), x(\cdot), y(\cdot)) \in \cT
      \text{ such that } u|_{\Omega} = \{ \delta_{n,0}e\}_{n \in
\Omega},
      x|_{\partial \Omega} = 0, y|_{\bbZ^{d}\setminus \Omega} = 0
      \label{iEF}  \\
 &  i_{*} \colon e_{*} \mapsto (u(\cdot), x(\cdot), y(\cdot)) \in \cT
   \text{ such that } \notag \\
   & \qquad
   u|_{\Omega'} = 0, x|_{\partial \Omega'} = 0, y|_{\bbZ^{d} \setminus
   \Omega'} = \{ \delta_{n,0} e_{*} \}_{n \in \bbZ^{d} \setminus
\Omega'}.
   \label{iE*F*}
 \end{align}

  Let us say that a
  multievolution scattering system ${\mathfrak S}$ of the form
  ${\mathfrak S}(\Sigma(U))$ for a conservative Givone--Roesser system
  $\Sigma(U)$ is a {\em multievolution scattering system with
   embedded unitary colligation $U$}; one of the main connections
   between $U$ and ${\mathfrak S}(\Sigma(U))$ is the property
   \eqref{trans-scat-func}:  {\em the scattering
  matrix for ${\mathfrak S}(\Sigma(U))$ coincides with the transfer
  function for $\Sigma(U)$}.

  \begin{remark}  \label{R:Schaffer}
      The above analysis shows that a unitarily equivalent version
       of the multievolution
      scattering system is
      determined by using the unitary map $\Gamma$ given by
      \eqref{Gamma} to map the space of admissible trajectories $\cT$
      to the $\Omega$-coordinatized version $\cT_{\Omega}$;
      specifically, we may define ${\mathfrak S}_{\Omega}(\Sigma(U))$
      by
      $$
    {\mathfrak S}_{\Omega}(\Sigma(U)) = (\cT_{\Omega}; \,
\cU_{\Omega} =
    (\cU_{\Omega,1}, \dots, \cU_{\Omega,d});\,  \cF_{\Omega},
    \cF_{_{\Omega} *})
     $$
     where, for $k = 1, \dots, d$, $\cU_{\Omega,k}$ on $\cT_{\Omega}$
is given by
     $$
     \cU_{\Omega,k} = \Gamma \cU_{k} \Gamma^{*}, \cF_{\Omega} =
     \Gamma \cF,\, \cF_{\Omega*} = \Gamma \cF_{*}
     $$
     (with $\cU_{k}$ on $\cT$ given by \eqref{U-calU}).
     Of course,
     $\cT$ depends on $U$ and $\Gamma$ and $\Gamma^{*}$ depend on $U$
     and $\Omega$; as $U$ and $\Omega$ are considered fixed, we
     suppress this dependence from the notation.  In the preceding
     discussion (as well as in \cite{BabyBear}) we avoided working
     out these operators
     $\cU_{\Omega,k}$ explicitly.  A convenient special case is the
     {\em balanced} shift-invariant sublattice $\Omega^{B}$ (see
     Section 4.4.1 of \cite{BabyBear} and also Section 1.5 of
\cite{K-JOT2000}) given by
     $$
     \Omega^{B} = \{ n = (n_{1}, \dots, n_{d}) \in {\mathbb Z}^{d}
     \colon n_{1} + \cdots + n_{d} \ge 0 \}
     $$
     as there are no infinite boundary components in this case.  For
     each $k$ the finite boundary is equal to the subset $\Xi$ of
     ${\mathbb Z}^{d}$ given by
     $$
     \Xi = \{ n = (n_{1}, \dots, n_{d}) \in {\mathbb Z}^{d} \colon
     n_{1} + \cdots + n_{d} = 0\}.
     $$
     Hence the $\Omega^{B}$-coordinatized version of $\cT$ reduces to
     $$
     \cT_{\Omega^{B}} = \begin{bmatrix} \ell^{2}({\mathbb
     Z}^{d}\setminus \Omega^{B}, \cE_{*}) \\ \bigoplus_{k=1}^{d}
     \ell^{2}(\Xi, \cH_{k}) \\ \ell^{2}(\Omega^{B}, \cE)
 \end{bmatrix}.
 $$
 We view elements of $\ell^{2}({\mathbb Z}^{d} \setminus
 \Omega^{B},\cE_{*})$ as functions $n \mapsto \vec e_{*}(n)$ of $n \in
 {\mathbb Z}^{d} \setminus \Omega^{B}$ with values $\vec e_{*}(n) \in
 \cE_{*}$, and similarly for $\ell^{2}(\Omega^{B}, \cE)$.  We view
 elements of $\bigoplus_{k=1}^{d} \ell^{2}(\Xi, \cH_{k})$ as
functions
 $(n,j) \mapsto \vec x(n,j)$ of $(n,j) \in \Xi \times \{1, \dots,
 d\}$ such that the value $x(n,j)$ is in $\cH_{j}$. For the
 computation to follow, we drop the subscript $\Omega^{B}$ and write
 $\cU_{k}$ rather than $\cU_{\Omega^{B},k}$.
 For a fixed $k
 \in \{1, \dots, d\}$, the operator $\cU_{k}^{*}$ on
 $\cT_{\Omega^{B}}$ can be viewed as a $3 \times 3$-block matrix of
 the form
 $$
 \cU_{k}^{*} = \begin{bmatrix} [\cU_{k}^{*}]_{11} &
 [\cU_{k}^{*}]_{12} & [\cU^{*}]_{13} \\
 0 & [ \cU_{k}*]_{22} & [ \cU_{k}^{*}]_{23} \\
 0 & 0 & [ \cU_{k}^{*}]_{33} \end{bmatrix}.
 $$
 The various matrix entries can be given explicitly as follows:
 \begin{align}
    & \left( [\cU^{*}_{k}]_{11} \,  \vec e_{*}\right)(n) =
\begin{cases}
    \vec e_{*}(n + \be_{k}) &\text{if } n + \be_{k} \in {\mathbb
    Z}^{d} \setminus \Omega^{B}, \\
    0 & \text{otherwise,} \end{cases} \notag
    \\
    & \left( [\cU_{k}^{*}]_{12}\, \vec x\right) (n) = \begin{cases}
    \sum_{j=1}^{d} C_{j} \vec x(n+ \be_{k},j) &\text{if } n +
    \be_{k} \in \Xi, \\
    0 &\text{otherwise,}
    \end{cases}  \notag
    \end{align}
    \begin{align*}
    & \left( [\cU_{k}^{*}]_{13}\,  \vec e \right)  (n) =
\begin{cases} D
    \vec e(n+e_{k}) &\text{if } n + \be_{k} \in \Xi \subset
    \Omega^{B}, \notag \\
    0 & \text{otherwise,}  \end{cases} \\
    & \left( [ \cU^{*}_{k}]_{22} \,  \vec x \right) (n,j) = \sum_{j'
= 1}
    ^{ d} A_{j j'} \vec x (n + \be_{k} - \be_{j}, j'), \notag \\
    &  \left( [\cU_{k}^{*}]_{23} \, \vec e \right) (n,j) = B_{j} \vec
    u(n+\be_{k} - \be_{j}), \notag \\
    & \left( [\cU_{k}^{*}]_{33} \, \vec e\right) (n) = \vec e(n +
    \be_{k})
    \end{align*}

    One can understand the geometric structure of $\cU_{k}^{*}$ as
    follows.  In the following discussion, let us identify
    $\ell^{2}({\mathbb Z}^{d} \setminus \Omega^{B}, \cE_{*})$ with
    its embedded version $\sbm{ \ell^{2}({\mathbb Z}^{d}
    \setminus \Omega^{B}, \cE_{*}) \\ 0 \\ 0 }$ inside $\cT_{\Omega^{B}}$,
    and similarly for $\ell^{2}(\Xi, \cH_{k})$,
    $\ell^{2}(\Omega_{B}, \cE)$ and subspaces thereof.  Then one can
    see that $\cU_{k}^{*}$ maps the first subspace unitarily onto the
    second for the following pairs of subspaces:
    \begin{align}
    & \cU_{k}^{*} \colon \ell^{2}({\mathbb Z}^{d} \setminus
    \Omega^{B}, \cE_{*}) \to \ell^{2}( - \be_{k} + ({\mathbb Z}^{d}
    \setminus \Omega^{B}), \cE_{*})  \label{action1}\\
    & \cU_{k}^{*} \colon \ell^{2}(\be_{k} + \Omega^{B}, \cE) \to
    \ell^{2}_{\cE}(\Omega^{B}), \label{action2} \\
    & \cU_{k}^{*} \colon \begin{bmatrix} \bigoplus_{k=1}^{d}
    \ell^{2}(\Xi, \cH_{k}) \\ \ell^{2}_{\cE}(\Xi) \end{bmatrix} \to
    \begin{bmatrix} \bigoplus_{k=1}^{d} \ell^{2}(\Xi, \cH_{k}) \\
    \ell^{2}(-\be_{k} + \Xi, \cE_{*}) \end{bmatrix}.
    \label{action3}
    \end{align}
  where the actions in \eqref{action1} and \eqref{action2} are via the
  $k$-th backward shift operator $\xi(n) \mapsto \xi(n + \be_{k})$.
  Given that each of these three actions is unitary, the unitarity of
  $\cU_{k}^{*}$ now becomes obvious from the fact that
  $\cT_{\Omega^{B}}$ can be decomposed orthogonally as
  \begin{align*}
    \cT_{\Omega^{B}}& = \ell^{2}({\mathbb  Z}^{d} \setminus
\Omega^{B}, \cE_{*})
     \oplus \bigoplus_{k=1}^{d} \ell^{2}(\Xi, \cH_{k})
    \oplus \left( \ell^{2}(\Xi, \cE) \oplus \ell^{2}( \be_{k}
    + \Omega^{B}, \cE)  \right)  \\
    & = \left( \ell^{2}(-\be_{k} + ({\mathbb Z}^{d}
    \setminus \Omega^{B}), \cE_{*}) \oplus \ell^{2}(-\be_{k}
    + \Xi, \cE_{*}) \right)
    \oplus \bigoplus_{k=1}^{d} \ell^{2}(\Xi, \cH_{k}) \oplus
    \ell^{2}(\Omega^{B}, \cE).
  \end{align*}
  Moreover, the last action \eqref{action3} can be decomposed
  further; for each $n^{0}\in \Xi$ we have that
  \begin{equation} \label{action4}
  \cU_{k}^{*} \colon \begin{bmatrix} \delta_{n^{0}}
\Big(\bigoplus_{j=1}^{d}
  \ell^2(\Xi, \cH_{j})\Big) \\ \delta_{n^{0}} \ell^2(\Xi, \cE)
\end{bmatrix} \mapsto \begin{bmatrix}
  \bigoplus_{j=1}^{d} \delta_{n^{0} + \be_{j} - \be_{k}}
\ell^2(\Xi, \cH_{j}) \\ \delta_{n^{0}-
  \be_{k}}\ell^2(\Xi, \cE_{*}) \end{bmatrix}
  \end{equation}
  unitarily (where, for $m \in \Xi$,  $\delta_{m}$ denotes the delta
  function on $\Xi$:  $\delta_{m}(n) = 0$ for $n \ne m$ and $
  \delta_{m}(n) = 1$ for $n=m$ for $n \in \Xi$. Given that
  \eqref{action4} is unitary, it follows that \eqref{action3} is
  unitary from the fact that the mapping
  $$
    (n,j) \mapsto (n+\be_{j} - \be_{k}, j)
   $$
   is bijective on $\Xi \times \{1, \dots, d\}$.

  Explicitly, the
  mapping \eqref{action4} is given by
  \begin{multline*}
  \cU_{k}^{*} \colon \begin{bmatrix} \delta_{n^{0}} \Big(
\bigoplus_{j=1}^{d}
  x_{j}(\cdot)\Big) \\ \delta_{n^{0}} e(\cdot) \end{bmatrix} =
  \begin{bmatrix}   \bigoplus_{j=1}^{d}
  x_{j}(n^0) \\ e(n^0) \end{bmatrix}\\
    \mapsto \begin{bmatrix}
  \delta_{n^{0} + \be_{1} - \be_{k}} & & & \\
  & \ddots & & \\ & & \delta_{n^{0}+ \be_{d} - \be_{k}} \\ & & &
  \delta_{n^{0}-\be_{k}}
  \end{bmatrix} \begin{bmatrix} A & B \\ C & D \end{bmatrix}
  \begin{bmatrix} \bigoplus_{j=1}^dx_{j}(\cdot) \\ e(\cdot)
\end{bmatrix}\\
  = \begin{bmatrix} A & B \\ C & D \end{bmatrix}
  \begin{bmatrix} \bigoplus_{j=1}^dx_{j}(n^0+e_j-e_k) \\ e(n^0-e_k)
\end{bmatrix},
  \end{multline*}
  from which the unitarity of the mapping \eqref{action4} follows
  immediately from the unitarity of the colligation matrix $U = \sbm{
A
  & B \\ C & D }$.  With a little more work one can verify explicitly
  that the  $d$-tuple $(\cU_{1}, \dots, \cU_{d})$ is commutative.
  This representation for the scattering system with given embedded
  GR-unitary colligation reduces to the well-known Sch\"affer matrix
  for the classical $d=1$ case (see \cite{NF}).
 \end{remark}

  It is also shown in \cite{BabyBear} that the multievolution
scattering
  system ${\mathfrak S}(\Sigma(U))$  with given embedded unitary
  colligation $U$
  carries some additional geometric structure.  Given such a
  scattering system ${\mathfrak S}(\Sigma(U))$,  define subspaces
  ${\mathcal L}_{k}$, ${\mathcal M}_{k}$ and ${\mathcal
  M}_{*k}$ by
  \begin{align*}
     & {\mathcal L}_{k} = \{ (u,x,y) \in {\mathcal T} \colon
     u|_{\Omega} = 0, \
     x|_{\partial \Omega_{\infty}} = 0,\
     x|_{\partial_{j}\Omega_{\text{fin}}} = 0 \text{ for } j \ne k,
     \notag \\
    & \qquad  x|_{\partial_{k}\Omega_{\text{fin}} \setminus \{0\}} = 0, \
     y|_{{\mathbb Z}^{d} \setminus \Omega} = 0\}
     \notag \\
     & \qquad
      \text{where } \Omega \text{ is any shift-invariant
     sublattice with } 0 \in \partial_{k}\Omega_{\text{fin}},
       \\
  & {\mathcal M}_{k} = \{ (u,x,y) \in {\mathcal T} \colon u|_{\Omega}
=
  0, \ x|_{\partial \Omega_{\text{fin}}} = 0, \ x|_{\partial
  \Omega_{-\infty}} = 0, \ x|_{\partial_{j} \Omega_{+\infty}} = 0
\text{
  for } j \ne k, \notag \\
  & \qquad x|_{\partial_{k} \Omega_{+\infty} \setminus \ell_{0,k}} =
  0, \  y|_{{\mathbb Z}^{d} \setminus \Omega} = 0 \} \notag \\
  & \qquad
      \text{where } \Omega \text{ is any shift-invariant
      sublattice with }  \ell_{0,k} \in \partial_{k}\Omega_{+\infty},
      \\
  & {\mathcal M}_{*k} = \{ (u,x,y) \in {\mathcal T} \colon
u|_{\Omega} =
  0, \ x|_{\partial \Omega_{\text{fin}}} = 0, \ x|_{\partial
  \Omega_{+\infty}} = 0, \ x|_{\partial_{j} \Omega_{-\infty}} = 0
\text{
  for } j \ne k,\notag \\
  & \qquad x|_{\partial_{k} \Omega_{-\infty} \setminus \ell_{0,k}} =
  0, \  y|_{{\mathbb Z}^{d} \setminus \Omega} = 0 \} \notag \\
  & \qquad
      \text{where } \Omega \text{ is any shift-invariant
      sublattice with }  \ell_{0,k} \in \partial_{k}\Omega_{-\infty}.
  \end{align*}
  It is shown in \cite{BabyBear} that $\cL_{1}, \dots, \cL_{d}$ are
mutually
  orthogonal and that $\cL_{k}$ is wandering for the commuting
  unitary $(d-1)$-tuple $\widehat \cU^{k}$ (equal to the tuple $\cU$
  with $\cU_{k}$ left out), i.e.,
  $$
  (\widehat \cU^{k})^{\widehat n^{k}} \cL_{k} \perp (\widehat
  \cU^{k})^{\widehat n^{\prime k}}
  \cL_{k} \text{ if } n \ne n' \text{ in } \bbZ^{d-1};
  $$
  here we use the notation
  $$ (\widehat \cU^{k})^{\widehat n^{k}} = \cU_{1}^{n_{1}} \cdots
  \cU_{k-1}^{n_{k-1}} \cU_{k+1}^{n_{k+1}} \cdots \cU_{d}^{n_{d}}
  $$
  for $n = (n_{1}, \dots, n_{d}) \in {\mathbb Z}^{d}$ and $k \in \{1,
  \dots, d\}$.
 Also shown in \cite{BabyBear} is that fact that
 $\cM_{k}$ and $\cM_{*k}$ are doubly invariant for $\cU_{k}$ for
 $k = 1, \dots, d$.
Moreover, if we let $\Omega$ be any shift-invariant sublattice,
then the scattering subspace ${\mathcal V}^{\Omega}$ of
${\mathfrak S}(\Sigma(U))$ associated with $\Omega$ given by
 $$
       {\mathcal V}^{\Omega} = {\mathcal T} \ominus \left[
      \left(\bigoplus_{n \in {\mathbb Z}^{d} \setminus \Omega}
      {\mathcal U}^{n} {\mathcal F}_{*}\right)
      \oplus \left( \bigoplus_{n \in \Omega}
      {\mathcal U}^{n} {\mathcal F}\right)  \right]
      $$
 has the orthogonal decomposition
       \begin{align}
       {\mathcal V}^{\Omega} & =
       \left(\bigoplus_{k=1}^{d}\,
       \bigoplus_{n \in \partial_{k}\Omega_{\text{\rm fin}}} {\mathcal
       U}^{n}{\mathcal L}_{k} \right) \notag \\
      & \oplus \left( \bigoplus_{k=1}^{d}\, \bigoplus_{n' \colon
\ell_{n',k} \in
       \partial_{k}\Omega_{+\infty}} ( \widehat{\mathcal
       U}^{k})^{\widehat{n'}^{k}}
       {\mathcal M}_{k} \right) \oplus \left( \bigoplus_{k=1}^{d}\,
       \bigoplus_{n''\colon \ell_{n'',k} \in
       \partial_{k} \Omega_{-\infty}} (\widehat{\mathcal
       U}^{k})^{\widehat{n''}^{k}}
       {\mathcal M}_{*k}  \right) .
       \tag{OD}
      \label{Omegadecom}
      \end{align}
In addition, the subspaces ${\mathcal F}$, ${\mathcal F}_{*}$,
${\mathcal L}_{k}$, ${\mathcal M}_{k}$ and ${\mathcal M}_{*k}$
satisfy the Compatible Decomposition Property
\begin{equation}
     \left( \bigoplus_{k=1}^{d} {\mathcal L}_{k}\right)  \oplus
{\mathcal F}
     = {\mathcal  F}_{*} \oplus \left( \bigoplus_{j=1}^{d} {\mathcal
U}_{j}
     {\mathcal L}_{j}\right), \tag{CDP} \label{CDP}
\end{equation}
and the strong limit properties
\begin{align}
     &  \slim_{m \to \infty} P_{{\mathcal U}_{k}^{m}
       {\mathcal L}_{k}} = P_{{\mathcal M}_{k}},
       \tag{SL} \label{slim}  \\
     &  \slim_{m \to -\infty}
       P_{{\mathcal U}_{k}^{m}{\mathcal L}_{k}} = P_{{\mathcal
M}_{*k}}
       \tag{SL$_*$} \label{slim*}
\end{align}
(see \cite[Theorem 4.6]{BabyBear}).  One of the main results from
\cite{BabyBear} is the converse.

\begin{theorem} \label{T:BabyBear} (See \cite[Theorem 4.8]{BabyBear}.)
    Let ${\mathfrak S} = ({\mathcal K};
    {\mathcal U} = ({\mathcal U}_{1}, \dots, {\mathcal U}_{d});
    {\mathcal F}, {\mathcal F}_{*})$  be a $d$-evolution
    unitary scattering system and let $\Omega$ be a shift-invariant
    sublattice.  Assume:
    \begin{enumerate}
    \item There exist subspaces ${\mathcal L}_{1}, \dots, {\mathcal
    L}_{d}$, ${\mathcal M}_{1}, \dots, {\mathcal M}_{d}$ and ${\mathcal
    M}_{*1}, \dots, {\mathcal M}_{*d}$ of ${\mathcal K}$ such that
    \begin{enumerate}
        \item

    ${\mathcal M}_{k}$ and ${\mathcal M}_{*k}$ are doubly invariant for
    ${\mathcal U}_{k}$ for each $k = 1, \dots, d$, and
    \item the scattering subspace
    $$ {\mathcal V}^{\Omega}: = {\mathcal K} \ominus \left[ \left(
    \bigoplus_{n \in \Omega} {\mathcal U}^{n} {\mathcal F} \right)
    \oplus \left( \bigoplus_{n \in {\mathbb Z}^{d} \setminus \Omega}
    {\mathcal U}^{n} {\mathcal F}_{*} \right) \right]
    $$
    has the orthogonal decomposition
    \begin{align}
        {\mathcal V}^{\Omega} = &
    \left(\bigoplus_{k=1}^{d}
       \bigoplus_{n \in \partial_{k}\Omega_{\text{fin}}} {\mathcal
       U}^{n}{\mathcal L}_{k} \right)
       \oplus  \left( \bigoplus_{k=1}^{d} \bigoplus_{n' \colon
\ell_{n',k} \in
       \partial_{k}\Omega_{+\infty}} ( \widehat{\mathcal
       U}^{k})^{\widehat{n'}^{k}}
       {\mathcal M}_{k} \right)   \notag \\
       & \qquad \oplus
     \left( \bigoplus_{k=1}^{d}
       \bigoplus_{n''\colon \ell_{n'',k} \in
       \partial_{k} \Omega_{-\infty}} (\widehat{\mathcal
       U}^{k})^{\widehat{n''}^{k}}
       {\mathcal M}_{*k}  \right).
       \tag{\ref{Omegadecom}}
    \end{align}
    \end{enumerate}
    \item $({\mathfrak S}, {\mathcal L}_{1}, \dots, {\mathcal L}_{d})$
    has the compatible decomposition property
    \begin{equation}
         \left( \bigoplus_{k=1}^{d} {\mathcal L}_{k} \right) \oplus
         {\mathcal F} = {\mathcal F}_{*} \oplus \left(
         \bigoplus_{k=1}^{d} {\mathcal U}_{k} {\mathcal L}_{k} \right).
         \tag{\ref{CDP}}
    \end{equation}
    \item
    The subspaces ${\mathcal L}_{k}$ and ${\mathcal M}_{k}$ are
    connected via the formula
    \begin{equation}  \label{slim'}
        \slim_{m \to \infty} P_{{\mathcal
        U}_{k}^{m}{\mathcal L}_{k}} = P_{{\mathcal M}_{k}}.
        \tag{\ref{slim}}
    \end{equation}
    Similarly the subspaces ${\mathcal L}_{k}$ and ${\mathcal M}_{*k}$
are
    connected via the formula
    \begin{equation}  \label{slim*'}
        \slim_{m \to -\infty} P_{{\mathcal
        U}_{k}^{m}{\mathcal L}_{k}} = P_{{\mathcal M}_{*k}}.
        \tag{\ref{slim*}}
    \end{equation}
    \end{enumerate}
    Let $\cH_{k}$, $\cE$, $\cE_{*}$ be copies of $\cL_{k}$, $\cF$ an
    $\cF_{*}$ respectively with associated unitary identification maps
    $$i_{\cH_{k}} \colon \cH_{k} \to \cL_{k} \text{ for } k = 1, \dots,
d,
    \qquad
    i_{\cE} \colon \cE \to \cF, \qquad
    i_{\cE_{*}} \colon \cE_{*} \to \cF_{*}.
    $$
    Define the operator $U \colon \left(\widehat  \bigoplus_{k=1}^{d}
{\mathcal
    H}_{k} \right) \widehat \oplus {\mathcal E} \to  \left(
    \widehat \bigoplus_{k=1}^{d} {\mathcal
    H}_{k} \right) \widehat \oplus {\mathcal E}_{*}$ by
    \begin{equation}
           U = \begin{bmatrix} (i_{{\mathcal H}_{1}})^{*} {\mathcal
U}_{1}^{*} \\
          \vdots \\ (i_{{\mathcal H}_{d}})^{*} {\mathcal U}_{d}^{*} \\
          (i_{{\mathcal E}_{*}})^{*}  \end{bmatrix}
          \begin{bmatrix} i_{{\mathcal H}_{1}} & \dots & i_{{\mathcal
          H}_{d}} & i_{{\mathcal E}} \end{bmatrix}
    \colon \begin{bmatrix} {\mathcal H}_{1} \\ \vdots \\ {\mathcal H}_{d}
    \\  {\mathcal E} \end{bmatrix}  \to
    \begin{bmatrix} {\mathcal H}_{1} \\ \vdots \\ {\mathcal H}_{d}
    \\  {\mathcal E}_{*} \end{bmatrix}.
    \label{UfromcalU'}
    \end{equation}
        Then $U$ is a $d$-variable unitary colligation such that
${\mathfrak
    S}$ is unitarily equivalent to  ${\mathfrak S}(\Sigma(U))$ under the
    map  ${\mathcal J} \colon {\mathcal K} \to
    {\mathcal T}$ given by \eqref{Jmap1}--\eqref{Jmap2}:
    \begin{equation} \label{Jmap1}
    {\mathcal J} \colon \xi \mapsto (u(\cdot), x_{1}(\cdot), \dots,
    x_{d}(\cdot), y(\cdot))
    \end{equation}
    with
    \begin{equation} \label{Jmap2}
    u(n)  = i_{\cE}^{*}{\mathcal U}^{*n} \xi, \qquad
    x_{k}(n)  = i_{\cH_{k}}^{*} {\mathcal U}^{*n} \xi \text{ for } k
    = 1, \dots, d, \qquad
    y(n)  = i_{\cE_{*}}^{*}{\mathcal U}^{*n} \xi
    \end{equation}
    for $n \in {\mathbb Z}^{d}$.    Moreover, if $\Omega'$
    is any other shift-invariant sublattice, then
    property \eqref{Omegadecom} holds with $\Omega'$ in place of
    $\Omega$ as well.

    The Fourier representation
    operators $\Phi$ and $\Phi_{*}$ (see \eqref{Fourierrep}) are
    given by
    \begin{equation}  \label{Fourierrep-colmod}
     ( \Phi \xi)(z) = \sum_{n \in \bbZ^{d}} i_{\cE}u(n)
     z^{n}, \qquad
     (\Phi_{*}\xi)(z) = \sum_{n \in \bbZ^{d}} i_{\cE_{*}} y(n)
     z^{n} \text{ if } {\mathcal J} \xi = (u(\cdot), x(\cdot), y(\cdot)).
     \end{equation}

    In addition, the scattering system ${\mathfrak S}(\Sigma(U))$
    is minimal (see Definition \ref{D:minimal-scat}) if and only
    if the GR-unitary colligation $U$ is {\em scattering-minimal}
    in the sense that for some (or equivalently,
    for any) shift-invariant sublattice $\Omega$ the map
    \begin{equation}  \label{col-scatmin}
    \Pi^{dBR,\partial \Omega}_{U} \colon \begin{bmatrix}
    \widehat \bigoplus_{k=1}^{d} \ell^{2}(\partial_{k}\Omega,
    \operatorname{im}P_{k}) \\
       \widehat \bigoplus_{k=1}^{d}
\ell^{2}(\partial_{k}\Omega_{+\infty},
    \cR_{k}) \\  \widehat \bigoplus_{k=1}^{d}
    \ell^{2}(\partial_{k}\Omega_{-\infty}, \cR_{*k}) \end{bmatrix} \to
              (\cE_{*} \oplus \cE) [[ z^{\pm 1} ]] \text{ \em is
injective}
    \end{equation}
    where $\Pi^{dBR, \partial \Omega}_{U}$ is   given by
    $$ \Pi^{dBR, \partial \Omega}_{U} \colon
    \begin{bmatrix}  \widehat \bigoplus_{k=1}^{d}
        \widehat{x_{k}^{0,fin}}^{\partial_{k}\Omega_{\text{fin}}}(z) \\
       \widehat  \bigoplus_{k=1}^{d} \widehat{ x_{k}^{0,+\infty}}(z) \\
       \widehat  \bigoplus_{k=1}^{d} \widehat{ x_{k}^{0,-\infty}}(z)
    \end{bmatrix} \mapsto
     \begin{bmatrix} \widehat y(z) \\ \widehat u(z) \end{bmatrix}
    $$
    where
 \begin{align}
     & \widehat y(z) =
     C (I - Z_{\text{diag}}(z)A)^{-1}
     \sbm{
\widehat{x_{1}^{0,\text{fin}}}^{\partial_{1}\Omega_{\text{\text{fin}}}}(z)
     \\ \vdots \\
     \widehat{x_{d}^{0,\text{fin}}}^{\partial_{d}\Omega_{\text{fin}}}(z)}
     \notag \\
     & \,  +
     C \sbm{ \widehat{x_{1}^{0,-\infty}}(z) \\ \vdots \\
     \widehat{x_{d}^{0,-\infty}}(z)} - C(I - Z_{\text{diag}}(z) A)^{-1}
     \left(\sbm{ \widehat{x_{1}^{0,-\infty}}(z) \\ \vdots \\
     \widehat{x_{d}^{0,-\infty}}(z)} - Z_{\text{diag}}(z)A
     \sbm{ \widehat{x_{1}^{0,-\infty}}(z) \\ \vdots \\
     \widehat{x_{d}^{0,-\infty}}(z)} \right),
     \notag
     \end{align}
     \begin{align}
 & \widehat u(z) =  B^{*} Z_{\text{diag}}(z)^{-1} (I - A^{*}
     Z_{\text{diag}}(z)^{-1})^{-1}
     \sbm{
\widehat{x_{1}^{0,\text{fin}}}^{\partial_{1}\Omega_{\text{fin}}}(z)
     \\ \vdots \\
\widehat{x_{d}^{0,\text{fin}}}^{\partial_{d}\Omega_{\text{fin}}}(z)
}
     \notag \\
     & \,  + B^{*} Z_{\text{diag}}(z)^{-1}
     \sbm{ \widehat{x_{1}^{0, +\infty}}(z) \\ \vdots \\
     \widehat{x_{d}^{0, +\infty}}(z) }  +
     B^{*} Z_{\text{diag}}(z)^{-1} (I - A^{*}
     Z_{\text{diag}}(z)^{-1})^{-1} \cdot \notag \\
     & \qquad \cdot \left( A^{*}Z_{\text{diag}}(z)^{-1}
     \sbm{ \widehat{x_{1}^{0, +\infty}}(z) \\ \vdots \\
    \widehat{x_{d}^{0, +\infty}}(z) } -
    \sbm{ \widehat{x_{1}^{0, +\infty}}(z) \\ \vdots \\
        \widehat{x_{d}^{0, +\infty}}(z) } \right).
    \label{uyfromx}
 \end{align}
 where
 $$
 \widehat{x_{k}^{0,+\infty}}(z) = \sum_{m  \colon \ell_{m,k} \in
 \partial_{k}\Omega_{+\infty}} \sum_{t=-\infty}^{\infty} \left(
 x_{k}^{0,+\infty}|_{\ell_{m,k}}\right)(t)
 z_{k}^{t}(\widehat{z}^{k})^{\widehat m^{k}}
 $$
 and where
 $$
 \widehat{x_{k}^{0,-\infty}}(z) = \sum_{m \colon \ell_{m,k} \in
 \partial_{k}\Omega_{-\infty}} \sum_{t=-\infty}^{\infty} \left(
 x_{k}^{0,-\infty}|_{\ell_{m,k}}\right)(t)
 z_{k}^{t}(\widehat{z}^{k})^{\widehat m^{k}}.
 $$

    \end{theorem}

    \begin{proof}
  The only new observation beyond \cite[Theorem 4.8]{BabyBear} is the
  criterion  \eqref{col-scatmin} for minimality of the scattering
  system ${\mathfrak S}(\Sigma(U))$ expressed directly in terms of the
  GR-unitary colligation $U$. In general, a vector $\xi$ in the
  ambient space $\cK$ for the scattering system ${\mathfrak S}$ is
  orthogonal to the minimal subspace equal to the closure of
  $\widetilde \cW + \widetilde
  \cW_{*} \subset \cK$ for ${\mathfrak S}$ (see Definition
  \ref{D:minimal-scat}) if and only if both $\Phi \xi = 0$ and
  $\Phi_{*} \xi = 0$.  From \eqref{Fourierrep-colmod} we see that
  this is equivalent to the element ${\mathcal J} \xi$ having the
form $(0,
  x(\cdot), 0)$.  Hence minimality of the scattering system
  ${\mathfrak S}(\Sigma(U))$ translates to:
  $$
    (0, x(\cdot), 0) \in \cT \Rightarrow x(\cdot ) = 0.
  $$
  If we use the isometric isomorphism $\Gamma \colon \cT \to
  \cT_{\Omega}$ as in \eqref{Gamma}, we see that the condition is:
  $$
  \Gamma^{-1} \begin{bmatrix} 0 \\ \widehat \bigoplus_{k=1}^{d}
  x^{\partial_{k}\Omega_{\text{fin}},0} \\
     \widehat  \bigoplus_{k=1}^{d} x^{(0,\infty),0} \\ \widehat
\bigoplus_{k=1}^{d}
      x^{(-\infty,0),0} \\ 0 \end{bmatrix} \text{ has the form } (0,
      x(\cdot), 0) \Rightarrow x(\cdot) = 0.
  $$
  The formula \eqref{uyfromx} is just the formula from
  \cite[Theorem 4.15]{MammaBear} for the input string $u$ and output
string $y$
  associated with the unique trajectory $(u,x,y)$ which solves the
  initial value problem
  $$
   u|_{\Omega} = 0, \qquad x|_{\partial \Omega} =
   \begin{bmatrix} \widehat \bigoplus_{k=1}^{d}
       x^{\partial_{k}\Omega_{\text{fin}},0} \\
      \widehat \bigoplus_{k=1}^{d} x^{(0,\infty),0} \\ \widehat
\bigoplus_{k=1}^{d}
      x^{(-\infty,0),0} \end{bmatrix}, \qquad y|_{\bbZ^{d} \setminus
      \Omega} = 0;
   $$
 for more complete details on how to make rigorous sense of the
 formulas in \eqref{uyfromx}, we refer to \cite{MammaBear}.
Hence minimality of the scattering system ${\mathfrak
S}(\Sigma(U))$ is equivalent to injectivity of the map
$\Pi^{dBR,\partial \Omega}_{U}$ as wanted.
  \end{proof}

  \begin{remark}
   Note that the formula for $\widehat u$ in \eqref{uyfromx} involves
only $x|_{\partial
   \Omega_{\text{fin}} \cup \partial \Omega_{- \infty}}$ while the
formula
   for $\widehat y$ in \eqref{uyfromx} involves only $x|_{\partial
   \Omega_{\text{fin}} \bigcup \partial \Omega_{+\infty}}$.  In the
spirit of
   the terminology introduced in \cite{Kailath}, given a position $n
   \in \bbZ^{d}$ and a shift-invariant sublattice with $n \in
   \cap_{k=1}^{d} \partial_{k}\Omega_{\text{fin}}$, it makes sense to
view $x(n)$ as the
   {\em local state} at $n$, $x|_{\partial \Omega_{fin} \cup
\Omega_{-\infty}}$
   as the {\em global state} for the forward system \eqref{polysys} at
   $n$, and $x|_{\partial \Omega_{\text{fin}} \cup \partial
   \Omega_{+\infty}}$ as the {\em global state} for the backward
   system \eqref{polysysback} at position $n$.
   \end{remark}

    As a corollary, we have the following result.
    \begin{corollary}  \label{C:summary}
    Suppose that $S \in \cS(\cE, \cE_{*})$ is in the $d$-variable
    Schur--Agler class.
    Then the following conditions are equivalent.
    \begin{enumerate}

    \item $S(z)$ has an Agler decomposition
     $$ I - S(z) S(w)^{*} = \sum_{k=1}^{d} (1 - z_{k}
\overline{w_{k}})
    H_{k}(z) H_{k}(w)^{*}
     $$
     for some operator-valued functions $H_{k} \colon \bbD^{d} \to
      \cL(\cH'_{k}, \cE)$ and some Hilbert spaces $\cH'_{1}, \dots,
\cH'_{d}$.

\item There exists a multievolution scattering system ${\mathfrak
S}(\Sigma(U))$ with scattering matrix $S_{{\mathfrak S}}(z)$ equal
to $S(z)$ which satisfies conditions \eqref{Omegadecom},
\eqref{CDP}, \eqref{slim} and \eqref{slim*}.
\end{enumerate}

\end{corollary}

\begin{proof}  For the direction (2) $\Rightarrow$ (1), use
formula \eqref{UfromcalU'} to define $U = \sbm{ A & B \\ C & D }$
and then set $H_{k}(z) = C (I - Z_{\diag}(z) A)^{-1} P_{k}$ for $k
= 1, \dots, d$.  For the direction (1) $\Rightarrow$ (2), use the
result of \cite{Ag90} to deduce that $S(z)$ has a realization
$S(z) = T_{\Sigma(U)}(z)$ as the transfer function of a
conservative Givone--Roesser system $\Sigma(U)$ and then use the
analysis in \cite{BabyBear} to deduce that the ambient space $\cK$
of the associated scattering system ${\mathfrak S}(\Sigma(U))$
contains subspaces $\cL_{k}$, $\cM_{k}$, $\cM_{*k}$ ($k = 1,
\dots, d$) meeting all the conditions \eqref{Omegadecom},
\eqref{CDP}, \eqref{slim} and \eqref{slim*}.
\end{proof}

While the proof of (2) $\Rightarrow$ (1) is quite explicit, the
proof of (1) $\Rightarrow$ (2) in Corollary \ref{C:summary} is
rather indirect.  A goal of the present paper is to provide a
direct proof of the implication (1) $\Rightarrow$ (2) in Corollary
\ref{C:summary}.  In detail, given a Schur-class function $S$, it
is always possible to form the model scattering system ${\mathfrak
S}_{dBR}^{S}$ (see \eqref{frakSdBR}) which has $S$ as its
scattering function.  If $S$ has an Agler decomposition, then we
know that there is a possibly nonminimal scattering system
${\mathfrak S}$ having $S$ as its scattering matrix for which
there are subspaces $\cL_{k}$, $\cM_{k}$, $\cM_{*k}$ ($k = 1,
\dots, d$) satisfying (OD), (CDP), (SL), (SL$*$).  As a goal for
the present paper, we seek to construct such a scattering system
with ambient space $\cK^{S}_{dBR} \oplus \cK_{\text{non-min}}$ and
identify these spaces explicitly inside the functional model space
$\cK^{S}_{dBR}$ in terms of reproducing kernel representations.
In case the various kernels in \eqref{dBRkernels} have no
overlapping, the associated scattering system is minimal and the
identification is relatively straightforward.    In the end, one
can explicitly construct subspaces $\cL_{k}$, $\cM_{k}$,
$\cM_{*k}$ ($k = 1, \cdots, d$) of $\cK^{S}_{dBR}$ from the Agler
decomposition data functions $\{H_{1}(z), \dots, H_{d}(z)\}$.
>From these subspaces, by using the results of \cite{BabyBear} one
is led to a unitary colligation $U = \sbm{A & B \\ C & D}$ so that
$S = S_{{\mathfrak S}^{S}_{dBR}} = T_{\Sigma(U)}$.  In this way we
arrive at a new proof of (1) $\Rightarrow$ (2) in Corollary
\ref{C:summary} which bypasses Theorem \ref{T:Agler}.  More
precisely, we carry out this construction under the assumption
that we have the data functions $\Big\{ \sbm{H^{1}_{k}(z) \\
H^{2}_{k}(z)} \colon k = 1, \dots, d \Big\}$ for an augmented
Agler decomposition, and thus obtain an explicit
function-theoretic proof for (2$^{\prime}$) $\Rightarrow$ (3) in
Theorem \ref{T:Agler}.

\section{Functional models for scattering systems containing an
embedded unitary colligation}  \label{S:scat-col-model}

In this section we assume that we are given a unitary colligation
$U = \sbm{ A & B \\ C & D }$ as in \eqref{colligation} to which we
associate the multievolution scattering system ${\mathfrak
S}(\Sigma(U))$ as in Section \ref{S:scat-col}. When viewed in a
coordinate-free way, ${\mathfrak S}(\Sigma(U))$ is a
multievolution scattering system ${\mathfrak S}$ as in
\eqref{scatsys} with the additional structure that there exist
subspaces $\cL_{1}, \dots, \cL_{d}$, $\cM_{1}, \dots, \cM_{d}$ and
$\cM_{*1}, \dots, \cM_{*d}$ so that \eqref{Omegadecom},
\eqref{CDP}, \eqref{slim} and \eqref{slim*} all hold. As explained
in Section \ref{S:scatmodels} (see Lemma \ref{L:PidBR}), the map
$$ \Pi_{dBR} \colon k \mapsto  \begin{bmatrix} \Phi_{*}  k \\ \Phi k
\end{bmatrix} = \begin{bmatrix} \sum_{n \in \bbZ^{d}} (P_{\cF_{*}}
\cU^{*n} k) z^{n} \\ \sum_{n \in \bbZ^{d}} (P_{\cF} \cU^{*n} k)
z^{n} \end{bmatrix}
$$
is a coisometry from the ambient space $\cK$ onto the ambient
space $\cK^{S}_{dBR}$ for the de Branges--Rovnyak model scattering
system ${\mathfrak S}^{S}_{dBR}$.  When expressed in terms of
trajectory coordinates, $\Pi_{dBR}$ assumes the form
$$
 \Pi_{dBR} \colon (u(\cdot), x(\cdot), y(\cdot)) \mapsto
   \begin{bmatrix}  \sum_{n \in \bbZ^{d}} y(n) z^{n} \\
         \sum_{n \in \bbZ^{d}} u(n) z^{n} \end{bmatrix} : =
     \begin{bmatrix} \widehat y (z) \\ \widehat u(z) \end{bmatrix}.
$$

If $S(z) = \sum_{n \in \bbZ^{d}} S_{n} z^{n}$ is in the
Schur--Agler class $\mathcal{SA}(\cF, \cF_{*})$, then we may view
$S(z)$ either in the classical sense as an analytic $\cL(\cF,
\cF_{*})$-valued function on the unit polydisk, or purely formally
as in Section \ref{S:FRKHS} as a formal power series in the
indeterminates $z_{1}, \dots, z_{d}$.  In the former case we view
$S(w)^{*}$ as the conjugate analytic function defined on ${\mathbb
D}^{d}$ given by $ S(w)^{*} = \sum_{n \in \bbZ^{d}_+} S_{n}^{*}
\overline{w}^{n}$ while in the latter we follow the convention
\eqref{convention} and define $S(w)^{*} \in \cL(\cF_{*},
\cF)[[z^{\pm 1}]]$ as $S(w)^{*} = \sum_{n \in \bbZ^{d}_+}
S_{n}^{*} w^{-n}$.  It is relatively clear how to view an Agler
decomposition \eqref{Aglerdecom} or an augmented Agler
decomposition \eqref{augAglerdecom-Intro} either in the
sesquianalytic sense or in the formal power series sense.  However
the Schur matrix product (i.e., entrywise matrix product) in
\eqref{augAglerdecom-Intro} is not so convenient.  The next lemma
gives a more convenient form of the augmented Agler decomposition
which avoids the Schur-matrix product.

 \begin{lemma} \label{L:Aglerdecom}
     Suppose that $S(z)$ is an analytic $\cL(\cF, \cF_{*})$-valued
     function on the unit polydisk ${\mathbb D}^{d}$ which has a
     (classical) sesquianalytic augmented Agler decomposition as in
     \eqref{augAglerdecom-Intro}.  Then:
     \begin{enumerate}
     \item
     If we view $S(z) = \sum_{n \in \bbZ^{d}_+} S_{n} z^{n}$
     as a formal power series in $\cL(\cF, \cF_{*})[[z^{\pm 1}]]$,
     the augmented Agler decomposition
     \eqref{augAglerdecom-Intro} can be reexpressed as the following
formal
     augmented Agler-decomposition:
 \begin{equation}  \label{enlargedAglerdecom}
 \begin{bmatrix}
     I - S(z) S(w)^{*} &  S(w) - S(z) \\ S(z)^{*} - S(w)^{*} &
     S(z)^{*} S(w)  - I \end{bmatrix} =
     \sum_{k=1}^{d} ( 1 - z_{k} w_{k}^{-1} ) \widetilde  K_{k}(z,w)
 \end{equation}
 where
 $$ \widetilde K_{k}(z,w) = \begin{bmatrix} \widetilde H^{1}_{k}(z)
 \\ \widetilde H^{2}_{k}(z^{-1}) \end{bmatrix}
 \begin{bmatrix} \widetilde  H^{1}_{k}(w)^{*} & \widetilde
H^{2}_{k}(w^{-1})^{*} \end{bmatrix}
 $$
 where we have set
 $$ \widetilde H^{1}_{k}(z) = H^{1}_{k}(z), \qquad
 \widetilde H^{2}_{k}(z^{-1}) = z_{k}^{-1} H^{2}_{k}(z^{-1}).
 $$

 \item In case we know a realization \eqref{realization} for $S$ with
 $U$ as in \eqref{colligation} unitary, then we may take
 \begin{align}
  \widetilde H^{1}_{k}(z) =  C (I - Z_{\text{\rm diag}}(z) A)^{-1}
 P_{k} \text{ and }
 \widetilde H^{2}_{k}(z^{-1}) =z_{k}^{-1} B^{*} (I - Z_{\text{\rm
diag}}(z^{-1})
 A^{*})^{-1} P_{k}
 \label{H's}
 \end{align}
 in \eqref{enlargedAglerdecom}, and the kernel $\widetilde
 K_{k}(z,w)$ appearing in \eqref{enlargedAglerdecom} is given
 explicitly as
 \begin{align*}
  &  \widetilde  K_{k}(z,w) = \begin{bmatrix} C (I - Z_{\text{\rm
diag}}(z)
    A)^{-1} \\ z_{k}^{-1} B^{*}(I -
    Z_{\text{\rm diag}}(z)^{-1}A^{*})^{-1} \end{bmatrix} P_{k} \cdot
    \notag \\
    & \qquad
    \cdot \begin{bmatrix}  (I - A^{*} Z_{\text{\rm
diag}}(w)^{-1})^{-1} C^{*} &
    (I - A Z_{\text{\rm diag}}(w))^{-1} B w_{k} \end{bmatrix}.
  \end{align*}
 \end{enumerate}
 \end{lemma}

 \begin{proof}
    Given that
    \eqref{augAglerdecom-Intro} holds, a change of sign in the second
    column gives
    \begin{align*}
 & \begin{bmatrix} I - S(z) S(w)^{*} & S(\overline{w}) - S(z) \\
    S(\overline{z})^{*} - S(w) & S(\overline{z}))^{*} S(\overline{w})
- I
    \end{bmatrix} \\
    & \qquad
    = \sum_{k=1}^{d} \begin{bmatrix} H^{1}_{k}(z) \\
    H^{2}_{k}(z) \end{bmatrix}
    \begin{bmatrix} H^{1}_{k}(w)^{*} & H^{2}_{k}(w)^{*} \end{bmatrix}
    \circ \begin{bmatrix} 1 - z_{k} \overline{w}_{k} &
    \overline{w}_{k} - z_{k} \\ z_{k} - \overline{w}_{k} & z_{k}
    \overline{w}_{k} - 1 \end{bmatrix}.
  \end{align*}
  We then replace $w$ by $\overline{w}$ in the second column and $z$
  by $\overline{z}$ in the second row to get
  \begin{align*}
 & \begin{bmatrix} I - S(z) S(w)^{*} & S(w) - S(z) \\
  S(z)^{*} - S(w)^{*} & S(z)^{*} S(w) - I
  \end{bmatrix} \\
  & \qquad
  = \sum_{k=1}^{d} \begin{bmatrix} H^{1}_{k}(z) \\
  H^{2}_{k}(\overline{z}) \end{bmatrix}
  \begin{bmatrix} H^{1}_{k}(w)^{*} & H^{2}_{k}(\overline{w})^{*}
\end{bmatrix}
      \circ \begin{bmatrix} 1 - z_{k} \overline{w}_{k} &
      w_{k} - z_{k} \\ \overline{z}_{k} - \overline{w}_{k}
      & \overline{z}_{k}  w_{k} - 1 \end{bmatrix}.
\end{align*}
We then replace $\overline{w}_{k}$ by $w_{k}^{-1}$,
$\overline{z}_{k}$ by $z_{k}^{-1}$ and interpret $S(z)$ and
$H^{j}_{k}(z)$ ($j=1,2$, $k = 1, \dots, d$) in the formal sense
with the convention \eqref{convention} in force to arrive at
\begin{align*}
   & \begin{bmatrix} I - S(z) S(w)^{*}  & S(w) - S(z) \\
    S(z)^{*} - S(w)^{*} & S(z)^{*} S(w) - I \end{bmatrix} = \\
    & \qquad
    \sum_{k=1}^{d} \begin{bmatrix} H^{1}_{k}(z) \\
    H^{2}_{k}(z^{-1}) \end{bmatrix}
  \begin{bmatrix} H^{1}_{k}(w)^{*} & H^{2}_{k}(w^{-1})^{*}
  \end{bmatrix} \circ
  \begin{bmatrix} 1 - z_{k} w_{k}^{-1} & w_{k} - z_{k} \\ z_{k}^{-1}
      - w_{k}^{-1} & z_{k}^{-1} w_{k} - 1 \end{bmatrix}.
   \end{align*}
   Next observe that
   $$
   \begin{bmatrix} 1 - z_{k} w_{k}^{-1} & w_{k} - z_{k} \\
       z_{k}^{-1} - w_{k}^{-1} & z_{k}^{-1} w_{k} - 1 \end{bmatrix} =
       \begin{bmatrix} 1 & w_{k} \\ z_{k}^{-1} & z_{k}^{-1}
       w_{k}^{-1} \end{bmatrix} (1 - z_{k} w_{k}).
 $$
 Hence, if we set
 \begin{equation}  \label{tildeH}
 \widetilde H^{1}_{k}(z) = H^{1}_{k}(z), \qquad
 \widetilde H^{2}_{k}(z^{-1}) = z_{k}^{-1}
 H_{k}^{2}(z^{-1}),
 \end{equation}
 we arrive at \eqref{enlargedAglerdecom} as wanted.

 If we know a
 realization \eqref{realization} for $S(z)$, we can arrive at the
 formulas for $\widetilde H^{j}_{k}(z)$ via direct substitution using
 the unitary relations \eqref{unitary} for $U$.  Alternatively,
 by Theorem \ref{T:Agler} we know that \eqref{augAglerdecom-Intro}
holds with $H_{k}^{1}(z)$ and
 $H^{2}_{k}(z)$ as in \eqref{H-intro2}.   Setting $\widetilde
H^{j}_{k}(z)$
 as in \eqref{tildeH} then gives us the formulas for $\widetilde
 H^{j}_{k}(z)$ given in \eqref{H's}.
 \end{proof}

 \begin{remark} Part of the content of Theorem \ref{T:Agler} is that
     $S$ has an augmented Agler decomposition (and hence the formal
     augmented Agler decomposition \eqref{enlargedAglerdecom})
     whenever $S$ has the simpler Agler decomposition
     \eqref{Aglerdecom}. The usual proof for this fact goes through
     the implication (2) $\Rightarrow$ (3) in Theorem
\ref{T:Agler}
     (achieved through the ``lurking isometry'' argument---see
     \cite{BT}) to arrive at a GR-unitary realization for $S$; one
     then arrives at (2$^{\prime}$) (or the amended form
     \eqref{enlargedAglerdecom}) via the formulas
     \eqref{H-intro2} or \eqref{H's}.  For the
     completeness of the function-theoretic approach of this paper, it
     would be nice to have a direct function-theoretic proof of (2)
     $\Rightarrow$ (2$^{\prime}$) in Theorem \ref{T:Agler} which
     does not pass through  a
     conservative-GR realization for $S$.
 \end{remark}

 In the sequel we assume that we are given a Schur-class function $S
 \in \cS(\cE, \cE_{*})$ together with an augmented Agler
 decomposition; to lighten the notation, we drop the tildes and write
 simply
 \begin{equation}  \label{augAglerdecom}
  \begin{bmatrix}
      I - S(z) S(w)^{*} & S(w) - S(z) \\ S(z)^{*} - S(w)^{*} &
      S(z)^{*} S(w) - I \end{bmatrix} =
      \sum_{k=1}^{d}  (1 - z_{k} w_{k}^{-1})  K_{k}(z,w)
  \end{equation}
  for the (formal) augmented Agler decomposition.  In case we are
given a
  realization $S(z) = D + C(I - Z_{\diag}(z) A)^{-1} Z_{\diag}(z) B$
  for $S$ with $U = \sbm{ A & B \\ C & D }$ unitary, then by  part
  (3) of Lemma \ref{L:Aglerdecom} one can
  take $K_{k}(z,w)$ to be given by
  \begin{align}
   K_{k}(z,w)  = &
   \begin{bmatrix} C (I - Z_{\diag}(z) A)^{-1} \\ z_{k}^{-1} B^{*} (I
-
   Z_{\diag}(z)^{-1} A^{*})^{-1} \end{bmatrix} P_{k}  \notag \\
   & \qquad \cdot  \begin{bmatrix} (I - A^{*} Z_{\diag}(w)^{-1})^{-1}
C^{*} &
       (I - A Z_{\diag}(w) )^{-1}B w_{k}\end{bmatrix}.
       \label{augKernel}
   \end{align}

 The augmented Agler decomposition \eqref{augAglerdecom} leads to
 the following kernel identities.

 \begin{proposition}  \label{P:dBRkerneldecom}
     Suppose that $ S \in \cS(\cE, \cE_{*})$ has augmented Agler
     decomposition \eqref{augAglerdecom} and that $\Omega \subset
     \bbZ^{d}$ is a shift-invariant sublattice.  We shall use the
     decomposition
     \begin{equation} \label{Kdecom}
     K_{k}(z,w) = \begin{bmatrix} K^{11}_{k}(z,w) & K^{12}_{k}(z,w)
     \\ K^{21}_{k}(z,w) & K^{22}_{k}(z,w) \end{bmatrix}
     \end{equation}
     of the kernels appearing in \eqref{augAglerdecom} for $k = 1,
     \dots, d$.
     Then:
    \begin{enumerate}
 \item  The kernel
     function $K_{\cV^{S, \Omega}_{dBR}}(z,w)$ appearing in
     \eqref{dBRkernels} has the decomposition
  \begin{equation}  \label{dBRkerneldecom}
      K_{\cV^{S, \Omega}_{dBR}}(z,w)  =
      \sum_{k=1}^{d}
      \left[\left(\sum_{n \in \partial_{k}\Omega_{\text{\rm{fin}}}}
z^{n}w^{-n}
      \right)
       K_{k}(z,w) + K_{\partial \Omega_{\infty},k}(z,w) \right]
   \end{equation}
   where $ K_{\partial \Omega_{\infty},k}(z,w)$ has the form
   $$
   K_{\partial \Omega_{\infty},k}(z,w) = \begin{bmatrix}
K^{11}_{\partial
   \Omega_{\infty},k}(z,w) & 0 \\ 0 &  K^{22}_{\partial
\Omega_{\infty},k}(z,w)
   \end{bmatrix}
   $$
   with
   \begin{align*} & K^{11}_{\partial \Omega_{\infty},k}(z,w) =
   \sum_{\ell_{n',k} \in \partial_{k}\Omega_{-\infty}}
   (\widehat z^{k})^{(\widehat{n^{\prime}})^{ k}}
   ({\widehat w}^{k})^{-\widehat{n'}^{k}}
      K^{11}_{-\infty,k}(z,w), \\
       &  K^{22}_{\partial \Omega_{\infty},k}(z,w) =
  \sum_{ \ell_{n'',k} \in \partial_{k}\Omega_{+\infty}}
 (\widehat z^{k})^{\widehat{n''}^{ k}}(\widehat
    w^{k})^{- (\widehat{n''})^{ k}}
    K^{22}_{+\infty,k}(z,w)
    \end{align*}
    where we  set
   \begin{align}
       K^{11}_{-\infty,k}(z,w) & = \lim_{t \to -\infty} z_{k}^{t}
       w_{k}^{-t}K_{k}^{11}(z,w), \notag \\
       K^{22}_{+\infty,k}(z,w) & = \lim_{t \to +\infty} z_{k}^{t}
       w_{k}^{-t} K^{22}_{k}(z,w)
       \label{Kkinf}
   \end{align}
   with the limit interpreted in the sense of strong convergence
   of power-series coefficients, and where
   $$
   (\widehat z^{k})^{\widehat{n'}^{k}} = z_{1}^{n_{1}} \cdots
\widehat{z_{k}} \cdots
   z_{d}^{n_{d}}  \text{ and } \widehat{n}^{k} =
   (n'_{1}, \dots, \widehat{n'_{k}}, \dots, n'_{d}),
   $$
   i.e., the $k$-th term is omitted.

   \item We have the following kernel analogue of \eqref{CDP}:
   \begin{equation}
    \sum_{k = 1}^{d}  K_{k}(z,w) + \begin{bmatrix}
       S(z) \\ I \end{bmatrix} \begin{bmatrix} S(w)^{*} & I
       \end{bmatrix}
    = \sum_{k=1}^{d} z_{k}  K_{k}(z,w) w_{k}^{-1} +
    \begin{bmatrix} I \\ S(z)^{*} \end{bmatrix}
           \begin{bmatrix} I & S(w) \end{bmatrix}.
      \label{kernelCDP}
      \end{equation}
      \end{enumerate}
  \end{proposition}

  \begin{remark} \label{R:kernel-col-scat-geometry}
      We view conditions \eqref{dBRkerneldecom},
      \eqref{Kkinf} and \eqref{kernelCDP} as kernel function analogues
      of the subspace conditions \eqref{Omegadecom}, \eqref{slim'}
      and \eqref{slim*'}, and \eqref{CDP} in Theorem
      \ref{T:BabyBear}.  We indicate below (see Theorems
      \ref{T:minfuncmodel} and \ref{T:scc-real})
      how, under certain conditions, these kernel decompositions can
be translated to the
      validity of the subspace conditions \eqref{Omegadecom},
      \eqref{slim'} and \eqref{slim*'}, and \eqref{CDP} for a
      functional-model multievolution scattering system having
      scattering matrix $S$.
      \end{remark}

  \begin{proof}[Proof of Proposition \ref{P:dBRkerneldecom}]
       In addition to
      \eqref{Kdecom}, let us decompose $K_{\cV_{dBR}^{S,
      \Omega}}(z,w)$ as
  \begin{equation*}
      K_{\cV_{dBR}^{S, \Omega}}(z,w) = \begin{bmatrix}
      K^{11}_{\cV_{dBR}^{S, \Omega}}(z,w) &
      K^{12}_{\cV_{dBR}^{S, \Omega}}(z,w)  \\
      K^{21}_{\cV_{dBR}^{S, \Omega}}(z,w) &
      K^{22}_{\cV_{dBR}^{S, \Omega}}(z,w)  \end{bmatrix}.
   \end{equation*}
   We shall first verify the (1,1)-entry in the equation
   \eqref{dBRkerneldecom}.  The computations must be regularized
properly
   in order to avoid the traps illustrated in Remark \ref{R:trap}. To
perform the
      regularization we introduce the approximate shift-invariant
      sublattice
   $$ \Omega^{M} = \{ n = (n_{1}, \dots, n_{d}) \in \Omega \colon
   n_{k} \ge M \text{ for all } k = 1, \dots, d \}.
   $$
   We assert that
   \begin{equation}  \label{assert}
      k_{\Sz, \Omega}(z,w)  (I - S(z) S(w)^{*}) = \lim_{M \to
       -\infty}k_{\Sz, \Omega^{M}}(z,w) (I - S(z) S(w)^{*})
   \end{equation}
   (with convergence taken to be strong coefficientwise).  Indeed,
   the coefficient of $z^{n}w^{-m}$ on the left-hand side of
\eqref{assert} is
   \begin{align}
   & [k_{\Sz,\Omega}(z,w)(I - S(z) S(w)^{*})]_{n,m} \notag \\
   & \qquad = \chi_{\Omega}(n)\delta_{n,m}I_{\cE_{*}} \notag\\
   & +
\text{weak}\sum_{n'\in\mathbb{Z}^d_+}\Big(\text{weak}\sum_{m' \in
{\mathbb
   Z}^{d}_{+}}\chi_{\Omega}(n-n') \chi_{\Omega}(m-m')
\delta_{n-n',m-m'} S_{n'}S^{*}_{m'}
   \Big) \notag\\
   & = \chi_{\Omega}(n) \delta_{n,m} I_{\cE_{*}} +\text{weak}
\sum_{n' \in
   {\mathbb Z}^{d}_{+} \colon n -
   n' \in \Omega} S_{n'}S^{*}_{n'+m-n}
   \label{K11}
   \end{align}
   (with $\chi_{\Omega}$ equal to the characteristic function of the
   set $\Omega$ and with $(k,\ell) \mapsto \delta_{k,\ell}$ equal to
   the Kronecker delta-function); note that
     the infinite series in
   \eqref{K11} converges weakly since $S(z)$ is a bounded
multiplier from
   $\cH(k_{\Sz}  I_{\cE})$ to $\cH(k_{\Sz}
   I_{\cE_{*}})$ and, hence, is a bounded multiplier
$\cH(k_{\Sz,\Omega}  I_{\cE})$ to $\cH(k_{\Sz,\Omega}
   I_{\cE_{*}})$ (see Section \ref{S:multipliers}).
   The coefficient of $z^{n}w^{-m}$ inside the limit on the
   right-hand side of \eqref{assert} is
   \begin{align}
     &  [ k_{\Sz,\Omega^{M}}(z,w)(I - S(z) S(w)^{*})]_{n,m}
    \notag \\
   & \,  =
       \chi_{\Omega^{M}}(n) \delta_{n,m}I_{\cE_{*}} +
     \text{weak}  \sum_{n' \in
     {\mathbb Z}^{d}_{+} \colon n -
     n' \in \Omega^{M}} S_{n'}S^{*}_{n'+m-n}
       \label{KM11}
    \end{align}
    Note that the series in \eqref{KM11} is just a finite truncation
of
    the series in \eqref{K11}.  Since the series \eqref{K11} defining
    the $(n,m)$-coefficient of $k_{\Sz,
    \Omega}(z,w)(I - S(z) S(w)^{*})$ converges, the assertion
\eqref{assert} follows as wanted.

    Combining \eqref{assert} with \eqref{dBRkernels} then gives
    \begin{align}
    K^{11}_{\cV_{dBR}^{S, \Omega}}(z,w) & =
    k_{\Sz,\Omega}(z,w)(I - S(z) S(w)^{*}) \notag \\
    & = \lim_{M \to -\infty} k_{\Sz,
    \Omega^{M}}(z,w)(I - S(z) S(w)^{*}).
    \label{KV11-1}
    \end{align}
    If we plug the augmented Agler decomposition
    \eqref{augAglerdecom} ((1,1)-component
    only for the moment) into \eqref{KV11-1} we arrive at
    \begin{equation} \label{KV11-2}
    K^{11}_{\cV_{dBR}^{S, \Omega}}(z,w) =
    \lim_{M \to -\infty}    k_{\Sz, \Omega^{M}}(z,w) \left( \sum_{k=1}^{d}
   (1 - z_{k}w_{k}^{-1})    K^{11}_{k}(z,w) \right).
    \end{equation}
    Since $K_{k}^{11}(z,w) = \sum_{n,m \in \bbZ^{d}_{+}}
    [K_{k}^{11}]_{n,m} z^{n} w^{-m}$ (i.e., the coefficients are
    supported on $\bbZ^{d}_{+} \times \bbZ^{d}_{+}$) and since
$\Omega^{M}$
    has no boundary at $-\infty$, only finite sums are
    involved in the expressions defining the coefficients of
    $k_{\Sz, \Omega^{M}}(z,w) K^{11}_{k}(z,w)  $ and of
    $ z_{k} w_{k}^{-1} k_{\Sz, \Omega^{M}}(z,w)K^{11}_{k}(z,w) $, and
the
    associativity is valid.\footnote{Note in particular that the
starting
    point of the example in Remark \ref{R:trap} is the observation
that the
    collections of kernels $K_{1}(z,w) = \sum_{\ell=0}^{+\infty}
    z_{1}^{\ell} w_{1}^{-\ell}$, $K_{j}(z,w) = 0$ for $2 \le j \le d$
    form an Agler decomposition for the scalar-valued
    Schur-Agler-class function $S(z) = 0$.}  We are therefore able to
continue
    \eqref{KV11-2} in the form
    \begin{align}
    K^{11}_{\cV_{dBR}^{S, \Omega}}(z,w) & =
    \lim_{M \to -\infty}k_{\Sz, \Omega^{M}}(z,w)  \left( \sum_{k=1}^{d}
(1 - z_{k} w_{k}^{-1})
    K_{k}^{11}(z,w) \right) \notag \\
    & = \lim_{M \to -\infty} \sum_{k=1}^{d} \Big(k_{\Sz,
\Omega^{M}}(z,w) (1 - z_{k} w_{k}^{-1})\Big)
    K_{k}^{11}(z,w).
    \label{KV11-3}
   \end{align}
   Since (as has already been pointed out and used) $\partial_{k}
   \Omega^{M}_{-\infty} = \emptyset$ for each $k = 1, \dots, d$, one
   can verify the useful identity
   \begin{equation} \label{useful}
       k_{\Sz, \Omega^{M}}(z,w) (1 - z_{k} w_{k}^{-1}) = k_{\Sz,
       \partial_{k}\Omega_{\text{fin}}^{M}}(z,w).
    \end{equation}
    It is convenient to decompose the finite boundary components
    $\partial_{k}\Omega^{M}_{\text{fin}}$ of $\Omega^{M}$ as
    \begin{align*}
    \partial_{k}\Omega^{M}_{\text{fin}} & =
    ( \partial_{k} \Omega_{\text{fin}} \cap
    \partial_{k}\Omega^{M}_{\text{fin}}) \cup
    \left[ \partial_{k} \Omega^{M}_{\text{fin}} \setminus
    (\partial_{k}\Omega^{M}_{\text{fin}} \cap \partial
    \Omega_{\text{fin}}) \right] \notag \\
    & =: \partial_{k,I}^{M}(\Omega) \cup
    \partial_{k,II}^{M}(\Omega)
     \end{align*}
     Note that the asymptotics of $\partial^{M}_{k, II}(\Omega)$ can
     be understood as follows:  for a given $n' \in {\mathbb Z}^{d}$,
     we have
     \begin{align}
     &   n' \in \partial^{M}_{k,II}(\Omega) \Rightarrow n'_{k} =
     M, \notag \\
   & \ell_{n',k} \in \partial_{k}\Omega_{-\infty} \Rightarrow
   \ell_{n',k}|_{n_k'=M} \in \partial^{M}_{k, II}(\Omega), \notag \\
   & \ell_{n',k} \notin \partial_{k}\Omega_{-\infty} \Rightarrow
   \text{ there exists } N_{n'} > -\infty \text{ so that} \notag \\
   & \qquad \qquad \ell_{n',k}|_{n'_k=M} \ \notin
   \partial^{M}_{k,II}(\Omega) \text{ once } M < N_{n'}
   \label{asymptotics}
   \end{align}
   where we have used the (presumably transparent) notation
   $$
   \ell_{n',k}|_{n'_k=M} = (n'_{1}, \dots, n'_{k-1}, M, n'_{k+1},
\dots,
   n'_{d}) \in {\mathbb Z}^{d}.
   $$
   From \eqref{KV11-3} combined with \eqref{useful} we write
     \begin{equation} \label{KV11-4}
   K^{11}_{\cV_{dBR}^{S, \Omega}}(z,w) =
   \lim_{M \to -\infty}
   \sum_{k=1}^{d} \left[
   k_{\text{Sz},\partial_{k,I}^{M}(\Omega)}(z,w) K^{11}_{k}(z,w) +
   k_{\text{Sz},\partial_{k,II}^{M}(\Omega)}(z,w) K^{11}_{k}(z,w)
\right].
   \end{equation}
   Note that the term $k_{\text{Sz},\partial_{k,I}^{M}(\Omega)}(z,w)
K_{k}^{11}(z,w)$,
   as a function of $M$, is an increasing sequence of positive
   kernels as $M \to -\infty$ (i.e.,
   $k_{\Sz,\partial_{k,I}^{M}(\Omega)}(z,w)K_{k}^{11}(z,w)$ is a
positive kernel
   for each fixed $M$ and the difference
   $$
   k_{\Sz,\partial_{k,I}^{M'}(\Omega)}(z,w) K_{k}^{11}(z,w) -
   k_{\Sz,\partial_{k,I}^{M}(\Omega)}(z,w) K_{k}^{11}(z,w)
   $$
   is a positive kernel
   for $M' < M$).  As each term $k_{Sz,\partial_{k,II}^{M}}(z,w)
   K^{11}_{k}(z,w)$ is also a positive kernel, we see that the
   sequence $\{k_{\Sz,\partial_{k,I}^{M}(\Omega)}(z,w)
K_{k}^{11}(z,w)\}_{M = -1,
   -2, \dots}$ is bounded above by the positive kernel
   $K^{11}_{\cV_{ddBR}^{S, \Omega}}(z,w)$.  We conclude that the limit
   $$
   \lim_{M \to -\infty} k_{Sz,\partial_{k,I}^{M}(\Omega)}(z,w)
K^{11}_{k}(z,w)
   $$
   exists (strongly coefficientwise).  Moreover, since the sequence
   $$
   \{k_{Sz,\partial_{k,I}^{M}(\Omega)}(z,w) K^{11}_{k}(z,w)\}_{M=-1,
-2, \dots}
   $$
   can be
   identified as the sequence of partial sums for
   $k_{Sz,\partial_{k}\Omega_{\text{fin}}}(z,w)  K_{k}^{11}(z,w)$ (see
   \cite[page 173]{MammaBear}), we
   can identify the limit explicitly as
  $$
  \lim_{M \to -\infty} k_{Sz,\partial_{k,I}^{M}(\Omega)}(z,w)
K^{11}_{k}(z,w) =
  k_{Sz,\partial_{k}\Omega_{\text{fin}}}(z,w) K_{k}(z,w).
  $$
  From \eqref{KV11-4} we conclude that
 $ \lim_{M \to -\infty} \sum_{k=1}^{d}
       k_{Sz,\partial_{k,II}^{M}(\Omega)}(z,w) K_{k}^{11}(z,w)$ with
value
  \begin{equation}  \label{dagger}
      \lim_{M \to -\infty} \sum_{k=1}^{d}
      k_{Sz,\partial_{k,II}^{M}(\Omega)}(z,w) K_{k}^{11}(z,w)  =
      K^{11}_{\cV_{dBR}^{S, \Omega}}(z,w) - \sum_{k=1}^{d}
      k_{Sz,\partial_{k}\Omega_{\text{fin}}}(z,w) K^{11}_{k}(z,w).
   \end{equation}

   Our next goal is to verify that the limit
   \begin{equation}  \label{want-lim}
       \lim_{M \to -\infty} z_{k}^{M} w_{k}^{-M} K_{k}^{11}(z,w)
       =:K^{11}_{-\infty,k}(z,w)
   \end{equation}
   exists for each $k = 1, \dots, d$.  To this end, let us fix an
   index $k_{0} \in \{ 1, \dots, d\}$ and define the shift-invariant
   sublattice $\Omega_{k_{0}}$ by
   $$
   \Omega_{k_{0}} = \{ n \in {\mathbb Z}^{d}_{+} \colon n_{j} \ge 0
   \text{ for } j \ne k_{0}\}.
   $$
   We compute
   \begin{align*}
       & \partial_{k_{0}}(\Omega_{k_{0}})_{\text{fin}} = \emptyset, \\
       & \partial_{k}(\Omega_{k_{0}})_{\text{fin}} =
       \{n \in {\mathbb Z}^{d} \colon n_{k}=0, n_{j} \ge 0 \text{ for
       } j \ne k_{0} \} \text{ for } k \ne k_{0}.
   \end{align*}
   Similarly, under the assumption that $M < 0$, we have
   \begin{align*}
       & \partial_{k_{0}}(\Omega_{k_{0}})^{M}_{\text{fin}} =
       \{n \in {\mathbb Z}^{d} \colon n_{j} \ge 0 \text{ for } j \ne
       k_{0} \text{ and } n_{k_{0}} = M \}, \\
       & \partial_{k} (\Omega_{k_{0}})^{M}_{\text{fin}} =
       \{ n \in {\mathbb Z}^{d} \colon n_{k}=0, \, n_{j} \ge 0 \text{
       for } j \ne k_{0}, \, n_{k_{0}} \ge M \} \text{ for } k \ne
       k_{0}.
   \end{align*}
   We conclude that, for $M < 0$, we have
   \begin{align}
       & \partial^{M}_{k_{0},II}(\Omega_{k_{0}}) = \{ n \in {\mathbb
       Z}^{d} \colon n_{j} \ge 0 \text{ for } j \ne k_{0}, \,
       n_{k_{0}} = M \}, \notag \\
       & \partial^{M}_{k, II}(\Omega_{k_{0}}) = \emptyset \text{ for
       } k \ne k_{0}.
       \label{dagger1}
   \end{align}
   From \eqref{dagger} applied to the case $\Omega = \Omega_{k_{0}}$
   we conclude that
   \begin{equation}  \label{dagger1-lim}
       \lim_{M \to -\infty} \sum_{n \in
       \partial^{M}_{k_{0},II}(\Omega_{k_{0}})} z_{k_{0}}^{n}
w_{k_{0}}^{-n}
       K^{11}_{k_{0}}(z,w) \text{ exists.}
   \end{equation}
   We now do a similar analysis for the shift-invariant sublattice
   $\Omega_{k_{0}}^{0}$ given by
   $$
   \Omega_{k_{0}}^{0} = \Big\{ n \in {\mathbb Z}^{d} \colon n_{j} \ge
0
   \text{ for } j \ne k_{0}, \, \sum_{j \colon j \ne k_{0}} n_{j} > 0
   \Big\}.
   $$
   We again compute
   \begin{align*}
       & \partial_{k_{0}}(\Omega_{k_{0}}^{0})_{\text{fin}} =
         \emptyset, \\
       & \partial_{k}(\Omega_{k_{0}}^{0})_{\text{fin}} =
    \{ n \in {\mathbb Z}^{d} \colon n_{j} = 0 \text{ for } j
       \notin \{k,k_{0}\} \text{ and } n_{k} = 1 \} \\
       & \qquad \cup \Big\{ n \in {\mathbb Z}^{d} \colon n_{k} = 0, \,
     n_{j} \ge 0  \text{ for } j \ne k_{0}, \, \sum_{j \colon j \ne
       k_{0}} n_{j} > 0 \Big\} \text{ for } k \ne k_{0},
   \end{align*}
   while, under the assumption that $M < 0$, we have
   \begin{align*}
      & \partial_{k_{0}}(\Omega_{k_{0}}^{0})^{M}_{\text{fin}} =
 \Big\{ n \in {\mathbb Z}^{d} \colon n_{j} \ge 0 \text{ for } j \ne
k_{0},
 \, \sum_{j \colon j \ne k_{0}} n_{j} > 0, \, n_{k_{0}} = M \Big\},
\\
  & \partial_{k}(\Omega_{k_{0}}^{0})^{M}_{\text{fin}} =
 \{ n \in {\mathbb Z}^{d} \colon n_{j}=0 \text{ for } j \notin \{k,
 k_{0}\}, \, n_{k} = 1, \, n_{k_{0}} \ge M \}  \\
 & \qquad \cup \Big\{ n \in {\mathbb Z} \colon n_{k}=0,\,  n_{j} \ge 0
 \text{ for } j \ne k_{0}, \, \sum_{j \colon j \ne k_{0}} n_{j} > 0,
\,
 n_{k_{0}} \ge M \Big\} \text{ for } k \ne k_{0}.
 \end{align*}
 It then follows that
 \begin{align}
     & \partial^{M}_{k_{0}, II}(\Omega^{0}_{k_{0}}) =
     \Big\{ n \in {\mathbb Z}^{d} \colon n_{j} \ge 0 \text{ for } j
\ne
     k_{0}, \, \sum_{j \colon j \ne k_{0}} n_{j} > 0, \, n_{k_{0}} =
     M \Big\}, \notag \\
     & \partial^{M}_{k,II}(\Omega^{0}_{k_{0}}) = \emptyset \text{ for
     } k \ne k_{0}.
     \label{dagger2}
     \end{align}
     If we apply \eqref{dagger} with $\Omega = \Omega_{k_{0}}^{0}$ we
     then get
     \begin{equation}  \label{dagger2-lim}
     \lim_{M \to -\infty} \sum_{n \in
     \partial^{M}_{k_{0},II}(\Omega_{k_{0}}^{0})} z_{k_{0}}^{n}
     w_{k_{0}}^{-n}
     K^{11}_{k_{0}}(z,w) \text{ exists.}
 \end{equation}
 If we observe from \eqref{dagger1} and \eqref{dagger2} that
$ \partial^{M}_{k_{0},II}(\Omega^{0}_{k_{0}}) \subset
 \partial^{M}_{k_{0}, II}(\Omega_{k_{0}})$   with
 $$
  \partial^{M}_{k_{0}, II}(\Omega_{k_{0}}) \setminus
 \partial^{M}_{k_{0},II}(\Omega^{0}_{k_{0}}) = \{ (n_{1}, \dots,
 n_{d}) \colon n_{j} = 0 \text{ for } j \ne k_{0}, \, n_{k_{0}} = M \}
 $$
 and then subtract \eqref{dagger2-lim} from \eqref{dagger1-lim}, we
 finally arrive at the existence of the limit
 $$
 \lim_{M \to -\infty} \sum_{n \in
 \partial^{M}_{k_{0},II}(\Omega_{k_{0}}) \setminus
 \partial^{M}_{k_{0},II}(\Omega^{0}_{k_{0}})} z_{k_{0}}^{n}
w_{k_{0}}^{-n} K^{11}(z,w)
 = \lim_{M \to -\infty} z_{k_{0}}^{M} w_{k_{0}}^{-M}
 K^{11}_{k_{0}}(z,w)
 $$
 and \eqref{want-lim} follows.
  From the observation that
  \begin{equation}  \label{observation}
  z_{k_{0}}^{M} w_{k_{0}}^{-M}K^{11}_{k_{0}}(z,w) =
  k_{\Sz,\Xi_{M,k}}(z) (I - S(z) S(w)^{*})
  \end{equation}
  where the subset
  $$
  \Xi_{M,k} = \{ n \in {\mathbb Z}^{d} \colon n_{j} \ge 0 \text{ for }
  j \ne k_{0}, \, n_{k_{0}}  \ge M \}
  $$
  is increasing as $M$ decreases to $-\infty$, we see that the
  sequence $\{ z_{k_{0}}^{M} w_{k_{0}}^{-M} K^{11}_{k_{0}}(z,w)
  \}$ is an increasing sequence of positive kernels
  as $M \to -\infty$.

  We now return to the setting where $\Omega$ is a general
  shift-invariant sublattice.
  To check the validity of the $(1,1)$-entry of
\eqref{dBRkerneldecom},
  it remains now only to verify that
  $$
 \sum_{k=1}^{d}
  k_{Sz,\partial^{M}_{k,II}(\Omega)}(z,w) K^{11}_{k}(z,w)
  \to
  \sum_{k=1}^d\sum_{\widehat{n'}^{k} \colon \ell_{n',k} \in
\partial_{k}
  \Omega_{-\infty}} (\widehat{z}^{k})^{(\widehat{n'})^{k}} ({\widehat
  w}^{k})^{-\widehat{n'}^{k}} K^{11}_{-\infty,k}(z,w)
  $$
  as $M \to -\infty$.   To do this it suffices to show that the
  identity holds termwise:
  \begin{align}
     & \lim_{M \to -\infty} k_{Sz,\partial^{M}_{k,II}(\Omega)}(z,w)
      K^{11}_{k}(z,w) \notag  \\
      & \qquad = \sum_{\widehat{n'}^{k} \colon \ell_{n',k} \in
      \partial_{k} \Omega_{-\infty}}
      (\widehat{z}^{k})^{\widehat{n'}^{k}}
      (\widehat{w}^{k})^{-\widehat{n'}^{k}} K^{1}_{-\infty,k}(z,w)
      \label{toshow-k}
  \end{align}
  for each fixed $k=1, \dots, d$.  From \eqref{asymptotics} we see
that
  $$
  k_{\Sz, \partial^{M}_{k,II}(\Omega)}(z,w) = \sum_{n \in
  \partial^{M}_{k,II}(\Omega) \colon n = \ell_{n',k}|_{n'_k=M} \text{
  for some } n'} z^{n} w^{-n},
  $$
  and, for a fixed $n' \in {\mathbb Z}^{d}$, if we set $n'(M) =
  \ell_{n',k}|_{n'_k=M}$ we have
  \begin{align*}
 &  \lim_{M \to -\infty}  \chi_{\partial^{M}_{k,II}(\Omega)}n'(M)
  z^{n'(M)} w^{-n'(M)} K^{11}_{k}(z,w) \\
  & \qquad \qquad = \begin{cases}
  0 & \text{if } \ell_{n',k} \notin \partial_{k}\Omega_{-\infty}, \\
  (\widehat{z}^{k})^{\widehat{n'}^{k}}
  (\widehat{w}^{k})^{-\widehat{n'}^{k}} K^{11}_{-\infty,k}(z,w) &
  \text{if } \ell_{n',k} \in \partial_{k}\Omega_{-\infty}.
  \end{cases}
  \end{align*}
  We conclude that \eqref{toshow-k} holds up to an interchange of
  the limit and summation signs.  However, this interchange is
  justified by the monotone convergence theorem since, as we have
  observed in \eqref{observation}, the sequence
  $\{z_{k}^{M} w_{k}^{-M} K^{11}_{k}(z,w)\}$ converges monotonically
  (with positivity of a difference measured as kernel positivity)
  to $K^{11}_{-\infty,k}(z,w)$ as $M \to -\infty$.  This completes
  the proof of the validity of the $(1,1)$-block entry in
  \eqref{dBRkerneldecom}.

   The (1,2)-entry of \eqref{dBRkerneldecom} can be verified by a
   completely parallel argument.  In this case, however, it is
   automatic that
   $$
    \lim_{M \to -\infty} z_{k}^{M}w_{k}^{-M}K^{12}_{k}(z,w) = 0
   $$
   since $K^{12}(z,w)$ has the form
   $$
    K^{12}(z,w) = \sum_{n \in \bbZ^{d}_{+},m \in -\bbZ^{d}_{+}}
    [K^{12}]_{n,m} z^{n}
    w^{-m}.
   $$

   For the $(2,1)$ and $(2,2)$ terms, one should define
$\Omega^{M}_{b}$
   to be the approximating backward shift-invariant sublattice
   $$
     \Omega^{M}_{b} = \{ n = (n_{1}, \dots, n_{d}) \in \Omega \colon
     n_{k} \le M \text{ for all } k = 1, \dots, d \}
   $$
   and then take a limit as $M \to +\infty$.  With these
   adjustments, the arguments for the (2,2) and (2,1) cases
   are exactly the same as the arguments for the  (1,1) and (1,2)
   cases respectively.  This completes the verification of all four
   entries of \eqref{dBRkerneldecom}

   To verify \eqref{kernelCDP}, spell out \eqref{enlargedAglerdecom}
as
  \begin{align*}
    & \sum_{k=1}^{d}(1 - z_{k}w_{k}^{-1}) K^{11}_{k}(z,w)  = I - S(z)
      S(w)^{*}, \\
   & \sum_{k=1}^{d}  (1 - z_{k}w_{k}^{-1} )  K^{12}_{k}(z,w) = S(w) -
S(z),
   \end{align*}
   \begin{align*}
    & \sum_{k=1}^{d} (1 - z_{k}w_{k}^{-1})   K^{21}_{k}(z,w)  =
    S(z)^{*} - S(w)^{*}, \\
    & \sum_{k=1}^{d} (1 - z_{k} w_{k}^{-1} ) K^{22}_{k}(z,w)  =
    S(z)^{*}S(w) - 1.
  \end{align*}
  Then rearrange to get
  \begin{align*}
     & \sum_{k=1}^{d}  K^{11}_{k}(z,w) + S(z) S(w)^{*}  =
      \sum_{k=1}^{d} z_{k}  w_{k}^{-1}K^{11}_{k}(z,w)  + I \\
     & \sum_{k=1}^{d} \ K^{12}_{k}(z,w) + S(z)  =
      \sum_{k=1}^{d} z_{k} w_{k}^{-1} K^{12}_{k}(z,w)  + S(w), \\
     & \sum_{k=1}^{d}  K^{21}_{k}(z,w) + S(w)^{*}  =
      \sum_{k=1}^{d} z_{k} w_{k}^{-1} K^{21}_{k}(z,w)  + S(z)^{*}, \\
     & \sum_{k=1}^{d}  K^{22}_{k}(z,w) + I  = \sum_{k=1}^{d}
      z_{k} w_{k}^{-1} K^{22}_{k}(z,w)  + S(z)^{*} S(w).
  \end{align*}
  This amounts to the spelling out of \eqref{kernelCDP} as wanted.
  \end{proof}

  \begin{remark}  \label{R:d=1lim} For the case $d=1$, the limit
      \eqref{want-lim} can be evaluated explicitly as follows.
      We note that $K^{11}_{k}(z,w) = K^{11}(z,w)$ where
      \begin{align*} K^{11}(z,w) & = ({1 - z w^{-1}})^{-1}(I - S(z)
S(w)^{*})=
      \sum_{n=0}^{\infty} z^{n} w^{-n} -  \Big( \sum_{n = 0}^\infty
      z^{n} w^{-n} \Big)S(z) S(w)^{*} \\
      & = \sum_{n \ge 0} z^{n} w^{-n} - \sum_{i,j,\ell \ge 0} S_{i}
      S_{j}^{*} z^{i+\ell} w^{-j -\ell}  \\
      & = \sum_{n \ge 0} z^{n} w^{-n} - \sum_{n,m \ge 0} \sum_{ \ell =
      0}^{\operatorname{min} \{n, m\}} S_{n-\ell} S^{*}_{m - \ell}
      z^{n} w^{-m}.
      \end{align*}
      We write
      $$
      K^{11}(z,w) = \sum_{n,m \ge 0}[ K^{11}]_{n,m} z^{n} w^{-m}
\text{
      where } [K^{11}]_{n,m} = \delta_{n,m}
     -  \sum_{ \ell =  0}^{\operatorname{min} \{n, m\}} S_{n-\ell}
S^{*}_{m - \ell}.
      $$
      Hence
      $$ z^{M}K^{11}(z,w) w^{-M} = \sum_{n,m \ge M} K^{11}_{n-M,
      m-M} z^{n} w^{-m}
      $$
      where,
      $$
       [K^{11}]_{n-M, m-M} =
      \delta_{n - M, m - M} - \sum_{\ell = 0}^{
      \operatorname{min}\{ n-M, m-M \}}
       S_{n - M - \ell} S^{*}_{m-M - \ell}.
  $$
  For $n \le m$ we can then write
  $$
   [K^{11}]_{n-M, m-M} = \delta_{n,m} - \sum_{\ell = 0}^{ n-M}
   S_{\ell} S^{*}_{m-n+\ell}
   $$
 and we have
   $$
  \lim_{M \to -\infty} [ K^{11} ]_{n-M, m-M} =
  \delta_{n, m} - \sum_{\ell = 0}^{\infty} S_{\ell} S_{m
  - n + \ell}^{*} \text{ if } n \le m.
  $$
  Similarly we see that
  $$
  \lim_{M \to -\infty} [ z^{M} K^{11}(z,w) w^{-M} ]_{\alpha, \beta} =
  \delta_{\alpha, \beta} - \sum_{\ell=0}^{\infty} S_{n-m +
  \ell} S_{\ell}^{*} \text{ if } n > m.
  $$
  We note that the limiting matrix $[X]_{n, m}: = \lim_{M \to
-\infty} [
  K^{11}]_{n-M, m-M}$ is necessarily {\em Toeplitz}, i.e.,
  the matrix entry $[X]_{n,m}$ depends only on the difference
  $n-m$. One can say in general that the existence of the limit
  $\lim_{M \to -\infty} [K^{11}]_{n-M, m-M}$ means that $[K^{11}]$ is
  {\em asymptotically Toeplitz}.  The content of the assertion
  \eqref{want-lim} is that the matrices $[K^{11}_{k}]_{n,m}$
associated
  with the kernels
  $$
  K^{11}_{k}(z,w) = \sum_{n,m \in {\mathbb Z}^{d}_{+}}
  [K^{11}_{k}]_{n,m} z^{n} w^{-m}
  $$
  are asymptotically Toeplitz in direction $k$ for $k = 1, \dots, d$.
  \end{remark}

  Our next goal is to show how the kernel decompositions
  \eqref{dBRkerneldecom}, \eqref{Kkinf} and \eqref{kernelCDP}
  associated with an augmented Agler decomposition
  \eqref{augAglerdecom} can be
  turned in the subspace decompositions \eqref{Omegadecom},
  \eqref{slim'} and \eqref{slim*'}, and \eqref{CDP} for a
  functional-model multievolution scattering system having scattering
  matrix $S$.  The formula \eqref{UfromcalU'} in Theorem
  \ref{T:BabyBear} then gives rise to a GR-unitary realization for $S$
  and we have an alternative functional-model proof of (2$^{\prime}$)
  $\Rightarrow$ (3) in Theorem \ref{T:Agler}.  We carry out the
  details of this procedure only for a more tractable special case.
  For this purpose we introduce the following notion of {\em strictly
  closely connected} augmented Agler decomposition.  Section
  \ref{S:minimal} below explores the meaning of a strictly closely
  connected Agler decomposition in terms of a GR-unitary realization
  $U = \sbm{A & B \\ C & D}$ for the given augmented Agler
  decomposition.

  \begin{definition}  \label{D:augAglerdecom-scc}
       Suppose that we are given an augmented Agler
       decomposition \eqref{augAglerdecom} for a given
Schur-Agler-class
       function $S(z)$.  We say that this augmented Agler
decomposition is
       {\em strictly closely connected} if both collections of
reproducing
       kernel Hilbert spaces $(\cH(K_{1}), \dots,
       \cH(K_{d}))$  and $(z_{1} \cH(K_{1}), \dots,
       z_{d}\cH(K_{d}))$ have no overlap,
       i.e., if both associated overlapping spaces (see
      Section \ref{S:overlapping}) are trivial:
      $$ \boldsymbol{\cL}: =\cL(K_{1}, \dots, K_{d}) = \{0\} \text{ and }
      \boldsymbol{\cL}':= \cL(K'_{1}, \dots,  K'_{d}) = \{0\}
      $$
      where we set $K'_{k}(z,w) = z_{k} K(z,w) w_{k}^{-1}$ for $k
      = 1, \dots, d$.
       \end{definition}

    The connection between the kernel CDP property \eqref{kernelCDP}
    and the subspace CDP property \eqref{CDP} is particularly
    transparent for the case of strictly closely connected augmented
    Agler decompositions.

    \begin{lemma}  \label{L:sccAglerdecom}
Suppose that we are given a strictly closely connected augmented
Agler decomposition
     \eqref{augAglerdecom}
    for a Schur--Agler class function $S \in \cS(\cE, \cE_{*})$.
    Then the functional-model version of the property \eqref{CDP}
holds,
    i.e.,  \eqref{CDP} holds with $\cL_{k} = \cH(K_{k})$, $\cF =
    \cH\left( \sbm{S(z) \\ I} \sbm{ S(w)^{*} & I} \right)$,
    $\cF_{*} = \cH \left( \sbm{ I \\ S(z)^{*}} \sbm{I & S(w)}
    \right)$ and with $\cU_{k}$ equal to the multiplication operator
    $M_{z_{k}}$:
    \begin{equation}  \label{funcmodelCDP}
    \bigoplus_{k=1}^{d} \cH(K_{k}) \oplus \cH\left( \sbm{ S(z) \\
    I}\sbm{ S(w)^{*} & I} \right) =
    \bigoplus_{k=1}^{d} z_{k} \cH(K_{k}) \oplus \cH\left( \sbm{I \\
    S(z)^{*}} \sbm{I & S(w)^{*}} \right).
    \end{equation}
    \end{lemma}

    \begin{proof}
    From
    the decomposition \eqref{dBRkerneldecom} we read off that
    $$
     \operatorname{span}_{1 \le k \le d} \cH(K_{k}) \subset
     \cH(K_{\cV^{S, \Omega}_{dBR}})
    $$
    if $\Omega$ is a shift-invariant sublattice with $0 \in
    \partial_{k} \Omega_{fin}$ for each $k = 1, \dots, d$.  From
    \eqref{dBRkerneldecom'} combined with the fourth of the identities
    \eqref{dBRkernels}, we see that it is always the case that
    $$
    \operatorname{span}_{1 \le k \le d} \cH(K_{k}) \cap \cH\left(
    \sbm{ S(z) \\ I } \sbm{ S(w)^{*} & I } \right) = \{0\}.
    $$
    Similarly, by choosing $\Omega'$ with $\be_{k} \in
    \partial_{k}\Omega_{fin}$ for $k=1, \dots, d$, one can read off
    that it is always the case that
    $$
    \operatorname{span}_{1 \le k \le d} z_{k} \cH(K_{k}) \cap
    \cH\left( \sbm{ I \\ S(z)^{*}} \sbm{ I & S(w) } \right) = \{0\}.
    $$
    Hence if $\{ K_{k}(z,w) \colon k = 1, \dots, d \}$ is a strictly
    closely connected augmented Agler decomposition, we conclude that
    \eqref{funcmodelCDP} holds as wanted.
    \end{proof}

   We next introduce the notion of {\em minimal} augmented Agler
   decomposition.

   \begin{definition} \label{D:minAglerdecom}
     Let $\Omega$ be a given shift-invariant
    sublattice of $\bbZ^{d}$.  We say that the augmented Agler
decomposition
    \eqref{enlargedAglerdecom} is {\em minimal} if both of the
    following conditions hold:
    \begin{enumerate}
    \item $\{K_{k}\}$ is   strictly closely connected (see
    Definition \ref{D:augAglerdecom-scc}).

       \item The overlapping space $\boldsymbol{\cL}_{\cV^{S,
       \Omega}_{dBR}}$ associated with the sum-decomposition of
       $K_{\cV^{S, \Omega}_{dBR}}$ in \eqref{dBRkerneldecom} is
       trivial, i.e., the collection of subspaces
     \begin{align*}
   & z^{n} \cH( K_{k}) \text{ for } n \in
    \partial_{k}\Omega_{\rm fin}, \\
   &(\widehat z^{k})^{n'} \cH\left( \begin{bmatrix} 0 & 0 \\ 0 &
     K_{+\infty,k}(z,w) \end{bmatrix} \right) \text{ for }
    \ell_{n',k} \in \partial_{k} \Omega_{+\infty},  \\
    & (\widehat z^{k})^{n''} \cH\left( \begin{bmatrix} \
    K_{-\infty,k}(z,w) & 0 \\ 0 & 0 \end{bmatrix} \right) \text{ for }
    \ell_{n'',k} \in \partial_{k} \Omega_{-\infty}
    \end{align*}
    form a direct-sum decomposition of $\cH(K_{\cV^{S,
    \Omega}_{dBR}})$.
     \end{enumerate}
     \end{definition}

     We shall see from Corollary \ref{C:minAglerdecom} below that
     condition (1) in Definition \ref{D:minAglerdecom} follows
     automatically from condition (2) in case all finite components
     $\partial_{k}\Omega_{\text{fin}}$ of the boundary of $\Omega$
     are nonempty.  On the other hand, we shall see as a consequence
     of Example \ref{E:2} and Remark \ref{R:E2} below that condition
     (1) in general does not imply condition (2).  At this stage it is
     possible to give the following explicit reproducing-kernel-space
construction
     for a scattering system realizing a given Agler decomposition
     under this minimality assumption.

    \begin{theorem} \label{T:minfuncmodel}
    \begin{enumerate}
        \item  The notion of minimal augmented Agler decomposition is
independent
      of the choice of shift-invariant sublattice.
\item   Assume that $S \in
    \cS(\cE, \cE_{*})$ is a Schur-Agler-class function with
    augmented Agler decomposition which is minimal (see Definition
    \ref{D:minAglerdecom}) with respect to the shift-invariant
sublattice
    $\Omega$. Let ${\mathfrak S}^{S}_{dBR}$ be the associated de
    Branges--Rovnyak model multievolution scattering system with
    scattering matrix $S$ (see \eqref{frakSdBR}). Define subspaces of
    $\cK^{S}_{dBR}$ according to
    \begin{align}
    & \cL_{k} = \cH(K_{k}) \text{ for } k = 1, \dots, d, \notag  \\
    & \cM_{k} = \cH\left( \begin{bmatrix} 0 & 0 \\ 0 &
    K_{+\infty,k}\end{bmatrix} \right)  \notag \\
    & \cM_{*k} = \cH\left( \begin{bmatrix} K_{-\infty,k} & 0 \\ 0
    & 0 \end{bmatrix} \right).
    \label{defLMM*}
  \end{align}
  Then the collection of spaces $\cL_{k}, \cM_{k}, \cM_{*k}$ ($k =1,
  \dots, d$) satisfies properties \eqref{Omegadecom}, \eqref{slim'},
  \eqref{slim*'} and \eqref{CDP}.  The associated GR-unitary
  colligation $U$ given by \eqref{UfromcalU'} can be written
  explicitly as
     $$
      U^{S}_{dBR} = \begin{bmatrix} A^{S}_{dBR,11} & \cdots &
      A^{S}_{dBR,1d} & B^{S}_{dBR,1} \\
      \vdots &  & \vdots & \vdots \\
      A^{S}_{dBR,d1} & \cdots & A^{S}_{dBR,dd} & B^{S}_{dBR,d}  \\
      C^{S}_{dBR,1} & \dots & C^{S}_{dBR,d} & D \end{bmatrix} \colon
      \begin{bmatrix} \cH( K_{1}) \\ \vdots \\ \cH(K_{d}) \\ \cE
\end{bmatrix} \to
      \begin{bmatrix} \cH(K_{1}) \\ \vdots \\ \cH(K_{d}) \\ \cE_{*}
\end{bmatrix}
     $$
      where
      \begin{align}
      A_{ij} \begin{bmatrix} f_{j} \\ g_{j} \end{bmatrix} =
      P_{\cH(K_{i})} \begin{bmatrix} z_{i}^{-1}f_{j}(z) \\
      z_{i}^{-1} g_{j}(z) \end{bmatrix}, &\qquad
      B_{i} e = P_{\cH(K_{i})} \begin{bmatrix} z_{i}^{-1}
      S(z) e \\ z_{i}^{-1}e \end{bmatrix}, \notag \\
      C_{j} \begin{bmatrix} f_{j} \\ g_{j} \end{bmatrix} =
      f_{j,0},  & \qquad
      D e = S(0) e
      \label{modelU}
       \end{align}
       for $i,j = 1, \dots, d$, $\sbm{f_{j} \\ g_{j}} \in \cH(K_{j})$
       and $e \in \cE$.
       \end{enumerate}
       \end{theorem}

       \begin{proof}
  Let $\cL_{k}, \cM_{k}, \cM_{*k}$ be defined as in \eqref{defLMM*}.
   By Lemma \ref{L:sccAglerdecom} we know that the property
   \eqref{CDP} holds.
   Then the import of the kernel identity
   \eqref{dBRkerneldecom} combined with condition (2) in Definition
   \ref{D:minAglerdecom} is that the space
   $\cH(K_{\cV^{S,\Omega}_{dBR}})$ has the internal orthogonal
decomposition
   \begin{align}
     \cH(K_{\cV^{S,\Omega}_{dBR}})  = & \bigoplus_{k=1}^{d}\
\bigoplus_{n
    \in \partial_{k} \Omega_{fin}}
  z^{n} \cH(K_{k})
     \oplus \bigoplus_{k=1}^{d}\ \bigoplus_{\ell_{n'', k} \in
    \partial_{k}\Omega_{+\infty}}
     \begin{bmatrix} 0 & 0 \\ 0 &
    (\widehat z^{k})^{\widehat{n''}^{k}} \cH(K^{22}_{+\infty,k})
    \end{bmatrix} \notag \\
       &  \oplus \bigoplus_{k=1}^{d}\ \bigoplus_{ \ell_{n',k} \in
\partial_{k}\Omega_{-\infty}}
    \begin{bmatrix} (\widehat z^{k})^{\widehat{n'}^{k}}
    \cH(K^{11}_{-\infty, k}) & 0 \\ 0 & 0 \end{bmatrix}
    \label{minimal-decom}
   \end{align}
      If we combine \eqref{minimal-decom} with
   the general decompositions \eqref{shiftdecom} and
   \eqref{dBRkerneldecom'},
   then we see that condition \eqref{Omegadecom} is satisfied.
    Furthermore the existence of the limits in
   \eqref{Kkinf} amounts to the conditions \eqref{slim'} and
   \eqref{slim*'} in Theorem \ref{T:BabyBear} and it is easy to see
   from the limit definition of $\cH(K^{11}_{-\infty,k})$ and
   $\cH(K^{22}_{+\infty,k})$ that $\cH(K^{11}_{-\infty,k})$ and
   $\cH(K^{22}_{+\infty, k})$ are invariant under multiplication by
   $z_{k}$ and $z_{k}^{-1}$.
   Hence Theorem \ref{T:BabyBear} applies and we
   arrive at a GR-unitary functional-model realization for the
   Schur-class function $S$ by applying the formula \eqref{UfromcalU'}
   to this situation.  If we apply the general formula
   \eqref{UfromcalU'} to this situation, we arrive at the formula
   \eqref{modelU} for the associated GR-unitary colligation.

   Finally, the fact that the notion of minimal Agler decomposition
   is independent of the choice of shift-invariant sublattice follows
   from the general fact from \cite{BabyBear} that the property
   \eqref{Omegadecom} is independent of the choice of shift-invariant
   sublattice $\Omega$ once it is established that the multievolution
   scattering system arises from an embedded GR-unitary colligation.
   \end{proof}

   \begin{remark} \label{R:scatmincol}  We note that the
       functional-model unitary colligation \eqref{modelU}
       constructed from a minimal Agler decomposition $\{K_{k}(z,w)
       \colon k = 1, \dots, d\}$ as in Theorem \ref{T:minfuncmodel}
       is scattering-minimal (see the end of Theorem
       \ref{T:BabyBear}) since the de Branges--Rovnyak model
       multievolution scattering system ${\mathfrak S}^{S}_{dBR}$ is a
       minimal scattering system.  We do not know in general if every
       Schur-Agler-class function has a minimal augmented Agler
       decomposition.
    \end{remark}

  The next two subsections give independent approaches for computing
  explicitly the kernels at infinity $K^{11}_{-\infty,k}$ and
  $K^{22}_{+\infty,k}$  in terms of the coefficients
  $A,B,C,D$ of a conservative Givone--Roesser realization of $S(z)$.
  These two subsections (Sections \ref{S:inf1} and \ref{S:inf2} below)
  can be omitted on first reading without loss of continuity.

  \subsection{Direct computation of $K^{11}_{-\infty,k}$ in terms of
  $A,B,C,D$} \label{S:inf1}
      Assuming a conservative Givone--Roesser realization for $S$, we
can
      identify the kernels at infinity $K_{\pm \infty,k}(z,w)$
      more explicitly in terms of $A,B,C,D$ as follows.  We consider
      only the case
      $K^{11}_{-\infty,k}(z,w)$ as $K^{22}_{+\infty,k}(z,w)$ is
      similar.  A useful notation is the Givone--Roesser functional
calculus
      on a block matrix $A = \sbm{ A_{11} & \cdots & A_{1d} \\
      \vdots & & \vdots \\ A_{d1} & \cdots & A_{dd} }$ given by
      $$
      A^{w} = P_{i_{N}} AP_{i_{N-1}}A \cdots
        P_{i_{1}}A \text{ if } w = i_{N} i_{N-1} \cdots i_{1}
        \in \cF_{d}
      $$
      where $\cF_{d}$ denotes the free semigroup consisting of all
     finite  words $i_{N} \cdots i_{1}$ in the letters $\{1, \dots,
d\}$
     (so $i_{\ell} \in \{1, \dots, d\}$ for each $\ell = 1, \dots,
N$).
      It is useful to introduce the {\em abelianization map} $\ba
\colon
    \cF_{d} \to \bbZ^{d}_{+}$ given by
    \begin{equation}  \label{ba-map}
    \ba(i_{N}\cdots i_{1}) = (n_{1}, \dots, n_{d}) \text{ if }
    n_{k} = \#\{\ell \colon i_{\ell} = k\} \text{ for } k = 1, \dots,
d
    \end{equation}
    together with the abelianized Givone--Roesser functional calculus
    \begin{equation}  \label{ba-funccal}
     (A^{\ba})^{n} = \sum_{w \in \cF_{d} \colon \ba(w) = n} A^{w}.
    \end{equation}
    Then we have the formal power series expansion
    $$
      C(I - Z_{\text{diag}}(z) A)^{-1} = \sum_{n \in \bbZ^{d}_{+}} C
      (A^{\ba})^{n} z^{n}.
    $$
     From \eqref{augKernel} we know that
      $$ K^{11}_{k}(z,w) = C (I - Z_{\text{diag}}(z) A)^{-1} P_{k} (I
-
      A^{*}Z_{\text{diag}}(w)^{-1})^{-1} C^{*}.
      $$
      Therefore
      \begin{align*}
      & K^{11}_{-\infty,k}(z,w)  = \lim_{t \to -\infty} z_{k}^{t}
w_{k}^{-t}
      K^{11}_{k}(z,w) \\
      & \qquad = \lim_{t \to -\infty} z_{k}^{t}w_{k}^{-t} C (I -
Z_{\text{diag}}(z)
      A)^{-1} P_{k} (I - A^{*} Z_{\text{diag}}(z)^{-1})^{-1} C^{*} \\
      & \qquad =\lim_{t \to -\infty} \sum_{n,m \in \bbZ^{d}_{+}} C
(A^{\ba})^{n} P_{k}
      (A^{* \ba})^{m} C^{*} z_{k}^{t} z^{n} w^{-m} w_{k}^{-t}  \\
      & \qquad = \sum_{\alpha, \beta \in \bbZ^{d} \colon (\widehat
\alpha^{k}),
      (\widehat \beta^{k}) \in \bbZ^{d-1}_{+}}
      \left[ \lim_{t \to -\infty} C (A^{\ba})^{\alpha -t \be_{k}}
P_{k}
      (A^{*\ba})^{\beta - t \be_{k}} C^{*} \right] z^{\alpha}
w^{-\beta}.
    \end{align*}
    i.e.,
     $$ K^{11}_{-\infty,k}(z,w) = \sum_{\alpha, \beta \in \bbZ^{d}
\colon (\widehat \alpha^{k}),
      (\widehat \beta^{k}) \in \bbZ^{d-1}_{+}}
      [K^{11}_{-\infty,k}]_{\alpha, \beta} z^{\alpha} w^{-\beta}
      $$
      where
      \begin{equation}  \label{K-infk-coeff}
    [K^{11}_{-\infty,k}]_{\alpha, \beta} =
    \lim_{t \to -\infty} C (A^{\ba})^{\alpha -t \be_{k}} P_{k}
          (A^{*\ba})^{\beta -  t \be_{k}} C^{*}.
  \end{equation}
 In case $(\widehat \alpha^{k}) =  (\widehat \beta^{k}) = 0 \in
 \bbZ^{d-1}$, formula \eqref{K-infk-coeff} can be simplified somewhat:
 $$
  [ K_{-\infty,k}]_{\alpha, \beta} = \begin{cases}
   C_k \Delta_{A_{kk}^{*}} A_{kk}^{* \beta_{k}-\alpha_{k}}C_k^{*} &
   \text{if } \alpha_{k} \le \beta_{k} \\
   C_k A^{\alpha_{k} - \beta_{k}}\Delta_{A_{kk}^{*}} C_k^{*} &
\text{if }
   \alpha_{k} > \beta_{k}
  \end{cases}
  $$
  where we have set
  $$
    \Delta_{A_{kk}^{*}} = \operatorname{strong\ limit}_{s \to +\infty}
     A_{kk}^{s} A_{kk}^{*s}.
  $$
  In particular, this formula applies for the case $d=1$ and gives the
  kernel associated with the Sz.-Nagy--Foia\c{s} defect space
  ${\mathcal R}_{*}$ (see \cite[page 139]{MammaBear}).

 \subsection{Computation of $K^{11}_{-\infty,k}$ via solution of
 initial value problem}  \label{S:inf2}

 Suppose that $\Omega$ is a shift-invariant sublattice such that
     $\ell_{0,k}$ is part of $\partial_{k}\Omega_{-\infty}$.
     Theorem 4.15 in \cite{MammaBear} indicates how to construct a
     trajectory $\xi = (u,x,y)$ by specifying initial condition
$x^{0}$ for
     $x$ on $\partial \Omega$, a future input $u^{+} \in
\ell^{2}(\Omega,\cE)$ and a past output $y^{-} \in
\ell^{2}(\bbZ^{d}
     \setminus \Omega, \cE_{*})$.   In particular, if we
     specify
     \begin{align*}
     & u^{+} = 0,  \qquad y^{-}=0,  \\
     & x^{0}|_{\partial_{j}\Omega_{\text{fin}}} = 0 \text{ and }
     x^{0}|_{\partial_{j}\Omega_{+\infty}} = 0 \text{  for all } j =
1,
     \dots, d, \\
     & x^{0}|_{\partial_{j}\Omega_{\-\infty}} = 0 \text{ for } j \ne
     k, \\
     & x^{0}|_{\ell_{n',k}} = 0 \text{ for } \ell_{n',k} \in
     \partial_{k} \Omega_{-\infty} \text{ with } n' \ne 0, \\
     & x^{0}|_{\ell_{0,k}} = \vec h
     \end{align*}
     where $\vec h$ is a prespecified element of the defect space
     \begin{align*}
     {\mathcal R}_{*k} & = \{ \vec h = \{ h(t) \}_{t \in \bbZ} \in
     \ell(\bbZ, \cH_{k}) \colon h(t+1) = A_{kk} h(t) \\
     & \qquad \text{ and }
     \| \vec h\|^{2}_{{\mathcal R}_{*k}} := \lim_{t \to -\infty} \|
     h(t) \|^{2}_{\cH_{k}} < \infty\},
     \end{align*}
     then the solution of the initial value problem is given by
     $\xi = (u,x,y)$ where
     \begin{align*}
     &  \widehat u(z) = 0, \\
     &  \widehat x(z) = i_{k} \widehat h^{k}(z_{k}) -
     (I - Z_{\diag}(z) A)^{-1} (i_{k} \widehat h^{k}(z_{k}) -
     Z_{\diag}(z) A i_{k} \widehat h^{k}(z_{k}), \\
     &  \widehat y(z) =
    C_{k} \widehat h^{k}(z_{k}) -
         C(I - Z_{\diag}(z) A)^{-1} (i_{k} \widehat h^{k}(z_{k}) -
         Z_{\diag}(z) A i_{k} \widehat h^{k}(z_{k})
      \end{align*}
      where $i_{k}$ is the injection
      $$
      i_{k} \colon h_{k} \mapsto \sbm {0 \\ \vdots \\ h_{k} \\ \vdots
\\ 0 }
      $$
      of $\cH_{k}$ into the $k$-th component of $\bigoplus_{j=1}^{d}
      \cH_{d}$ and
      where we have set
      $$
       \widehat h^{k}(z_{k}) = \sum_{t = -\infty}^{+\infty} h(t)
z_{k}^{t}.
      $$
      The space $\cH(K^{11}_{-\infty,k})$ is the first component of
the
      image under the de
     Branges-Rovnyak map $\Pi^{dBR}$ of all such trajectories, where,
     in general, the de Branges--Rovnyak map is given by
     $$ \Pi^{dBR} \colon (u,x,y) \mapsto \begin{bmatrix} \widehat
     y(z) \\ \widehat u(z) \end{bmatrix}.
     $$
     We thus see that $\cH(K^{11}_{-\infty,k})$ is the image of the
map
     giving rise to the $Z$-transformed output signal $\widehat y(z)$
     from an initial condition $\vec h \in {\mathcal R}_{*k}$, i.e.,
     $\cH(K^{11}_{-\infty,k}) = \operatorname{im}
     \Pi^{dBR,1}_{-\infty,k}$ where $\Pi^{dBR,1}_{-\infty,k} \colon
     {\mathcal R}_{*k} \to \cE_{*}[[z^{\pm 1}]]$ is given by
     \begin{equation}  \label{PidBR1}
       \Pi^{dBR,1}_{-\infty,k} \colon \vec h \mapsto
       C_{k} \widehat h^{k}(z_{k}) -
       C(I - Z_{\diag}(z) A)^{-1} (i_{k} \widehat h^{k}(z_{k}) -
       Z_{\diag}(z) A i_{k} \widehat h^{k}(z_{k}))
     \end{equation}
     and where the norm on $\cH(K^{11}_{-\infty,k})$ is equal to the
     lifted norm from ${\mathcal R}_{*k}$. If we define an element
     $\Pi^{dBR,1}_{-\infty, k}(z) \in \cL({\mathcal R}_{*k},
     \cE_{*})[[ z^{\pm 1} ]]$ by
     \begin{equation} \label{PidBR(z)}
     \Pi^{dBR,1}_{-\infty, k}(z) (\vec h) =
      \left( \Pi^{dBR,1}_{-\infty, k} (\vec h) \right)(z).
     \end{equation}
     it then follows from the results in Section \ref{S:liftednorm}
that the
     kernel $K^{11}_{-\infty,k}$ can alternatively be computed as
     \begin{equation}  \label{K11-infk-alt}
     K^{11}_{-\infty,k}(z,w) = \Pi^{dBR,1}_{-\infty,k}(z)\left(
     \Pi^{dBR,1}_{-\infty,k}(w) \right)^{*}.
     \end{equation}
     In this subsection we show how this formula for
     $K^{11}_{-\infty, k}(z,w)$ is consistent with
\eqref{K-infk-coeff}.
     We shall use the notation $(A^{\ba})^n$ introduced in Section
     \ref{S:inf1} (see displays \eqref{ba-map} and
     \eqref{ba-funccal}).

     We first verify the following.

    \begin{proposition}  \label{P:Kinfk-IC}
     The coefficients $[\Pi^{dBR,1}_{-\infty,k}]_{m}$ of the formal power
     series
     $$ \Pi^{dBR,1}_{-\infty,k}(z) = \sum_{m \in \bbZ^{d}}
     [\Pi^{dBR,1}_{-\infty,k}]_{m} z^{m}
     $$
     defined in \eqref{PidBR1} and \eqref{PidBR(z)} are given by
     \begin{equation}  \label{PidBR-infk2}
        [\Pi^{dBR,1}_{-\infty,k}]_{m} =
          \begin{cases} 0 & \text{if } \widehat m^{k} \notin
               \bbZ^{d-1}_{+}, \\
               \lim_{T \to -\infty} C (A^{\ba})^{m - T \be_{k}} i_kh(T) &
               \text{if } \widehat{m}^k \in \bbZ^{d-1}_{+}
          \end{cases}
       \end{equation}
   with, for given $e_{*} \in \cE_{*}$, adjoint given by
   \begin{align}
      & [ \Pi^{dBR,1}_{-\infty,k}]_{m}^{*}e_{*} = 0 \text{ if }
\widehat m^{k}
      \notin \bbZ^{d-1}_{+}, \notag \\
   & [ \Pi^{dBR,1}_{-\infty,k}]_{m}^{*}e_{*} = \vec h \text{ where }
\vec h
    \text{ is the unique element of } {\mathcal R}_{*k} \text{ such
    that } \notag \\
    & \qquad \vec h \underset{-\infty} \sim \{ P_{k}(A^{* \ba})^{m -
    t \be_{k}}
    C^{*} e_{*}\}_{t \le m_{k}} \text{ if } \widehat{m}^{k}
    \in \bbZ^{d-1}_{+}.
       \label{PidBR-infk2*}
       \end{align}
     Here $\vec h \underset{-\infty} \sim \{g(t)\}_{t \le m_{k}}$
     means that $\vec h$ and $\{g(t)\}_{t \le m_{k}}$ have the
     same asymptotics at $-\infty$ in the sense that
     $$
      \lim_{t \to -\infty} \langle f(t), h(t) - g(t)
\rangle_{\cH_{k}} = 0
      \text{ for all } \vec f = \{f(t)\}_{t \in \bbZ} \in {\mathcal
      R}_{*k}.
     $$
     In case $\widehat m^{k} = 0$, the formula for $[
     \Pi^{dBR,1}_{-\infty,k}]_{m}^{*} e_{*}$ can be given explicitly
as
     \begin{align}
    [ \Pi^{dBR,1}_{-\infty,k}]_{m}^{*} e_{*} =
     \vec h = \{h(t)\}_{t \in \bbZ} \in
     {\mathcal R}_{*k} \text{ where } \notag \\
     h(t) = \begin{cases}
     \Delta_{A_{kk}^{*}}A_{kk}^{* m_{k}-t} C_k^{*} e_{*} &\text{if }
     t \le m_{k}  \\
     A_{kk}^{t-m_{k}} \Delta_{A_{kk}^{*}} C_k^{*} e_{*} & \text{if } t >
m_{k}.
     \end{cases}
     \label{PidBR-infk2*0}
       \end{align}
       where we have set
       $$ \Delta_{A_{kk}^{*}} = \operatorname{strong \ limit}_{s \to
       +\infty} A_{kk}^{s} A_{kk}^{*s}.
       $$

   \end{proposition}

   \begin{proof}
     We compute
     \begin{align*}
    & i_{k} \widehat h^{k}(z_{k}) - Z_{\diag}(z) A i_{k}\widehat
h^{k}(z_{k})
    = \sum_{t = -\infty} ^{+\infty} i_kh(t) z_{k}^{t} -
    \sum_{j=1}^{d} z_{j}P_{j} A  \Big(\sum_{t=-\infty}^{+\infty}
i_kh(t)
    z_k^{t} \Big) \\
    & \qquad =
    \sum_{t=-\infty}^{+\infty} i_kh(t) z_{k}^{t} - z_{k} i_kA_{kk}
\Big(
    \sum_{t = -\infty }^{+\infty} h(t) z_{k}^{t}\Big) - \sum_{j
\colon j \ne k}
    z_{j} P_{j} A  \Big( \sum_{t = -\infty }^{+\infty} i_kh(t)
    z_{k}^{t} \Big)\\
    & \qquad =
    \sum_{t=-\infty}^{+\infty} i_kh(t) z_{k}^{t} -
    \sum_{t=-\infty}^{+\infty} i_kh(t+1) z_{k}^{t+1} -\sum_{j \colon
j \ne
    k} z_{j} P_{j}A \Big( \sum_{t=-\infty }^{+\infty} i_kh(t)
    z_{k}^{t}\Big) \\
    & \qquad \qquad \text{(where we use that $h(t+1) = A_{kk} h(t)$
since
    $\vec h \in {\mathcal R}_{*k}$)} \\
    & \qquad =
    -\sum_{j\colon j \ne k}\ \sum_{t=-\infty} ^{+\infty} P_{j}A
    i_kh(t) z_{j} z_{k}^{t}.
    \end{align*}
    Thus
    \begin{align*}
    &  \Pi^{dBR,1}_{-\infty,k}(z) \vec h = C_{k} \widehat
    h^{k}(z_{k}) -
    C (I - Z_{\diag}(z)A)^{-1} (i_{k}\widehat h^{k}(z_{k}) -
    Z_{\diag}(z) A i_{k} \widehat h^k(z_{k})) \\
    & \qquad =
    \sum_{t=-\infty}^{+\infty} C_{k} h(t) z_{k}^{t} +
    \sum_{n \in \bbZ^{d}_{+}}\ \sum_{j \colon j \ne k}\
    \sum_{t=-\infty}^{+\infty} C (A^{\ba})^{n} P_{j} A i_{k}h(t)
    z^{n} z_{j} z_{k}^{t}
    \\
    & \qquad =
    \sum_{m \in \bbZ^{d}} [\Pi^{dBR,1}_{-\infty,k}]_{m} \vec hz^m
    \end{align*}
    where we have set $[\Pi^{dBR,1}_{-\infty,k}]_{m} \vec h$ equal to
    \begin{align}
    & 0  \text{ if } \widehat m^{k} \notin
    \bbZ^{d-1}_{+}, \notag \\
      & C_k h(m_{k})  \text{ if } \widehat m^{k} = 0, \notag \\
      & \sum_{t = -\infty}^{m_{k}}\
   \sum_{j \colon j \ne k, m_{j}\ge 1}
    C(A^{\ba})^{m-\be_{j}-t\be_{k}} P_jAi_kh(t)
    \text{ if } \widehat m^{k}
    \in \bbZ^{d-1}_{+} \setminus \{0\}.
     \label{PidBR-infk}
     \end{align}
     For the case where $\widehat m^{k} \in \bbZ^{d-1}_{+}
     \setminus \{0\}$, we may simplify the formula as follows:
     \begin{multline}
     [\Pi^{dBR,1}_{-\infty,k}]_{m}  =
     \lim_{T \to -\infty} \sum_{t=T}^{m_{k}} \sum_{j \colon j \ne k,
     m_{j} \ge 1} C (A^{\ba})^{m-\be_{j}-t \be_{k}}P_j A  i_k h(t) \\
     = \lim_{T \to -\infty} \left[ \sum_{t=T}^{m_{k}} C (A^{\ba})^{m -
    t \be_{k}} i_k h(t) - \sum_{t=T}^{m_{k}-1} C (A^{\ba})^{m -
(t+1)\be_{k}}P_k A i_k
    h(t) \right]  \\
    =\lim_{T \to -\infty} \left[\sum_{t=T}^{m_k} C (A^{\ba})^{m - t
\be_{k}}i_k h(T)-\sum_{t=T}^{m_k-1} C (A^{\ba})^{m - (t+1)
\be_{k}}i_k h(t+1)\right]\\
     = \lim_{T \to -\infty} C (A^{\ba})^{m - T \be_{k}} i_k h(T).
    \label{PidBR-infk1}
    \end{multline}
    For the case where $\widehat m^{k} = 0$, this formula makes
    sense and can be simplified further as follows.  Since
    $$
     (A^{\ba})^{m-T \be_{k}} i_kh(T) = (P_kA)^{m_{k} - T} i_kh(T) =
       i_kh(m_{k} - T + T) = i_kh(m_{k}),
    $$
    the formula \eqref{PidBR-infk1} in this case collapses to
    $ C_k h(m_{k})$ agreeing with the formula for
    $[\Pi^{dBR,1}_{-\infty,k}]_{m}$ in \eqref{PidBR-infk} for this
case.
    Thus the formula \eqref{PidBR-infk} can be simplified to
    $$
    [\Pi^{dBR,1}_{-\infty,k}]_{m} =
      \begin{cases} 0 & \text{if } (\widehat m^{k}) \notin
           \bbZ^{d-1}_{+}, \\
           \lim_{T \to -\infty} C (A^{\ba})^{m - T \be_{k}} i_kh(T) &
           \text{if } \widehat m^k \in \bbZ^{d-1}_{+}.
      \end{cases}
   $$
   agreeing with the formula \eqref{PidBR-infk2*0} in the statement of
   the proposition.
   (If we interpret $(A^{\ba})^{n}=0$ whenever $n \notin
\bbZ^{d}_{+}$,
   then the second formula in \eqref{PidBR-infk2} in fact suffices for
   all cases.)

   If $\widehat m^{k} \notin \bbZ^{d-1}_{+}$, then trivially
   $[\Pi^{dBR,1}_{-\infty,k}]_{m}^{*} = 0$.  For the case where
   $\widehat m^{k} \in \bbZ^{d-1}_{+}$, the element $\vec h =
   [\Pi^{dBR,1}_{-\infty,k}]_{m}^{*} e_{*}$ is by definition the
   necessarily unique element of ${\mathcal R}_{*k}$ such that
   $$
   \langle [\Pi^{dBR,1}_{-\infty,k}]_{m} \vec g, e_{*}
   \rangle_{\cE_{*}} = \langle \vec g, \vec h \rangle_{{\mathcal
R}_{*k}},
   $$
   or, equivalently, such that
   $$
   \left\langle \lim_{T \to -\infty} C (A^{\ba})^{m - T \be_{k}}
P_{k} g(T),
   e_{*}\right\rangle_{\cE_{*}} = \lim_{T \to -\infty }\langle g(T),
h(T)
   \rangle_{\cH_{k}}
   $$
   which in turn is equivalent to
   $$
   \lim_{T \to -\infty} \langle g(T), P_{k}(A^{*\ba})^{m-T \be_{k}}
C^{*}
   e_{*} \rangle = \lim_{T \to -\infty} \langle g(T), h(T)
\rangle_{\cH_{k}}
   $$
   for all $\vec g = \{g(t)\}_{t \in \bbZ} \in {\mathcal R}_{*k}$.
   We conclude that $\vec h = \{h(t) \}_{t \in \bbZ} \in {\mathcal
   R}_{*k}$ is uniquely determined by the constraint that
   $$ \vec h \underset{-\infty}\sim \{P_{k}(A^{*\ba})^{m-t\be_{k}}
C^{*}
   e_{*}\}_{t \le m_{k}}$$
   as asserted in \eqref{PidBR-infk2*}.

   Let us now continue the computation under the assumption that
   $\widehat m^{k} = 0$.  In this case we have the simplification
   $$
   (A^{*\ba})^{m-T \be_{k}} C^{*}
     e_{*} =i_k A_{kk}^{*m_{k}-t} C^{*} e_{*} \text{ for } t \le
m_{k}.
   $$
   For convenience,  let us set
   $$
   h_{\text{\rm pre}}(t) = \begin{cases}
    A_{kk}^{* m_{k}-t}C_k^{*} e_{*} &\text{if } t \le m_{k}, \\
    A_{kk}^{t-m_{k}} C_k^{*} e_{*} &\text{if } t > m_{k}.
    \end{cases}
   $$
   This candidate for $[\Pi^{dBR,1}_{-\infty,k}]_{m}^{*} e_{*}$ has
the
   desired asymptotics at $-\infty$ but fails to satisfy the
   recursion $h_{\text{\rm pre}}(t+1) = A_{kk} h(t)$ for all $t \in
   \bbZ$.  We check that $\vec h = \{h(t) \}_{t \in \bbZ}$ given as in
   \eqref{PidBR-infk2*0}
   $$
   h(t) = \begin{cases} \Delta_{A_{kk}^{*}}
   A_{kk}^{*m_{k}-t}C_k^{*}e_{*} & \text{if } t \le m_{k} \\
   A_{kk}^{t-m_{k}}\Delta_{A_{kk}^{*}} C_k^{*} e_{*} &\text{if } t >
m_{k}
   \end{cases}
   $$
   maintains the same asymptotics at $-\infty$ and satisfies the
   recurrence relation.  Indeed, from the fundamental identity
   \begin{equation}  \label{fundid}
    A_{kk} \Delta_{A_{kk}^{*}} A_{kk}^{*} = \Delta_{A_{kk}^{*}}
   \end{equation}
   it is easily checked that $\vec h$ satisfies the recursion
   $$
     h(t+1) = A_{kk} h(t).
   $$
   As $0 \le \Delta_{A_{kk}^{*}} \le I$, $\Delta_{A_{kk}^{*}}$ has a
   nonnegative square root $(\Delta_{A_{kk}^{*}})^{1/2}$.
   Moreover, the computation
   \begin{align*}
       \|h(t)\|^{2} & \le \| \Delta_{A_{kk}^{*}}^{1/2}\|^{2}
     \cdot  \| \Delta_{A_{kk}^{*}}^{1/2} A_{kk}^{* m_{k} -t}C_k^{*}
      e_{*}\|^{2} \\
     & \le   \| \Delta_{A_{kk}^{*}}\| \cdot
      \sup_{s \in \bbZ_{+}} \langle A_{kk}^{s} \Delta_{A_{kk}^{*}}
      A_{kk}^{*s} C_k^{*}e_{*}, C_k^{*} e_{*} \rangle_{\cH_{k}} \\
      & = \| \Delta_{A_{kk}^{*}}\| \cdot
   \langle \Delta_{A_{kk}^{*}} C_k^{*} e_{*}, C_k^{*} e_{*}
   \rangle_{\cH_{k}} \text{ (using \eqref{fundid})} \\
   & \le \| \Delta_{A_{kk}^{*}}\| \cdot \| \Delta_{A_{kk}^{*}}^{1/2}
C_k^{*}
   e_{*} \|^{2} < \infty
   \end{align*}
   shows that $\vec h \in {\mathcal R}_{*k}$.

   It remains to show that
   $ \vec h \underset{-\infty} \sim \vec h_{\text{\rm pre}}$, i.e.,
   that, for all $\vec f \in   {\mathcal R}_{*k}$ we have
   \begin{equation} \label{toshow1} \lim_{t \to -\infty}
   \langle h_{\text{\rm pre}}(t) - h(t), f(t) \rangle_{\cH_{k}} :=
   \lim_{t \to -\infty} \langle (I - \Delta_{A_{kk}^{*}}) A_{kk}^{*
   m_{k}-t} C_k^{*} e_{*}, f(t) \rangle_{\cH_{k}} = 0.
   \end{equation}
   As $\| f(t) \|$ is bounded as $t \to -\infty$, so also is
   $ \| (I - \Delta_{A_{kk}^{*}})^{1/2} f(t) \|$.  By the
Cauchy-Schwarz
   inequality, \eqref{toshow1} follows if we show
   \begin{equation}  \label{toshow2}
       \lim_{s \to +\infty} \| (I - \Delta_{A_{kk}^{*}})^{1/2}
       A_{kk}^{*s} C_k^{*} e_{*} \| = 0.
   \end{equation}
   We compute
   \begin{align*}
     &  \lim_{s \to +\infty} \| (I - \Delta_{A_{kk}^{*}})^{1/2}
       A_{kk}^{*s} C_k^{*} \|^{2}  =
       \lim_{s \to +\infty} \langle A_{kk}^{s} (I -
       \Delta_{A_{kk}^{*}}) A_{kk}^{*s} C_k^{*}e_{*}, C_k^{*} e_{*}
       \rangle_{\cH_{k}} \\
       & \qquad  = \lim_{s \to +\infty}[ \langle A_{kk}^{s}
A_{kk}^{*s}
       C_k^{*}e_{*}, C_k^{*} e_{*} \rangle_{\cH_{k}} -
       \langle \Delta_{A_{kk}^{*}} C_k^{*}e_{*}, C_k^{*}e_{*}
\rangle_{\cH_{k}}]
       \text{ (using \eqref{fundid} again)} \\
       & \qquad  = 0
    \end{align*}
    and \eqref{toshow2} follows as needed.  This completes the proof
of
    all parts of Proposition \ref{P:Kinfk-IC}.
   \end{proof}

   \begin{remark}
      The expressions \eqref{PidBR-infk} and \eqref{PidBR-infk2}
      represent the two distinct regularization methods used in
      \cite{MammaBear}.  The proof above that \eqref{PidBR-infk}
      $\Rightarrow$ \eqref{PidBR-infk2} gives the connection
      between these two regularization methods which was left
      implicit in \cite{MammaBear}.
       \end{remark}

  \begin{remark}  We conjecture that an explicit formula for
  $[\Pi^{dBR,1}_{-\infty,k}]_{m}^{*}$ is
  $$
  [\Pi^{dBR,1}_{-\infty,k}]_{m}^{*} e_{*} =
  \lim_{T \to -\infty} (P_kA)^{t-T} P_{k} (A^{*\ba})^{m -T \be_{k}}
  C^{*}e_{*}
  $$
  but we have not been able to verify that this expression has the
  required asymptotics at $-\infty$.
  \end{remark}

  With the information from Proposition \ref{P:Kinfk-IC} in hand
  combined with the formula \eqref{K11-infk-alt}, we may compute the
  coefficients $[K^{11}_{-\infty,k}]_{\alpha, \beta}$ for the kernel
  $K^{11}_{-\infty,k}(z,w)$ as follows.
  For $e_{*} \in \cE_{*}$ and $\beta \in \bbZ^{d}$ with $(\widehat
  \beta^{k}) \in \bbZ^{d-1}_{+}$, set
  $$
  h(t) = [\Pi^{dBR,1}_{-\infty,k}]_{\beta}^{*} e_{*}, \qquad
  h_{\text{\rm pre}}(t) =
  P_{k}(A^{* \ba})^{m - t\be_{k}}  C^{*} e_{*}
  $$
  so $\vec h \underset{-\infty} \sim \vec h_{\text{\rm pre}}$.
  Then, at least at a formal level, using that
  $$
  \lim_{T \to -\infty} h(T) = \lim_{T \to -\infty }h_{\text{\rm
  pre}}(T),
  $$
  we then have, for $\alpha, \beta \in \bbZ^{d}$ with $(\widehat
  \alpha^{k}), (\widehat \beta^{k}) \in \bbZ^{d-1}_{+}$,
  \begin{align*}
      [K^{11}_{-\infty,k}]_{\alpha, \beta} e_{*} & =
      [ \Pi^{dBR,1}_{-\infty,k}]_{\alpha} [
      \Pi^{dBR,1}_{-\infty,k}]_{\beta}^{*} e_{*} \\
      & = \lim_{T \to -\infty} C (A^{\ba})^{\alpha -T \be_{k}}i_k
h(T) \\
      & = \lim_{T \to -\infty} C(A^{\ba})^{\alpha -T \be_{k}}i_k
      h_{\text{\rm pre}}(T) \\
      & = \lim_{T \to -\infty} C (A^{\ba})^{\alpha -T \be_{k}} P_{k}
      (A^{* \ba})^{\beta -T \be_{k}} C^{*} e_{*}
   \end{align*}
   which agrees with the formula \eqref{K-infk-coeff} for
   $[K^{11}_{-\infty,k}]_{\alpha, \beta}$.  This concludes our
   computation of $K^{11}_{-\infty,k}(z,w)$ via  the formula for
   solution of the initial value problem from \cite{MammaBear}.

   \section{Finer structure of GR-unitary colligations}
\label{S:minimal}

   In this section it will be convenient to work with a
Givone--Roesser
   unitary colligation in more coordinate-free form.  By a {\em
   Givone--Roesser unitary colligation in coordinate-free form}, we
mean
   a unitary operator $U$ of the form
   \begin{equation} \label{colligation-free}
    U = \begin{bmatrix} A & B \\ C & D \end{bmatrix} \colon
   \begin{bmatrix} \cH \\ \cE \end{bmatrix} \to \begin{bmatrix} \cH \\
   \cE_{*} \end{bmatrix}
   \end{equation}
   together with a collection $\{P_{1}, \dots, P_{d}\}$ of orthogonal
   projections on $\cH$ with pairwise-orthogonal ranges having span
equal
   to the whole space $\cH$.  Given such a coordinate-free
Givone--Roesser
   unitary colligation, the associated transfer function is
   $$
     S(z) = D + C (I_{\cH} - Z_{\diag}(z) A)^{-1} Z_{\diag}(z) B
    $$
    where now $Z_{\diag}(z)$ stands for the operator-pencil
    $$
     Z_{\diag}(z) = z_{1} P_{1} + \cdots + z_{d} P_{d}.
    $$
    with associated system equations
    \begin{equation}  \label{sys-free}
    \Sigma(U) \colon \left\{
    \begin{array}{rcl}
    P_{k}x(\sigma_{k}(n)) & = & P_{k}A x(n) +  P_{k} B u(n) \text{
    for } k = 1, \dots, d, \\
    y(n) & = & C x(n) + D u(n).
     \end{array} \right.
     \end{equation}
    We consider here only the {\em conservative} case corresponding
to the
    condition that $U$ be unitary.
    This of course is equivalent to our original definition
    \eqref{colligation} of Givone--Roesser unitary colligation after
    specification of the block basis for $\cH$ associated with the
family of
    orthogonal projections $\{P_{1}, \dots, P_{d} \}$. This
    coordinate-free version gives a little more flexibility in the
    construction of examples.

   Given a Givone--Roesser unitary colligation $U$, we already
   introduced the notion of {\em scattering-minimal}  at the end of
   Theorem \ref{T:BabyBear}.
    We have the following additional notions of minimality for a
GR-unitary
    colligation.

    \begin{definition} \label{D:minimal}
    \begin{enumerate}
        \item We say that the GR-unitary colligation $U$ as in
        \eqref{colligation-free} is {\em closely connected} if
     $\cH$ is equal to the smallest subspace of $\cH$
      containing $\operatorname{im} B$ and $\operatorname{im} C^{*}$
      which is invariant for $A$,  $A^{*}$ and $P_{k}$ for all $k=1,
      \dots, d$.
      \item We say that $U$ is {\em strictly closely connected} if the
      restriction of the de Branges--Rovnyak identification map to
      $\operatorname{im} i_{\cH} = \cL$ in the multievolution
scattering
      system in which $U$ is embedded (see Section \ref{S:scat-col})
is
      injective, or, more concretely, the map
      $\Pi^{dBR}_{U} \colon \cH \to (\cE_{*} \oplus \cE) [[
      z^{\pm 1} ]] $ given by
      \begin{equation}   \label{PidBRU}
       \Pi^{dBR}_{U} \colon h \mapsto \begin{bmatrix} C (I -
       Z_{\diag}(z) A)^{-1}  \\ B^{*} (I - Z_{\diag}(z)^{-1}
       A^{*})^{-1} Z_{\diag}(z)^{-1} \end{bmatrix} h
      \end{equation}
      is injective.
      \item We say that the GR-unitary colligation $U$
      is {\em shifted strictly closely connected} if the map
      $\Pi^{dBR \prime}_{U} \colon \cH \to (\cE_{*} \oplus cE)\langle
       \langle z^{\pm 1} \rangle \rangle$ given by
       \begin{equation}  \label{PidBRU'}
       \Pi^{dBR \prime}_{U} \colon h \mapsto
      \begin{bmatrix} C (I - Z_{\diag}(z) A)^{-1} Z_{\diag}(z) \\
B^{*}
      (I - Z_{\diag}(z)^{-1} A^{*})^{-1} \end{bmatrix} h
       \end{equation}
       is injective.
      \end{enumerate}
      \end{definition}

     If we introduce subspaces
      \begin{align}
\cH_{cc} & = \overline{\operatorname{span}} \{ p(A,A^{*},P_{1},
\dots, P_{d}) (\operatorname{im} B + \operatorname{im} C^{*})
\colon p = \text{
polynomial in} \notag \\
& \qquad d+2 \text{ noncommuting variables}\} \notag \\
\cH_{scc} & = \overline{\operatorname{span}} \{ \operatorname{im}
(A^{*\ba})^{n}
  C^{*}, \, \operatorname{im} (A^{\ba})^{n} P_{k} B \colon n \in
  \bbZ^{d}_{+}, \ k=1, \dots, d \}, \notag \\
  \cH_{sscc} & = \overline{\operatorname{span}} \{
  \operatorname{im}(A^{* \ba})^{n}P_{k}C^{*},
 \operatorname{im}
  (A^{\ba})^{n}B \colon n \in \bbZ^{d}_{+},\ k = 1, \dots, d \}.
\label{minimalsubspaces}
   \end{align}
   Then the following follows easily from the definitions.

   \begin{proposition}  \label{P:col-prop}
       Let $U$ be a GR-unitary colligation. Then:
   \begin{enumerate}
       \item $U$ is closely connected $\Leftrightarrow$
       $\cH_{cc} = \cH$,
       \item $U$ is strictly closely connected $\Leftrightarrow$
       $\cH_{scc} = \cH$, and
       \item $U$ is shifted strictly closely connected
       $\Leftrightarrow$ $\cH_{sscc} = \cH$.
   \end{enumerate}
   \end{proposition}

   We may also view a GR-unitary colligation $U$ as embedded in a
   mutievolution Lax--Phillips scattering system ${\mathfrak
   S}(\Sigma(U))$ via the
   admissible-trajectory-space construction form \cite{BabyBear}
   described in Section \ref{S:scat-col}.  This scattering system
   ${\mathfrak S}(\Sigma(U))$ then has the additional geometric
   structure described in Theorem \ref{T:BabyBear}:  the ambient
   space ${\mathcal K}$ for ${\mathfrak S}(\Sigma(U))$ contains
   subspaces $\cL_{1}, \dots, \cL_{d}$, $\cM_{1}, \dots, \cM_{d}$ and
   $\cM_{*1}, \dots, \cM_{*d}$ so that properties \eqref{Omegadecom},
   \eqref{CDP}, \eqref{slim} and \eqref{slim*} are satisfied for any
   shift-invariant sublattice $\Omega$.  In this situation there are
   natural unitary identification maps
   $$ i_{\cH_{k}} \colon \cH_{k} \to \cL_{k}, \qquad
   i_{\cE} \colon \cE \to \cF, \qquad i_{\cE_{*}} \colon \cE_{*} \to
   \cF_{*}
   $$
   and the multievolution $\cU = (\cU_{1}, \dots, \cU_{d})$ on $\cK$
   is connected with the GR-unitary colligation $U = \sbm{ [A_{ij}] &
   [B_{i}] \\ [C_{j}] & D}$ via equation \eqref{UfromcalU'}.  As the
   unitary colligation $U$ uniquely determines the multievolution
   scattering system ${\mathfrak S}(\Sigma(U))$, properties of the
   GR-unitary colligation translate to properties of the scattering
   system ${\mathfrak S}(\Sigma(U))$.  Note that the last statement
   of Theorem \ref{T:BabyBear} is an instance of the reverse
   direction:  the property of {\em minimal} for the scattering system
   ${\mathfrak S}(\Sigma(U))$ is translated to the property of {\em
   scattering-minimal} for the colligation $U$.  The next proposition
does such
   an analysis for the colligation properties mentioned above.  Here
   we use the notation $\cK_{\text{scat-min}}$ to denote the ambient
   space for the minimal part of the scattering system ${\mathfrak
   S}(\Sigma(U))$, namely in the notation of Definition
   \ref{D:minimal-scat},
   $$ \cK_{\text{scat-min}} = \overline{\operatorname{span}}\,  [
\widetilde
   \cW + \widetilde \cW_{*}].
   $$

   \begin{proposition}  \label{P:col-scat}  Let $U$ be a GR-unitary
       colligation embedded in the multievolution scattering system
       ${\mathfrak S}(\Sigma(U))$ as described in Section
       \ref{S:scat-col}.  Then:
       \begin{enumerate}
       \item $U$ is scattering minimal if and only if
       $(\cK_{\text{scat-min}})^{\perp} = \{0\}$, or,
       equivalently, for any
       (or, equivalently, for some)
       shift-invariant sublattice $\Omega$, we have
     $$ \cV^{\Omega} \cap (\cK_{\text{scat-min}})^{\perp} =
     \{0\}.
     $$

     \item $U$ is strictly closely connected if and only if
     $$
     \cL \cap (\cK_{\text{scat-min}})^{\perp} = \{0\}.
     $$

     \item  $U$ is shifted strictly closely connected if and only if
     the space $\cL':=\cU_{1}  \cL_{1} \oplus \cdots \oplus \cU_{d}
     \cL_{d}$ has the property
     $$
      \cL' \cap (\cK_{\text{scat-min}})^{\perp} = \{0\}.
      $$

      \item  $U$ is closely connected if and only if
           $\cV^{\Omega} = \cV^{\Omega}_{cc}$ where
     \begin{align*}
    &    \cV^{\Omega}_{cc} = \bigoplus_{k=1}^{d}\
    \bigoplus_{n \in \partial_{k}\Omega_{\text{fin}}} z^{n} \cL_{k,cc}
      \oplus
       \left( \bigoplus_{k=1}^{d} \ \bigoplus_{n' \colon \ell_{n',k}
\in
          \partial_{k}\Omega_{+\infty}} ( \widehat{\mathcal
          U}^{k})^{\widehat{n'}^{k}}
          {\mathcal M}_{cc,k} \right)   \notag \\
          & \qquad \oplus
        \left( \bigoplus_{k=1}^{d}\
          \bigoplus_{n''\colon \ell_{n'',k} \in
          \partial_{k} \Omega_{-\infty}} (\widehat{\mathcal
          U}^{k})^{\widehat{n''}^{k}}
          {\mathcal M}_{*cc,k}  \right)
\end{align*}
     where $\cL_{cc,k} \subset \cL_{k}$ and $\cL_{cc}: = \cL_{cc,1}
     \oplus \cdots \oplus \cL_{cc,d}$ is given by
     \begin{itemize}
     \item[] $\cL_{cc}$ is the smallest subspace of $\cL$ which
     contains $P_{\cL_{j}} \cU_{j}^{*} \cF$ and $P_{\cL_{j}} \cF_{*}$
for each
     $j = 1, \dots, d$ and which is invariant under $P_{\cL_{j}}
     \cU_{j}^{*}P_{\cL_{k}}$ and $P_{\cL_{k}} \cU_{nj} P_{\cL_{j}}$
     for each pair of indices $j,k = 1, \dots, d$,
     \end{itemize}
  while $\cM_{cc,k}$ and $\cM_{*cc,k}$ are given by
  $$
  P_{\cM_{cc,k}} = \operatorname{strong\ limit}_{m \to \infty}
  P_{\cU_{k}^{m} \cL_{cc,k}}.
  $$
  Equivalently, $U$ is closely connected if and only if
     $$
       \cL = \cL_{cc}.
     $$
     In addition, $\cL_{cc} = i_{\cH}(\cH_{cc})$ where $\cH_{cc}$
     is given by the first line of \eqref{minimalsubspaces}.
   \end{enumerate}
       \end{proposition}

       \begin{proof}
       The first assertion follows from the first of the general
       identities \eqref{simplify} given below.

       As for the second assertion, note that the map
       $\Pi^{dBR}_{U}$, after using the identification $i_{\cH}
       \colon \cH \to \cL$ between $\cH $ and $\cL$, can be
       viewed as the restriction (to the state space $\cL$)
       of the de Branges--Rovnyak model
       map $\Pi_{dBR} \colon \cK \to \cK^{S}_{dBR}$ to the de
       Branges-Rovnyak-model multievolution scattering system
       given in \eqref{PidBR}.  As discussed in Section
       \ref{S:scatmodels} above, we know that $\Pi_{dBR}$ is a
       partial isometry with initial space equal to
       $\cK_{scat-min}$ and hence kernel equal to
       $(\cK_{scat-min})^{\perp}$.  Hence the condition that $U$
       is strictly closely-connected translates to the assertion
       that $\Pi_{dBR}|_{\cL}$ has trivial kernel, i.e., to
       $ \cL \cap (\cK_{\text{scat-min}})^{\perp} = \{0\}$.

       In a similar way, the map $\Pi^{dBR'}_{U}$, after appropriate
       identifications, amounts to the restriction of $\Pi_{dBR}$
       to $\cU_{1} \cL_{1} \oplus \cdots \oplus \cU_{d}
       \cL_{d}$.  The third assertion now follows in the same way
       as the second assertion.

       The fourth assertion follows immediately from the
    definitions and from the connection \eqref{UfromcalU'}
    between $U$ and $\cU = (\cU_{1}, \dots, \cU_{d})$.

       \end{proof}

We now obtain the following relations among these various notions
as a corollary to Proposition \ref{P:col-scat}.

\begin{corollary}  \label{C:implications}
   Let $U$ be a GR-unitary colligation.  Then:
   \begin{enumerate}
    \item
     $U$  scattering minimal  $\Rightarrow$  $U$
    strictly closely
    connected $\Rightarrow$  $U$  closely connected.
    \item
    $U$ scattering minimal  $\Rightarrow$ $U$
    shifted strictly closely
    connected  $Rightarrow$ $U$  closely connected.
    \end{enumerate}
    \end{corollary}

    \begin{proof}
    If $U$ is scattering minimal, by definition
    $(\cK_{\text{scat-min}}))^{\perp} = \{0\}$; hence $U$ is also
    strictly closely connected and shifted strictly closely
    connected by the criteria given in statements (2) and (3) of
    Proposition \ref{P:col-scat}.

    To verify the second assertions in (1) and (2), we note that
    $\cL_{cc \perp}: = \cL \ominus \cL_{cc}$ decomposes as
    $$ \cL_{cc \perp} = \cL_{cc \perp, 1} \oplus \cdots \oplus
    \cL_{cc \perp, d}
    $$
    with $\cL_{cc \perp, k} \subset \cL_{cc \perp} \cap \cL_{k}$,
reducing for
    $\cU_{k}$ and wandering for $\widehat \cU^{k} = \{ \cU_{j}
    \colon j \ne k \}$.  One can then check that
    \begin{equation}  \label{cc-scatmin}
    \cV^{\Omega} \ominus \cV^{\Omega}_{cc}= \bigoplus_{k=1}^{d}
\bigoplus_{n \in
    {\mathbb Z}^{d}}( \widehat \cU^{k})^{\widehat n^{k}}
    \cL_{cc\perp,k} \subset (\cK_{\text{scat-min}})^{\perp};
    \end{equation}
 indeed, a nice exercise is to determine how the orthogonal
decomposition in
 \eqref{cc-scatmin} fits with the orthogonal decomposition of
 $\cV^{\Omega}$ itself in \eqref{Omegadecom}.
 In particular, the condition $\cL_{cc \perp} \ne \{0\}$ forces
    also $\cL \cap (\cK_{\text{scat-min}})^{\perp} \ne 0$ and
    $\cL' \cap (\cK_{\text{scat-min}})^{\perp} \ne \{0\}$.
    Reading off from the criteria (2), (3) and (4) in Proposition
    \ref{P:col-scat}, we see that $U$ not closely connected
    forces both $U$ not strictly closely connected and $U$ not
    shifted strictly closely connected.  This completes the proof
    of the second implications in both (1) and (2) of the
    corollary.

 \end{proof}

    We now give a couple of examples to show that the converse of
    each implication in Corollary \ref{C:implications} in general
fails.

    \begin{example} \label{E:1} {\em The 2-variable Schur--Agler class
    function $S(z_{1}, z_{2}) = z_{1} z_{2}$ has a closely connected
    GR-unitary realization $U$ which is not strictly closely
    connected and not shifted strictly closely connected.}  This
    example is based on Example 3.8 in \cite{K-JOT2000}. We take $\cH
=
    \bbC^{4}$, $d=2$, $\cE = \cE_{*} = \bbC$ with
    $$ P_{1} = \begin{bmatrix} 1 &0 & 0 & 0 \\ 0 & 0 & 0 & 0 \\
                0 & 0 & \frac{1}{2} & \frac{1}{2} \\
                0 & 0 & \frac{1}{2} & \frac{1}{2} \end{bmatrix},
                \qquad
     P_{2} = \begin{bmatrix}
    0 & 0 & 0 & 0 \\ 0 & 1 & 0 & 0 \\
    0 & 0 & \frac{1}{2} & -\frac{1}{2} \\
    0 & 0 & -\frac{1}{2} & \frac{1}{2} \end{bmatrix}
    $$
    and hence $Z(z) : = z_{1}P_{1} + z_{2} P_{2}$ is given by
   $$
   Z(z) = \begin{bmatrix}
     z_{1} & 0 & 0 & 0 \\
     0    & z_{2} & 0 & 0 \\
     0 &   0    &\frac{z_{1}+z_{2}}{2} &\frac{z_{1} - z_{2}}{2} \\
     0 &   0 &  \frac{ z_{1}-z_{2} }{2} & \frac{ z_{1}+ z_{2} }{2}
     \end{bmatrix}
   $$
   We take $U = \sbm{ A & B \\ C & D }$ where
    $$ A = \begin{bmatrix} 0 & 0 & -\frac{1}{\sqrt{2}} & 0 \\
        0 & 0 & \frac{1}{\sqrt{2}} & 0 \\
       -\frac{1}{\sqrt{2}} & \frac{1}{\sqrt{2}} & 0 & 0 \\
       \frac{1}{\sqrt{2}} & \frac{1}{\sqrt{2}} & 0 & 0 \end{bmatrix},
     \qquad
     B = \begin{bmatrix} \frac{1}{\sqrt{2}} \\ \frac{1}{\sqrt{2}} \\ 0
     \\ 0 \end{bmatrix}, \qquad
     C = \begin{bmatrix} 0 & 0 & 0 & 1 \end{bmatrix}, \qquad D = 0.
    $$
    We check that $U$ is closely connected as follows.  Note that
    \begin{align*}
    & \operatorname{im} P_{1}B = \begin{bmatrix} \bbC \\ 0 \\ 0 \\ 0
    \end{bmatrix},  \qquad
    \operatorname{im} P_{2}B = \begin{bmatrix} 0 \\ \bbC \\ 0 \\ 0
    \end{bmatrix}, \\
     & \operatorname{im} A P_{1}B = \operatorname{im} A|_{\bbC \oplus
0
    \oplus 0 \oplus 0}
    = \begin{bmatrix} 0 \\ 0 \\ -\frac{1}{\sqrt{2}}
    \\ \frac{1}{\sqrt{2}} \end{bmatrix} \bbC,  \qquad
    \operatorname{im} C^{*} = \begin{bmatrix} 0 \\ 0 \\ 0 \\ \bbC
    \end{bmatrix}.
    \end{align*}
    As the span of these images is already the whole space
$\bbC^{4}$, we
    see that $U$ is indeed closely connected.

    We check that $U$ is not strictly closely connected as follows.
    First note that $U$ being strictly closely connected can be
    equivalently expressed as
    \begin{equation} \label{SCC'}
    \bigcap_{n \ge 0} \ker C (Z(z) A)^{n} \cap
    \bigcap_{n \ge 0} \ker B^{*}Z_{\diag}(z)^{-1} ( A^{*}
    Z(z)^{-1})^{n} = \{0\}
     \end{equation}
     (where we set $Z(z) = z_{1}P_{1} + z_{2} P_{2}$ and
     where the operators involved are considered as acting on constant
     vectors in $\cH$).
    For the case here, we compute
    \begin{align*}
    & Z(z) A = \begin{bmatrix} 0 & 0 & -\frac{1}{\sqrt{2}}z_{1} & 0 \\
    0 & 0 & \frac{1}{\sqrt{2}}z_{2} & 0  \\
    -\frac{z_{2}}{\sqrt{2}} & \frac{z_{1}}{\sqrt{2}} & 0 & 0 \\
    \frac{ z_{2}}{\sqrt{2}} & \frac{z_{1}}{\sqrt{2}} & 0 & 0
    \end{bmatrix}, \\
    & C Z(z) A = \begin{bmatrix} \frac{z_{2}}{\sqrt{2}} &
    \frac{z_{1}}{\sqrt{2}} & 0 & 0 \end{bmatrix}, \qquad
    C (Z(z)A)^{2} = \begin{bmatrix} 0 & 0 & 0 & 0 \end{bmatrix}.
    \end{align*}
    We conclude that
    \begin{equation} \label{intersection1}
     \bigcap_{n \ge 0} \ker C(Z(z) A)^{n}
    = \ker C \cap \ker C Z(z) A = \begin{bmatrix} 0 \\ 0 \\ \bbC
    \\ 0 \end{bmatrix}.
    \end{equation}
    Similarly, we compute
    $$
    Z(z) B = \begin{bmatrix} \frac{z_{1}}{\sqrt{2}} \\
    \frac{z_{2}}{\sqrt{2}} \\ 0 \\ 0 \end{bmatrix}, \,
    Z(z) A Z(z) B = \begin{bmatrix} 0 \\ 0 \\ 0 \\
    z_{1}z_{2}\end{bmatrix}, \, (Z(z) A)^{2} Z(z)B
    = \begin{bmatrix} 0 \\ 0 \\ 0 \\ 0 \end{bmatrix}.
    $$
    Hence
    \begin{align}
   & \bigcap_{n\ge 0} \ker B^{*}Z(z)^{-1} ( A^{*}
    Z(z)^{-1})^{n} \notag \\
    & \qquad   = \ker B^{*} Z(z)^{-1} \cap \ker B^{*} Z(z)^{-1}
     A^{*} Z(z)^{-1}
     \notag \\
     & \qquad  =
    \ker \begin{bmatrix} \frac{z_{1}^{-1}}{\sqrt{2}} &
    \frac{z_{2}^{-1}}{\sqrt{2}} & 0 & 0 \end{bmatrix}
    \cap \ker \begin{bmatrix} 0 & 0 & 0 & z_{1}^{-1} z_{2}^{-1}
    \end{bmatrix}
     = \begin{bmatrix} 0
    \\ 0 \\ \bbC \\ 0 \end{bmatrix}.
    \label{intersection2}
    \end{align}
    Intersecting \eqref{intersection1} and \eqref{intersection2} gives
    $$
    \bigcap_{n \ge 0} \ker C (Z(z) A)^{n} \cap
    \bigcap_{n \ge 0} \ker B^{*}Z(z)^{-1} ( A^{*}
    Z(z)^{-1})^{n} = \begin{bmatrix} 0 \\ 0 \\ \bbC \\ 0
    \end{bmatrix}
    $$
    and we conclude from the criterion \eqref{SCC'} that $U$ for this
    example is not strictly closely connected as claimed.

    To show that $U$ is not shifted strictly closely connected, it
    suffices to show that
    \begin{equation*}
    \bigcap_{n \ge 0} \ker C (Z(z)A)^{n} Z(z) \cap \bigcap_{n \ge 0}
    \ker B^{*} (Z(z) A^{*})^{n} = \{0\}.
    \end{equation*}
    Straightforward computations show that
    \begin{align*}
    & CZ(z) = \begin{bmatrix} 0 & 0 & \frac{z_{1}-z_{2}}{2} &
\frac{z_{1} +
    z_{2}}{2} \end{bmatrix}, \quad
     CZ(z)A Z(z) = \begin{bmatrix} z_{1}z_{2}/\sqrt{2} & z_{1}z_{2}/
    \sqrt{2} & 0 & 0 \end{bmatrix}, \\
    & C(Z(z)A)^{n}Z(z) = \begin{bmatrix} 0 & 0 & 0 & 0 \end{bmatrix}
    \text{ for } n \ge 2
    \end{align*}
    and we conclude that
    $$
    \bigcap_{n\ge0} \ker C (Z(z)A)^{n} Z(z) = \operatorname{span}
    \left\{ \left[ \begin{smallmatrix} 1\\ -1 \\ 0 \\ 0
\end{smallmatrix} \right] \right\}.
$$
Similarly, from
$$
A Z(z) B = \left[ \begin{smallmatrix} 0 \\ 0 \\ (-z_{1} +
z_{2})/\sqrt{2} \\
(z_{1} + z_{2})/ \sqrt{2} \end{smallmatrix} \right], \quad (A
Z(z))^{n} B = \left[ \begin{smallmatrix} 0 \\ 0 \\ 0 \\ 0
\end{smallmatrix} \right] \text{ for } n \ge 2,
$$
we see that
$$
 \bigcap_{n  \ge 0} \ker B^{*} (Z(z)A^{*})^{n} = \operatorname{span}
 \left\{ \left[ \begin{smallmatrix} 1 \\ -1 \\ 0 \\ 0
\end{smallmatrix} \right] \right\}
$$
and it follows that $U$ also is not shifted strictly closely
connected.

Finally, we compute
\begin{align*}
    S(z_{1},z_{2}): = &  D + C(I - Z(z)A)^{-1} Z(z) B \\
      = & D+\sum_{n=0}^{\infty} C (Z(z)A)^{n}Z(z) B  \\
      & = 0 + CZ(z)B + C Z(z) AZ(z) B \\
      = & 0 + 0 + \begin{bmatrix} \frac{z_{2}}{\sqrt{2}} &
      \frac{z_{1}}{\sqrt{2}} \end{bmatrix} \left[ \begin{smallmatrix}
      \frac{z_1}{\sqrt{2}} \\ \frac{z_2}{\sqrt{2}} \end{smallmatrix}
      \right] \\
      = & z_{1} z_{2}.
 \end{align*}

    \end{example}

    \begin{example} \label{E:2} {\em There exists a
    GR-unitary colligation matrix $U$  which is both strictly closely
    connected and shifted strictly closely connected but not
scattering minimal.}
    We take $\cH = {\mathbb C}^{8}$, $\cE = \cE_{*} = {\mathbb C}^{3}$
    with
    \begin{align*}
      A=\sbm{
      0 & 0 & 0 & \frac{1}{\sqrt{6}} & \frac{1}{\sqrt{6}} & 0 & 0 &
0\\
      0 & 0 & 0 & 0 & \sqrt{\frac{2}{3}} &  0 & 0 & 0\\
      0 & 0 & 0 & \sqrt{\frac{2}{3}} & 0 &  0 & 0 & 0\\
      \frac{1}{\sqrt{6}} & 0 & \sqrt{\frac{2}{3}} & 0 & 0 &  0 & 0 &
0\\
      \frac{1}{\sqrt{6}} & \sqrt{\frac{2}{3}}  & 0 & 0 & 0 &  0 & 0 &
0\\
      0 & 0 & 0 & \frac{1}{\sqrt{6}} & -\frac{1}{\sqrt{6}}  &  0 & 0
& 0\\
      0 & -\frac{1}{\sqrt{6}} & \frac{1}{\sqrt{6}} & 0 & 0 & 0 & 0 &
0\\
      -\sqrt{\frac{2}{3}} & \frac{1}{\sqrt{6}} & \frac{1}{\sqrt{6}} &
0 & 0
      &  0 & 0 & 0 }, \qquad &
      B= \sbm{
           0 & 0 & -\sqrt{\frac{2}{3}}\\
           0 & -\frac{1}{\sqrt{6}} & \frac{1}{\sqrt{6}}\\
           0 & \frac{1}{\sqrt{6}} & \frac{1}{\sqrt{6}}\\
           \frac{1}{\sqrt{6}} & 0 & 0\\
           -\frac{1}{\sqrt{6}} & 0 & 0\\
          0 & -\sqrt{\frac{2}{3}} & 0\\
           -\sqrt{\frac{2}{3}} & 0 & 0\\
           0 & 0 & 0 }, \\
       C=\begin{bmatrix}
    0 & 0 & 0 & 0 & 0 & 1 & 0 & 0\\
    0 & 0 & 0 & 0 & 0 &  0 & 1 & 0 \\
    0 & 0 & 0 & 0 & 0 &  0 & 0 & 1 \end{bmatrix}, \qquad &
    D= \begin{bmatrix}
           0 & 0 & 0\\
           0 & 0 & 0\\
           0 & 0 & 0 \end{bmatrix}, \\
 P_1  = \sbm{
    \frac{5}{6} & 0 & -\frac{1}{3} & 0 & 0 &  -\frac{1}{6} & 0 & 0\\
    0 & 0 & 0 & 0 & 0                  &  0 & 0 & 0
\\
    -\frac{1}{3} & 0 & \frac{1}{3} & 0 & 0 & -\frac{1}{3} & 0 & 0\\
    0 & 0 & 0 & 1 & 0  & 0 & 0 & 0\\
    0 & 0  & 0 & 0 & \frac{2}{3} &  0 & -\frac{1}{3} & \frac{1}{3}\\
    -\frac{1}{6} & 0 & -\frac{1}{3} & 0 & 0  & \frac{5}{6} & 0 & 0\\
    0 & 0 & 0 & 0 & -\frac{1}{3} &  0 & \frac{1}{6} & -\frac{1}{6}\\
    0 & 0 & 0 & 0 & \frac{1}{3} &  0 & -\frac{1}{6} & \frac{1}{6}
    }, \qquad &
   P_2  = \sbm{
    \frac{1}{6} & 0 & \frac{1}{3} & 0 & 0 &  \frac{1}{6} & 0 & 0\\
    0 & 1 & 0 & 0 & 0                   & 0 & 0 & 0
\\
    \frac{1}{3} & 0 & \frac{2}{3} & 0 & 0  & \frac{1}{3} & 0 & 0\\
    0 & 0 & 0 & 0 & 0 &  0 & 0 & 0\\
    0 & 0  & 0 & 0 & \frac{1}{3} &  0 & \frac{1}{3} & -\frac{1}{3}\\
    \frac{1}{6} & 0 & \frac{1}{3} & 0 & 0 & \frac{1}{6} & 0 & 0\\
    0 & 0 & 0 & 0 & \frac{1}{3} &  0 & \frac{5}{6} & \frac{1}{6}\\
    0 & 0 & 0 & 0 & -\frac{1}{3}  & 0 & \frac{1}{6} & \frac{5}{6}
    },
 \end{align*}
   and we set $Z(z) = z_{1} P_{1} + z_{2} P_{2}$.
It is straightforward to see that $P_{1}$ and $P_{2}$ are
complementary orthogonal projections on ${\mathbb C}^{8}$ and that
$U = \sbm{A & B \\ C & D}$ is a unitary matrix. Hence the
associated system $\Sigma(U, P_{1}, P_{2})$ is a GR-conservative
linear system. Since
$$
P_1B= \sbm{
 0 & 0 & -\sqrt{\frac{2}{3}}\\
 0 & 0 & 0                  \\
 0 & \frac{1}{\sqrt{6}} & \frac{1}{\sqrt{6}}\\
 \frac{1}{\sqrt{6}} & 0 & 0\\
 0 & 0 & 0\\
 0 & -\sqrt{\frac{2}{3}} & 0\\
 0 & 0 & 0\\
 0 & 0 & 0
},\qquad  P_2B= \sbm{
 0 & 0 & 0\\
 0 & -\frac{1}{\sqrt{6}} & \frac{1}{\sqrt{6}}\\
 0 & 0 & 0\\
 0 & 0 & 0\\
 -\frac{1}{\sqrt{6}} & 0 & 0\\
 0 & 0 & 0\\
 -\sqrt{\frac{2}{3}} & 0 & 0\\
 0 & 0 & 0},
 $$
and $C^*\mathbb{C}^3=\{ 0\}\oplus\mathbb{C}^3$, we have
$$\mathbb{C}^8=P_1B\mathbb{C}^3+P_2B\mathbb{C}^3+C^*\mathbb{C}^3\subset\mathcal{H}_{scc}
\subset\mathcal{H}=\mathbb{C}^8,$$ thus $\Sigma(U, P_{1}, P_{2})$
is strictly closely connected.  One can also verify that
$$
{\mathbb C}^{8} = B {\mathbb C}^{3} + P_{1} C^{*} {\mathbb C}^{3}
+ P_{2} C^{*} {\mathbb C}^{3} \subset \cH_{sscc} \subset \cH =
{\mathbb C}^{8}
$$
and hence $U$ is also shifted strictly closely connected.

Consider now the subspace $\mathcal{H}_0$ in $L^2_{\mathbb{C}^8}$
consisting of vectors of the form
\begin{equation}  \label{x}
x(z_1,z_2)=\left[\begin{array}{c} \left[\begin{array}{c} z_1\\
z_2\\ z_1
\end{array}\right]f(z_1,z_2)\\
\left[\begin{array}{c} z_1\\ z_2 \end{array}\right]g(z_1,z_2)\\
0\\
0\\
0
\end{array}\right],\quad f,g\in L_{\mathbb{C}}^2.
\end{equation}
 It is routine to check that $\mathcal{H}_0\subset\ker C,\
\mathcal{H}_0\subset\ker B^* Z(z)^{-1}$, and that $\mathcal{H}_0$
reduces $Z(z)A$.  Therefore
 $$
 \{0\}\neq\mathcal{H}_0\subset\left(\bigcap_{n\in \bbZ_{+}}\ker C
(Z(z) A)^n\right)\cap \left(\bigcap_{n\in \bbZ_{+}}\ker
B^*(Z(z))^{-1}(A^*(Z(z))^{-1})^n\right).
$$
If we choose $\Omega$ to be the so-called {\em balanced}
shift-invariant sublattice
$$
  \Omega^{B} = \{ (n_{1}, n_{2} ) \in \bbZ^{2} \colon n_{1} + n_{2}
  \ge  0 \}
$$
then the infinite boundary of $\Omega^{B}$ is empty
($\partial_{k}\Omega^{B}_{\pm \infty} = \emptyset$ for $k = 1,2$)
while the finite boundary is such that
$$
\left( H^{2}(\partial_{1}\Omega^{B}_{fin}, \operatorname{im}
P_{1}) \right)  \oplus
 \left( H^{2}(\partial_{2} \Omega^{B}_{fin}, \operatorname{im}
 P_{2})\right) \cong H^{2}(\Xi, {\mathbb C}^{8})
$$
where we set
$$
  \Xi = \{ (n_{1}, n_{2}) \in \bbZ^{2} \colon n_{1} + n_{2} = 0\}.
$$
Hence, if we take $f(z_{1},z_{2}) =  z_{2}^{-1}
f_{0}(z_{1},z_{2})$ and $g(z_{1}, z_{2}) = z_{2}^{-1} g_{0}(z_{1},
z_{2})$ where $f_{0}$ and $g_{0}$ are in $H^{2}(\Xi)$ and not both
zero, then the resulting function $x(z_{1},z_{2})$ in  \eqref{x}
is the $Z$-transform of a nonzero element  of $\bigoplus_{k=1}^{2}
\ell^{2}(\partial_{k}\Omega^{B}_{fin}, \operatorname{im} P_{k})$
which is in the kernel $\Pi_{U}^{dBR}$.  We conclude that
$\Sigma(U, P_{1}, P_{2})$ is not scattering minimal.
\end{example}

 Let us now introduce the subspaces
 \begin{equation}  \label{Lscc}
     \cL_{scc}: = i_{\cH}(\cH_{scc}), \qquad \cL_{sscc}: =
     i_{\cH}(\cH_{sscc}).
  \end{equation}
  Then it is not difficult to verify that
  \begin{equation*}
      \cL_{scc} = \cL \ominus (\cL \cap
      (\cK_{\text{scat-min}})^{\perp}), \qquad
      \cL_{sscc} = \cL' \ominus (\cL' \cap
      (\cK_{\text{scat-min}})^{\perp})
   \end{equation*}
    where  we use the notation
   $$
   \cL': = \cU_{1} \cL_{1} \oplus \cdots \oplus \cU_{d} \cL_{d}.
   $$
  From Proposition \ref{P:col-prop} we know that the condition
  $\cL_{scc} = \cL$ characterizes the GR-unitary colligation $U$
  being strictly closely connected and we know that $\cL_{sscc} =
  \cL'$  characterizes $U$ being shifted strictly closely connected.
We now
  give connections between these subspaces and the various formal
  reproducing kernel Hilbert spaces arising from the kernels in the
  associated augmented Agler decomposition.

  \begin{proposition}  \label{P:Aglerdecomcc}
      Suppose that $U$ is a GR-unitary colligation with transfer
      function $S(z) = D + C (I - Z_{\diag}(z)A)^{-1}
       Z_{\diag}(z) B$ and induced augmented Agler decomposition
       \eqref{augAglerdecom} with
       $$ K_{k}(z,w) = H_{k}(z) H_{k}(w)^{*} \text{ with } H_{k}(z) =
       \begin{bmatrix} C (I - Z_{\diag}(z) A)^{-1} \\
      z_{k}^{-1} B^{*} (I - Z(z)^{-1}A^{*} )^{-1} \end{bmatrix} P_{k}.
       $$
       Let the subspaces $\cL_{scc}$ and $\cL_{sscc}$ be defined by
       \eqref{Lscc}  and \eqref{minimalsubspaces}.  Set
       \begin{align*}
       & K(z,w):  = \sum_{k=1}^{d} K_{k}(z,w), \qquad
       K'_{k}(z,w) : = z_{k} K_{k}(z,w) w_{k}^{-1}, \\
       &
       K'(z,w) : =  \sum_{k=1}^{d} K'(z,w) = \sum_{k=}^{d} z_{k}
       K_{k}(z,w) w_{k}^{-1}.
       \end{align*}
        Then:
       \begin{enumerate}
       \item The mappings
       \begin{align}
     &      \Pi_{dBR}|_{\cL_{scc}} \colon \cL_{scc} \to \cH(K),
     \notag \\
      & \Pi_{dBR}|_{\cL_{k} \ominus (\cL_{k} \cap
           (\cK_{\text{scat-min}})^{\perp})} \colon \cL_{k} \ominus
(\cL_{k} \cap
           (\cK_{\text{scat-min}})^{\perp}) \to \cH(K_{k}),
           \notag \\
   & \Pi_{dBR}|_{\cL{sscc}} \colon \cL_{sscc} \to \cH(K'), \notag \\
    & \Pi_{dBR}|_{ \cU_{k} \cL_{k} \ominus
      (\cU_{k} \cL_{k} \cap  (\cK_{\text{scat-min}})^{\perp})}
      \colon  \notag \\
      & \qquad \qquad  \cU_{k} \cL_{k} \ominus  (\cU_{k} \cL_{k} \cap
(\cK_{\text{scat-min}})^{\perp})
      \to \cH(K'_{k})
      \label{unitarymaps}
\end{align}
are all unitary between the indicated spaces.

\item The overlapping space (see \eqref{overlapping} above)
  $$
 \boldsymbol{\mathcal L}: = \cL(K_{1}, \dots, K_{d})
 $$
 is trivial (equal to $\{0\}$) if and only if one (or, equivalently,
 both) of the subspaces
 $ \cL \cap (\cK_{\text{scat-min}})^{\perp}$ and $\cL_{scc}$ is/are
 invariant under each of the orthogonal projection operators
 $P_{\cL_{k}} \colon \cL \to \cL_{k}$.

 \item The overlapping space
 $$ \boldsymbol{\cL}' : = \cL(K'_{1}, \dots, K'_{d})
 $$
 is trivial if and only if one (or, equivalently, both) of the
 subspaces $\cL' \cap (\cK_{\text{scat-min}})^{\perp}$ and
 $\cL_{sscc}$ is/are invariant under each of the orthogonal
 projection operators $P_{\cU_{} \cL_{k}} \colon \cL' \to \cU_{k}
 \cL_{k}$ for $k = 1, \dots, d$.
  \end{enumerate}

    \end{proposition}

  \begin{proof} The unitarity of each of the four maps
      \eqref{unitarymaps} can be seen as a direct consequence of the
      last assertion in Theorem \ref{T:FRKHS}.

     By definition (see \eqref{overlapping}), the overlapping space
     $\boldsymbol{\cL}:  = \cL(K_{1}, \dots, K_{d})$ is trivial if
and only if
     \begin{equation}  \label{criterion1}
     h_{k} \in \cH(K_{k}),\, h_{1} + \cdots + h_{d} = 0
     \Rightarrow h_{k} = 0 \text{ for } k = 1, \dots, d.
     \end{equation}
     Using that $\cH(K_{k}) = \operatorname{im} \Pi_{dBR}|_{\cL_{k}}$
     and that $\cH(K) = \operatorname{im} \Pi_{dBR}|_{\cL}$ with
     lifted norm in each case, we see that  \eqref{criterion1}
     translates to
     $$ \ell_{k} \in \cL_{k},\, \ell_{1} + \cdots + \ell_{d} \in \cL
     \cap (\cK_{\text{scat-min}})^{\perp} \Rightarrow \ell_{k}
     \in \cL_{k} \cap (\cK_{\text{scat-min}})^{\perp} \text{ for } k =
     1, \dots, d.
     $$
     This in turn is just the assertion that $\cL \cap
     (\cK_{\text{scat-min}})^{\perp}$ is invariant under each
     $P_{\cL_{k}}$, i.e., for $Q =
     P_{ \cL \cap (\cK_{\text{scat-min}})^{\perp}}$, we have
     \begin{equation}  \label{know1'}
     P_{\cL_{k}} Q = Q P_{\cL_{k}} Q.
  \end{equation}
  Note that $Q = I - P$ where we set $P =
  P_{\cL_{scc}}$ (and all operators are considered as acting on
$\cL$).
  Substituting this expression into
  \eqref{know1'} gives
  $$
   P_{\cL_{k}}(I - P) = (I - P) P_{\cL_{k}} (I - P)
   $$
  which simplifies to
  $$
    P P_{\cL_{k}} = P P_{\cL_{k}} P.
  $$
  Taking adjoints finally gives
  $$
    P_{\cL_{k}} P = P P_{\cL_{k}}P,
  $$
  i.e., that $\cL_{scc}$ is invariant under each
  $P_{\cL_{k}}$ as well.  The steps are reversible, so we see that
  one of $\cL_{scc}$ and $\cL \ominus \cL_{scc} = \cL \cap
  (\cK_{\text{scat-min}})^{\perp}$ is invariant under each
  $P_{\cL_{k}}$ if and only if the other is as well.

  The third assertion in Proposition \ref{P:Aglerdecomcc} is proved
  in a completely analogous way by noting the triviality of the
  overlapping space $\boldsymbol{\cL}' : = \cL(K'_{1}, \dots,
  K'_{d})$ is characterized by
  \begin{align*}
  &  \ell'_{k} \in \cU_{k} \cL_{k}, \, \ell'_{1} + \cdots + \ell'_{d}
\in
  \cL' \cap (\cK_{\text{scat-min}})^{\perp}  \Rightarrow
  \ell'_{k}  \in   \cU_{k} \cL_{k} \cap
  (\cK_{\text{scat-minus}})^{\perp}  \\
  & \qquad \qquad  \text{ for } k = 1, \dots, d.
  \end{align*}

  \end{proof}

  A consequence of the next result is that we may drop the term
  ``shifted strictly closely connected'' for GR-unitary colligations
since it is equivalent to
  ``strictly closely connected''.

  \begin{theorem}  \label{T:scc&sscc}
      Suppose that we are given a GR-unitary colligation
      \eqref{colligation-free} with associated transfer function
      $S(z)$ and augmented Agler decomposition \eqref{augAglerdecom}
      as in Proposition \ref{P:Aglerdecomcc}.  Then
      the following are equivalent:
      \begin{enumerate}
      \item $\cH_{scc} =  \cH_{cc}$.

          \item $\cH_{sscc} = \cH_{cc}$.

      \item $\cH_{scc}$ and $\cH_{sscc}$ are invariant under $P_{k}$
      for $k = 1, \dots, d$.

      \item The augmented Agler decomposition is  strictly closely
      connected (see Definition \ref{D:augAglerdecom-scc}).
  \end{enumerate}
  \end{theorem}

  \begin{proof}
     We first note that (3) $\Leftrightarrow$ (4) amounts to the
     content of Proposition \ref{P:Aglerdecomcc}.  Thus it remains
only to
     prove (1) $\Leftrightarrow$ (2) $\Leftrightarrow$ (3).

      \textbf{Proof of (1) $\Leftrightarrow$ (2):}
      Without loss of generality we may assume that the unitary
          colligation $U = \sbm{A & B \\ C & D }$ is closely connected,
as
          cutting down to the closely connected part has no effect on the
          Agler decomposition or on the strictly closely connected and
          shifted strictly closely connected parts.
          Then we have the equivalences
          \begin{align*}
          & \cH_{scc} = \cH_{cc}  \Leftrightarrow
          \ker \Pi^{dBR}_{U} = \{0\},
          & \cH_{sscc} = \cH_{cc} \Leftrightarrow
          \ker \Pi^{dBR \prime}_{U} = \{0\}
          \end{align*}
where $\Pi^{dBR}_{U}$ and $\Pi^{dBR'}_{U}$ are as in
\eqref{PidBRU} and \eqref{PidBRU'}.
          As $\operatorname{im} \Pi^{dBR}_{U} = \cL^{S}_{dBR} \subset
          \cK^{S,\bbZ^{d}_{+}}_{\cV,dBR}$ is orthogonal to
$\cF^{S}_{dBR} =
          \sbm{S(z) \\ I } \cE$, we see that
          $$
          \begin{bmatrix} \ker \Pi^{dBR}_{U} \\ \{0\} \end{bmatrix} =
    \ker \begin{bmatrix} \Pi^{dBR}_{U} & \sbm{ S(z) \\ I} \end{bmatrix}
          \subset \begin{bmatrix} \cH \\ \cE \end{bmatrix}.
          $$
       and similarly
       $$
       \begin{bmatrix} \ker \Pi^{dBR \prime}_{U} \\ \{0\}
\end{bmatrix} =
       \ker \begin{bmatrix} \Pi^{dBR \prime}_{U} & \sbm{ I \\ S(z)^{*}}
       \end{bmatrix} \subset \begin{bmatrix} \cH \\ \cE_{*}
\end{bmatrix}.
       $$
       On the other hand a direct calculation shows that
       $$
    \begin{bmatrix}
        \Pi^{dBR}_{U} & \sbm{ S(z) \\ I } \end{bmatrix} =
        \begin{bmatrix} \Pi^{dBR \prime}_{U} &
        \sbm{I \\ S(z)^{*}} \end{bmatrix} U,
       $$
       i.e.,
       \begin{align*}
    & \begin{bmatrix}
       C(I - Z_{\diag}(z)A)^{-1}  & S(z)  \\
       B^{*}(I - Z_{\diag}(z)^{-1}A^{*})^{-1}  Z_{\diag}(z)^{-1}
     & I\end{bmatrix}  \\ & \qquad \qquad  =
      \begin{bmatrix} C(I - Z_{\diag}(z)A)^{-1} Z_{\diag}(z) & I \\
         B^{*}(I - Z_{\diag}(z)^{-1}A^{*})^{-1}  & S(z)^{*} \end{bmatrix}
      \cdot \begin{bmatrix} A & B \\ C & D \end{bmatrix}.
       \end{align*}
       In fact, this identity amounts to a coordinate form of the
property \eqref{CDP}
       $$
    \cL_{1} \oplus \cdots \oplus \cL_{d} \oplus \cF =
    \cU_{1} \cL_{1} \oplus \cdots \oplus \cU_{d} \cL_{d} \oplus \cF_{*}.
       $$
       We conclude that $\ker \begin{bmatrix} \Pi^{dBR}_{U} & \sbm{
S(z) \\ I }
       \end{bmatrix} = \{0\}$ if and only if $\ker \begin{bmatrix}
       \Pi^{dBR \prime}_{U} & \sbm{ I \\ S(z)^{*}} \end{bmatrix} =
\{0\}$ and the
       equivalence of (1) and (2) in Theorem \ref{T:scc&sscc} follows.

       \textbf{Proof of (1) or (2) $\Leftrightarrow$ (3):}
       Assume (1)  and hence also (2).  From (1) we see in particular
       that $\cH_{scc}$ is invariant under each $P_{k}$ while from (2)
       we see in particular that $\cH_{sscc}$ is invariant under each
       $P_{k}$ and (3) follows.

       Conversely, assume (3) or equivalently, {\em both $\cH
       \ominus \cH_{scc}$ and $\cH \ominus \cH_{sscc}$ are invariant
       under $P_{k}$ for $k = 1, \dots, d$}.  This version of
       condition (3) can be expressed as the simultaneous validity of
       the two conditions:
       \begin{align}
&  \begin{bmatrix} C (I - Z_{\diag}(z)A)^{-1} \\
B^{*}(I - Z_{\diag}(z)^{-1}A^{*})^{-1}Z_{\diag}(z)^{-1}
\end{bmatrix} h = 0
           \Rightarrow  \notag \\
        & \qquad   \begin{bmatrix} C(I - Z_{\diag}(z)A)^{-1} \\
    B^{*}(I - Z_{\diag}(z)^{-1}A^{*})^{-1}Z_{\diag}(z)^{-1}
\end{bmatrix}
    P_{k}h = 0  \text{ for each } k = 1, \dots, d,\label{know1''}
    \end{align}
    along with
      \begin{align}
 & \begin{bmatrix} C(I - Z_{\diag}(z)A)^{-1} Z_{\diag}(z) \\
       B^{*}(I - Z_{\diag}(z)^{-1}A^{*})^{-1} \end{bmatrix}
       h = 0
       \Rightarrow  \notag \\
  & \quad      \begin{bmatrix} C(I - Z_{\diag}(z)A)^{-1} Z_{\diag}(z)
\\
         B^{*}(I - Z_{\diag}(z)^{-1}A^{*})^{-1} \end{bmatrix}
         P_{k}h =  0 \text{ for each } k = 1, \dots, d. \label{know2''}
 \end{align}
Note that the second expression in \eqref{know1''} can be
reexpressed as
$$
  \begin{bmatrix} C(I - Z_{\diag}(z)A)^{-1} \\
    B^{*}(I - Z_{\diag}(z)^{-1}A^{*})^{-1}Z_{\diag}(z)^{-1}
\end{bmatrix}
    P_{k}h = \begin{bmatrix} C(I - Z_{\diag}(z)A)^{-1} \\
      z_{k}^{-1} B^{*}(I - Z_{\diag}(z)^{-1}A^{*})^{-1} \end{bmatrix}
       P_{k}h
  $$
and hence \eqref{know1''} can be rewritten as
\begin{align}
    &  \begin{bmatrix} C (I - Z_{\diag}(z)A)^{-1} \\
B^{*}(I - Z_{\diag}(z)^{-1}A^{*})^{-1}Z_{\diag}(z)^{-1}
\end{bmatrix} h = 0
           \Rightarrow  \notag \\
        & \qquad   \begin{bmatrix} C(I - Z_{\diag}(z)A)^{-1} \\
    B^{*}(I - Z_{\diag}(z)^{-1}A^{*})^{-1} \end{bmatrix}
    P_{k}h = 0 \text{ for each } k = 1, \dots, d. \label{know1}
    \end{align}
 Similarly, the condition \eqref{know2''} can be rewritten as
 \begin{align}
      & \begin{bmatrix} C(I - Z_{\diag}(z)A)^{-1} Z_{\diag}(z) \\
       B^{*}(I - Z_{\diag}(z)^{-1}A^{*})^{-1} \end{bmatrix}
       h = 0
       \Rightarrow  \notag \\
  & \quad      \begin{bmatrix} C(I - Z_{\diag}(z)A)^{-1}  \\
         B^{*}(I - Z_{\diag}(z)^{-1}A^{*})^{-1} \end{bmatrix}
         P_{k}h =  0 \text{ for each } k = 1, \dots, d. \label{know2}
 \end{align}

To prove (1), it suffices to show that $\cH \ominus \cH_{scc}$ is
invariant also under $A$ and $A^{*}$, or in detail
\begin{align}
    &  \begin{bmatrix} C (I - Z_{\diag}(z)A)^{-1} \\
    B^{*}(I - Z_{\diag}(z)^{-1}A^{*})^{-1}Z_{\diag}(z)^{-1}
\end{bmatrix} h = 0
    \Rightarrow \notag \\
    & \qquad
    \begin{bmatrix} C (I - Z_{\diag}(z)A)^{-1} \\
    B^{*}(I - Z_{\diag}(z)^{-1}A^{*})^{-1}Z_{\diag}(z)^{-1}
    \end{bmatrix} x = 0 \text{ for } x = Ah \text{ and } x = A^{*}h.
    \label{toshow1'}
 \end{align}
 By \eqref{know1} we can assume instead that
 \begin{equation}  \label{know3}
     \begin{bmatrix} C(I - Z_{\diag}(z)A)^{-1} \\
          B^{*}(I - Z_{\diag}(z)^{-1}A^{*})^{-1} \end{bmatrix}
          P_{k}h = 0 \text{ for } k = 1, \dots, d.
  \end{equation}
Note that
  \begin{equation}  \label{observe1}
  C(I - Z_{\diag}(z)A)^{-1} Z_{\diag}(z) A h = -Ch + C(I -
  Z_{\diag}(z)A)^{-1} h.
  \end{equation}
  From \eqref{know3} we have in particular that $C(I -
  Z_{\diag}(z)A)^{-1} h = 0$ and hence also the constant term $Ch$
  vanishes.  These observations combined with the identity
  \eqref{observe1} tell us that
  \begin{equation*}
      C(I - Z_{\diag}(z) A)^{-1} Z_{\diag}(z) A h = 0.
   \end{equation*}
   We now have the hypothesis for the implication \eqref{know2} with
   $Ah$ in place of $h$.  From \eqref{know2} we therefore conclude
   that $C(I - Z_{\diag}(z)A)^{-1} Ah = 0$ and
   the top component of \eqref{toshow1'} is verified for the case
   $x = Ah$.

   We next note that
   \begin{align*}
   & B^{*}(I - Z_{\diag}(z)^{-1}A^{*})^{-1} A h  =
      B^{*}Ah + B^{*}(I - Z_{\diag}(z)^{-1} A^{*})^{-1}
      Z_{\diag}(z)^{-1} A^{*}A h \notag \\
      & \qquad = -D^{*}C h + B^{*}(I - Z_{\diag}(z)^{-1}A^{*})^{-1}
      Z_{\diag}(z)^{-1} (I - C^{*}C) h \notag \\
      &\qquad  = -D^{*}C h - B^{*}(I - Z_{\diag}(z)^{-1} A^{*})^{-1}
      Z_{\diag}(z)^{-1} C^{*}Ch  \notag \\
      & \qquad = 0
    \end{align*}
    where we used the fact that $U$ is unitary and again
     the top component of \eqref{know3}.  We have
    now verified the bottom component of \eqref{toshow1'} for the
    case $x = Ah$.

    We next observe, again using that $U$ is unitary,
    \begin{align}
   & C(I - Z_{\diag}(z)A)^{-1} A^{*} h =
  C A^{*}h + C (I - Z_{\diag}(z)A)^{-1} Z_{\diag}(z) AA^{*}h \notag \\
  & \qquad =
  -DB^{*}h + C (I - Z_{\diag}(z) A)^{-1} Z_{\diag}(z) (I - B B^{*})h
  \notag \\
  & \qquad
  = [-D - C (I - Z_{\diag}(z) A)^{-1} Z_{diag}(z) B] B^{*}h + C (I -
  Z_{\diag}(z) A)^{-1} Z_{\diag}(z)h.
  \label{observe3}
  \end{align}
  From the bottom component of \eqref{know3} we read off that
  $B^{*}h = 0$ and hence the first term of \eqref{observe3} vanishes.
  The vanishing of the second term follows from the top component of
  \eqref{know3}.  In this way we have verified the validity of the top
  component of \eqref{toshow1'} for the case $x = A^{*}h$.

  Next observe that
  \begin{equation*}
    B^{*}(I - Z_{\diag}(z)^{-1}A^{*})^{-1} Z_{\diag}(z)^{-1}A^{*}h =
  -B^{*}h + B^{*} (I - Z_{\diag}(z)^{-1} A^{*})^{-1}h.
  \end{equation*}
  The vanishing of this quantity follows from the bottom component of
  \eqref{know3}.  We have now verified the bottom component of
  \eqref{toshow1'} for the case $x = A^{*}h$.  We now have completed
  the proof of Theorem \ref{T:scc&sscc}.
 \end{proof}

  \begin{remark} \label{R:Hscc-Hsscc}
      An open question is whether it can happen that
  $\cH_{scc}$ is invariant under the projections $P_{k}$ ($k = 1,
  \dots, d$) without also $\cH_{sscc}$ being invariant under $P_{k}$
  for each $k$. An
  equivalent version of the question is whether it can happen,
  for a given augmented Agler decomposition
  \eqref{augAglerdecom}, that the collection of subspaces
   $\{ \cH(K_{1}), \dots, \cH(K_{d}) \}$ has no overlapping while
   $\{z_{1}\cH(K_{1}), \dots, z_{d}\cH(K_{d})\}$ has nontrivial
   overlapping, or vice versa.  Note that we have defined the
   augmented Agler decomposition \eqref{augAglerdecom} to be {\em
   strictly closely connected} when both collections $\{\cH(K_{1}),
   \dots, \cH(K_{d})\}$ and $\{ z_{1} \cH(K_{d}), \dots,
   z_{d}\cH(K_{d})\}$ have trivial overlapping.
   Another
  equivalent version of the open question is whether $\cH_{scc}$
being invariant under
  each $P_{k}$ implies that $\cH_{scc} = \cH_{cc}$;  the work in the
  proof of Theorem \ref{T:scc&sscc} shows that $\cH_{scc} = \cH_{cc}$
  if we assume that both $\cH_{scc}$ and $\cH_{sscc}$ are invariant
  under each $P_{k}$.
  \end{remark}

  The analysis in Proposition \ref{P:Aglerdecomcc}  dissects the
interplay between the two
  decompositions $\cL = \cL_{1} \oplus \cdots \oplus \cL_{d}$ and the
  decomposition $\cK = \cK_{\text{scat-min}} \oplus
  (\cK_{\text{scat-min}})^{\perp}$ of the ambient space of the
  scattering system.   A similar analysis can be done for the
  decomposition \eqref{Omegadecom} of the scattering space
  $\cV^{\Omega}$ associated with any shift-invariant sublattice
  $\Omega$.  One distinguishing feature of this situation versus that
  in Proposition \ref{P:Aglerdecomcc} is that we always have that
  $(\cK_{\text{scat-min}})^{\perp} \subset \cV^{\Omega}$; indeed from
  $$
  \widetilde \cW_{*} + \widetilde \cW \supset \bigoplus_{n \in
  {\mathbb Z}^{d} \setminus \Omega} \cU^{n}\cF_{*} \oplus
  \bigoplus_{n \in \Omega} \cU^{n} \cF = (\cV^{\Omega})^{\perp},
  $$
  we see that
  $$
  (\cK_{\text{scat-min}})^{\perp}: = (\widetilde \cW_{*} + \widetilde
  \cW)^{\perp} \subset \cV^{\Omega}.
  $$
  Hence we  have the simplifications
  \begin{equation}  \label{simplify}
  \cV^{\Omega} \cap (\cK_{\text{scat-min}})^{\perp} =
  (\cK_{\text{scat-min}})^{\perp}, \qquad
  \cV^{\Omega} \ominus (\cV^{\Omega} \cap
  (\cK_{\text{scat-min}})^{\perp}) = \cV^{\Omega} \cap
  \cK_{\text{scat-min}}.
  \end{equation}
  The following summarizes the situation.

  \begin{proposition}  \label{P:scatspacecc}
      Suppose that $U$ is a GR-unitary colligation embedded in a
      multievolution scattering system as in Theorem \ref{T:BabyBear}
      with associated transfer function $S$ and augmented Agler
      decomposition \eqref{augAglerdecom}.
      \begin{enumerate}
      \item The mappings
      \begin{align*}
        &  \Pi_{dBR}|_{\cV^{\Omega}} \colon \cV^{\Omega} \cap
          \cK_{\text{scat-min}} \to \cH(K_{\cV^{S,\Omega}_{dBR}}), \\
       & \Pi_{dBR}|_{\chi} \colon \chi \ominus (\chi \cap
       (\cK_{\text{scat-min}})^{\perp}) \to \cH(K_{\chi^{S}_{dBR}})
       \end{align*}
       are all unitary between the indicated spaces. Here $\chi$
       is any one of the direct summand spaces
       \begin{align}
          & \chi = \cU^{n} \cL_{k} \text{ for } n \in \partial_{k}
          \Omega_{\text{fin}}, \label{DS1} \\
        &   \chi = (\widehat \cU^{k})^{\widehat{n'}^{k}} \cM_{k} \text{
          for } \ell_{n',k} \in \partial_{k}
          \Omega_{+\infty},  \label{DS2} \\
          & \chi = (\widehat \cU^{k})^{\widehat{n''}^{k}} \cM_{*k}
          \text{ for } \ell_{n'',k} \in \partial_{k} \Omega_{-\infty}
          \label{DS3}
        \end{align}
(with $k = 1, \dots, d$) in the orthogonal decomposition
\eqref{Omegadecom} with associated formal positive kernels
\begin{align*}
      & K_{\chi^{S}_{dBR}}(z,w) = z^{n}K_{k}(z,w) w^{-n} \text{ if }
\chi
      \text{ is as in \eqref{DS1},} \\
      & K_{\chi^{S}_{dBR}}(z,w) = (\widehat z^{k})^{\widehat{n'}^{k}}
      \sbm{ K^{11}_{-\infty,k}(z,w) & 0 \\ 0 & 0 }
    (\widehat w^{k})^{ -\widehat{n'}^{k} } \text{ if } \chi
      \text{ is as in \eqref{DS2}, and} \\
       & K_{\chi^{S}_{dBR}} = (\widehat z^{k})^{\widehat{n''}^{k}}
      \sbm{0 & 0 \\ 0 &  K^{11}_{+\infty,k}(z,w)  } (\widehat
      w^{k})^{- \widehat{n''}^{k} } \text{ if } \chi
      \text{ is as in \eqref{DS3}.}
  \end{align*}

 \item  Let
      $\boldsymbol{\cL}_{\cV^{\Omega}}$ be the overlapping space
      associated with the decomposition \eqref{Omegadecom} of
      $\cV^{\Omega}$.  Then the following are equivalent:
      \begin{enumerate}
      \item $\boldsymbol{\cL}_{\cV^{\Omega}}$ is trivial.

      \item $\cV^{\Omega} \cap (\cK_{\text{scat-min}})^{\perp} =
      (\cK_{\text{scat-min}})^{\perp}$
      is invariant under  $P_{\chi}$ for each of the subspaces $\chi$
described in
      \eqref{DS1}, \eqref{DS2}, \eqref{DS3}.

      \item $\cV^{\Omega}  \cap \cK_{\text{scat-min}}$ is
      invariant under $P_{\chi}$ for each of the subspaces $\chi$
      described in \eqref{DS1}, \eqref{DS2}, \eqref{DS3}.
      \end{enumerate}

  \item Suppose that $\partial_{k}\Omega_{\text{fin}} \ne  \emptyset $
  for each $k = 1, \dots, d$ and that
  $\boldsymbol{\cL}_{\cV^{\Omega}} = \{0\}$.  Then
  $(\cK_{\text{scat-min}})^{\perp}$ is invariant under $P_{\cL_{k}}$
  for $k = 1, \dots, d$.  Conversely, if
  $(\cK_{\text{scat-min}})^{\perp}$ is invariant under $P_{\cL_{k}}$
  for each $k = 1, \dots, d$ and $\Omega$ is any shift-invariant
  sublattice, then $\boldsymbol{\cL}_{\cV^{\Omega}} = \{0\}$.

  \item The augmented Agler decomposition $\{K_{k} \colon k = 1,
  \dots, d\}$ associated with $U$ is minimal (in the sense of
  Definition \ref{D:minAglerdecom}) if and only if
  $(\cK_{\text{scat-min}})^{\perp}$ is invariant under $P_{\cL_{k}}$
  for each $k = 1, \dots, d$.  In this case we have the identity
  \begin{equation}  \label{cc-scatmin'}
      \cV^{\Omega} \ominus \cV^{\Omega}_{cc} =
      (\cK_{\text{scat-min}})^{\perp}
   \end{equation}
   and $U$ is closely connected
  if and only if ${\mathfrak S}(\Sigma(U))$ is scattering-minimal.

\end{enumerate}
  \end{proposition}

  \begin{proof} Given the validity of the identities in
      \eqref{simplify}, the proofs of statements (1) and  (2) proceed
exactly
      as the proof of Proposition \ref{P:Aglerdecomcc} given above;
      we leave the precise details to the reader.

      Let us suppose now that $\Omega$ is a shift-invariant sublattice
      such that $\partial_{k}\Omega_{\text{fin}} \ne \emptyset$ for
      each $k = 1, \dots, d$ and that $\boldsymbol{\cL}_{\cV^{\Omega}}
      = \{0\}$.   Fix $k \in \{1, \dots, d\}$. By assumption we can
find at least one $n \in
      \partial_{k}\Omega_{\text{fin}}$.  Then $\chi =
      \cU^{n}\cL_{k}$ is one of the direct-summand spaces in
      the list \eqref{DS1}.  Since by assumption
      $\boldsymbol{\cL}_{\cV^{\Omega}} = \{0\}$, by the equivalence
      of (a) and (b) in statement (2) of the proposition, we know that
      $(\cK_{\text{scat-min}})^{\perp}$ is invariant under
      $P_{\cU^{n}\cL_{k}}$.  Since $P_{\cL_{k}} = \cU^{-n}
      P_{\cU^{n} \cL_{k}} \cU^{n}$ and
      $(\cK_{\text{scat-min}})^{\perp}$ is reducing for $\cU_{1},
      \dots, \cU_{d}$, it follows that
      $(\cK_{\text{scat-min}})^{\perp}$ is invariant for
      $P_{\cL_{k}}$ as well.  As $k \in \{1, \dots, d\}$ is
      arbitrary, the first part of statement (3) follows.

      Conversely, suppose that $(\cK_{\text{scat-min}})^{\perp}$ is
      invariant under $P_{\cL_{k}}$ for each $k = 1, \dots, d$
      and that  $\Omega$ is any shift-invariant  sublattice.  Since
      $(\cK_{\text{scat-min}})^{\perp}$ is reducing for $\cU_{1},
      \dots, \cU_{d}$ and $P_{\cU^{n} \cL_{k}} = \cU^{n} P_{\cL_{k}}
      \cU^{-n}$ for any $n \in {\mathbb Z}^{d}$, we see that
      $(\cK_{\text{scat-min}})^{\perp}$ is invariant under
      $P_{\cU^{n} \cL_{k}}$ for any $n \in {\mathbb Z}^{d}$, in
      particular, for $n \in \partial_{k}\Omega_{\text{fin}}$.
      From the characterization of $P_{\cM_{k}}$ via the strong limit
      \eqref{slim} in Theorem \ref{T:BabyBear}, it follows that
      $(\cK_{\text{scat-min}})^{\perp}$ is also invariant under
      $P_{\cM_{k}}$.  Since, for any $n \in {\mathbb Z}^{d}$,
      $$
      P_{(\widehat \cU^{k})^{\widehat{n'}^{k}} \cM_{k}} =
      (\widehat \cU^{k})^{\widehat{n'}^{k}} P_{\cM_{k}}
      (\widehat \cU^{k})^{-\widehat{n'}^{k}}
      $$
      and $(\cK_{\text{scat-min}})^{\perp}$ is reducing for
      $(\widehat \cU^{k})^{\widehat{n'}^{k}}$, it follows that
      $(\cK_{\text{scat-min}})^{\perp}$ is invariant for
     $P_{(\widehat \cU^{k})^{\widehat{n'}^{k}} \cM_{k}}$ as well for
     any $n' \in {\mathbb Z}^{d}$, in particular, for $n'$ such that
     $\ell_{n',k} \in \partial_{k}\Omega_{+\infty}$.  That
     $(\cK_{\text{scat-min}})^{\perp}$ is invariant for
    $ P_{(\widehat \cU^{k})^{\widehat{n''}^{k}} \cM_{*k}}$ for each
    $n''$ such that $\ell_{n'',k} \in \partial\Omega_{-\infty}$
    follows in a similar way by starting with \eqref{slim*} in place
    of \eqref{slim}.  We have now verified condition (b) in statement
    (2);  by statement (2) of the proposition it follows that
    $\boldsymbol{\cL}_{\cV^{\Omega}} = \{0\}$ as wanted. This
    completes the proof of statement (3) in the proposition.

    Suppose that the augmented Agler decomposition $\{K_{k} \colon
k=1,
    \dots, d\}$ associated with $U$ is minimal in the sense of
    Definition \ref{D:minAglerdecom}.  By statement (1) in Theorem
    \ref{T:minfuncmodel}, we know that
    $\boldsymbol{\cL}_{\cV^{\Omega}} = \{0\}$ for {\em any}
    shift-invariant sublattice $\Omega$; in particular we may choose
    $\Omega$ having nonempty finite boundary components.  Then from
    statement (3) of the proposition already proved, we see that
    $(\cK_{\text{scat-min}})^{\perp}$ is invariant under
    $P_{\cL_{k}}$ for each $k = 1, \dots, d$.  Conversely, if
    $(\cK_{\text{scat-min}})^{\perp}$ is invariant under each
    $P_{\cL_{k}}$, then $\boldsymbol{\cL}_{\cV^{\Omega}} = \{0\}$ by
   the converse side of statement (3) above, and both $\cL$ and
   $\cL'$ are trivial by statements (2) and (3) in Proposition
   \ref{P:Aglerdecomcc} from which it follows that both conditions in
   Definition \ref{D:minAglerdecom} are satisfied and $\{K_{k} \colon
   k = 1, \dots, d\}$ is minimal.

   As noted in \eqref{cc-scatmin}, the containment $\subset$ in
   \eqref{cc-scatmin'} always holds; we prove the reverse containment
   under the hypothesis that $(\cK_{\text{scat-min}})^{\perp}$ is
   invariant under $P_{\cL_{k}}$ for $k = 1, \dots, d$.
   Equivalently, the assumption is that $\cK_{\text{scat-min}}$ is
   invariant under each $P_{\cL_{k}}$.  From the formula for
   $\cL_{cc}$ in statement (4) of Proposition \ref{P:col-scat}, we
   see that $\cL_{cc}$ is then contained in $\cK_{\text{scat-min}}$.
It
   then follows from the way that $\cV^{\Omega}_{cc}$ is constructed
   that $\cV^{\Omega}_{cc} \subset \cK_{\text{scat-min}}$.  Taking
   orthogonal complements inside $\cV^{\Omega}$ then gives
   $\cV^{\Omega} \ominus \cV^{\Omega}_{cc} \supset
   (\cK_{\text{scat-min}})^{\perp}$ and \eqref{cc-scatmin'} follows.
   It then follows immediately that  $U$ closely connected
   ($\cV^{\Omega} = \cV^{\Omega}_{cc}$) if and only if ${\mathfrak
   S}(\Sigma(U))$ is scattering-minimal
   ($(\cK_{\text{scat-min}})^{\perp} = \{0\}$).
     \end{proof}

     As a corollary we obtain the following complement to Definition
     \ref{D:minAglerdecom}.

     \begin{corollary}  \label{C:minAglerdecom}
     Suppose that we are given an   augmented Agler decomposition
     \eqref{enlargedAglerdecom} for a Schur class formal power
     series $S \in \cL(\cE, \cE_{*})[[z^{\pm 1}]]$ and suppose that
     $\Omega$ is a shift-invariant sublattice such that each
     finite boundary component $\partial_{k}\Omega_{\rm fin}$
     ($k = 1, \dots, d$) is nonempty.  Assume that
     $\boldsymbol{\cL}_{\cV^{\Omega}} = \{0\}$.  Then also the
     overlapping spaces
     $\boldsymbol{\cL} = \cL(K_{1}, \dots, K_{d})$ and
$\boldsymbol{\cL}' =
     \cL(K'_{}, \dots, K'_{d})$ are trivial,
     i.e., for the case that all finite boundary components of $\Omega$
     are nonempty, condition (1) is a consequence of condition (2) in
     Definition \ref{D:minAglerdecom}.
     \end{corollary}

     \begin{proof} Simply combine statements (2) and (3) of
         Proposition \ref{P:Aglerdecomcc} with statement (3) of
         Proposition \ref{P:scatspacecc}.
      \end{proof}

  \begin{remark}  \label{R:E2}  We note that Example \ref{E:2} can be
      used to illustrate an additional point: {\em there exists a
      strictly closely connected Agler decomposition which is not
      minimal}, i.e., {\em condition (1) does not in general imply
      condition (2) in Definition \ref{D:minAglerdecom}.}  Indeed,
      consider the GR-unitary colligation presented in Example
      \ref{E:2} and its associated Agler decomposition $\{K_{k} \colon
      k=1, \dots, d\}$, say.
       As $U$ is both strictly closely
      connected and shifted strictly closely connected, the
      associated Agler decomposition is strictly closely connected.
      We also verified for this example that $U$ is closely
      connected.  Were it the case that the Agler decomposition were
      minimal, then as a consequence of statement (4) in Proposition
      \ref{P:scatspacecc} it would follow that $U$ is also
      scattering-minimal, contrary to the main point of the example.
     Of course it should also be possible to check directly that
certain
      subspaces in the list \eqref{DS1} for the balanced cut in this
      example have nontrivial intersection in order to verify that
      the augmented Agler decomposition for this example is not
      minimal.
  \end{remark}
  \begin{remark}\label{R:min-summary}
  For the reader's convenience, we summarize here the various notions
of minimality: for augmented Agler decompositions, for
Givone--Roesser unitary colligations, and for multievolution
scattering systems, and the interrelations among these notions.
  \begin{description}
    \item[I] Augmented Agler decompositions
    \begin{itemize}
      \item[I.1.] \emph{Strictly closely connected} --- see
Definition \ref{D:augAglerdecom-scc}.
      \item[I.2.] \emph{Minimal} --- see Definition
\ref{D:minAglerdecom}.
    \end{itemize}
    \item[II] GR unitary colligations
    \begin{itemize}
    \item[II.0.] \emph{Closely connected} --- see Definition
\ref{D:minimal}(1).
      \item[II.1a.] \emph{Strictly closely connected} --- see
Definition \ref{D:minimal}(2).
    \item[II.1b.] \emph{Shifted strictly closely connected} --- see
Definition \ref{D:minimal}(3).
      \item[II.2.] \emph{Kernel minimal}: the augmented Agler
decomposition induced by the colligation has property I.2.
      \item[II.3.] \emph{Scattering minimal} --- see
\eqref{col-scatmin}--\eqref{uyfromx}.
    \end{itemize}
        \item[III] Multievolution scattering systems
    \begin{itemize}
      \item[III.3.] \emph{Minimal} --- see Definition
\ref{D:minimal-scat}.
    \end{itemize}
    \end{description}
  First, we note that II.1a and II.1b are equivalent by Theorem
\ref{T:scc&sscc}. Second, we remark that every unitary colligation
can be replaced by its closely connected part by compressing the
state space $\mathcal{H}$ to $\mathcal{H}_{cc}$, which does not
affect the transfer function, the induced Agler decomposition, or
the minimal part of the associated scattering system. Assume now
that the GR unitary colligation is closely connected, i.e., II.0
holds. Then (1) by Theorem \ref{T:scc&sscc}, conditions II.1a and
II.1b are each equivalent to condition I.1 for the induced Agler
decomposition; (2) by Theorem \ref{T:BabyBear}, condition II.3 is
equivalent to condition III.3 for the associated scattering
system; (3) by Definition \ref{D:minAglerdecom}, I.2 implies I.1,
and thus II.2 implies both II.1a and II.1b; (4) by Proposition
\ref{P:scatspacecc}, II.2 implies II.3, and II.3 together with
either II.1a or II.1b imply II.2. Thus we conclude that, under
assumption II.0, we have II.3\&II.1a $\Leftrightarrow$ II.3\&II.1b
$\Leftrightarrow$ II.2.

  \end{remark}

  \section{Closely connected GR-unitary colligations compatible with
  given augmented Agler decomposition} \label{S:cc}

  By the results of \cite{Ag90,AgMcC,BT} we know that GR-unitary
  realizations of a given Schur-Agler-class function (viewed here as
  a formal power series) $S(z) \in \mathcal{S A}(\cE, \cE_{*})$ can
  be constructed from a given augmented (or even nonaugmented) Agler
  decomposition \eqref{augAglerdecom}. Conversely, a given GR-unitary
  realization for $S$ picks out a particular augmented Agler
  decomposition via \eqref{augKernel}.

  When this is the case (i.e., $U$ and $\{K_{k}\}$ are connected
via\emph{}
  \eqref{augKernel}),
   we shall say that the GR-unitary colligation $U = \left(\sbm{A & B
\\ C & D }, P_{1}, \dots,
  P_{d}\right)$  is {\em compatible} with
  the augmented Agler decomposition $\{K_{k} \colon k=1, \dots, d\}$
for $S$.
  We next present a
  description of the set of
   all GR-unitary colligations which
  realize $S(z)$ and are compatible with a given augmented Agler
  decomposition $K_{1}(z,w), \dots, K_{d}(z,w)$ for $S(z)$.
  The statement of the result requires a few preliminaries.

  We start with an augmented Agler decomposition for $S(z)$:
  $$
  \begin{bmatrix} I - S(z) S(w)^{*} & S(w) - S(z) \\
      S(z)^{*} - S(w)^{*} & S(z)^{*} S(w) - I \end{bmatrix} =
      \sum_{k=1}^{d} (1 - z_{k} w_{k}^{-1}) K_{k}(z,w).
  $$
  Let us introduce the notation
   \begin{align*}
       & K'_{k}(z,w) = z_{k}K_{k}(z,w) w_{k}^{-1}, \\
   & K_{\cF}(z,w) = \begin{bmatrix} S(z) \\ I \end{bmatrix}
   \begin{bmatrix} S(w)^{*} & I \end{bmatrix}, \qquad
    K_{\cF_{*}}(z,w) = \begin{bmatrix} I \\ S(z)^{*} \end{bmatrix}
    \begin{bmatrix} I & S(w) \end{bmatrix}
    \end{align*}
    so we may rewrite  \eqref{kernelCDP} as
   \begin{equation}  \label{boldK}
   \sum_{k=1}^{d} K_{k}(z,w) + K_{\cF}(z,w)  =
       \sum_{k=1}^{d}  K'_{k}(z,w)  + K_{\cF_{*}}(z,w) =: {\mathbf
       K}(z,w).
   \end{equation}
   As a consequence of the reproducing kernel property, this in turn
   means that, for
   each $e,e' \in \cE$ and $e_{*},e'_{*} \in \cE_{*}$,
   \begin{align}
    &   \sum_{k=1}^{d}\left\langle K_{k}(\cdot, w) \begin{bmatrix}
       e_{*} \\ e \end{bmatrix}, \, K_{k}(\cdot, z) \begin{bmatrix}
       e'_{*} \\ e' \end{bmatrix} \right \rangle +
      \left \langle K_{\cF}(\cdot, w) \begin{bmatrix} e_{*} \\ e
   \end{bmatrix}, \, K_{\cF}(\cdot, z) \begin{bmatrix} e'_{*} \\ e'
\end{bmatrix} \right\rangle  \notag \\
& \qquad =
\sum_{k=1}^{d} \left\langle K'_{k}(\cdot, w)\begin{bmatrix} e_{*} \\
e \end{bmatrix}, \, K'_{k}(\cdot, z) \begin{bmatrix} e'_{*} \\ e'
\end{bmatrix} \right\rangle + \left \langle K_{\cF_{*}}(\cdot, w)
\begin{bmatrix} e_{*} \\ e \end{bmatrix}, K_{\cF_{*}}(\cdot, z)
    \begin{bmatrix} e'_{*} \\ e' \end{bmatrix} \right \rangle
    \label{kernelCDP'}
\end{align}
where the inner products are of the form
\begin{align*}
& \cH(K_{k})[[w^{\pm 1}]] \times \cH(K_{k})[[ z^{\pm 1}]], \qquad
  \cH(K_{\cF})[[w^{\pm 1}]] \times \cH(K_{\cF})[[z^{\pm 1}]], \\
  & \cH(K'_{k})[[w^{\pm 1}]] \times \cH(K'_k)[[z^{\pm 1}]], \qquad
  \cH(K_{\cF_{*}})[[w^{\pm 1}]] \times \cH(K_{\cF_{*}})[[z^{\pm 1}]]
\end{align*}
with values in ${\mathbb C}[[z^{\pm 1}, w^{\pm 1}]]$.  As
\eqref{kernelCDP'} is an identity between formal power series in
${\mathbb C}[[z^{\pm 1}, w^{\pm 1}]]$, the identity must hold
coefficientwise for coefficients of $z^{-n} w^{m}$. Hence, if use
the notation \eqref{notation'} and \eqref{notation''}, we have,
for all $e, e' \in \cE$, $e_{*}, e'_{*} \in \cE_{*}$ and $m,n \in
{\mathbb Z}^{d}$,
\begin{multline*}
     \sum_{k=1}^{d} \left\langle [K_{k}]_{n}(z) \begin{bmatrix}
    e_{*} \\ e \end{bmatrix},\, [K_{k}]_{m}(z) \begin{bmatrix}
    e'_{*} \\ e' \end{bmatrix} \right \rangle_{\cH(K_{k})} \\
    +
    \left \langle [K_{\cF}]_{n}(z) \begin{bmatrix} e_* \\ e
\end{bmatrix}, [K_{\cF}]_{m}(z) \begin{bmatrix} e'_{*} \\ e'
\end{bmatrix} \right\rangle_{\cH(K_{\cF})} \\
 = \sum_{k=1}^{d} \left \langle [K'_{k}]_{n}(z)
\begin{bmatrix} e_{*} \\ e \end{bmatrix}, \,
    [ K'_{k}]_{m}(z) \begin{bmatrix} e'_{*} \\ e' \end{bmatrix}
    \right \rangle_{\cH(K'_{k})}\\
    + \left\langle [K_{\cF_{*}}]_{n}(z)
    \begin{bmatrix} e_{*} \\ e \end{bmatrix}, [K_{\cF_{*}}]_{m}(z)
    \begin{bmatrix} e'_{*} \\ e' \end{bmatrix}
    \right\rangle_{\cH(K_{\cF_{*}})}.
\end{multline*}
If we now define subspaces $\widehat {\mathcal D}$ and $\widehat
{\mathcal R}$ according to
  \begin{align*}
      & \widehat  {\mathcal D} =
\overline{\operatorname{span}}\left\{ \begin{bmatrix}
     [ K_{1}]_{m}(z) \\ \vdots \\ [ K_{d}]_{m}(z) \\
     [K_{\cF}]_{m}(z) \end{bmatrix} \begin{bmatrix} e_{*} \\ e
\end{bmatrix}
      \colon e_{*} \in \cE_{*}, \, e \in \cE, \, m \in {\mathbb
Z}^{d} \right\}
       \subset \begin{bmatrix} \cH(K_{1}) \\ \vdots \\ \cH(K_{d}) \\
      \cH(K_{\cF}) \end{bmatrix}, \\
      &\widehat  {\mathcal R} = \overline{\operatorname{span}} \left
\{
      \begin{bmatrix} [K'_{1}]_{m}(z) \\ \vdots \\
      [K'_{d}]_{m}(z)
      \\ [K_{\cF_{*}}]_{m}(z) \end{bmatrix}
     \begin{bmatrix}
      e_{*} \\ e \end{bmatrix} \colon e_{*} \in \cE_{*},\, e \in
      \cE, \, m \in {\mathbb Z}^{d} \right\}
     \subset \begin{bmatrix} \cH(K'_{1}) \\ \vdots
      \\ \cH(K'_{d}) \\ \cH(K_{\cF_{*}}) \end{bmatrix},
 \end{align*}
 we see that the map $\widehat V$ given by
  \begin{equation*}
     \widehat V \colon \begin{bmatrix}[ K_{1}]_{m}(z) \\ \vdots \\
   [ K_{d}]_{m}(z) \\ [K_{\cF}]_{m}(z) \end{bmatrix}
    \begin{bmatrix} e_{*} \\ e \end{bmatrix} \mapsto
  \begin{bmatrix}[ K'_{1}]_{m}(z)\\ \vdots \\ [K'_{d}]_{m}(z)
\\
     [ K_{\cF_{*}}]_{m}(z) \end{bmatrix}
\begin{bmatrix} e_{*} \\ e \end{bmatrix}
  \end{equation*}
  extends by linearity and continuity to define a unitary operator
  from $\widehat {\mathcal D}$ onto $\widehat {\mathcal R}$.
Alternatively, to see
  that $\widehat V \colon \widehat {\mathcal D} \to \widehat
{\mathcal R}$ extends to a unitary,
  one can note that $\widehat V$ extends to
  $$
  \widehat  V = ({\mathbf s}')^{*}
  {\mathbf s}|_{\widehat{\mathcal D}}
  $$
  where
  $$
  {\mathbf s}  \colon \begin{bmatrix}
 \widehat \bigoplus_{k=1}^{d} \cH(K_{k}) \\ \cH(K_{{\mathcal F}})
\end{bmatrix} \to \cH({\mathbf K}), \qquad
{\mathbf s}' \colon \begin{bmatrix}
  \widehat \bigoplus_{k=1}^{d} \cH(K'_{k})  \\ \cH(K_{{\mathcal
F_{*}}})
\end{bmatrix} \to \cH({\mathbf K})
  $$
  are the coisometric sum maps as in Proposition \ref{P:overlapping}
  with respective initial spaces $\widehat {\mathcal D}$  and
  $\widehat {\mathcal R}$.
  From Proposition \ref{P:overlapping} we know that we
  can identify the defect spaces as overlapping spaces:
  \begin{align*}
    &  \cL : = \begin{bmatrix} \widehat
      \bigoplus_{k=1}^{d} \cH(K_{k}) \\ \cH(K_{\cF}) \end{bmatrix}
      \ominus \widehat {\mathcal D}  = \begin{bmatrix} \cL(K_{1},
\dots,
      K_{d}) \\ 0 \end{bmatrix}, \\
     & \cL': = \begin{bmatrix} \widehat
      \bigoplus_{k=1}^{d} \cH(K'_{k}) \\ \cH(K_{\cF_{*}})
\end{bmatrix}
      \ominus \widehat  {\mathcal R} = \begin{bmatrix} \cL(K'_{1},
\dots,
      K'_{d}) \\ 0 \end{bmatrix}.
  \end{align*}
  Using the identification maps explained in Proposition
  \ref{P:overlapping}, we
  have the following unitary identifications,
  where ${\mathbf K}$ denotes the common value of the two expressions
  in \eqref{boldK}:
  \begin{align*}
    &  \left(\widehat\bigoplus_{k=1}^{d} \cH(K_{k})  \widehat \oplus
    \cH(K_{\cF})\right)  \widehat \oplus
      \cL' \cong \left(\cH({\mathbf K}) \widehat \oplus
      \cL\right) \widehat \oplus  \cL' \\
    & \qquad   \cong \left(\cH({\mathbf K})\widehat \oplus
    \cL'\right) \widehat \oplus  \cL \\
    & \qquad \cong \left(\widehat \bigoplus_{k=1}^{d} \cH(K'_{k})
    \widehat\oplus \cH(K_{\cF_{*}}) \right) \widehat \oplus \cL.
   \end{align*}
   The composition of these identification operators gives us a
   unitary map
   $$
   \widehat U_{0} \colon \begin{bmatrix} \widehat \bigoplus_{k=1}^{d}
\cH(K_{k})
   \\ \cH(K_{\cF}) \\ \cL' \end{bmatrix} \to
   \begin{bmatrix} \widehat \bigoplus_{k=1}^{d} \cH(K'_{k}) \\
       \cH(K_{\cF_{*}}) \\ \cL \end{bmatrix}
   $$
   such that
   \begin{align*}
  &  \widehat  U_{0} \begin{bmatrix}\widehat  d \\ 0 \end{bmatrix} =
   \begin{bmatrix} \widehat V\widehat  d
    \\ 0 \end{bmatrix} \text{ if }\widehat  d \in \widehat {\mathcal
    D}, \\
    & \widehat U_{0} \colon \begin{bmatrix} \ell \\ 0 \end{bmatrix}
    \mapsto \begin{bmatrix} 0 \\ \ell \end{bmatrix} \text{ if } \ell
    \in {\mathcal L}, \\
    & \widehat U_{0} \colon \begin{bmatrix} 0 \\ \ell' \end{bmatrix}
    \mapsto \begin{bmatrix} \ell' \\ 0 \end{bmatrix} \text{ if }
    \ell' \in \cL'.
    \end{align*}
    In particular, $\widehat U_{0}$ can be viewed as a unitary
extension of
    $\widehat V$.
      Note that $e \mapsto \sbm{S(z) \\ I } e$ is an isometry from
$\cE$
  onto $\cH(K_{\cF})$ and that $e_{*} \mapsto \sbm{ I \\ S(z)^{*}}
  e_{*}$ is an isometry from $\cE_{*}$ onto $\cH(K_{\cF_{*}})$.
  In addition, multiplication by $\sbm{ z_{1}^{-1} & & \\ & \ddots &
  \\ & & z_{d}^{-1}}$ maps $\widehat \bigoplus_{k=1}^{d} \cH(K'_{k})$
  unitarily onto $\widehat \bigoplus_{k+1}^{d} \cH(K_{k})$.
  In addition one can check directly the following formulas:
  \begin{align*}
& \left[ K_{k}(\cdot, w) w_{k}^{-1} \right]_{m}(z) = [ K_{k}]_{m - \be_{k}}(z), \\
& \left[ \begin{bmatrix} I \\ S(\cdot)^{*} \end{bmatrix}
  \begin{bmatrix} S(w)^{*} & I \end{bmatrix} \right]_{m}(z) =
\begin{bmatrix} I \\ S(z)^{*} \end{bmatrix} \begin{bmatrix} S_{m}^{*} & \delta_{m,0} I \end{bmatrix}, \\
& \left[ \begin{bmatrix} I \\ S(z)^{*} \end{bmatrix}
\begin{bmatrix} I & S(w) \end{bmatrix} \right]_{m}(z) =
\begin{bmatrix} I \\ S(z)^{*}
\end{bmatrix} \begin{bmatrix} \delta_{m,0} I & S_{-m} \end{bmatrix}.
\end{align*}
If we incorporate the identifications
$$
\cE \sim \cH(K_{\cF}), \quad \cE_{*} \sim \cH(K_{\cF_{*}}), \quad
\cH(K_{k}') \sim \cH(K_{k})
$$
described above with these observations,  then the modified
versions
  ${\mathcal D}$ and ${\mathcal R}$ of $\widehat {\mathcal
  D}$ and $\widehat {\mathcal R}$ respectively are given by
   \begin{align}
     {\mathcal D} & =
     \overline{\operatorname{span}} \left \{
          \begin{bmatrix} [K_{1}]_{m}(z) \\ \vdots \\
         [ K_{d}]_{m}(z)
          \\ \begin{bmatrix} S_{m}^{*} & \delta_{m,0}I
          \end{bmatrix} \end{bmatrix} \begin{bmatrix}
          e_{*} \\ e \end{bmatrix} \colon e_{*} \in \cE_{*},\, e \in
          \cE, \, m \in {\mathbb Z}^{d} \right\} \subset \begin{bmatrix}
\cH(K_{1}) \\ \vdots
          \\ \cH(K_{d}) \\ \cE \end{bmatrix}  \label{defD} \\
 {\mathcal R} &  = \overline{\operatorname{span}} \left \{
    \begin{bmatrix} [K_{1}]_{m- \be_{1}}(z) \\ \vdots \\
       [ K_{d}]_{m- \be_{d}}(z)
        \\ \begin{bmatrix}  \delta_{m,0}I & S_{-m} \end{bmatrix} \end{bmatrix} \begin{bmatrix}
        e_{*} \\ e \end{bmatrix} \colon e_{*} \in \cE_{*},\, e \in
        \cE, \, m \in {\mathbb Z}^{d} \right\} \subset \begin{bmatrix}
\cH(K_{1}) \\ \vdots
        \\ \cH(K_{d}) \\ \cE_{*} \end{bmatrix}
        \label{defR}
   \end{align}
   and the modified version $U_{0}$ of $\widehat U_{0}$,
   $$
   U_{0} \colon \begin{bmatrix} \widehat \bigoplus_{k=1}^{d}
   \cH(K_{k}) \\ \cE \\ \cL' \end{bmatrix} \to \begin{bmatrix}
   \widehat \bigoplus_{k=1}^{d} \cH(K_{k}) \\ \cE_{*}\\  \cL
\end{bmatrix},
   $$
   is an extension of the unitary $ V
    \colon {\mathcal D} \to {\mathcal R}$
   given by
   \begin{equation}  \label{defV}
       V \colon \begin{bmatrix} [K_{1}]_{m}(z) \\ \vdots \\
         [ K_{d}]_{m}(z)
          \\ \begin{bmatrix} S_{m}^{*} & \delta_{m,0}I
          \end{bmatrix} \end{bmatrix} \begin{bmatrix}
          e_{*} \\ e \end{bmatrix} \colon e_{*}  \mapsto
    \begin{bmatrix} [K_{1}]_{m- \be_{1}}(z) \\ \vdots \\
       [ K_{d}]_{m- \be_{d}}(z)
        \\ \begin{bmatrix}  \delta_{m,0}I & S_{-m} \end{bmatrix} \end{bmatrix} \begin{bmatrix}
        e_{*} \\ e \end{bmatrix}.
\end{equation}

Let us write $U_{0}$ in matrix form
 $$
 U_{0} = \begin{bmatrix} U_{0,11} & U_{0,12} &
 U_{0, 13} \\ U_{0,21} & U_{0,22} & U_{0,23} \\
 U_{0,31} & U_{0,32} & 0 \end{bmatrix}
 \colon \begin{bmatrix}  \widehat \bigoplus_{k=1}^{d}
 \cH(K_{k}) \\ \cE \\ \cL' \end{bmatrix} \to
 \begin{bmatrix} \widehat \bigoplus_{k=1}^{d} \cH(K_{k}) \\ \cE_{*}
\\ \cL
     \end{bmatrix}.
 $$
Then we have the following result.

  \begin{theorem}  \label{T:parametrize}
      Suppose that $S(z) \in \cL(\cE, \cE_{*})[[z]]$ is in
       $\mathcal{S A}(\cE, \cE_{*})$ and that we are given
      an augmented Agler decomposition $K_{1}(z,w), \dots,
      K_{d}(z,w)$ for $S(z)$ as in \eqref{augAglerdecom}. Then:
      \begin{enumerate}
        \item Any GR-unitary colligation $U$ which
      gives a realization of $S$ compatible with the given augmented
      Agler decomposition $K_{1}(z,w), \dots, K_{d}(z,w)$ is
unitarily equivalent to a unitary extension
      $$ U = \begin{bmatrix} A_{11} & A_{12} & B_{1} \\ A_{21} &
      A_{22} & B_{2} \\ C_{12} & C_{2} & D \end{bmatrix}
      \colon \begin{bmatrix} \widehat  \bigoplus_{k=1}^{d} \cH(K_{k})
\\
      \widehat \bigoplus_{k=1}^{d} \widetilde \cH_{k} \\  \cE
      \end{bmatrix} \to
      \begin{bmatrix} \widehat \bigoplus_{k=1}^{d}  \cH(K_{k})
      \\ \widehat \bigoplus_{k=1}^{d} \widetilde \cH_{k}  \\ \cE_{*}
\end{bmatrix}
      $$
      of the partially defined unitary operator $V$
      \eqref{defV}.
        \item       Such unitary extensions $U$ are parametrized by
      GR-unitary colligations of the form
      $$ \widetilde U  =
      \begin{bmatrix} \widetilde U_{11} & \widetilde U_{12} \\
      \widetilde U_{21} & \widetilde U_{22} \end{bmatrix}
      \colon \begin{bmatrix} \widehat \bigoplus_{k=1}^{d} \widetilde
      \cH_{k} \\ \cL  \end{bmatrix} \to
      \begin{bmatrix}  \widehat \bigoplus_{k=1}^{d} \widetilde
\cH_{k} \\
     \cL' \end{bmatrix}
      $$
      via the feedback connection:  the system of equations
      \begin{equation}  \label{FB1}
   U = \begin{bmatrix}  A_{11} & A_{12} & B_{1} \\ A_{21} &
      A_{22} &
      B_{2} \\ C_{1} & C_{2} & D \end{bmatrix}
    \colon \begin{bmatrix}
      \widehat \bigoplus_{k=1}^{d} h_{k} \\
       \widehat  \bigoplus_{k=1}^{d}  \widetilde h_{k} \\ e
\end{bmatrix}
       \mapsto \begin{bmatrix}  \widehat \bigoplus_{k=1}^{d} h_{k}' \\
      \widehat \bigoplus_{k=1}^{d} \widetilde h_{k}'   \\ e_{*}
\end{bmatrix}
      \end{equation}
      means that there are uniquely determined vectors $\ell \in \cL$
     and $ \ell' \in \cL'$ so
      that the following system of equations is verified:
      \begin{align}
     & U_{0} \colon \begin{bmatrix} \widehat \bigoplus_{k=1}^{d}
     h_{k} \\ e \\\ell' \end{bmatrix}
     \mapsto \begin{bmatrix}\widehat \bigoplus_{k=1}^{d} h_{k}'
     \label{FB2} \\
     e_{*} \\ \ell  \end{bmatrix}, \\
     & \widetilde U \colon \begin{bmatrix}  \widehat \bigoplus_{k=1}^{d}
     \widetilde h_{k} \\  \ell \end{bmatrix} \mapsto
    \begin{bmatrix} \widehat \bigoplus_{k=1}^{d} \widetilde h_{k}' \\
\ell' \end{bmatrix}.
        \label{FB3}
     \end{align}
     Explicitly, $U = \sbm{ A_{11} & A_{12} & B_{1} \\ A_{21} &
     A_{22} & B_{2} \\ C_{1} & C_{2} & D }$ is given by
     \begin{align}
    & \begin{bmatrix} A_{11} & A_{12} & B_{1} \\ A_{21} & A_{22} &
         B_{2} \\ C_{1} & C_{2} & D \end{bmatrix}
      = \cF_{\ell}
         \left( \left[ \begin{array}{ccc|cc}
         U_{0,11} & 0 & U_{0,12} & 0 & U_{0,13} \\ 0 & 0 & 0 & I &
         0 \\ U_{0,21} & 0 & U_{0,22} & 0 & U_{0,23} \\ \hline
         0 & I & 0 & 0 & 0 \\ U_{0,31} & 0 & U_{0,32} & 0 & 0
     \end{array} \right], \,
     \begin{bmatrix} \widetilde U_{11} &
     \widetilde U_{12} \\ \widetilde U_{21} & \widetilde U_{22}
     \end{bmatrix} \right) \notag \\
     & \qquad = \begin{bmatrix}  U_{0,11} + U_{0,13} \widetilde U_{22}
     U_{0,31}  & U_{0,13} \widetilde U_{21} & U_{0,12} + U_{0,13}
     \widetilde U_{22} U_{0,32} \\
     \widetilde U_{12} U_{0,31} & \widetilde U_{11}  & \widetilde
     U_{12} U_{0,32} \\
     U_{0,21} + U_{0,23} \widetilde U_{22} U_{0,31} & U_{0,23}
     \widetilde U_{21} & U_{0,22} + U_{0,23} \widetilde U_{22}
     U_{0,32} \end{bmatrix}.
     \label{col-param}
  \end{align}
  Here we use the general notation
  $$
  \cF_{\ell}\left( \sbm{ \bf A & \bf B \\ \bf C & \bf D}, X \right) =
\bf A + \bf B (I -
  X\bf D)^{-1} X \bf C
  $$
  for the result of the lower Redheffer linear-fractional map induced
  by $\sbm{\bf A & \bf B \\ \bf C & \bf D}$ acting on $X$.
  \item The colligation $\widetilde{U}$ is closely connected if and
only if there is no reducing subspace of $U$ (that is, a subspace
that $U$ maps onto itself) contained in
$\widetilde{\mathcal{H}}=\widehat\bigoplus_{k=1}^d\widetilde{\mathcal{H}}_k$
and invariant under the projections $P_k$, $k=1,\ldots,d$.
        \item $\widetilde{U}$ is unitarily equivalent to
$\widetilde{U}'$ (as colligations) if and only if $U$ is unitarily
equivalent to $U'$ (as colligations) by a unitary equivalence
which is the identity on
$\widehat\bigoplus_{k=1}^d\mathcal{H}(K)$.
      \end{enumerate}

     \end{theorem}

     \begin{proof}
    This result is the reproducing-kernel version of
     a result essentially contained in
     \cite[Theorem 5.1]{BBQ2}---one must make the identification
between formal
     power series kernels $K(z,w)$ and sesquianalytic kernels
     $K(z,w)$. In addition one should use the factorization
$K_{k}(z,w) = H_{k}(z)
     H_{k}(w)^{*}$ for an $H_{k}(z) = \sum_{n \in \bbZ^{d}_{+}}
H_{k,n}
     z^{n}$ where $H_{k,n} \in \cL(\cH_{k}, \cF_{*} \widehat \oplus
\cF)$
     and the fact that (under the assumption that multiplication by
     $H_{k}(z)$ acting on $\cH_{k}$ is injective) the map $x \mapsto
H_{k}(z)$
     is an isometric isomorphism of $\cH_{k}$ onto $\cH(K_{k})$.
     In addition, one should specialize the result in \cite{BBQ2} to
the case
     where (1) ${\mathbf Q}(z) = \sbm{ z_{1} & & \\ & \ddots & \\ & &
     z_{d} }$ so that the associated domain ${\mathcal D}_{{\mathbf
     Q}}$ is the unit polydisk $\bbD^{d}$ and where (2) the
     interpolation problem is taken to be the full-matrix-value
     (rather than tangential) interpolation problem with set of
     prescribed interpolation nodes equal to the whole polydisk
$\bbD^{d}$.
     The result also appears in \cite[Theorem 2.1 (3) and Theorem
6.1]{BT}
     in more implicit form.
     In addition, we are here parametrizing unitary colligations
     rather than the associated characteristic functions; formulas of
     the type \eqref{col-param} for the result of the lower feedback
     loop of the type \eqref{FB1}, \eqref{FB2}, \eqref{FB3}
     have been known in the control literature
     for some time (see e.g.~\cite[Section 2.1]{IS}).
     \end{proof}

\begin{remark}\label{R:goal}
Our ultimate goal here would be to obtain a complete description
of closely connected GR-unitary realizations of $S$ which are
compatible with a given augmented Agler decomposition up to a
unitary equivalence. The description in Theorem
\ref{T:parametrize} falls short of achieving this goal in two
ways:
\begin{itemize}
  \item letting the load $\widetilde{U}$ run over all closely
connected GR-unitary colligations with input space $\mathcal{L}$
and output space $\mathcal{L}'$ may give us some
non-closely-connected GR-unitary colligations $U$;
  \item unitarily non-equivalent loads $\widetilde{U}$ can give
unitarily equivalent GR-unitary colligations $U$.
\end{itemize}
\end{remark}

     What is added here (as opposed to what already appears in
     \cite{BT, BBQ2}) to the results on parametrization of unitary
     colligations realizing a given augmented Agler decomposition is
     the result coming out of Proposition \ref{P:overlapping}
     that the defect spaces
     $$
     \begin{bmatrix} \widehat \bigoplus_{k=1}^{d} \cH(K_{k}) \\ \cE
\end{bmatrix}
     \ominus {\mathcal D}, \qquad
     \begin{bmatrix} \widehat \bigoplus_{k=1}^{d} \cH(K_{k}) \\
         \cE_{*} \end{bmatrix} \ominus {\mathcal R}
     $$
      can be identified with the overlapping
     spaces  $\cL(K_{1}, \dots, K_{d})$ and $\cL(K'_{1}, \dots,
    K'_{d})$ respectively. As a result we have the following
     corollary.

     \begin{corollary}  \label{C:unique-real}
Suppose that the augmented Agler decomposition for the
     Schur-Agler-class function $S(z)$ is strictly closely connected,
i.e., such that both the corresponding overlapping spaces
     $$
     \cL(K_{1}, \dots, K_{d}) \text{ and }
     \cL(K'_{1}, \dots,  K'_{d})
     $$
     are trivial.  Then there
     is a unique, up to a unitary equivalence of colligations,
GR-unitary closely-connected colligation realizing
     $S$ compatible with this given augmented Agler decomposition.
Moreover, this colligation is necessarily strictly closely
connected, and if, in addition, the augmented Agler decomposition
is minimal, then this colligation is scattering minimal (see
Remark \ref{R:min-summary}).
     \end{corollary}

     \begin{remark}  \label{R:weaklyunitary}
 Canonical functional models for a Schur-Agler-class function with
 are also studied in \cite{BB-Gohberg} using the
 classical (rather than formal) reproducing kernel Hilbert spaces
 associated with the kernels in an augmented Agler decomposition.
 In these models the state space is taken to be $\widehat
 \bigoplus_{k=1}^{d} \cH(K_{k})$ (no overlapping spaces).  To make this
 possible, the class of colligations considered is ``weakly unitary''
 (whereby the colligation operator is required only to be contractive
 with it and its adjoint isometric only when restricted to certain
 canonical subspaces).  Then it is shown that any weakly unitary
 ``closely connected'' realization is unitarily equivalent to
 such a (two-component) canonical functional-model realization.
The meaning of  ``closely connected'' as defined in
\cite{BB-Gohberg} is closely related to but not quite the same as
``strictly closely connected'' or ``shifted strictly closely
connected'' as given in Definition \ref{D:minimal} above: namely,
the map $\widetilde \Pi_{U}^{dBR}$ given by
$$
    \widetilde \Pi_{U}^{dBR} \colon h \mapsto \begin{bmatrix}
    C(I - Z_{\rm diag}(z) A)^{-1} \\ B^{*} (I - Z_{\rm diag}(z)
    A^{*})^{-1} \end{bmatrix} h
$$
should be injective.  Also the work in \cite{BB-Gohberg} makes no
contact with multievolution scattering systems and the associated
scattering geometry.

\end{remark}

       \section{Functional-model
     colligations and scattering systems realizing a given
     augmented Agler decomposition}   \label{S:Aglerdecom-scatsys}

 Suppose that we are given a Schur-class function $S \in \cS(\cE,
 \cE_{*})$ with given augmented Agler decomposition $\{K_{k} \colon
 k=1, \dots, d\}$ (so \eqref{augAglerdecom} holds).  Proposition
\ref{P:dBRFRKHSs} presents a
 functional-model  minimal scattering system having $S$ as its
 scattering matrix.  As a consequence of Proposition
 \ref{P:scatspacecc}, we see that this model has the extra geometric
structure
 (i.e., the existence of auxiliary subspaces $\cL_{k}, \cM_{k},
 \cM_{*k}$ satisfying conditions \eqref{Omegadecom}, \eqref{CDP},
 \eqref{slim} and \eqref{slim*}) to lead to a functional-model
 GR-unitary colligation compatible with the given augmented Agler
 decomposition if and only if the augmented Agler decomposition is
 {\em minimal}, and then this extra structure (and the associated
 compatible GR-unitary colligation) can be presented in functional
 form as in Theorem \ref{T:minfuncmodel}.  Another situation where a
 (not necessarily minimal) multievolution scattering system with
 embedded GR-unitary colligation compatible with a preassigned
augmented
 Agler decomposition can be presented in functional-model form is the
 following.  The result is immediate from Remark \ref{R:Schaffer} and
Corollary \ref{C:unique-real}.

 \begin{theorem} \label{T:scc-real}  Suppose that we are given
     an augmented Agler decomposition for the Schur-class function $S
     \in \cS(\cE, \cE_{*})$ which is strictly closely connected (see
     Definition \ref{D:augAglerdecom-scc}).  Then there is an
     essentially unique closely connected GR-unitary colligation
     compatible with $\{K_{k} \colon k = 1, \dots, d\}$ and an
     essentially unique (not necessarily minimal) associated
scattering system
     ${\mathfrak S}(\Sigma(U))$ which can be constructed in
     functional model form as follows.
     Note that the subspaces $\cD$
     and $\cR$ given by \eqref{defD} and \eqref{defR} are equal to
     all of $\widehat \bigoplus_{k=1}^{d} \cH(K_{k}) \oplus \cE$ and
     $\widehat \bigoplus_{k=1}^{d} \cH(K_{k}) \oplus \cE_{*}$
respectively and
     the operator $V$ given by \eqref{defV} is unitary.
     Define an operator
     $$
     U = \begin{bmatrix} A & B \\ C & D \end{bmatrix} \colon
     \begin{bmatrix}  \widehat \bigoplus_{k=1}^{d} \cH(K_{k})  \\ \cE
     \end{bmatrix} \to \begin{bmatrix} \widehat \bigoplus_{k=1}^{d}
     \cH(K_{k}) \\ \cE_{*} \end{bmatrix}
     $$
     implicitly by $U = V$ where $V$ is as in \eqref{defV}.
     Construct a scattering system, denoted as ${\mathfrak
     S}_{\Omega^{B}}(\Sigma(U))$, with scattering space
     $$
      \cV^{\Omega^B}: = \widehat \bigoplus_{k=1}^{d} \widehat
\bigoplus_{n \in
      \Xi} z^{n} \cH(K_{k}
     $$
     as in Remark \ref{R:Schaffer}.  Then ${\mathfrak S}$ is a
     multievolution scattering system with embedded GR-unitary
     colligation $U$ compatible with the given augmented Agler
     decomposition $\{K_{k} \colon k=1, \dots, d\}$.
     \end{theorem}

     Note that the main difference between Theorem
\ref{T:minfuncmodel}
     and Theorem \ref{T:scc-real} is that the construction
     of the scattering space in Theorem \ref{T:minfuncmodel}
     involved internal orthogonal direct sums while the construction
     in Theorem \ref{T:scc-real} necessarily calls for
     external direct sums.  We note also that the orthogonal
     complement of the minimal-scattering subspace
     $\cK_{\text{scat-min}}^{\perp}$ in Theorem
     \ref{T:scc-real} can be identified with the
     overlapping space $\boldsymbol{\cL}_{\cV^{\Omega^{B}}}$.

    For the general case, by Theorem \ref{T:parametrize} we know
    that the amount of information determined by a given augmented
    Agler decomposition $\{K_{k} \colon k = 1, \dots, d\}$ in the
    construction of a compatible GR-unitary colligation is precisely
    the following: {\em we may assume that $\cH(K_{k}) \subset
\cL_{k}$ and that
    $U|_{\cD}$ agrees with $V$ given by \eqref{defV}.}  Following the
    construction in Remark \ref{R:Schaffer}, this translates to the
    following concerning the construction of the associated
    multievolution scattering system: {\em  we know that $\cH(z^{n}
    K_{k}(z,w) w^{-n})=z^n\cH(K_{k}(z,w))$ may be assumed to be a
subspace of $\cU^{n}
    \cL_{k}$ for each $n \in \Xi$ and, for each $k = 1, \dots, d$, we
    know how each $\cU_{k}^{*}$ is defined on }
    $\cH(K_{\cW^{S, \Omega}_{dBR}*})
     \widehat \oplus  \Big(\widehat \bigoplus_{k=1}^{d} \bigoplus_{n
    \in \Xi} z^{n} \cD_{0}\Big)\widehat  \oplus \cH(K_{\cW^{S,
\Omega}_{dBR}})$
    (where we write $\cD = \cD_{0} \oplus \cF$).
    The problem of constructing a multievolution scattering system
    compatible with the given augmented Agler decomposition then
    reduces to the following: {\em find a simultaneous commuting
unitary extension of
    partially defined isometries ${\mathcal  V} = ({\mathcal V}_{1},
\dots, {\mathcal V}_{d})$
    (defined on  spaces in functional-model form) while maintaining
    the additional geometric structure \eqref{Omegadecom},
    \eqref{CDP}, \eqref{slim}, \eqref{slim*}.}
    Remarkably,
    due to the constructions in Section \ref{S:scat-col} (i.e., the
    general trajectory-space construction or, more concretely, the
Sch\"affer-matrix
    construction in Remark \ref{R:Schaffer}), this a priori hard
    problem comes down to finding a single unitary extension $U$ of
the partially
    defined unitary colligation $V$; this latter problem in turn can
    be handled as in Theorem \ref{T:parametrize}.

\end{document}